\documentclass[aos]{imsart}

\RequirePackage{amsthm,amsmath,amsfonts,amssymb}
\RequirePackage[authoryear]{natbib}

\usepackage{xr-hyper}
\usepackage{nicefrac}
\usepackage{breakurl}

\RequirePackage[colorlinks,citecolor=blue,filecolor=black,urlcolor=blue]{hyperref}
\startlocaldefs

\usepackage{macros}

\usepackage{mathrsfs}

\usepackage{breakurl}

\usepackage{url}
\usepackage{graphicx}

\usepackage{xcolor}

\usepackage{enumitem}
\usepackage{mathtools}
\usepackage{thmtools}

\newlist{thmlist}{enumerate}{1}
\setlist[thmlist]{label=(\roman{thmlisti}), ref=\thetheorem(\roman{thmlisti})}

\declaretheorem[
    name=Theorem,
    Refname={Theorem,Theorems},
    ]{theorem}
 
\declaretheorem[
    name=Theorem,
    Refname={Theorem, Theorems},
   numbered=no
   ]{theorem*} 
 
\declaretheorem[
    name=Lemma,
    Refname={Lemma,Lemmas},
   sibling=theorem
   ]{lemma}

\declaretheorem[
    name=Proposition,
    Refname={Proposition,Propositions},
	sibling=theorem
	]{proposition}

\declaretheorem[
    name=Corollary,
    Refname={Corollary,Corollaries},
	sibling=theorem
	]{corollary}
\newtheorem{remark}[theorem]{Remark}

\newtheorem{definition}{Definition}

\usepackage{nameref}
\usepackage[capitalize]{cleveref}

\Crefname{theorem}{Theorem}{Theorems}
\Crefname{lemma}{Lemma}{Lemmas}

\addtotheorempostheadhook[theorem]{\crefalias{thmlisti}{theorem}}
\addtotheorempostheadhook[lemma]{\crefalias{thmlisti}{lemma}}

\usepackage{amsmath}
\makeatletter
\let\save@mathaccent\mathaccent
\newcommand*\if@single[3]{%
  \setbox0\hbox{${\mathaccent"0362{#1}}^H$}%
  \setbox2\hbox{${\mathaccent"0362{\kern0pt#1}}^H$}%
  \ifdim\ht0=\ht2 #3\else #2\fi
  }
\newcommand*\rel@kern[1]{\kern#1\dimexpr\macc@kerna}
\newcommand*\widebar[1]{\@ifnextchar^{{\wide@bar{#1}{0}}}{\wide@bar{#1}{1}}}
\newcommand*\wide@bar[2]{\if@single{#1}{\wide@bar@{#1}{#2}{1}}{\wide@bar@{#1}{#2}{2}}}
\newcommand*\wide@bar@[3]{%
  \begingroup
  \def\mathaccent##1##2{%
    \let\mathaccent\save@mathaccent
    \if#32 \let\macc@nucleus\first@char \fi
    \setbox\z@\hbox{$\macc@style{\macc@nucleus}_{}$}%
    \setbox\tw@\hbox{$\macc@style{\macc@nucleus}{}_{}$}%
    \dimen@\wd\tw@
    \advance\dimen@-\wd\z@
    \divide\dimen@ 3
    \@tempdima\wd\tw@
    \advance\@tempdima-\scriptspace
    \divide\@tempdima 10
    \advance\dimen@-\@tempdima
    \ifdim\dimen@>\z@ \dimen@0pt\fi
    \rel@kern{0.6}\kern-\dimen@
    \if#31
      \overline{\rel@kern{-0.6}\kern\dimen@\macc@nucleus\rel@kern{0.4}\kern\dimen@}%
      \advance\dimen@0.4\dimexpr\macc@kerna
      \let\final@kern#2%
      \ifdim\dimen@<\z@ \let\final@kern1\fi
      \if\final@kern1 \kern-\dimen@\fi
    \else
      \overline{\rel@kern{-0.6}\kern\dimen@#1}%
    \fi
  }%
  \macc@depth\@ne
  \let\math@bgroup\@empty \let\math@egroup\macc@set@skewchar
  \mathsurround\z@ \frozen@everymath{\mathgroup\macc@group\relax}%
  \macc@set@skewchar\relax
  \let\mathaccentV\macc@nested@a
  \if#31
    \macc@nested@a\relax111{#1}%
  \else
    \def\gobble@till@marker##1\endmarker{}%
    \futurelet\first@char\gobble@till@marker#1\endmarker
    \ifcat\noexpand\first@char A\else
      \def\first@char{}%
    \fi
    \macc@nested@a\relax111{\first@char}%
  \fi
  \endgroup
}
\makeatother

\endlocaldefs

\begin{document}

\begin{frontmatter}
\title{Plugin Estimation of Smooth Optimal Transport Maps}
\runtitle{Plugin Estimation of Smooth Optimal Transport Maps}

\begin{aug}
\author[A]{\fnms{Tudor} \snm{Manole}\ead[label=e1,mark]{tmanole@andrew.cmu.edu}},
\author[A,B]{\fnms{Sivaraman} \snm{Balakrishnan}\ead[label=e2,mark]{siva@stat.cmu.edu}},

\author[C]{\fnms{Jonathan} \snm{Niles-Weed}\ead[label=e3,mark]{jnw@cims.nyu.edu}}
\and
\author[A,B]{\fnms{Larry} \snm{Wasserman}\ead[label=e4,mark]{larry@stat.cmu.edu}}
\address[A]{Department of Statistics and Data Science, Carnegie Mellon University \\ \printead{e1,e2,e4}}
\address[B]{Machine Learning Department, Carnegie Mellon University}
\address[C]{Courant Institute of Mathematical Sciences and Center for Data Science, New York University \\ \printead{e3}}
\end{aug}

%
%

\begin{abstract}

We analyze a number of natural estimators for the optimal transport map
between two distributions and show that they are minimax optimal.
We adopt the plugin approach: our estimators are simply optimal couplings between measures derived 
from our observations, appropriately extended so that they define functions on $\bbR^d$.
When the underlying map is assumed to be Lipschitz, we show that computing the optimal coupling
between the empirical measures, and extending it using linear smoothers,
already gives a minimax optimal estimator.
When the underlying map enjoys higher regularity, we show that the optimal coupling between
appropriate nonparametric density estimates yields faster rates.
Our work also  provides new bounds on the risk of corresponding plugin estimators for the quadratic Wasserstein distance, 
and we show how this problem relates to that of estimating optimal transport maps
using stability arguments  for smooth and strongly convex Brenier potentials. 
As an application of our results, we derive  central limit theorems for   plugin estimators of the squared Wasserstein 
distance, which are centered at their population counterpart when the underlying distributions have sufficiently smooth densities. 
In contrast to known central limit theorems for  empirical  estimators, this result
easily lends itself to statistical inference for the quadratic
Wasserstein distance.
\end{abstract}



\end{frontmatter}

%
%
\tableofcontents

\section{Introduction}

Optimal transport maps play a central role in the theory of optimal transport~\citep{rachev1998, 
villani2003, santambrogio2015}, 
and have received many recent methodological applications in 
statistics and machine learning~\citep{kolouri2017, panaretos2019a}. 
Given two distributions $P$ and $Q$ with support contained in a set $\Omega \subseteq \bbR^d$,
an optimal transport map $T_0$ from $P$ to $Q$ 
is any solution to the {\it Monge problem}~\citep{monge1781},
\begin{equation}
\label{eq:monge}
\argmin_{T \in \calT(P,Q)} \int_\Omega \norm{x-T(x)}^2 dP(x),
\end{equation}
where $\calT(P,Q)$ is the set of transport maps between $P$ and $Q$, 
that is, the set of Borel-measurable functions $T: \Omega \to \Omega$ 
such that $T_{\#} P := P(T^{-1}(\cdot)) = Q$. Equivalently, 
we write $T_{\#} P = Q$ whenever $X \sim P$ implies $T(X) \sim Q$. 
As we shall see in Section~\ref{sec:background}, 
a sufficient condition for the Monge problem to admit a 
solution  $T_0$ is for $P$ to be absolutely continuous with respect to the Lebesgue measure.

A wide range of statistical applications involve transforming 
random variables to ensure 
they follow a desired distribution. Optimal transport maps form natural
choices of such transformations when no other canonical choice is available. 
  For instance, optimal transport maps form a useful tool for addressing label shift between train  and test distributions
 in classification problems, 
and have more generally been applied to various
domain adaptation and transfer learning problems~\citep{courty2016, redko2019, rakotomamonjy2021, zhu2021}. 
A large body of recent work has also employed optimal transport
maps for defining notions of multivariate ranks and quantiles~\citep{chernozhukov2017,hallin2021,ghosal2022}, and has applied them
to a variety of nonparametric hypothesis testing problems~\citep{shi2020,deb2021a, deb2021b}. 
We also note their recent uses
in distributional regression~\citep{ghodrati2021}, 
generative modeling~\citep{finlay2020,onken2021}, fairness in machine learning~\citep{gordaliza2019,black2020, delara2021},   
and in a wide range of statistical applications to the sciences~\citep{read1999, wang2011, schiebinger2019, komiske2020}.

An important question arising in many of these applications is that of estimating the optimal transport map between unknown 
distributions, based on independent samples. The aim of this paper is
to develop practical estimators of optimal transport maps achieving near-optimal risk. 
Specifically, given i.i.d. samples $X_1, \dots, X_n \sim P$ and $Y_1, \dots, Y_m \sim Q$, 
we derive estimators $\hat T_{nm}$ which achieve the minimax rate of convergence\footnote{Here and throughout, minimax rate-optimality is tacitly understood up to polylogarithmic factors.}, under the loss function
\begin{equation}
\label{eq:L2P}
\big\|\hat T_{nm} - T_0 \big\|_{L^2(P)}^2 = \int_{\Omega} \big\|\hat T_{nm}(x) - T_0(x)\big\|^2 dP(x).
\end{equation} 
The theoretical study of such estimators was recently initiated by~\cite{hutter2021}, who proved that for any estimator
$\hat T_{nm}$ with $n=m$,  
\begin{equation}
\label{eq:HR_lower_bound}
\sup_{(P,Q)} \bbE 
\big\|\hat T_{nm} - T_0\big\|_{L^2(P)}^2  \gtrsim 
   n^{-\frac {2\alpha}{2(\alpha-1)+d}}\vee \frac 1 {n},
\end{equation}
where the supremum is taken over all pairs of distributions $(P,Q)$ 
admitting densities bounded away from zero over a compact set $\Omega$, 
for which $T_0$ lies in an $\alpha$-H\"older ball for some $\alpha \geq 1$, and satisfies 
a key curvature condition~\ref{assm:curvature} which we define below. The lower bound~\eqref{eq:HR_lower_bound}
is reminiscent of, but generally faster than, the classical $n^{-2\alpha/(2\alpha+d)}$ minimax rate of estimating
an $\alpha$-H\"older continuous nonparametric regression function~\citep{tsybakov2008},
and is shown by~\cite{hutter2021} to be achievable up to a polylogarithmic factor. 
Nevertheless, their estimator is computationally intractable in general dimension, 
and their work leaves open the question of developing practical optimal transport map estimators which  achieve 
comparable risk.

In this paper, we establish the minimax optimality of several natural and intuitive
 estimators of optimal transport maps, 
several of which have already been proposed in the statistical optimal transport 
literature, but have resisted sharp statistical analyses thus far. 
We focus on the following two classes of plugin estimators.

\begin{enumerate}
\item[(i)] \textbf{Empirical Estimators.} When no smoothness assumptions are placed on 
$P$ and $Q$, it is natural to study the plugin estimator based on the empirical
measures 
$$P_n = \frac 1 n \sum_{i=1}^n \delta_{X_i}, \quad \text{and}\quad Q_m = \frac 1 m \sum_{j=1}^m \delta_{Y_j}.$$ 
In the special case $n=m$, there is an optimal transport map $T_{nm}$ from $P_n$ to $Q_m$, and more generally there is an optimal coupling of these measures.
While the in-sample estimator $T_{nm}$
is only defined over the support of $P_n$, we readily obtain  estimators
defined over the entire domain by casting the extension problem as one of
nonparametric regression. We show how linear smoothers
and least-squares estimators can be used to interpolate $T_{nm}$, 
leading to estimators $\widehat T_{nm}$ defined over $\Omega$.
Such estimators are new in the literature to the best of our knowledge, 
and achieve the minimax rate
for estimating Lipschitz optimal transport maps~$T_0$.

\item[(ii)] \textbf{Smooth Estimators.} In order to obtain faster rates of convergence when $P$ and $Q$ admit smooth densities $p$ and $q$, we next
analyze the risk of the unique optimal transport
map  between kernel or wavelet density estimators of $p$ and $q$. 
In contrast to our empirical optimal transport map estimators, 
we show that such smooth plugin estimators
are able to take advantage of additional regularity of the densities $p$ and $q$, and achieve minimax-optimal rates when these densities are H\"older smooth.
\end{enumerate}
While our emphasis is on optimal transport maps, 
an equally important target of estimation is the optimal objective value 
in the Monge problem~\eqref{eq:monge}, which gives rise to the 
squared 2-Wasserstein distance
$W_2^2(P,Q)$ (defined formally in Section~\ref{sec:background}). Our optimal transport map estimators
naturally yield estimators for the Wasserstein distance, and we provide upper
bounds on their risk, and derive limit laws, as a byproduct of our study.  \\[-0.1in]

\noindent {\bf Our Contributions.} The primary contributions of this paper are summarized as follows.
\begin{enumerate}
\item[(i)] In Sections~\ref{sec:one_sample} and~\ref{sec:two_sample}, we develop new stability bounds which relate the risk of plugin transport map estimators to the plugin density estimation risk, as measured in the Wasserstein distance. 
These stability bounds are quite general and enable the analysis of flexible, practical transport map estimators. The risk of density estimation under the Wasserstein distance has been extensively studied~\citep{weed2019a,divol2021}, and our
stability bounds enable us to leverage this past work. Additionally, our stability bounds enable the analysis of plugin estimators of the Wasserstein distance, once again relating the risk in this problem 
to the plugin density estimation risk.

\item[(ii)] We build upon our stability bounds to analyze the risk of empirical, kernel-based and wavelet-based transport map estimators in both the one-sample setup (where the source distribution
is known exactly, and the target distribution is sampled) and the two-sample setup (where both the source and target distributions are sampled). The rates we obtain are minimax optimal. 
For example, suppose that $\widehat T_n$ is the optimal transport map from~$P$ to~$\widehat Q_n$, where $\widehat Q_n$ is a
wavelet-estimator over the domain $[0,1]^d$. 
Then, whenever $P$ and $Q$ admit $(\alpha-1)$-H\"older densities and satisfy several additional conditions, we show that,
\begin{align}
\label{eq:informal_transport}
\mathbb{E} \big\|\widehat T_n - T_0\big\|_{L^2(P)}^2  &\lesssim
\begin{cases}
n^{-\frac{2\alpha}{2(\alpha-1)+d}}, & d \geq 3 \\
(\log n)^2/n, & d = 2 \\
1/n, & d = 1.
\end{cases}
\end{align}
As we explain in Section~\ref{sec:background}, the H\"older smoothness of $T_0$
is typically expected to be of one degree greater than
that of $p$ and $q$, and thus our estimator 
achieves the minimax lower bound~\eqref{eq:HR_lower_bound}
when these densities are $(\alpha-1)$-H\"older smooth, for any $\alpha > 1$\footnote{As discussed in Appendix E of~\cite{hutter2021}, 
the minimax lower bound~\eqref{eq:HR_lower_bound}
also holds under such smoothness conditions on the densities
$p$ and $q$, as opposed to smoothness conditions on  $T_0$.}. 
In the two-sample setting, we develop analogous minimax-optimal analyses, for the empirical plugin estimator~(Propositions~\ref{prop:curvature_empirical}--\ref{prop:transport_least_squares})
as well as for kernel-based and wavelet-based plugin estimators (Theorems~\ref{thm:two_sample_density}--\ref{thm:two_sample_kernel}) 
when $P$ and $Q$ admit H\"{o}lder-smooth densities.
In the latter case, as we discuss further in the sequel, we avoid complications that arise in the optimal transport problem due to boundary effects by working over the $d$-dimensional flat torus. 

\item[(iii)] In each of the above settings, we complement our results with upper bounds on the risk of plugin estimators of the Wasserstein distance. For instance, in the smooth 
setting discussed above, we show that,
\begin{align}
\label{eq:informal_wasserstein}
\mathbb{E} \big| W_2^2(P, \widehat Q_n) - W_2^2(P, Q) \big| &\lesssim \left(\frac{1}{n}\right)^{\frac{2\alpha}{2(\alpha-1)+d}} \vee \frac 1 {\sqrt n}.
\end{align}
We also develop analogous results in the one and two-sample settings, for various empirical and smooth plugin estimators. 

\item[(iv)] We build upon these estimation results to 
address {\it inference} for Wasserstein distances in the high-smoothness regime $2(\alpha+1) > d$. 
We show in Section~\ref{sec:clt}, under regularity conditions, that there exists $\sigma^2 > 0$ such that 
\begin{equation}
\label{eq:informal_clt}
\sqrt n \Big( W_2^2(P, \widehat Q_n) - W_2^2(P, Q)\Big) \rightsquigarrow N(0, \sigma^2), \quad \text{as } n\to \infty.
\end{equation}
We also develop analogous results in the two-sample setting. To the best of our knowledge, these are the first
central limit theorems for a plugin estimator of the Wasserstein distance which is centered
at its population counterpart, for absolutely continuous distributions $P$ and $Q$ in arbitrary dimension. We further show that the 
variance $\sigma^2$ of the limiting distribution can be estimated
using our transport map estimators, leading to an asymptotic confidence interval
for $W_2^2(P, Q)$.

\item[(v)] We also develop the semiparametric efficiency theory for the Wasserstein distance functional.  
In Section~\ref{sec:efficiency}, we derive the efficient influence function of the Wasserstein distance, derive asymptotic local minimax lower bounds, and show that 
our plugin Wasserstein distance estimators are asymptotically efficient in the high-smoothness regime.

\end{enumerate}

\noindent {\bf Related Work.}
The two recent works of \cite{hutter2021} and \cite{gunsilius2021} establish $L^2(P)$ convergence rates
for transport map estimators. \cite{gunsilius2021}
derives upper bounds on the risk
of a  plugin estimator for Brenier potentials,
obtained via kernel density estimation of~$p$ and~$q$. 
This analysis results in suboptimal convergence rates for the optimal transport map $T_0$ itself. 
We show in this work that such plugin estimators do in fact achieve the optimal convergence rate
when the sampling domain is the $d$-dimensional torus.

Building upon a construction of~\cite{delbarrio2020}, a consistent estimator of $T_0$
was obtained by \cite{delara2021}  under mild assumptions, 
by regularizing a piecewise constant approximation of the  empirical optimal transport
map $T_n$. We do not know if quantitative convergence rates can be obtained for their estimator under stronger
assumptions. 
Beyond these works, a wide range of heuristic estimators have been proposed in the 
literature~\citep{perrot2016,nath2020,makkuva2020},
but their theoretical properties remain unknown to the best of our knowledge.

Rates of convergence for the problem of estimating Wasserstein distances have arguably received
more attention than that of estimating optimal transport maps. 
Characterizing the convergence rate of the empirical measure under the Wasserstein distance 
is a classical problem~\citep{dudley1969,boissard2014a, fournier2015,bobkov2019,weed2019,lei2020} which immediately leads to upper bounds on the convergence
rate of the empirical plugin estimator of the Wasserstein distance. 
While such upper bounds are generally unimprovable~\citep{liang2019,niles2022}, they have recently been sharpened 
by~\cite{chizat2020} and \cite{manole2021a}
when $W_2(P,Q)$ is bounded away from zero. We complement these results
by deriving a convergence rate that adapts to the magnitude of $W_2(P,Q)$,
in Sections~\ref{sec:one_sample_empirical} and~\ref{sec:two_sample_empirical}. 
Though the empirical plugin estimator of the Wasserstein distance is minimax optimal up to polylogarithmic
factors under no assumptions on $P$ and $Q$, it becomes suboptimal when $P$ and $Q$  have smooth densities.  
\cite{weed2019a} derive the minimax rate of estimating smooth densities under the Wasserstein distance,
and we build upon their results, together with those of~\cite{divol2021},
to characterize  the risk of our density plugin estimators  (cf. Sections~\ref{sec:one_sample_density},~\ref{sec:two_sample_combined}, and~\ref{sec:smooth_domains}).

Central limit theorems for the empirical quadratic cost 
$W_2^2(P_n, Q_m)$ around its expectation 
have been derived by~\cite{delbarrio2019a} 
under mild conditions on the underlying distributions. As we discuss in Section~\ref{sec:clt}, 
however, the centering sequence $\bbE W_2^2(P_n, Q_m)$ in these results cannot
generally be replaced by its population counterpart $W_2^2(P, Q)$, 
which is a barrier to their use  for statistical inference. 
Key exceptions are obtained when the support of $P$ and $Q$ is 
at most  
three-dimensional~\citep{munk1998, delbarrio1999a,freitag2005,delbarrio2005,delbarrio2019c,manole2022b,hundrieser2022}
or countable~\citep{sommerfeld2018,tameling2019}, in which case non-degenerate limiting distributions
for the process $W_2(P_n, Q_m) - W_2(P,Q)$ are known 
up to suitable scaling. In contrast, our work derives 
central limit theorems with desirable centering  in arbitrary dimension $d \geq 1$,
for  a large class of 
 absolutely continuous
distributions $P$ and $Q$ admitting sufficiently smooth densities.

\vspace{0.1in}
\noindent
{\bf Concurrent Work.} 
During the final stages of preparation of the first version of our manuscript, 
we became aware of the independent work of~\cite{deb2021},  
and of the most recently revised version of the work of~\cite{ghosal2022}. 
These papers bound the risk of certain plugin optimal transport map estimators that are closely related to those in our work.   
In particular, assuming for simplicity that $n=m$, they show that an estimator derived from the empirical plugin optimal transport coupling
achieves the $n^{-\left(\frac 1 2 \wedge \frac 2 d\right)}$ convergence rate under the squared $L^2(P_n)$ loss up to polylogarithmic
factors. 
Our work establishes an analogous result using a distinct proof, but 
further shows that empirical estimators achieve this rate in  squared $L^2(P)$ norm, 
once suitably extended using nonparametric smoothers. We also sharpen this result 
to the rate $n^{-\left(1 \wedge \frac 2 d\right)}$ under additional conditions. \cite{deb2021}
also analyze the convergence rate of plugin estimators based on wavelet and kernel density estimation. 
Their work shows that such estimators can achieve, for instance, 
the  faster rate $n^{-\left(\frac 1 2 \vee \frac{ \alpha}{d + 2(\alpha-1)}\right)}$, when 
the underlying densities lie in a $(\alpha-1)$-H\"older ball for some $\alpha > 1$. 
While this upper bound illustrates an improvement  over empirical
estimators in the presence of smoothness, it scales at a quadratically slower rate than the minimax rate~\eqref{eq:HR_lower_bound}. 
In contrast, our work shows that wavelet density plugin estimators do in fact   achieve
the minimax rate $n^{-\left(1 \vee \frac{ 2\alpha}{d + 2(\alpha-1)}\right)}$
(up to a polylogarithmic factor when $d=2$).
The current version of our manuscript extends this result to kernel density estimators, 
using a significantly different proof strategy than~\cite{deb2021}. Finally, we emphasize that our
sharp analysis of   estimators for the Wasserstein distance
allows us to deduce that their bias 
is of lower order
than their variance when $2(\alpha+1) > d$, which is a key component in our derivation of their limiting distribution. 
Indeed, our results in Section~\ref{sec:inference} on statistical inference 
for the 2-Wasserstein distance cannot be deduced from the works of~\cite{deb2021,ghosal2022}.

\vspace{0.1in}
\noindent 
 {\bf Notation.} 
The Euclidean norm on $\bbR^d$ is denoted $\norm\cdot$,  
and the $\ell_p$ norm of a sequence $(a_n)_{n\geq 1} \subseteq \bbR$ is written 
$\norm{(a_n)_{n\geq 1}}_{\ell_p} = (\sum_{n\geq 1} |a_n|^p)^{1/p}$ for all $1 \leq p \leq \infty$. 
Given a   set $\Omega$, which is either a closed
subset of $\bbR^d$ or   the $d$-dimensional
flat torus $\Omega = \bbT^d := \bbR^d/ \bbZ^d$, and given real numbers $\alpha >  0$, $s \in \bbR\setminus \{0\}$, $1 \leq p,q \leq \infty$, 
the H\"older spaces $\calC^\alpha(\Omega)$,  Besov spaces $\calB_{p,q}^s(\Omega)$, 
 homogeneous Sobolev spaces $\dH{s}(\Omega)$,  inhomogeneous Sobolev spaces $H^s(\Omega)$,
and their respective norms $\norm\cdot_{\calC^\alpha(\Omega)}$, $\norm\cdot_{\calB^s_{p,q}(\Omega)}$, 
$\norm\cdot_{\dH{s}(\Omega)}$, $\norm\cdot_{H^s(\Omega)}$,
are defined in Appendix~\ref{app:smoothness_classes}. We drop the suffix $\Omega$ 
when the underlying space can be understood from context. 
We also define, 
for any $M,\gamma > 0$,   
\begin{align}
\label{eq:holder_ball}
\calC^\alpha(\Omega; M) &:= \left\{f \in \calC^\alpha(\Omega): \|f \|_{\calC^\alpha(\Omega)} \leq M\right\}, \\
\label{eq:holder_ball_bound}
\calC^\alpha(\Omega; M,\gamma) &:= \left\{f \in \calC^\alpha(\Omega): \|f \|_{\calC^\alpha(\Omega)} \leq M,  f \geq 1/\gamma \text{ over } \Omega \right\}.
\end{align}
Furthermore, $\calC^\infty(\Omega)$ denotes the set of 
real-valued functions on $\Omega$ which are differentiable up to any order, and 
$\calC^\infty_c(\Omega)$ denotes the set of 
 functions in $\calC^\infty(\Omega)$ whose support is compactly contained
 in $\Omega$.
 Given a measure space $(\Omega, \calF, \nu)$, $L^p(\nu)$    
denotes the Lebesgue space of order $1 \leq p \leq \infty$, 
endowed with the norm $\norm f_{L^p(\nu)} = (\int_{\Omega} |f(x)|^p d\nu(x))^{1/ p}$, 
for any measurable function $f:\Omega\to\bbR$. We also write 
$L_0^p(\nu) = \{f \in L^p(\nu): \int fd\nu = 0\}.$
When $\nu$ is the Lebesgue measure 
$\calL$   on $\Omega \subseteq \bbR^d$, we write $L^p(\Omega)$ (or $L_0^p(\Omega)$) instead of $L^p(\calL)$
(or $L_0^p(\calL)$). We adopt the same convention when $\Omega \subseteq \bbT^d$, 
in which case, by abuse of notation,~$\calL$ denotes the uniform probability measure over $\bbT^d$.
We often abbreviate $\int f d\calL$ by  $\int f$. 
Given $T:\Omega \to \Omega$, we write by abuse of notation
$\|T\|_{L^2(P)} = (\int \|T(x)\|^2 dP(x))^{1/2}$. For any set $\calX$ and 
$f:\calX \to \bbR$, we write $\|f\|_\infty = \sup_{x\in \calX} |f(x)|$.
The Fourier transform of a function $K \in L^1(\bbR^d)$ is denoted
$\calF[K](\xi) = \int_{\bbR^d} f(x) e^{-2\pi i x^\top\xi}dx$ for all $\xi \in \bbR^d$. 
For any   $B \in \bbN$, the permutation group on $[B] = \{1, \dots, B\}$
is denoted $S_B$. For any $a,b \in \bbR$, let $a\vee b = \max\{a,b\}$, $a\wedge b = \min\{a,b\}$, and
$a_+ = a \vee 0$. Furthermore, let $\lfloor a \rfloor$ and $\lceil a \rceil $ denote
the respective floor and ceiling of $a$. 
The diameter of a set $\Omega \subseteq \bbR^d$ is denoted
$\diam(\Omega) = \sup\{\norm{x-y}:x,y \in\Omega\}$, 
and its interior and closure are respectively denoted
$\Omega^\circ$ and $\widebar \Omega$.
For all $x \in \bbR^d$ and $\epsilon > 0$, $B(x,\epsilon) = \{y \in \bbR^d: \norm{x-y} \leq \epsilon\}$. 
For sequences 
$(a_n)_{n=1}^\infty$ and $(b_n)_{n=1}^\infty$, we write $a_n \lesssim b_n$ if there exists
$C > 0$ such that $a_n \leq C b_n$ for all $n \geq 1$, and
we also write $a_n \asymp b_n$ if $b_n \lesssim a_n \lesssim b_n$.
The constant~$C$ is always permitted to depend on $\Omega$, $d$,    and other problem parameters when they are clear from context. 
We sometimes write
$\lesssim_{c_1, c_2, \dots}$ or $\asymp_{c_1, c_2, \dots}$,
to indicate that the 
suppressed constants depend on  problem
parameters $c_1,c_2,\dots$. 
 
\section{Background on Optimal Transport}
\label{sec:background}
\subsection{The Quadratic Optimal Transport Problem over $\bbR^d$}
We provide a brief background on the optimal transport
problem over $\bbR^d$ with respect
to the squared Euclidean cost function, and direct the reader to~\cite{villani2003,santambrogio2015} for further details. 
To simplify our exposition, we assume  throughout the rest of the paper, except where
otherwise specified, that 
all measures have support contained in a set $\Omega \subseteq \bbR^d$ satisfying the following condition.
\begin{enumerate}[leftmargin=1.65cm,listparindent=-\leftmargin,label=\textbf{(S\arabic*)}]   
\item  \label{assm:supp_global} $\Omega$
is a  compact, convex set  with nonempty interior such that $\Omega \subseteq [0,1]^d$. 
\end{enumerate} 
Notice that once $\Omega$ is assumed compact, the final assumption in condition~\ref{assm:supp_global} can always be guaranteed by rescaling.
Let $\calP(\Omega)$ denote the set of Borel probability measures with support contained in 
$\Omega$, and 
$\calP_{\mathrm{ac}}(\Omega)$ the subset of such measures which are absolutely continuous with respect
to the Lebesgue measure on $\bbR^d$.
As we shall recall in Theorem~\ref{thm:brenier} below,   
for any $P \in \calPac(\Omega)$ and $Q \in \calP(\Omega)$ 
the Monge problem~\eqref{eq:monge} admits a  minimizer $T_0$, which is uniquely defined $P$-almost everywhere.
The Monge problem may, however,  be infeasible
when the absolute continuity condition  on $P$ is removed.
This observation  
motivated \cite{kantorovich1942,kantorovich1948} to develop the following convex relaxation of the Monge problem,
\begin{equation}
\label{eq:kantorovich}
\argmin_{\pi \in \Pi(P,Q)} \int_\Omega \norm{x-y}^2 d\pi(x,y),
\end{equation}
known as the Kantorovich problem, 
where $\Pi(P,Q)$ denotes the set of joint distributions on $\Omega^2$ with marginal distributions $P$
and $Q$, known as couplings of $P$ and $Q$. 
It can be shown in our setting
that a minimizer $\pi$ in equation \eqref{eq:kantorovich} always exists (\cite{villani2008}, Theorem~4.1),
and is called an optimal coupling. 
In the special case where $\pi$ is supported
in the graph of a map $ T_0:\Omega \to \Omega$, it
must be the case that $ T_0 \in \calT(P,Q)$ due to the marginal constraints in the definition
of $\Pi(P,Q)$, and it must then follow that $T_0$ is precisely an optimal
transport map from $P$ to $Q$. As we shall elaborate below, this situation turns out to characterize
all optimal couplings when $P \in \calPac(\Omega)$, and for such measures the Monge and Kantorovich problems yield equivalent
solutions. 

While an optimal coupling represents a transference plan for reconfiguring $P$ into $Q$, the 
corresponding optimal value of the objective
function~\eqref{eq:kantorovich} represents the optimal cost of such a reconfiguration, which provides an easily interpretable
measure of divergence between $P$ and $Q$. Specifically, it gives rise to the 2-Wasserstein distance,
\begin{equation}
\label{eq:wasserstein}
W_2(P,Q) = \left(\inf_{ \pi \in \Pi(P,Q)} \int \norm{x-y}^2 d\pi(x,y)\right)^{\frac 1 2}.
\end{equation}
The above problem is an (infinite-dimensional) convex program with linear constraints, and it
admits a dual maximization problem, known as the Kantorovich dual problem, given by
\begin{equation}
\label{eq:kantorovich_dual}
W_2^2(P,Q) = \sup_{(\phi,\psi) \in \calK} \int \phi dP + \int \psi dQ,
\end{equation}
where  $\calK$ is the set of pairs $ (\phi,\psi)\in L^1(\Omega)\times  L^1(\Omega)$
such that $\phi(x) + \psi(y) \leq \norm{x-y}^2$ for all $x, y \in \Omega$.
In the present setting of the quadratic optimal transport problem over
the compact set $\Omega$, it can be shown that strong
duality indeed holds in equation~\eqref{eq:kantorovich_dual}, and 
that the supremum is always achieved by some pair $(\phi_0,\psi_0) \in \calK$. 
Any such pair of functions is  
called a pair of Kantorovich potentials. 
In this case, notice that $(\phi_0,\phi_0^c)$, with $\phi_0^c(y) = \inf_{x \in \Omega} \big\{\norm{x-y}^2 - \phi_0(x)\big\}$,
is itself a pair of Kantorovich potentials, since replacing $\psi_0$ by $\phi_0^c$ 
can only increase the objective value~\eqref{eq:kantorovich_dual}, 
while retaining the constraint $(\phi_0,\phi_0^c) \in \calK$. If we define $\varphi_0 = \frac 1 2( \norm\cdot^2 - \phi_0)$, then
$\phi_0^c = \norm\cdot^2 - 2\varphi_0^*$, where for any  $f:\Omega \to \bbR$,  
$$f^*(y) = \sup_{x \in \Omega}\big\{\langle x, y\rangle - f(x)\big\}, \quad y \in \Omega,$$
denotes the Legendre-Fenchel conjugate of $f$. Under this reparametrization, the
Kantorovich dual problem is equivalent to the so-called semi-dual problem
\begin{equation}
\label{eq:semi_dual}
\inf_{\varphi \in L^1(P)} \int \varphi dP + \int \varphi^* dQ,
\end{equation}
in the sense that $\varphi_0$ solves to the semi-dual
problem if and only if $(\norm\cdot^2 - 2\varphi_0, \norm\cdot^2 -2 \varphi_0^*)$ solves the Kantorovich  problem \eqref{eq:kantorovich_dual}.
The semi-dual problem is closely connected to the Monge problem, as shown by the following 
result
of~\cite{knott1984, brenier1991}. \vspace{-0.2in}
\begin{theorem}[Brenier's Theorem]
\label{thm:brenier}
Let $P \in \calPac(\Omega)$ and $Q \in \calP(\Omega)$. 
\begin{thmlist}
\item  There exists an optimal transport map $T_0$ between $P$ and $Q$
which takes the form $T_0 = \nabla\varphi_0$ for a convex function $\varphi_0: \bbR^d \to \bbR$
which solves the semi-dual problem~\eqref{eq:semi_dual}. Furthermore, $T_0$
  is uniquely determined $P$-almost everywhere.  
\item If we further have $Q \in \calPac(\Omega)$, then $\nabla \varphi_0^*$ is the ($Q$-almost everywhere uniquely determined)
gradient of a convex function such that ${\nabla\varphi_0^*}_\# Q = P$, and solves the Monge problem 
for transporting $Q$ onto $P$. Furthermore, for Lebesgue-almost every $x, y \in \Omega$ 
$$\nabla\varphi_0^* \circ \nabla \varphi_0(x) = x, \quad \nabla\varphi_0\circ\nabla\varphi_0^*(y) = y.$$
\end{thmlist} 
\end{theorem}
The convexity of $\varphi_0$ 
implies that it will be almost-everywhere twice differentiable. 
Further smoothness properties of Brenier potentials, and therefore of optimal transport maps, have been studied
via the regularity theory of partial differential equations of the Monge-Amp\`ere type, and we refer
to~\cite{dephilippis2014,figalli2017} for surveys. 
In short, denote by $p,q$ the respective Lebesgue densities of $P,Q \in \calPac(\Omega)$, and assume 
$\varphi_0$ is in fact everywhere twice continuously differentiable. Then, the constraint ${\nabla\varphi_0}_\# P = Q$
implies  by the change of variable formula that $\varphi_0$ solves the equation
\begin{equation}
\label{eq:monge_ampere}
\det\big( \nabla^2\varphi_0(x)\big) = \frac{p(x)}{q(\nabla\varphi_0(x))}, \quad x \in \Omega.
\end{equation}
As a direct consequence of equation~\eqref{eq:monge_ampere}, notice that the Hessian $\nabla^2\varphi_0$
admits a uniformly bounded determinant whenever $p$ and $q$ are bounded, and bounded away from zero. 
This observation leads to the following simple result noted by~\cite{gigli2011}.
\begin{lemma}
\label{lem:gigli}
Assume   $\varphi_0 \in \calC^2(\Omega)$  and $\gamma^{-1} \leq p,q \leq \gamma$ for some $\gamma > 0$.
Then, there exists a constant $\lambda > 0$, depending only on $\gamma$ and $\norm{\varphi_0}_{\calC^2(\Omega)}$, 
such that $\varphi_0$ is $\lambda$-strongly convex.
\end{lemma}
Lemma~\ref{lem:gigli} shows that, whenever equation~\eqref{eq:monge_ampere}
has positive and bounded  right-hand side, smooth Brenier potentials are also strongly convex. We shall require this property in Section~\ref{sec:one_sample_stability} to derive stability bounds for the $L^2(P)$ loss. 
To further obtain sufficient conditions for the H\"older smoothness of $\varphi_0$, notice that 
the Monge-Amp\`ere equation~\eqref{eq:monge_ampere} suggests that $\varphi_0$
admits two degrees of smoothness more than the densities~$p$ and~$q$. 
This intuition indeed turns out to hold true under suitable regularity conditions on $\Omega$, 
as was established in a series of 
publications by~\cite{caffarelli1991,caffarelli1992a,caffarelli1992,caffarelli1996}. 
The following is a summary of these results, as stated by~\citeauthor{villani2008} (\citeyear{villani2008}, Chapter 12). 
\begin{theorem}[Caffarelli's Regularity Theory]
\label{thm:caffarelli}
Assume $\Omega$  satisfies condition~\ref{assm:supp_global}. Assume further that there exists $\gamma > 0$ such that 
$\gamma^{-1} \leq p,q \leq \gamma$ over $\Omega$. Then, the 
Brenier potential $\varphi_0$ is unique up to an additive constant, and satisfies the following. 
\begin{thmlist}
\item \label{thm:caffarelli--interior} (Interior Regularity) Suppose there exists $\alpha > 1$, $\alpha \not\in\bbN$, such that $p,q \in \calC^{\alpha-1}(\Omega^\circ)$.
Then $\varphi_0 \in \calC^{\alpha+1}(\Omega^\circ)$. Moreover,  for  
any open subdomain  $\Omega'$ such that $\widebar{ \Omega'} \subseteq \Omega^\circ$,  there exists a constant $C > 0$ depending 
on $\gamma, \alpha, \Omega, \Omega', \|\varphi_0\|_{L^\infty(\Omega)},\norm{p}_{\calC^{\alpha-1}(\Omega^\circ)}, \norm{q}_{\calC^{\alpha-1}(\Omega^\circ)}$
such that 
$$\norm{\varphi_0}_{\calC^{\alpha+1}(\Omega')} \leq C.$$ 
\item \label{thm:caffarelli--global} (Global Regularity) Assume $\Omega$ admits a $\calC^2$ boundary and is uniformly convex. 
Assume further that there exists $\alpha > 1$, $\alpha \not\in \bbN$, such that $p,q \in \calC^{\alpha-1}(\Omega)$. Then, 
$\varphi_0 \in \calC^{\alpha+1}(\Omega)$.
\end{thmlist}
\end{theorem}
Theorem~\ref{thm:caffarelli--global} implies that, under suitable conditions,
the optimal transport map $T_0$ inherits one degree of smoothness more than the densities $p$ and $q$ over $\Omega$.  
Unlike the interior regularity result of Theorem~\ref{thm:caffarelli--interior}, however, Theorem~\ref{thm:caffarelli--global}
does not imply a uniform bound on $\norm{\varphi_0}_{\calC^{\alpha+1}(\Omega)}$, 
and therefore does not preclude the possibility that the latter quantity diverges when $p,q$ vary
in a $\calC^{\alpha-1}(\Omega)$ ball. Closely related global regularity results 
have also been established by~\cite{urbas1997} under slightly stronger conditions, 
but we do not know if either of these results can be made uniform up to the boundary 
in an analogous way to the interior result of Theorem~\ref{thm:caffarelli--interior}. 
Whenever global uniform regularity results are needed in our development, we  sidestep this
issue by working with the optimal transport problem over the torus, for which boundary considerations do not arise.
 
\subsection{The Quadratic Optimal Transport Problem over the Flat Torus}
\label{sec:background_torus}
Denote by $\bbT^d = \bbR^d /\bbZ^d$ the flat $d$-dimensional torus. Specifically, $\bbT^d$ is the set of equivalence classes $[x]=\{x+k:k\in \bbZ^d\}$, 
for all $x \in [0,1)^d$. 
Abusing notation, we typically write $x$ instead of $[x]$. 
$\bbT^d$ is endowed with the standard metric
$$d_{\bbT^d}(x,y) = \min\{\norm{x-y+k}: k \in \bbZ^d\},
\quad x,y \in \bbT^d.$$
We identify $\calP(\bbT^d)$ with the set of Borel measures $P$ on $\bbR^d$ such that $P([0,1)^d) = 1$ and which are $\bbZ^d$-periodic,
in the sense that $P(B) = P(k+B)$ for all $k \in \bbZ^d$ and all Borel sets $B \subseteq \bbR^d$.
Furthermore, $\calPac(\bbT^d)$ denotes the subset of measures in $\calP(\bbT^d)$ which are absolutely continuous with respect
to the Lebesgue measure on $\bbR^d$. A function $f:\bbT^d \to \bbR$ is understood to 
be a function on $\bbR^d$ which is $\bbZ^d$-periodic, and we write $T:\bbT^d \to \bbT^d$
when $T$ is a map from $\bbR^d$ to $\bbR^d$ such that $[T(x)] = [T(y)]$ whenever $[x]=[y]$. 

The optimal transport problem over $\bbT^d$ with the quadratic cost $d_{\bbT^d}^2$
largely mirrors that of the squared Euclidean cost over $\bbR^d$. 
Define for all $P,Q \in \calPac(\bbT^d)$ the Monge problem 
\begin{equation}
\label{eq:monge_torus}
\argmin_{T\in \calT(P,Q)} \int_{\bbT^d} d_{\bbT^d}^2(x,T(x))dP(x), 
\end{equation} 
where the integral is understood as being taken over $[0,1)^d$. 
The Kantorovich problem and its dual 
give rise to the squared Wasserstein distance over $\calP(\bbT^d)$,
\begin{equation}
\label{eq:wasserstein_torus}
\calW_2^2(P,Q) = \inf_{\pi\in \Pi(P,Q)} \int d_{\bbT^d}^2(x,y)d\pi(x,y) 
               = \sup_{(\phi,\psi) \in \calK_T} \int\phi dP + \int\psi dQ,
\end{equation}
where $\calK_T$ denotes the set of pairs of potentials 
$(\varphi,\psi) \in L^1(P) \times L^1(Q)$ satisfying the dual constraint
$\varphi(x) + \psi(y) \leq d_{\bbT^d}^2(x,y)$ for all $x,y\in \bbT^d$. 
We abuse notation by writing $W_2$
to denote both the 2-Wasserstein distance over $\bbR^d$ and  
$\bbT^d$. Whenever we speak of the optimal transport problem  or Wasserstein distance
between two measures $P,Q \in \calP(\Omega)$, the underlying cost function is tacitly understood
to be $\|\cdot\|^2$ when $\Omega \subseteq \bbR^d$, and $d_{\bbT^d}^2$ when $\Omega = \bbT^d$. 

The following result due to~\cite{cordero-erausquin1999}  is an analogue of Brenier's Theorem,
together with additional properties about the optimal transport problem  over $\bbT^d$.
%
\begin{proposition}
\label{prop:torus_ot}
Let $P \in \calPac(\bbT^d)$ and $Q \in \calP(\bbT^d)$.
Then, there exists a ($P$-a.e. uniquely determined) optimal transport map $T_0 = \nabla \varphi_0$ from $P$ to $Q$
which solves the Monge problem~\eqref{eq:monge_torus}, where
$\varphi_0: \bbR^d \to \bbR$ is a convex function satisfying the following properties.
\begin{thmlist}
\itemsep0.05em 
\item \label{prop:torus_ot--potential_periodicity} 
$\norm\cdot^2/2-\varphi_0$ is $\bbZ^d$-periodic. 
\item \label{prop:torus_ot--map_periodicity}
$T_0(x+k) = T_0(x) + k$ for almost every $x \in \bbR^d$
and $k \in \bbZ^d$.
\item \label{prop:torus_ot--map_locality} 
For $P$-almost all $x \in \bbR^d$, 
 $\norm{T_0(x) - x} = d_{\bbT^d}(x,T_0(x))$.
\end{thmlist}
Assume further that $Q \in \calPac(\bbT^d)$, and denote the respective densities of $P,Q$
by $p,q$. Then, 
\begin{thmlist}[start=5 ]\itemsep0.05em 
\item \label{prop:torus_ot--inverse} $\nabla \varphi_0^*$ is the ($Q$-a.e. uniquely determined)  optimal transport map from $Q$ to $P$.
\item \label{prop:torus_ot--kantorovich}   $(\norm\cdot^2 - 2\varphi_0, \norm\cdot^2 - 2\varphi_0^*)$
is a pair of optimal Kantorovich potentials in equation~\eqref{eq:wasserstein_torus}. 
\item \label{prop:torus_ot--curvature} 
If $\varphi_0 \in \calC^2([0,1]^d)$, 
then it solves the Monge-Amp\`ere equation
$$\det(\nabla^2\varphi_0(x)) q(\nabla\varphi_0(x)) = p(x),\quad x \in \bbR^d.$$
In particular, if $\gamma^{-1} \leq p,q \leq \gamma$ for some $\gamma > 0$, 
then $\varphi_0$ is $\lambda$-strongly convex, for some constant $\lambda > 0$ depending only on $\gamma$ and 
$\norm{\varphi_0}_{\calC^2([0,1]^d)}$. 
\end{thmlist}
\end{proposition}
With Proposition~\ref{prop:torus_ot} in place, regularity properties of Brenier potentials $\varphi_0$ may
be deduced from smoothness conditions on $p,q$. 
The following result was stated by~\cite{cordero-erausquin1999} without explicit mention of the uniformity 
of the H\"older norms appearing therein, but can readily be made uniform using Caffarelli's interior regularity
theory (Theorem~\ref{thm:caffarelli--interior}; \cite{figalli2017}, Chapter 4). We also note
that this result was stated by~\cite{ambrosio2012} in the special case $d=2$. 
\begin{theorem} 
\label{thm:torus_regularity}
Let $P, Q \in \calP(\bbT^d)$ be absolutely continuous with respect to the Lebesgue measure, with respective densities
$p,q$ satisfying $ \gamma^{-1}\leq p,q \leq \gamma$ for some $\gamma > 0$. Assume further that
$p,q \in \calC^{\alpha-1}(\bbT^d)$ for some $\alpha > 1$. Then, there exists
a constant $C  > 0$ depending only on $\alpha,\gamma, \norm{p}_{\calC^{\alpha-1}(\bbT^d)}$
and $\norm q_{\calC^{\alpha-1}(\bbT^d)}$ such that,
$\norm{\varphi_0}_{\calC^{\alpha+1}([0,1]^d)} \leq C.$
\end{theorem}
 
\section{The One-Sample Problem}
\label{sec:one_sample} 
Throughout this section,  we let $P \in \calPac(\Omega)$ denote a known distribution, and $Q \in \calPac(\Omega)$ denote
an unknown distribution from which an i.i.d. sample $Y_1, \dots, Y_n \sim Q$ is observed. 
Let $p$ and $q$ denote their respective densities,
and let $T_0 = \nabla\varphi_0$ denote the unique optimal transport map from $P$ to $Q$, with respect to 
a convex Brenier potential~$\varphi_0$. 
We also denote by $\phi_0 = \norm\cdot^2-2\varphi_0$ and $\psi_0 = \norm\cdot^2 - 2\varphi_0^*$ 
the Kantorovich potentials induced by $\varphi_0$. 
We assume condition~\ref{assm:supp_global} holds throughout this section, and 
we may therefore assume without loss of generality that $-d \leq \phi_0 \leq 0$ and $0\leq  \psi_0 \leq d$ 
over $\Omega$ (\cite{villani2003}, Remark 1.13).	

Unlike the two-sample case which we discuss in Section~\ref{sec:two_sample}, 
there exist canonical estimators of $T_0$ when the source distribution $P$ is known. 
Indeed, since $P$ is absolutely continuous, Brenier's Theorem
implies that there exists a unique optimal transport map $\hat T$ between $P$ and any estimator $\hat Q$ of $Q$, and we analyze two such examples below. 
We first take $\hat Q$ to be the empirical measure of $Q$ in Section~\ref{sec:one_sample_empirical}, and show that the resulting estimator 
$\hat T$ achieves the minimax risk of estimating Lipschitz optimal transport maps,
under essentially no smoothness conditions on the underlying measures. In Section~\ref{sec:one_sample_density}, we
then take $\hat Q$ to be a  density estimator, leading to an estimator $\hat T$ achieving faster
rates of convergence when $Q$ admits a smooth density.
In both cases, our analysis will hinge upon known upper bounds on the risk of $\hat Q$ under the Wasserstein distance,
by invoking a key stability bound which we turn to first.

\subsection{A General Stability Bound} 
\label{sec:one_sample_stability}
The main technical result of this section 
will be stated under
the following curvature condition. 
\begin{enumerate}[leftmargin=1.65cm,listparindent=-\leftmargin,label=\textbf{A1($\lambda$)}]   
\item  \label{assm:curvature} The Brenier potential $\varphi_0$ is a convex function such that 
$\varphi_0 \in \calC^2(\Omega)$ and 
$(1 /\lambda) I_d \preceq \nabla^2\varphi_0(x) \preceq \lambda I_d$  for all $x \in \Omega$.
\end{enumerate}
It can be seen that whenever condition~\ref{assm:curvature} holds for $\varphi_0$,  the same bounds also 
hold  for $\varphi_0^*$. Therefore, condition~\ref{assm:curvature} implies that $T_0$ is  $\lambda$-bi-Lipschitz
over $\Omega$. 
As noted in Lemma~\ref{lem:gigli}, whenever $P$ and $Q$ both admit densities 
satisfying $ \gamma^{-1} \leq p,q \leq \gamma$ over $\Omega$, for some $\gamma > 0$, 
the second inequality of~\ref{assm:curvature} is sufficient 
to imply the first, up to inflating $\lambda$ by a factor depending on $\gamma$.  
Under this condition, we prove the following stability bounds in Appendix~\ref{app:pf_stability}. 
\vspace{-0.2in}
\begin{theorem}
\label{thm:stability}
Let $P,Q \in \calPac(\Omega)$, and assume condition \ref{assm:curvature} holds for some $\lambda > 0$. 
For any $\hat Q \in \calP(\Omega)$, let 
$\hat T = \nabla\hat\varphi$ be the unique optimal transport map from $P$ to $\hat Q$. Then,  
\begin{equation}\label{eq:stability_bound} 
\frac 1 \lambda \|\hat T - T_0\|_{L^2(P)}^2
  \leq  W_2^2(P, \hat Q) -  W_2^2(P,Q)  - \int \psi_0 d(\hat Q-Q)
  \leq  \lambda W_2^2(\hat Q, Q).
\end{equation}   
\end{theorem} 

\noindent We make several remarks regarding Theorem~\ref{thm:stability}.
 
\begin{itemize}
\item Caffarelli's regularity theory  (cf. Theorem~\ref{thm:caffarelli}) provides sufficient conditions
on the smoothness of $P, Q$ and $\partial\Omega$
for assumption~\ref{assm:curvature} to hold, albeit for a non-universal constant $\lambda > 0$.
We note, however, that our assumption
is considerably weaker. For instance, condition~\ref{assm:curvature} is satisfied whenever $P$ and $Q$ differ by a location transformation, irrespective
of the regularity or positivity of their Lebesgue densities.  
\item We  show in Section~\ref{sec:efficiency} that, under weaker assumptions than those of Theorem~\ref{thm:stability},
the map $\psi_0 - \bbE_Q[\psi_0(Y)]$ is the efficient influence
function of the functional $Q \in \calP(\Omega) \mapsto W_2^2(P, Q)$ with respect to the tangent space
$L_0^2(Q)$. It follows that the linear
functional
\begin{equation}
\label{eq:first_order_von_mises}
L(\hat Q) = \int \psi_0 d(\hat Q-Q)
\end{equation}
is the first-order term in the von Mises expansion of $W_2^2(P, \hat Q)$ around $W_2^2(P, Q)$. 
The upper bound of Theorem~\ref{thm:stability} implies that the remainder of this expansion decays
quadratically in the topology induced by $W_2$, a fact which we shall use  to derive upper bounds 
and limit laws 
for plugin estimators of the Wasserstein distance. This fact combined with the lower bound of Theorem~\ref{thm:stability} 
further implies the following remarkable equivalence,
\begin{equation}
\label{eq:W2_L2P_asymp}
\frac 1 {\lambda} \|\hat T - T_0\|_{L^2(P)} \leq W_2(\hat Q, Q) \leq \|\hat T - T_0\|_{L^2(P)}.
\end{equation}
Notice that
the second inequality   always holds due to the fact that
$(\hat T, T_0)_\# P$ is a coupling of $\hat Q$ and $Q$. Equation~\eqref{eq:W2_L2P_asymp} thus
shows that the transport
cost of this coupling is  within
a universal factor of being optimal,
when the curvature condition~\ref{assm:curvature}
is in force. We use this result to 
obtain upper bounds on the risk of one-sample plugin estimators $\hat T$ 
by appealing to the corresponding risk of  $\hat Q$ under the Wasserstein distance. 

\item When $d=1$, it is easy to see by direct calculation that 
the inequalities~\eqref{eq:W2_L2P_asymp} 
hold with equality, with $\lambda=1$, even without assumption~\ref{assm:curvature}. 
For multivariate measures,
weaker analogues of equation~\eqref{eq:W2_L2P_asymp}, in which the left-hand side admits an exponent greater than
unity, have previously been derived  by~\cite{merigot2019,delalande2021}.   
Those works adopted a weaker assumption than ours, however.
\item Suppose that, in addition to the assumptions of Theorem~\ref{thm:stability},
the measures $Q$ and $\hat Q$ are both absolutely continuous with respect to the Lebesgue measure, with respective densities
$q$ and $\hat q$ which satisfy $\gamma^{-1} \leq q,\hat q \leq \gamma$
over $\Omega$, for some $\gamma > 0$. In this setting, it was shown by~\cite{peyre2018} that the 2-Wasserstein distance is equivalent to the 
negative-order homogeneous Sobolev norm $\|\cdot\|_{\dH{-1}(\Omega)}$, in the sense that, under suitable conditions on $\Omega$, 
\begin{equation}
\label{eq:peyre} 
\gamma^{-1} \|\hat q - q \|_{\dH{-1}(\Omega)}^2 \lesssim W_2^2(\hat Q, Q) \lesssim \gamma \|\hat q - q \|_{\dH{-1}(\Omega)}^2.
\end{equation}
Theorem~\ref{thm:stability} and the above display then imply
$$\frac 1 {\lambda\gamma} \|\hat\varphi - \varphi_0 \|_{\dH{1}(\Omega)}^2 \lesssim W_2^2(P, \hat Q) - W_2^2(P, Q) - \int \psi_0 d(\hat Q-Q) 
\lesssim \lambda\gamma \|\hat q - q\|_{\dH{-1}(\Omega)}^2.$$
It  follows from the upper bound that $W_2^2(P, \cdot)$, when viewed as a functional of $\hat q$, is Fr\'echet
differentiable at $q$ in the $\dot H^{-1}(\Omega)$ topology. It moreover implies that this functional is strongly
convex and smooth with respect to the duality of the spaces $\dot H^{-1}(\Omega)$ and $\dot H^1(\Omega)$. 
\item 
Theorem~\ref{thm:stability} is stated in a form which is sufficient
for our purposes, however it
is not the most general result possible. On the one 
hand, the  assumption of boundedness on $\Omega$ is superfluous: 
Theorem~\ref{thm:stability} continues to hold if $\Omega$ is an unbounded, closed, and convex set, 
such as the entire Euclidean space $\bbR^d$. 
It follows, for instance, that Theorem~\ref{thm:stability} 
is applicable whenever $P$ and $Q$ are  strongly log-concave measures, in which case
assumption~\ref{assm:curvature} holds by Caffarelli's contraction theorem~\citep{caffarelli2000}.
On the other hand,
assumption~\ref{assm:curvature} can be weakened 
in the following way:
the first inequality
of display~\eqref{eq:stability_bound} 
holds  under the mere condition $\nabla^2 \varphi_0 \preceq \lambda I_d$, 
whereas the second holds when $\nabla^2\varphi_0 \succeq \lambda^{-1}I_d$.

\item 
Finally, one may also infer from Theorem~\ref{thm:stability} and the Kantorovich duality that,
\begin{equation}
\label{eq:stability_semi_dual}
 \frac 1 {2\lambda} \|\nabla\hat\varphi - \nabla \varphi_0\|_{L^2(P)}^2 
\leq \int (\varphi_0-\hat\varphi) dP + \int (\varphi_0^*-\hat \varphi^*)d\hat Q
\leq \frac \lambda 2\|\nabla \hat\varphi - \nabla \varphi_0\|_{L^2(P)}^2.
\end{equation}
Equation~\eqref{eq:stability_semi_dual} is a direct analogue of a stability bound
proven by~\citeauthor{hutter2021} (\citeyear{hutter2021}, Proposition~10), 
who show that similar inequalities hold when 
the measure $\hat Q$ appearing in the above display is replaced by $Q$.
Their result assumes, however, that  
$\hat\varphi$ itself satisfies condition~\ref{assm:curvature}.  
In contrast, we do not place any  conditions on the estimator $\hat T$ 
 beyond it being the optimal transport
map from $P$ to $\hat Q$.  
This will permit our study of transport map estimators which are potentially 
nonsmooth but easy to compute, as we show next.
\end{itemize}

\subsection{Upper Bounds for   One-Sample Empirical Estimators}
\label{sec:one_sample_empirical}
Recall that we denote by $Q_n = (1/n)\sum_{i=1}^n \delta_{Y_i}$  the empirical measure. 
Since $P$ is known and absolutely continuous, a natural estimator for $T_0$ is the 
  optimal transport map $T_n$ from $P$ 
to $Q_n$, defined by
\begin{equation}
\label{eq:semi_discrete}
T_n = \argmin_{T \in \calT(P, Q_n)} \int \norm{x-T(x)}^2 dP(x).
\end{equation}
By Brenier's Theorem, the minimizer $T_n$ in the above display exists and is uniquely determined $P$-almost everywhere. 
The optimization problem~\eqref{eq:semi_discrete} is sometimes known as the semi-discrete optimal transport problem, 
for which efficient numerical solvers are well-studied \citep{merigot2011,levy2018}. 
 
In view of the stability bound in Theorem~\ref{thm:stability}, the risk of $T_n$ may be related to that of the empirical
measure $Q_n$ under the Wasserstein distance. For instance, from the work of \cite{fournier2015}
we obtain the following bound, under no assumptions beyond~\ref{assm:supp_global},
\begin{equation}
\label{eq:wasserstein_empirical}
\bbE W_2^2(Q_n,Q) \lesssim \kappa_{n} := \begin{cases}
n^{-1/2}, & d \leq 3 \\
n^{-1/2} \log n, & d =4 \\
n^{-2/d}, & d \geq 5.
\end{cases}
\end{equation}
The following bound on the risk of $T_n$ is now an immediate
consequence of Theorem~\ref{thm:stability}, together with the fact that the 
functional $L$ in equation~\eqref{eq:first_order_von_mises} satisfies 
$\bbE[L(Q_n)] = 0$.
\begin{corollary}
\label{cor:one_sample_emp}
Let $P,Q \in \calPac(\Omega)$ and assume condition~\ref{assm:curvature} holds. Then, 
$$\bbE \big\| T_n - T_0 \big\|_{L^2(P)}^2 \asymp_{\lambda}
 \bbE \big[W_2^2(P, Q_n) - W_2^2(P, Q)\big] \asymp_{\lambda} \bbE W_2^2(Q_n, Q) \lesssim \kappa_n.$$
\end{corollary} 
When $d \geq 5$, Corollary~\ref{cor:one_sample_emp} implies that the  empirical estimator $T_n$ achieves the minimax lower
bound~\eqref{eq:HR_lower_bound} 
for estimating Lipschitz transport maps $T_0$.
On the other hand, when $1 \leq d \leq 4$, the rate $\kappa_n$ does not improve beyond $n^{-1/2}$, 
unlike the minimax lower bound~\eqref{eq:HR_lower_bound} of~\cite{hutter2021}, which scales as fast as $1/n$. 
This observation does not imply that the plugin estimator $T_n$ is minimax suboptimal, 
since equation~\eqref{eq:HR_lower_bound} holds~under~stronger assumptions 
than those of Corollary~\ref{cor:one_sample_emp}. In particular, it assumes that these distributions admit densities
which are bounded away from zero, and thus have connected support. 
In contrast, Corollary~\ref{cor:one_sample_emp} applies to measures $P$ and $Q$ with possibly disconnected support, 
for which our upper bound of $\kappa_n$ cannot generally be improved up to a logarithmic factor---similar
considerations are discussed for the convergence rate of the empirical measure by~\cite{bobkov2019} when $d=1$, and 
more generally by~\cite{weed2019a}. 

Nevertheless, when we further assume that $Q$ has a positive density, the result of Corollary~\ref{cor:one_sample_emp}
can be strengthened to match the minimax rate of~\cite{hutter2021} even for $d \leq 4$. For instance, 
it is well-known (cf.~\cite{ajtai1984}, \cite{ledoux2019}) that, when $Q$ is the uniform distribution on
$[0,1]^d$, $Q_n$ achieves the following faster rate,  
\begin{equation}
\label{eq:ledoux}
\bbE W_2^2(Q_n, Q) \lesssim  \begin{cases} 
n^{-1}, & d = 1 \\
n^{-1} \log n, & d = 2 \\
n^{-2/d}, & d \geq 3.
\end{cases}
\end{equation}
Such a result is also known to hold for any measure $Q$ admitting positive density over a compact subset of the real
line~\citep{bobkov2019}, or over the   flat torus~\citep{divol2021}. 
Inspired by the  latter result and by the work of~\cite{weed2019a}, we prove an analogue of equation~\eqref{eq:ledoux}
for arbitrary measures
supported on the unit hypercube, at the expense of an inflated polylogarithmic factor when $d=2$.
\begin{corollary} 
\label{cor:one_sample_emp_hypercube}
Let $P,Q \in \calPac([0,1]^d)$ and assume that condition~\ref{assm:curvature} holds.
Assume further that $\gamma^{-1} \leq q \leq \gamma$ over $[0,1]^d$, for some $\gamma > 0$. Then, 
$$\bbE \big\| T_n - T_0 \big\|_{L^2(P)}^2 \asymp 
 \bbE \big[W_2^2(P, Q_n) - W_2^2(P, Q)\big] \asymp \bbE W_2^2(Q_n, Q) \lesssim  \widebar\kappa_n\
{:=}  \begin{cases} 
n^{-1}, & d = 1 \\
\frac{ (\log n)^2}{n}, & d = 2 \\
n^{-2/d}, & d \geq 3.
\end{cases}$$
\end{corollary} 
Under the assumptions of Corollary~\ref{cor:one_sample_emp_hypercube}, we deduce that the plugin estimator $T_n$ is minimax
optimal for all $d \geq 1$, up to a polylogarithmic factor when $d= 2$. 
The scale of this factor is further discussed following the statement
of Theorem~\ref{thm:two_sample_kernel}. 

This result also provides a sharper bound on the bias of $W_2^2(P, Q_n)$ 
than could have been deduced from~\cite{chizat2020}, who show that the risk of this estimator decays at the rate $\kappa_n$ using distinct techniques. 
Aditionally, Corollaries~\ref{cor:one_sample_emp}--\ref{cor:one_sample_emp_hypercube} 
can be extended to recover the risk bounds of~\cite{chizat2020} under stronger conditions, though with an improved rate of convergence when $P$ approaches $Q$ in Wasserstein distance.
\begin{corollary}
\label{cor:two_sided_wasserstein}
Let $P,Q \in \calPac(\Omega)$, and assume condition~\ref{assm:curvature} holds.
Then, 
\begin{equation}
\label{eq:empirical_rate_W2}
\bbE \big| W_2^2(P, Q_n) - W_2^2(P,Q)\big| \lesssim_\lambda   \bbE W_2^2(Q_n, Q) + n^{-\frac 1 2} \lesssim \kappa_n.
\end{equation}
If we further assume that $\Omega = [0,1]^d$ and $\gamma^{-1} \leq q \leq \gamma$ over $\Omega$ 
for some $\gamma > 0$, then
\begin{equation}
\label{eq:empirical_rate_W2_stronger}
\bbE \big| W_2^2(P, Q_n) - W_2^2(P,Q)\big| \lesssim_{\lambda,\gamma} \widebar\kappa_n +  W_2(P,Q) n^{-\frac 1 2} .
\end{equation}
\end{corollary}
Equation~\eqref{eq:empirical_rate_W2_stronger} exhibits an upper bound on the risk of $W_2^2(P,Q_n)$ that interpolates
between the fast rate $\widebar \kappa_n$ when $W_2(P,Q) \lesssim n^{-1/2}$, and 
the rate $\kappa_n$ of~\cite{chizat2020}, which is minimax optimal when the distance between $P$ and $Q$ 
is unconstrained~\citep{manole2021a}. We defer the proofs
of Corollaries~\ref{cor:one_sample_emp_hypercube}--\ref{cor:two_sided_wasserstein} 
to Appendix~\ref{app:one_sample_empirical}.

\subsection{Upper Bounds for   One-Sample Wavelet Estimators}
\label{sec:one_sample_density}

While the empirical estimator in the previous section  achieves the minimax rate of estimating
Lipschitz optimal transport maps,  we do not generally expect it to achieve faster rates
of convergence if $T_0$ is assumed to enjoy further regularity. 
We instead show that such improvements can be achieved when $Q$ admits a smooth density $q$, and when
the empirical measure $Q_n$ is replaced
by the distribution $\hat Q_{n}$ of a density estimator $\hat q_{n}$. Specifically, define  
\begin{equation}
\label{eq:semi_discrete_density}
\hat T_n = \argmin_{T \in \calT(P, \hat Q_{n})} \int \norm{x-T(x)}^2 dP(x).
\end{equation}
We focus on the  case where $\hat q_{n}$ is a wavelet density estimator, 
for which sharp risk estimates under the Wasserstein distance have been established by~\cite{weed2019a}. 
In order to appeal to their results, we   assume
that the sampling domain is the unit hypercube $\Omega=[0,1]^d$.
In Section~\ref{sec:smooth_domains}, we also extend some of the results
of this section to the case where $\Omega$ is a generic domain with smooth boundary. 

We briefly introduce notation from the theory of wavelets, 
and refer the reader to Appendix~\ref{app:smoothness_classes}  for a detailed summary and references.  
To define a basis over the unit cube $\Omega$,
we focus on the boundary-corrected $N$-th Daubechies wavelet system, for an integer
$N \geq 2$,  as introduced by 
\cite{cohen1993}. 
In short, given an integer $j_0 \geq \log_2 N$, their construction leads to 
respective families of scaling and wavelet functions 
$$\Phi^{\mathrm{bc}}     = \{\zeta_{j_0k}^{\mathrm{bc}}: 0 \leq k  \leq  2^{j_0}-1\}, \quad
  \Psi_{j}^{\mathrm{bc}} = \{ \xi_{jk\ell}^{\mathrm{bc}}:  0 \leq k  \leq  2^{j_0}-1,\ell \in \{0,1\}^d\setminus \{0\}\},\quad j \geq j_0,$$
such that $\Psi^{\mathrm{bc}} = \Phi^{\mathrm{bc}} \cup \bigcup_{j=j_0}^\infty \Psi_j^{\mathrm{bc}}$
forms an orthonormal basis of $L^2(\Omega)$, 
with the property that $\Phi\pbc$ spans all polynomials 
of degree at most $N-1$ over $\Omega$. 
Given a probability distribution $Q \in \calPac(\Omega)$ admitting density $q \in L^2(\Omega)$,
one then has  
$$q  = \sum_{\xi \in \Psi\pbc} \beta_\xi \xi = \sum_{\zeta \in \Phi\pbc} \beta_\zeta \zeta  + \sum_{j=j_0}^\infty \sum_{\xi \in \Psi_j\pbc}\beta_\xi \xi,\quad \text{where}\quad \beta_\xi  
= \int \xi dQ, \ \xi\in \Psi\pbc,$$
where the series converges at least in $L^2(\Omega)$.
The standard truncated wavelet estimator of $q$~\citep{kerkyacharian1992} with 
a  truncation level $J_n \geq j_0 > 0$ is then  given by  
$$\tilde q_n\sbc = \sum_{\xi \in \Psi\pbc} \hbeta_\xi \xi  =  \sum_{\zeta \in \Phi\pbc} \hbeta_\zeta \zeta  + \sum_{j=j_0}^{J_n} \sum_{\xi \in \Psi_j\pbc}\hbeta_\xi \xi,\quad \text{where}\quad \hbeta_\xi = \int \xi dQ_n, \ \xi\in \Psi\pbc.$$
Notice that $\tilde q_n\sbc$ is permitted to take on negative values, in which case
it does not define a probability density.  
We instead define the final density estimator $\hat q_n\equiv \hat q_n\sbc$   by
\begin{equation}
\label{eq:hatq_one_sample}
\hat q_n\sbc =\frac{ \tilde q_n\sbc I(\tilde q_n\sbc \geq 0)}{\int \tilde q_n\sbc I(\tilde q_n\sbc \geq 0)}, \quad \text{over } \Omega,
\end{equation}
and we denote by $\hat Q_n\sbc$ the distribution induced by $\hat q_n\sbc$. 
We drop all superscripts ``bc'' in the sequel whenever the choice of wavelet system is unambiguous. 
\cite{weed2019a} bounded the Wasserstein risk of a wavelet density estimator obtained from a distinct modification   
of $\tilde q_n$. By appealing to $L^\infty$ concentration inequalities
for wavelet density estimators~\citep{masry1997}, 
we show in Appendix~\ref{app:pf_wavelet_wasserstein}
that their result carries over to the estimator $\hat q_n$.  
Equipped with this result, we arrive at the following bound on the risk of the estimator $\hat T_n \equiv \hat T_n\sbc$ 
defined in equation~\eqref{eq:semi_discrete_density}, and of the corresponding plugin estimator
of the squared Wasserstein distance. Recall that the H\"older balls $\calC^\alpha(\Omega; \cdot)$ and $\calC^\alpha(\Omega; \cdot,\cdot)$
are defined in equations~\eqref{eq:holder_ball}--\eqref{eq:holder_ball_bound}. 

 \begin{theorem}[One-Sample Wavelet Estimators]
\label{thm:one_sample_wavelet}
Let $\alpha > 1$ and $M,\gamma > 0$. Let $P,Q\in \calPac([0,1]^d)$, and assume
 the density $q$  satisfies $q \in \calC^{\alpha-1}([0,1]^d; M,\gamma)$.
Let  $2^{J_n} \asymp n^{1 /({d+2(\alpha-1)})}$.  Then, the following assertions hold.
\begin{thmlist}
\item \label{thm:one_sample_wavelet--map} (Optimal Transport Maps) Assume $\varphi_0$ satisfies condition~\ref{assm:curvature} 
for some $\lambda > 0$.
Then, there exists
a constant $C > 0$ depending on $M,\lambda,\gamma,\alpha$ such that,
\begin{equation*}
\bbE \big\|\hat T_n - T_0\big\|_{L^2(P)}^2
\leq C \trate, \quad \text{where } \trate :=  
\begin{cases}
1/n, & d = 1\\
(\log n)^2 /n, & d = 2\\
n^{-\frac{2\alpha}{2(\alpha-1) + d}}, & d \geq 3.
\end{cases}
\end{equation*}
\item \label{thm:one_sample_wavelet--wasserstein} (Wasserstein Distances) Assume that for some 
$\lambda > 0$, $\varphi_0^* \in \calC^{\alpha+1}([0,1]^d;\lambda)$, and 
$\gamma ^{-1}\leq p \leq \gamma$ 
over $[0,1]^d$. Then, there exists
a constant $C > 0$ depending on $M,\lambda,\gamma,\alpha$ such that, 
\end{thmlist}
\vspace{-0.1in}
\begin{align*}
\big| \bbE W_2^2(P, \hat Q_n) - W_2^2(P,Q) \big| &\leq C \trate,\\
\bbE \big| W_2^2(P, \hat Q_n) - W_2^2(P,Q) \big|^2
&\leq \left[C \trate  +  \sqrt{\frac {\Var_Q[\psi_0(Y)]}{n}}\right]^2.
\end{align*}
\end{theorem} 
Theorem~\ref{thm:one_sample_wavelet} requires smoothness assumptions on both the density $q$  
 and  the potential $\varphi_0^*$; in particular, the assumption of Theorem~\ref{thm:one_sample_wavelet--wasserstein} 
 requires both $q \in \calC^{\alpha-1}(\Omega)$ and $\varphi_0^* \in \calC^{\alpha+1}(\Omega)$. 
Caffarelli's regularity theory (Theorem~\ref{thm:caffarelli}) suggests that the former condition on $q$
should be sufficient to imply the latter condition on $\varphi_0^*$,
but such results cannot be invoked here due to the lack of smoothness
of the boundary of the unit cube $[0,1]^d$. Even if the above analysis could be adapted
to a domain $\Omega$ with smooth boundary, the lack of uniformity in Caffarelli's global regularity
theory would prevent the bounds in 
Theorem~\ref{thm:one_sample_wavelet} from holding uniformly in $P$ and $Q$, in the absence of a smoothness condition on $\varphi_0^*$.
We refer to Appendix~E of~\cite{hutter2021} for related discussions. 
In Proposition~\ref{prop:one_sample_torus} of Appendix~\ref{app:pf_density_based}, we will show that 
an analogue of~Theorem~\ref{thm:one_sample_wavelet} holds merely under smoothness conditions on $p$ and $q$ 
when $\Omega$ is  the $d$-dimensional torus $\bbT^d$,
which enjoys the global regularity result of Theorem~\ref{thm:torus_regularity}.
Here, we instead impose smoothness conditions on both $\varphi_0^*$ and $q$, 
in which case $\hat T_n$  achieves the minimax rate~\eqref{eq:HR_lower_bound} of estimating an $\alpha$-H\"older 
optimal transport map.
 
Theorem~\ref{thm:one_sample_wavelet--wasserstein} also proves that
the bias of $W_2^2(P, \hat Q_n)$ achieves the same convergence rate, 
as does its risk when $d \geq 2(\alpha+1)$. In the high-smoothness regime $d < 2(\alpha+1)$,
the risk of $W_2^2(P, \hat Q_n)$, in squared loss, does not generally improve beyond the parametric rate $1/n$,
except when $\Var_Q[\psi_0(Y)]$ vanishes. Using Lemma~\ref{lem:kantorovich_L2} in Appendix~\ref{app:kantorovich}, the
latter quantity is bounded above by $W_2^2(P,Q)$ up to a constant, so Theorem~\ref{thm:one_sample_wavelet--wasserstein} also implies
\begin{equation} 
\bbE\big| W_2^2(P, \hat Q_n) - W_2^2(P,Q)\big| \lesssim_{M,\gamma,\lambda,\alpha} \trate + \frac{W_2(P,Q)}{\sqrt n}.
\end{equation} 

We briefly highlight the main components of the proof of Theorem~\ref{thm:one_sample_wavelet}.
Both assertions are proven 
by combining the stability results of Theorem~\ref{thm:stability}
with the bound $\bbE W_2^2(\hat Q_n, Q) \lesssim \trate$, which is stated 
formally in Lemma~\ref{lem:wavelet_wasserstein}, and 
extends a result due to~\cite{weed2019a}.
In particular, Theorem~\ref{thm:one_sample_wavelet--map}   follows immediately from the equivalence~\eqref{eq:W2_L2P_asymp}.
Our proof of Theorem~\ref{thm:one_sample_wavelet--wasserstein}
additionally requires us to analyze the evaluation $L(\hat Q_n)$
of the linear functional~$L$ defined in equation~\eqref{eq:first_order_von_mises}, 
for which we prove the following. 
\begin{lemma}
\label{lem:L}
Assume the same conditions as Theorem~\ref{thm:one_sample_wavelet--wasserstein}. Then, 
$$  \bbE[L(\hat Q_n)] = O\left( 2^{-2J_n\alpha}\right), \quad   
  \Var\big[L(\hat Q_n)\big] =   \frac 1 n \Var_Q[\psi_0(Y)]  + O\left(\frac{2^{-2J_n\alpha}}{n}\right),$$
where the implicit constants depend only on $M,\gamma,\lambda,\alpha$.
\end{lemma}
Lemma~\ref{lem:L} shows that the bias of $L(\hat Q_n)$ 
scales quadratically faster than the traditional bias of $\hat Q_n$
in estimating an $(\alpha-1)$-H\"older density, 
which is known to be of order $2^{-J_n(\alpha-1)}$. 
We obtain the faster rate $2^{-2J_n\alpha}$ due to the assumed $(\alpha+1)$-H\"older smoothness of the  
potential~$\varphi_0^*$. The proofs of Theorem~\ref{thm:one_sample_wavelet} and Lemma~\ref{lem:L} are deferred to
Appendix~\ref{app:one_sample_wavelet}. 

\begin{remark}[Adaptive Estimation]
When constructing the estimator $\hat T_n$, 
we  assumed that the smoothness
parameter $\alpha$ is known, and used it to tune the truncation parameter $J_n$. 
It   is also possible to construct an 
adaptive estimator, however. \citet[Theorem 2]{weed2019a} derived an adaptive density estimator $\hat Q_n^\circ$
which achieves the minimax rate of estimating $Q$ under the Wasserstein distance, up to polylogarithmic
factors. It is then natural to define a plugin estimator of $T_0$ as the unique 
optimal transport map from $P$ to $\hat Q_n^\circ$.
By reasoning similarly as in the proof of Theorem~\ref{thm:one_sample_wavelet--map}, 
this estimator has an $L^2(P)$ risk of order $\trate$, up to  polylogarithmic factors, and does not require knowledge of~$\alpha$. 
\end{remark} 

\section{The Two-Sample Problem}
\label{sec:two_sample}
In this section, we turn to analyzing two-sample estimators when both measures $P ,Q \in \calPac(\Omega)$  are unknown. 
As in the one-sample case, we  study two classes of plugin 
estimators. 
The first consists of 
estimators which interpolate the empirical in-sample optimal transport coupling  
using nonparametric smoothers. Such estimators
will achieve the optimal rate of estimating $T_0$ when it is Lipschitz. 
The second class will consist of plugin estimators based on density estimates of $P$ and $Q$, and will achieve
faster rates of convergence when $P$ and $Q$ have smooth densities. As before, our proofs will rely on stability bounds for the two-sample problem, to which we turn our
attention first. 

\subsection{Two-Sample Stability Bounds}
\label{sec:two_sample_stability}
The stability bounds of Theorem~\ref{thm:stability} admit the following one-sided extension 
when both measures $P$ and $Q$ are unknown. 
\begin{proposition}
\label{thm:two_sample_stability} 
Let $P,Q \in \calPac(\Omega)$, and assume condition \ref{assm:curvature} holds for some $\lambda > 0$. 
Then,  for any measures $\hat P,\hat Q \in \calP(\Omega)$,
\begin{equation}
\label{eq:two_sample_stability}
\begin{multlined}[0.85\linewidth]
0\leq W_2^2(\hat P, \hat Q) - W_2^2(P, Q) - \int \phi_0 d(\hat P-P) - \int \psi_0 d(\hat Q-Q) 
\\  \leq {\lambda} \left[ W_2(\hat P, P) + W_2(\hat Q, Q)\right]^2.
\end{multlined}
\end{equation}
\end{proposition}  

The proof is deferred to Appendix~\ref{app:pf_two_sample_stability}.  
Similarly to Theorem~\ref{thm:stability}, this result shows that the remainder of a first-order 
expansion of $W_2^2(\hat P, \hat Q)$ around $W_2^2(P, Q)$ decays quadratically in the $W_2$ topology. 
Unlike Theorem~\ref{thm:stability}, however, we do not generally expect that the lower bound in Proposition~\ref{thm:two_sample_stability} 
can be replaced by a squared distance between $(P,Q)$ and $(\hat P,\hat Q)$: for instance, the lower bound
of zero is achieved in equation~\eqref{eq:two_sample_stability} when $\hat P = \hat Q \neq Q = P$, even though $\hat P$ may
be arbitrarily far from $P$ in Wasserstein distance. This example shows more generally that
the bivariate functional $W_2^2(\cdot,\cdot)$ is not strictly convex 
over $\calPac(\Omega) \times\calPac(\Omega)$, unlike the univariate functional $W_2^2(P,\cdot)$ 
for a fixed absolutely continuous measure~$P$ (cf. Theorem~\ref{thm:stability} and Proposition 7.19
of~\cite{santambrogio2015}).

These observations do not preclude the possibility of replacing the lower bound
in Proposition~\ref{thm:two_sample_stability} by $\lambda^{-1}\|\hat T - T_0\|_{L^2(P)}^2$, 
for $\hat T$ the optimal transport map between $\hat P$ and $\hat Q$. We were not
able to derive such a result under the stated assumptions, except when these estimators are taken to be
empirical measures. We describe this special case  next, 
and show how it may be used to derive estimators of Lipschitz optimal transport maps $T_0$.

\subsection{Upper Bounds for Two-Sample Empirical Estimators}
\label{sec:two_sample_empirical}
Let 
$X_1, \dots, X_n \sim P$ and $Y_1, \dots, Y_m \sim Q$ denote i.i.d. samples, and 
define the empirical measures $P_n = (1/n)\sum_{i=1}^n \delta_{X_i}$ and $Q_m = (1/m)\sum_{j=1}^m \delta_{Y_j}$. 
Though the Monge problem between $P_n$ and $Q_m$ can be infeasible
when $n \neq m$, the Kantorovich problem is always feasible, and takes the following form
$$\hat\pi  \in 
  \argmin_{\pi \in \calQ_{nm}} \sum_{i=1}^n \sum_{j=1}^m \pi_{ij} \norm{X_i - Y_j}^2,$$
where $\calQ_{nm}$ denotes the set of doubly stochastic matrices $\pi=(\pi_{ij}:1 \leq i \leq n, \ 1 \leq j \leq m)$,
satisfying $\pi_{ij}\geq 0$, $\sum_{i=1}^n \pi_{ij} = 1/m$ and $\sum_{j=1}^m \pi_{ij} = 1/n$.
We shall formulate the main stability bound of this section in terms of the quantity
$$\Delta_{nm} 
 =  \sum_{i=1}^n \sum_{j=1}^m \hat\pi_{ij} \norm{T_0(X_i) - Y_j}^2.$$
Recall that $(\kappa_n)$ and $(\widebar \kappa_n)$ denote the sequences defined
in equation~\eqref{eq:wasserstein_empirical} and Corollary~\ref{cor:one_sample_emp_hypercube} respectively.  We obtain the following result,
which we prove in Appendix~\ref{app:pf_two_sample_stability_empirical}. 
\begin{proposition}
\label{prop:curvature_empirical}
Let $P,Q \in \calPac(\Omega)$, and assume \ref{assm:curvature} holds
for some $\lambda > 0$. Then, 
$$ \bbE [\Delta_{nm}] \asymp_\lambda 
\bbE\Big[W_2^2(P_n, Q_m) - W_2^2(P,Q)\Big] \lesssim \kappa_{n\wedge m}.$$
If, in addition, $\Omega = [0,1]^d$ and there exists $\gamma > 0$ such that $\gamma^{-1} \leq p,q \leq \gamma$ over $\Omega$, 
then,  
$$\bbE [\Delta_{nm}] \asymp_\lambda 
\bbE\Big[W_2^2(P_n, Q_m) - W_2^2(P,Q)\Big] \lesssim_\gamma   \widebar \kappa_{n\wedge m}.$$

\end{proposition}
To gain intuition about Proposition~\ref{prop:curvature_empirical}, it is fruitful to consider the special case $n=m$.
In this setting, there exists an optimal transport map $T_n$ from $P_n$ to 
$Q_n$, and we may take 
$$\hpi_{ij} =  I(T_n(X_i) = Y_j)/n, \quad \text{for all } 1 \leq i,j \leq n.$$ 
We then have $\Delta_{nn}  =    \norm{T_n-T_0}_{L^2(P_n)}^2$, and Proposition~\ref{prop:curvature_empirical}  implies
\begin{equation}
\label{eq:two_sample_stability_empirical}
\bbE \norm{T_n-T_0}_{L^2(P_n)}^2 \asymp \bbE\Big[ W_2^2(P_n, Q_n) - W_2^2(P,Q)\Big].
\end{equation}
Equation~\eqref{eq:two_sample_stability_empirical} is a two-sample analogue of Corollary~\ref{cor:one_sample_emp},
and shows that the $L^2(P_n)$ risk of the in-sample transport map estimator is of same
order as the bias of the two-sample empirical optimal transport cost. 
While the estimators $T_n$ and $\hat\pi$
are only defined over the support of $P_n$, we next show how they may  be extended to the entire domain $\Omega$.
We begin with an estimator inspired by 
the classical method of nearest-neighbor nonparametric regression~\citep{cover1968}.

\noindent {\bf One-Nearest Neighbor Estimator.} Define the Voronoi partition generated by $X_1, \dots, X_n$~as
\begin{align}
\label{eq:voronoi}
V_j = \{x \in \Omega: \norm{x-X_j} \leq \norm{x-X_i} , \  \forall i\neq j\}, \quad j=1, \dots, n.
\end{align} 
Then, we define the one-nearest neighbor estimator of $T_0$  by
\begin{equation}
\label{eq:1NN}
\hat T_{nm}^{\mathrm{1NN}}(x) = \sum_{i=1}^n \sum_{j=1}^m (n\hat\pi_{ij}) I(x \in V_i) Y_j,\quad x \in \Omega.
\end{equation}
In order to state an upper bound on the convergence rate of $\hat T_{nm}^{\mathrm{1NN}}$, we place the following 
mild condition on the support $\Omega$.
Recall that $\calL$ denotes the Lebesgue measure on $\bbR^d$.
\begin{enumerate}[leftmargin=1.65cm,listparindent=-\leftmargin,label=\textbf{(S2)}]    
\item  \label{assm:supp_standard}
$\Omega$ is a standard set, in the sense that
there exist $\epsilon_0,\delta_0 > 0$ such that for all $x \in \Omega$
and $\epsilon\in(0,\epsilon_0)$, we have 
$\calL(B(x,\epsilon)\cap \Omega) \geq \delta_0 \calL(B(x,\epsilon)).$
\end{enumerate}
Condition~\ref{assm:supp_standard} arises frequently in the literature on statistical set 
estimation~\citep{cuevas1997,cuevas2009}, 
and prevents $\Omega$ from admitting cusps. 
Under this condition, we arrive at the following upper bound, which we prove in Appendix~\ref{app:pf_transport_1nn}. 
\begin{proposition}
\label{prop:transport_1nn} 
Let $P \in \calPac(\Omega)$ admit a density $p$ such that $\gamma^{-1} \leq p \leq \gamma$ over $\Omega$, 
for some $\gamma > 0$,
and let $Q \in \calPac(\Omega)$. Assume conditions~\ref{assm:curvature}
and  \ref{assm:supp_global}--\ref{assm:supp_standard} hold. Then,
$$\bbE \big\|\hat T_{nm}^{\mathrm{1NN}} - T_0\big\|_{L^2(P)}^2 \lesssim_{\lambda,\gamma,\epsilon_0,\delta_0} (\log n)^{2 }\kappa_{n\wedge m}.$$
Furthermore, if $\Omega = [0,1]^d$ and we additionally assume
that $\gamma^{-1} \leq q \leq \gamma$ over $\Omega$, then
$$\bbE \big\|\hat T_{nm}^{\mathrm{1NN}} - T_0\big\|_{L^2(P)}^2 \lesssim_{\lambda,\gamma,\epsilon_0,\delta_0} (\log n)^{2 }\widebar\kappa_{n\wedge m}.$$
\end{proposition}
Proposition~\ref{prop:transport_1nn} proves that the one-nearest neighbor estimator achieves
the minimax rate in equation~\eqref{eq:HR_lower_bound}, up to a polylogarithmic factor. This result is in stark contrast to standard risk bounds for 
$K$-nearest neighbor nonparametric regression, for which
the number $K$ of nearest neighbors is typically required to diverge in order to achieve the minimax estimation
rate of a Lipschitz continuous regression function~\citep{gyorfi2006}. 
Though increasing~$K$ reduces the variance of 
such estimators, 
in our setting, 
Propositions~\ref{prop:curvature_empirical}--\ref{prop:transport_1nn} suggest that the variance of
$\hat T_{nm}^{\mathrm{1NN}}$ is already dominated by its large bias, stemming from that of the in-sample coupling $\hat\pi$.
Therefore, the choice $K=1$ is sufficient to obtain a near-optimal rate. 
While the one-nearest neighbor estimator is simplest to analyze, it is natural
to expect that any  linear smoother with sufficiently small bandwidth may be used to smooth 
the in-sample coupling $\hat\pi_{nm}$ and lead to a similar rate.

\noindent \textbf{Convex Least Squares Estimator.}
Though nearly minimax optimal, the estimator $\hat T_{nm}^{\mathrm{1NN}}$ 
is typically not the gradient of a convex function,
and is therefore not an admissible optimal transport map in its own right. 
We next show how this property can be enforced using an estimator inspired
by nonparametric least squares regression. 
Let $\calJ_\lambda$ denote the class of functions $\varphi:\bbR^d \to \bbR$
which are  convex and have  $\lambda$-Lipschitz  gradients $\nabla\varphi$ over~$\Omega$. 
Define the least squares estimator
$$\hat T_{nm}^{\mathrm{LS}} =\nabla \hat \varphi_{nm}^{\mathrm{LS}}, \quad \text{where } \ 
  \hat \varphi_{nm}^{\mathrm{LS}} \in \argmin_{\varphi\in \calJ_\lambda}  \sum_{i=1}^n \sum_{j=1}^m \hat\pi_{ij}\norm{Y_{j} - \nabla \varphi(X_i)}^2.$$
The computation of the above infinite-dimensional optimization problem
can be reduced to that of solving a finite-dimensional quadratic 
program, by a direct extension of well-known solvers for shape-constrained
nonparametric regression with Lipschitz and convex constraints (cf. \cite{seijo2011}, \cite{mazumder2019}, and references therein). 
We obtain the following upper bound
by a simple extension of Proposition~\ref{prop:transport_1nn}. 
\begin{proposition}
\label{prop:transport_least_squares}
Proposition~\ref{prop:transport_1nn} continues to hold when $\hat T_{nm}^{\mathrm{1NN}}$
is replaced by $\hat T_{nm}^{\mathrm{LS}}$.  
\end{proposition}

\subsection{Upper Bounds for Two-Sample Estimators over $\bbT^d$}
\label{sec:two_sample_combined} 
We next study two-sample estimators under stronger smoothness assumptions on $P$ and $Q$. 
As discussed in Section~\ref{sec:two_sample_stability}, 
we do not know of a two-sample stability bound for the $L^2(P)$ loss which is analogous
to Theorem~\ref{thm:stability}, placing regularity conditions only on the population potential~$\varphi_0$. 
Therefore, unlike   
Theorem~\ref{thm:one_sample_wavelet}, in which smoothness conditions on $q$ and $\varphi_0$ were
sufficient to obtain sharp upper bounds, in the two-sample case
our analysis will also rely on the smoothness of  {\it estimators}~$\hat \varphi_{nm}$ of the potential~$\varphi_0$.
In order to quantify their regularity, we shall require a uniform analogue of Caffarelli's global
regularity theory (Theorem~\ref{thm:caffarelli--global}). Since we are unaware of such results
for generic compact domains $\Omega \subseteq \bbR^d$, we instead assume throughout 
this subsection that $\Omega$ is taken to be the   $d$-dimensional torus~$\bbT^d$, thus allowing
us to appeal to \cref{thm:torus_regularity}. 
We emphasize that our restriction to the torus represents a common approach
in the optimal transport literature, whereby numerical methods~\citep{benamou2000,loeper2005}
 and theoretical results~\citep{bonnotte2013,guittet2003,santambrogio2015}
 are first derived on the torus before being
extended to more generic domains. The torus is an idealized sampling domain, which we believe captures
the main qualitative features of our problem, while removing technical issues that arise 
from boundaries or lack of compactness. As such, it serves as a useful prototype for more general results on compact Euclidean
domains with boundaries. In order to illustrate this point, we will prove in the next subsection that our results over the torus extend
to generic Euclidean domains~$\Omega$, provided that one is willing to assume uniformity in Caffarelli's global regularity
theory on~$\Omega$. 

Though we impose periodicity   for technical purposes, we note that 
optimal transport has recently been used as a methodological tool in 
several applications involving periodic data, such as high energy physics (cf. \cite{komiske2019a,komiske2020}, 
where proton collisions occur in toric
colliders)   
and computational biology (cf. \cite{gonzalez2023}, where protein structures are recorded with pairs
of dihedral angles). More generally, toric data arises in a variety of applications in 
directional statistics (e.g.~\cite{klein2020}, \cite{wiechers2023}, etc.), and our results are naturally 
applicable to such settings.

We also note that periodicity constraints are 
commonly imposed in nonparametric estimation
problems to mitigate boundary issues~\citep{efromovich1999, krishnamurthy2014, han2020}.
In many such cases, an alternative is to assume that the 
underlying probability measures place sufficiently small mass near the boundary.
Such an assumption cannot be used in our context since, as before, we shall require all densities
to be bounded away from zero throughout their support.
Optimal estimation rates under Wasserstein distances
differ dramatically in the absence of a density lower bound condition~\citep{bobkov2019, weed2019a}, and we do not 
address this setting here.

We now turn to our main results. Recall the background on the quadratic optimal transport problem over $\bbT^d$ 
in  Section~\ref{sec:background_torus}.  
Let $P,Q \in \calPac(\bbT^d)$ be absolutely
continuous measures admitting respective $\bbZ^d$-periodic densities $p$ and $q$.
We now denote by $T_0$ the optimal transport map from $P$ to $Q$, with respect to the cost $d_{\bbT^d}^2$. 
As outlined in Proposition~\ref{prop:torus_ot}, $T_0$ is the gradient of a convex
potential $\varphi_0:\bbR^d\to\bbR$, and is uniquely determined $P$-almost everywhere.
We continue to denote by $\phi_0 = \norm\cdot^2 - 2\varphi_0$ and $\psi_0 = \norm\cdot^2 - 2\varphi_0^*$
a corresponding pair of Kantorovich potentials. 
Let $X_1, \dots, X_n \sim P$ and $Y_1, \dots, Y_m \sim Q$
 denote i.i.d. samples, which are  independent of each other, and let $\hat P_n,\hat Q_m$ respectively denote the distributions induced by   density estimators $\hat p_n, \hat q_m$ of $p, q$ over $\bbT^d$, to be defined below. 
Our aim is to bound the risk of the estimator  
\begin{align}
\label{eq:two_sample_density_estimator}
\hat T_{nm} = \nabla\hat\varphi_{nm} = \argmin_{T \in \calT(\hat P_n,\hat Q_m)} \int d_{\bbT^d}^2(T(x), x) d\hat P_n(x).
\end{align}
Note that $\hat P_n$ and $\hat Q_m$ are absolutely continuous, thus there indeed exists a unique
solution to the above minimization problem, by Proposition~\ref{prop:torus_ot}.
We continue to quantify the risk of $\hat T_{nm}$ in terms of the $L^2(P)$ loss
$$\big\|\hat T_{nm} - T_0\big\|_{L^2(P)}^2 = \int_{\bbT^d} \big\|\hat T_{nm}(x) - T_0(x)\big\|^2 dP(x).$$
Notice that the integrand on the right-hand side of the above display is 
$\bbZ^d$-periodic by Proposition~\ref{prop:torus_ot--map_periodicity}
and by the optimality of $\hat T_{nm}$ and $T_0$, thus it indeed defines a map $\bbT^d\to \bbR$. 
As before, we shall also obtain upper bounds on the bias
and risk of $\calW_2^2(\hat P_n, \hat Q_m)$
as a byproduct of our proofs. Indeed, our main results hinge upon 
the stability bounds derived in previous sections, which can easily be shown to hold in the present context.
\begin{proposition}
\label{prop:torus_stability} 
Assume $\varphi_0$ satisfies condition~\ref{assm:curvature}, in the sense
that $\varphi_0$ is a twice differentiable convex function over $\bbR^d$ satisfying
$\lambda^{-1} I_d \preceq \nabla^2 \varphi_0(x) \preceq \lambda I_d$ for all $x \in \bbR^d.$
Then, 
Theorem~\ref{thm:stability} and Proposition~\ref{thm:two_sample_stability}  
hold  with $\Omega = \bbT^d$.
\end{proposition}  

We now turn to the choice of density estimators $(\hat p_n, \hat q_m)$.
The absence of a boundary on the sampling domain $\bbT^d$ facilitates
the analysis of kernel density estimation, which will be our main focus in this section. 
We  also study   periodic wavelet
density estimators, similarly to the one-sample case,
but we defer this analysis to Appendix~\ref{app:pf_density_based} in the interest of brevity.

 Given a kernel $K \in \calC_c^\infty(\bbR^d)$ and  a bandwidth $h_n > 0$, write $K_{h_n} = h_n^{-d} K(\cdot/h_n)$,
 and define the kernel density estimators of $p$ and~$q$ by
$$\tilde p_n^{\mathrm{(ker)}}  = P_n \star K_{h_n} = \int_{\bbR^d} K_{h_n}(\cdot-z)dP_n(z),
\quad
  \tilde q_m^{\mathrm{(ker)}} = Q_m \star K_{h_m} = \int_{\bbR^d} K_{h_m}(\cdot-z)dQ_m(z).$$
Recall that integration over $\bbR^d$ with respect to a measure in $\calP(\bbT^d)$ is understood as integration
with respect to this measure extended to $\bbR^d$ via translation by $\bbZ^d$-periodicity. 
The above estimators may take on negative values,  
thus we again define the final density estimators by
$$\hat p_n^{(\mathrm{ker})} \propto \tilde p_n^{(\mathrm{ker})} I(\tilde p_n^{(\mathrm{ker})} \geq 0),\quad
  \hat q_m^{(\mathrm{ker})} \propto \tilde q_m^{(\mathrm{ker})} I(\tilde q_m^{(\mathrm{ker})} \geq 0),$$
  where the proportionality constants are to be chosen such that $\hat p_n^{(\mathrm{ker})}$
  and $\hat q_m^{(\mathrm{ker})}$ are densities. We also denote
  their induced probability distributions by $\hat P_n^{(\mathrm{ker})}$ and $\hat Q_m^{(\mathrm{ker)}}$. 
Furthermore, $\hat T_{nm}^{(\mathrm{ker})}$ denotes the optimal transport map between
these measures. 

We shall require the following condition on the kernel $K$, for given real numbers $\zeta,\kappa > 0$. 
\begin{enumerate}[leftmargin=1.65cm,listparindent=-\leftmargin,label=\textbf{K1($\zeta,\kappa$)}]   
\item  \label{assm:kernel}
$K \in \calC_c^\infty(\bbR^d) $ is an even kernel, whose Fourier transform $\calF[K]$
satisfies
\begin{equation}
\label{eq:fourier_condition_kernel}
\sup_{x \in \bbR^d\setminus\{0\}} |\calF[ K](x) - 1| \|x\|^{-\zeta} \leq \kappa.
\end{equation}
\end{enumerate} 
A sufficient condition for equation~\eqref{eq:fourier_condition_kernel}
to hold is for $K\in \calC_c^\infty(\bbR^d)$ to be a kernel of order $\beta = \lceil \zeta-1\rceil$. 
Such a statement appears for instance in \cite{tsybakov2008}
when $d=1$, and can easily be generalized to $d > 1$. Multivariate kernels of order $\beta$ which additionally
lie in $\calC_c^\infty(\bbR^d)$ can readily be defined; for example, one may start with a univariate even 
kernel $K_0 \in \calC_c^\infty(\bbR)$ of order $\beta$, constructed for instance using the procedure of~\cite{fan1992}, 
and then set  $K(x) = \prod_{i=1}^d K_0(x_i)$~\citep{gine2016}.

\cite{divol2021} stated that their work may be used to show that $\hat P_n^{(\mathrm{ker})}$ achieves  
a comparable rate of convergence as the boundary-corrected wavelet estimator $\hat P_n^{\mathrm{(bc)}}$, in Wasserstein distance.
 We provide a formal statement and proof of this fact in Lemma~\ref{lem:w2_kernel_rate} of Appendix~\ref{app:pf_kernel_based}, and use it to derive the following result. 
\begin{theorem}[Kernel Estimators]
\label{thm:two_sample_kernel}
Let the distributions $P,Q \in \calPac(\bbT^d)$ admit  
densities $p,q\in \calC^{\alpha-1}(\bbT^d; M,\gamma)$
for some $\alpha > 1$ and $M,\gamma > 0$.  
Assume further that  $K$ is a kernel satisfying condition~\hyperref[assm:kernel]{\textbf{K1($2\alpha,\kappa$)}} for some $\kappa > 0$. 
Let $h_n \asymp n^{-1/({d+2(\alpha-1))}}$. Then, there exists a constant
$C > 0$ depending only on $K,M, \gamma, \alpha$ such that the following statements hold.
\begin{thmlist}
\item \label{thm:two_sample_kernel--maps} (Optimal Transport Maps) We have, 
\begin{align*}
\bbE \big\|\hat T_{nm}^{(\mathrm{ker})} - T_0\big\|_{L^2(P)}^2  &\leq C \nmtratesker,
\quad \text{where } \tratesker := \begin{cases}
n^{-\frac{2\alpha}{2(\alpha-1) + d}}, & d \geq 3 \\  
\log n/n, & d = 2 \\
1/n, & d = 1.
\end{cases}
\end{align*} 
\item \label{thm:two_sample_kernel--wasserstein} (Wasserstein Distances) Assume further that $\alpha\not\in\bbN$. Then, \vspace{-0.2in}
\end{thmlist}
\end{theorem}  
\begin{align*}
\big|\bbE \calW_2^2(\hat P_n^{(\mathrm{ker})} , \hat Q_m^{(\mathrm{ker})} ) - \calW_2^2(P,Q)\big| &\leq C \nmtratesker,\\
\bbE \big|\calW_2^2(\hat P_n^{(\mathrm{ker})} , \hat Q_m^{(\mathrm{ker})} ) - \calW_2^2(P,Q)\big|^2 
 &\leq \left[C\nmtratesker {+}  \sqrt{\frac{\Var_P[\phi_0(X)]}{n}  {+} \frac{\Var_Q[\psi_0(Y)]}{m}}\right]^2.
\end{align*}   
Theorem~\ref{thm:two_sample_kernel} shows that the plugin estimators $\hat T_{nm}$ and $\calW_2^2(\hat P_n,\hat Q_m)$
achieve similar convergence rates as in the one-sample setting of Theorem~\ref{thm:one_sample_wavelet}.
Unlike the latter result, we also note that Theorem~\ref{thm:two_sample_kernel} places
no conditions on the regularity of~$T_0$ or~$\varphi_0$. Indeed, over $\bbT^d$, these can be inferred from 
the assumption $p,q \in \calC^{\alpha-1}(\bbT^d;M,\gamma)$, due to  Theorem~\ref{thm:torus_regularity}.
We  exclude the case $\alpha\in\bbN$ from Theorem~\ref{thm:two_sample_kernel--wasserstein} due in part to our use of this result.
Nevertheless, even when $\alpha\in\bbN$, 
Theorem~\ref{thm:two_sample_kernel--wasserstein} implies that
$$\big|\bbE \calW_2^2(\hat P_n, \hat Q_m) - \calW_2^2(P,Q)\big| \lesssim_\epsilon  R_{K,n\wedge m}^{1-\epsilon}(\alpha),$$
for any $\epsilon > 0$, and similarly for the risk of $\calW_2^2(\hat P_n, \hat Q_m)$.

If one is willing to place assumptions on the regularity of the potentials
$\varphi_0$ and $\varphi_0^*$, then an analogue
of Theorem~\ref{thm:two_sample_kernel--wasserstein} 
can be derived when the periodicity assumption is removed, and
the sampling domain is simply  the unit cube $[0,1]^d$. 
Such a result is stated in Proposition~\ref{thm:two_sample_density--wasserstein_cube} of Appendix~\ref{app:pf_density_based},
and is made possible by the fact
that Proposition~\ref{thm:two_sample_stability}
does not require any regularity of the fitted potentials.

When $d=2$, Theorem~\ref{thm:two_sample_kernel--maps} exhibits an improved
convergence rate relative to  Theorem~\ref{thm:one_sample_wavelet--map},
scaling as $\log n/n$ instead of $(\log n)^2/n$, which we now briefly discuss. 
This rate arises from our upper bound on $\bbE \calW_2^2(\hat P_n\sker, P)$ 
in Lemma~\ref{lem:w2_kernel_rate}, which makes use of the inequality~\eqref{eq:peyre} comparing $\calW_2$ to
a negative-order homogeneous Sobolev norm  \citep{peyre2018}. This last implies
\begin{align}
\label{eq:besov_ub_torus}
\calW_2(\hat P_n\sker, P) \lesssim\|\hat p_n\sker - p\|_{\dot H^{-1}(\bbT^d)} 
\asymp \|\hat p_n\sker - p\|_{\calB_{2,2}^{-1}(\bbT^d)}.
\end{align}
In contrast, when $P \in \calPac([0,1]^d)$, our upper bounds for wavelet estimators 
(and implicitly for empirical estimators in Corollary~\ref{cor:one_sample_emp_hypercube}) employed the following distinct relation, arising from the work of~\cite{weed2019a},
\begin{align}
\label{eq:besov_ub_cube}
 W_2(\hat P_n\sbc, P) \lesssim \|\hat p_n\sbc - p\|_{\calB_{2,1}^{-1}([0,1]^d)},
 \end{align}
and similarly for the estimator $\hat P_n\sper$ described in Appendix~\ref{app:pf_density_based}. It can be seen that the   $\calB_{2,2}^{-1}$ norm
is weaker than the $\calB_{2,1}^{-1}$ norm. While either of these norms provide sufficiently tight
upper bounds in equations~\eqref{eq:besov_ub_torus} and~\eqref{eq:besov_ub_cube} to obtain the minimax rate for density estimation 
in Wasserstein distance when $d \neq 2$, the former allows for a tighter logarithmic factor
to be derived when $d=2$. Inspired by the celebrated Ajtai--Koml\'os--Tusn\'ady matching theorem~\citep{ajtai1984,talagrand1992}, 
it is natural to conjecture that
the rate $\log n/n$ in the definition of $\tratesker$ cannot be further improved when $d=2$, for any of the conclusions of Theorem~\ref{thm:two_sample_kernel}. 

Theorem~\ref{thm:two_sample_kernel} is proved in Appendix~\ref{app:pf_kernel_based}, where the main difficulty is to show that 
the evaluation $L(\hat Q_m^{(\mathrm{ker})})$, of the linear functional $L$ from equation~\eqref{eq:first_order_von_mises},
has bias decaying at the quadratic rate $h_m^{2\alpha}$. As for our analysis of wavelet estimators,
this rate improves upon the naive  upper bound $|\bbE L(\hat Q_m^{(\mathrm{ker})}) |\lesssim h_m^{\alpha-1}$, 
which could have been deduced from the traditional bias of kernel density estimators in estimating
an $(\alpha-1)$-H\"older continuous density~\citep{tsybakov2008}.
Similar considerations arise
in the analysis of  kernel-based estimators for other important functionals, such as the integral of a squared density~\citep{gine2008a}.

\begin{remark}[Dependence Between Samples]
Our assumption of independence between the sample points $X_i$ and $Y_j$  
is only used to derive the sharp constant in the final term of Theorem~\ref{thm:two_sample_kernel--wasserstein},
which implies that, when $2(\alpha+1) > d$, 
$$\bbE\big|W_2^2(\hat P_n,\hat Q_m) - W_2^2(P,Q)\big|^2 \leq (1+o(1))\left( \frac{\Var_P[\phi_0(X)]}{n} + \frac{\Var_Q[\psi_0(Y)]}{m}\right).$$
If one is willing to inflate these leading constants, then all assertions of Theorem~\ref{thm:two_sample_kernel}
continue to hold under arbitrary dependence structures between the two i.i.d. samples. 
\end{remark}

\subsection{Toward Two-Sample Estimation over Smooth Domains}
\label{sec:smooth_domains} 
 
Our aim is now to show that an analogue 
of Theorem~\ref{thm:two_sample_kernel} holds
over generic domains $\Omega \subseteq \bbR^d$, provided that one is willing to assume
that Caffarelli's global regularity theorem (Theorem~\ref{thm:caffarelli--global}) 
holds uniformly over $\Omega$. Specifically, we will use the following conditions throughout this section.
\begin{enumerate}[leftmargin=1.65cm,listparindent=-\leftmargin,label=\textbf{(C1)}]   
\item  \label{assm:smooth_domain} $\Omega$ 
is a known compact, convex subset of $\bbR^d$, such that $\partial\Omega$ is $\calC^\infty$and $\calL(\Omega) = 1$. 
\end{enumerate}
\begin{enumerate}[leftmargin=1.65cm,listparindent=-\leftmargin,label=\textbf{(C2)}]   
\item  \label{assm:caffarelli} There exists $\epsilon_0 > 0$ such that
for any $M,\gamma > 0$ and $\epsilon \in (0,\epsilon_0)$, there exists a constant $C > 0$ depending only on $\Omega,M,\gamma,\epsilon$ such that
for any densities $\hat p,\hat q \in \calC^{\epsilon}(\Omega;M,\gamma)$, 
the unique mean-zero Brenier potential $\hat\varphi$
whose gradient pushes forward $\hat p$ onto $\hat q$ satisfies
$$\|\hat \varphi\|_{\calC^{2}(\Omega)} \leq C.$$
\end{enumerate}
As discussed previously, we are only able to verify condition~\ref{assm:caffarelli} 
when $\Omega$ is replaced by the torus $\bbT^d$. However, in view of Theorem~\ref{thm:caffarelli}, it is natural
to conjecture that condition~\ref{assm:caffarelli} is satisfied for other domains of the type~\ref{assm:smooth_domain}, 
and if such a result is proven in future work, then the bounds appearing in this section
can be applied. For completeness, we will also state a one-sample result over $\Omega$, for which condition~\ref{assm:caffarelli}
is not needed.

It is well-known that kernel density estimators suffer from {leading-order} boundary bias,
and are thus not minimax optimal for estimating strictly positive densities on compact subsets of $\bbR^d$.
In order to develop a minimax optimal estimator for densities supported   on $\Omega$, 
we will impose Neumann boundary conditions on the densities, and 
we will introduce an orthonormal basis of $L^2(\Omega)$ generated by the eigenfunctions of the Neumann Laplacian.
To elaborate, 
define $H_N^2(\Omega)$ to be the set of functions $u \in H^2(\Omega)\cap L_0^2(\Omega)$
satisfying 
$$\frac{\partial u}{\partial \nu} = 0, \quad \text{over } \partial\Omega,$$
where $\nu$ is an outward-pointing normal vector to $\partial\Omega$, and the 
{normal derivative}
is to be understood in the weak sense.
Under condition~\ref{assm:smooth_domain}, it
 is a standard fact that the negative Laplace operator $-\Delta$ is a self-adjoint
bijection of $H_N^2(\Omega)$ onto $ L_0^2(\Omega)$, which 
admits a real and discrete spectrum $0 < \lambda_1 \leq \lambda_2 \leq   \dots$, 
with corresponding eigenfunctions $\{\eta_\ell\}_{\ell=1}^\infty \subseteq H^2_N(\Omega)$
{~\citep{dunlop2020,evans2022}}. 
The latter form an orthonormal basis of $L_0^2(\Omega)$. 

Let the distributions $P,Q \in \calPac(\Omega)$ admit densities $p,q \in L^2(\Omega)$, 
and let $X_1, \dots, X_n \sim P$ and 
$Y_1,\dots, Y_m \sim Q$ be i.i.d. observations. 
Under condition~\ref{assm:smooth_domain}, the densities may be expanded~as
$$p = 1+\sum_{\ell=1}^\infty \alpha_\ell \eta_\ell, \quad q = 1+\sum_{\ell=1}^\infty \beta_\ell \eta_\ell,$$
where $\alpha_\ell = \int \eta_\ell dP$ and $\beta_\ell = \int \eta_\ell dQ$. 
Let $\tau \in \calC^\infty(\bbR_+)$ be a smooth approximation
to the indicator function $I(\cdot  < 1)$. Specifically, assume 
that $\tau$ is a nonincreasing and smooth function such that 
$\tau(x) = 1$ for all $x < 1/2$, and $\tau(x) = 0$ for all $x \geq 1$.  
Given an integer $L_n\geq 1$, set
$$\omega_\ell  = \tau(\lambda_\ell/\lambda_{L_n}), \quad \ell=1, 2, \dots,$$
and define the density estimators  
$$\tilde p_n^{(\mathrm{lap})} = 1 + \sum_{\ell=1}^{L_n} \omega_\ell \halpha_\ell \eta_\ell,\quad 
  \tilde q_m^{(\mathrm{lap})} = 1 + \sum_{\ell=1}^{L_m}    \omega_\ell \hat\beta_\ell\eta_\ell,$$
where $\halpha_\ell = \int\eta_\ell dP_n$,  $\hbeta_\ell = \int \eta_\ell dQ_m$ for $\ell=1,2,\dots.$
Density estimators of this type have also appeared in the works of~\cite{hendriks1990} and~\cite{cleanthous2020}.  
They may be thought of as truncated series estimators for which the truncation
is smoothed by the weight function $\tau$. As such, they are closely related to kernel density
estimators. In fact, if one were to replace~$\Omega$ by the torus~$\bbT^d$ (in which case the Neumann boundary
condition is replaced by the periodic boundary condition), then $\tilde p_n^{(\mathrm{lap})} $
would precisely be the kernel density estimator {whose
kernel is the inverse Fourier
transform of the map $\xi \in \bbR^d \mapsto \tau(\|2\pi\xi\|^2)$,}
 and whose bandwidth is $\lambda_{L_n}^{-1/2}$.  
We also note that if one were to choose the nonsmooth function $\tau = I(\cdot < 1)$, then $\tilde p_n^{(\mathrm{lap})} $
would reduce to a traditional series estimator, but our analysis does not extend to this case: the smoothness of $\tau$ is crucial for our use of $L^r(\Omega)$ multiplier arguments,
as we discuss further in  Remark~\ref{rem:fefferman} of Appendix~\ref{app:smooth_domains}.  
  
As in previous sections, we define the final estimators to be the densities given by
$$\hat p_n^{(\mathrm{lap})} \propto \tilde p_n^{(\mathrm{lap})} I(\tilde p_n^{(\mathrm{lap})}  \geq 0), \quad 
  \hat q_m^{(\mathrm{lap})} \propto \tilde q_m^{(\mathrm{lap})} I(\tilde q_m^{(\mathrm{lap})}  \geq 0),$$
  and we let $\hat P_n^{\mathrm{(lap)}}$ and $\hat Q_m^{\mathrm{(lap)}}$ be the induced distributions. 
  Furthermore, let $\hat T_{nm}^{\mathrm{(lap)}}$ be the optimal transport map from $\hat P_n^{\mathrm{(lap)}}$
  to $\hat Q_m^{\mathrm{(lap)}}$, and $\widebar T_m^{\mathrm{(lap)}}$ the optimal transport
  map from $P$ to $\hat Q_m^{\mathrm{(lap)}}$. 
We omit the superscripts ``lap'' for the remainder of this section.
In order to state convergence rates for these estimators, we will work over the  constrained
H\"older spaces
$$\calC_N^s(\Omega) = \left\{ u \in \calC^s(\Omega): \frac{\partial \Delta^{j}u}{\partial\nu} = 0 \text{ on } \partial\Omega, ~ 
0 \leq j \leq \left\lfloor \frac{s-1}{2}\right\rfloor \right\},$$ 
and the associated balls 
$\calC_N^s(\Omega;M) = \calC^s(\Omega;M) \cap \calC^s_N(\Omega)$ and 
$\calC_N^s(\Omega;M,\gamma) = \calC^s(\Omega;M,\gamma) \cap \calC_N^s(\Omega)$, for $M,\gamma,s > 0$. 
Note that $\calC^s(\Omega)  =  \calC^s_N(\Omega)$ for $s < 1$. 
Our main result is the following. 
\vspace{-0.2in}
\begin{theorem}
\label{thm:smooth_domains}
Let $\Omega$ be a domain satisfying condition~\ref{assm:smooth_domain}. 
Assume the distributions $P,Q \in \calPac(\Omega)$ admit  
densities $p,q\in \calC_N^{\alpha-1}(\Omega; M,\gamma)$
for some $\alpha > 1$, $\alpha\not\in\bbN$,  and $M,\gamma > 0$.
Let $L_n^{1/d} \asymp n^{1/({d+2(\alpha-1))}}$. Then, there exists a constant
$C > 0$ depending only on $\Omega,M, \gamma, \alpha,\tau$ such that the following statements hold.
\begin{thmlist}
\item \label{thm:smooth_domains--one_sample} (One-Sample) 
Assume $\varphi_0 \in \calC^2(\Omega;M)$. Then, 
\begin{align*}
\bbE \big\|\widebar T_{m} - T_0\big\|_{L^2(P)}^2  &\leq C \mtratesker. 
\end{align*}
\item \label{thm:smooth_domains--two_sample} (Two-Sample) Assume  that
condition~\ref{assm:caffarelli} holds. 
Then,
\begin{align*}
\bbE \big\|\hat T_{nm} - T_0\big\|_{L^2(P)}^2  &\leq C \nmtratesker,
\end{align*} 
 
\end{thmlist}
\end{theorem} 
Under the strong condition~\ref{assm:caffarelli}, Theorem~\ref{thm:smooth_domains--two_sample} shows that
our two-sample results on transport 
map estimation over the torus can be extended to 
generic domains\footnote{While we assume for simplicity that $P$ and $Q$
 share the same support $\Omega$, Theorem~\ref{thm:smooth_domains}
 can readily be extended to the case where $P$ and $Q$ are   supported on 
 distinct domains which both satisfy condition~\ref{assm:smooth_domain},
 with natural modifications to the statement of condition~\ref{assm:caffarelli}
 and to the definitions of $\hat p_n$, $\hat q_m$.}
 $\Omega$, 
provided that one is willing to place Neumann boundary conditions   
on the true densities. While other boundary conditions may have been used in our analysis, it is important that they 
be chosen such that $p,q$ are permitted to be smooth and strictly positive over $\Omega$; in particular, 
one cannot impose the Dirichlet  condition $p=q = 0 $ over $\partial\Omega$. Our current
boundary conditions are satisfied by 
a wide range of densities, such as those whose gradient vanishes at the boundary, 
or those which are equal to finite linear combinations of the eigenfunctions $\{\eta_\ell\}_{\ell=1}^\infty$. 
We also emphasize that Theorem~\ref{thm:smooth_domains} imposes no boundary conditions when $\alpha < 2$.

To prove Theorem~\ref{thm:smooth_domains}, our primary contribution is to derive Propositions~\ref{prop:density_risk_smooth_domain} 
and~\ref{prop:smooth_domain_neg_exp} 
of Appendix~\ref{app:smooth_domains}, which state
convergence
rates for $\hat p_n$ under the spectral   Sobolev norms $\calH^{t,r}(\Omega)$ which we define therein.
Here, $t \in \bbR$ and $r > 1$ are smoothness and integrability indices, respectively. Under some conditions on $r$, our
results imply that for large enough $d$, and $-\infty < t <  \alpha-1$,   it holds that
\begin{equation}
\label{eq:illustrate_lap_rate}
\bbE \|\hat p_n - p\|_{\calH^{t,r}(\Omega)} \lesssim n^{-\frac{\alpha-1-t}{2(\alpha-1) + d}},
\end{equation} 
assuming the same conditions as Theorem~\ref{thm:smooth_domains}. This bound has two implications:
\begin{itemize}
\item  On the one hand, taking $t=-1$, $r=2$,  and applying equation~\eqref{eq:peyre},
we deduce that $\hat p_n$ is a minimax optimal density estimator under the 2-Wasserstein distance. To the best of our knowledge, this is the only
known convergence rate for density estimation under the Wasserstein distance over Euclidean domains
with non-rectangular boundary (apart from the special case $\alpha\in (1,2]$, for which density estimation
has been studied without support assumptions by~\cite{weed2019a}).
We also highlight that~\cite{divol2022} has studied the case of boundary-free manifolds. 
\item On the other hand, take $t=\epsilon/2$ for some $\epsilon > 0$, and let $r > 2d/\epsilon$. 
Then, using a Sobolev embedding argument, equation~\eqref{eq:illustrate_lap_rate} leads to
a convergence rate for $\hat p_n$ under the $\calC^\epsilon(\Omega)$ norm.
In particular, this allows us to infer that $\hat p_n$ and $\hat q_m$ satisfy the conditions on the densities 
in assumption~\ref{assm:caffarelli},
with high probability, thus allowing us to infer that $\hat\varphi_{nm}$ is of class $\calC^{2+\epsilon}(\Omega)$
with uniformly bounded H\"older norm.
  \end{itemize}
We close this section by noting that, if one is willing to settle for pointwise asympotics, 
then Theorem~\ref{thm:smooth_domains--one_sample} can be stated without 
any smoothness assumptions on the potential $\varphi_0$. Indeed, the following
is a consequence of Caffarelli's regularity theory (Theorem~\ref{thm:caffarelli}). 
\begin{corollary}
\label{cor:one_sample_smooth_domain}
Let $\Omega$ be a domain satisfying condition~\ref{assm:smooth_domain}. 
Assume the distributions $P,Q \in \calPac(\Omega)$ admit  
densities $p,q\in \calC_N^{\alpha-1}(\Omega)$
for some $\alpha > 1$.
Let $L_m^{1/d} \asymp m^{1/({d+2(\alpha-1))}}$. Then, there exists a constant
$C > 0$ depending  on $\Omega,\alpha,\tau,p,q$ such that 
\begin{align*}
\bbE \big\|\widebar T_{m} - T_0\big\|_{L^2(P)}^2  &\leq C \mtratesker. 
\end{align*}
\end{corollary}

\section{Efficient Statistical Inference for Wasserstein Distances}
\label{sec:inference}

We now complement our results on estimation rates for Wasserstein distances
by deriving limit laws, in Section~\ref{sec:clt}, 
for the   plugin estimators studied in Sections~\ref{sec:one_sample}--\ref{sec:two_sample}. 
We then derive lower bounds in Section~\ref{sec:efficiency}, 
which show that these estimators are asymptotically efficient under suitable conditions.

\subsection{Central Limit Theorems for Smooth Wasserstein Distances}
\label{sec:clt}  
 
Recall that we respectively denote by
$P_n, \hat P_n^{(\mathrm{bc})}, 
\hat P_n^{(\mathrm{ker})}$, 
the 
empirical measure and the distributions induced
by the  
boundary-corrected  and   
kernel   
density estimators 
of $p$ (and similarly for $q$), as  defined in Sections~\ref{sec:one_sample}--\ref{sec:two_sample}.
 Given a smoothness parameter $\alpha > 1$
to be specified, 
let their tuning parameters
be chosen as {$2^{J_n} \asymp h_n^{-1}  \asymp n^{1/ (d+2(\alpha-1))}$}, and assume that the kernel $K$ satisfies 
condition~\hyperref[assm:kernel]{\textbf{K1($2\alpha,\kappa$)}} for some $\kappa > 0$. 
Furthermore, in what follows,   $X_1, \dots, X_n \sim P$ and $Y_1, \dots, Y_m \sim Q$ 
denote i.i.d. samples which are independent of each other, and we write
$$\sigma_\rho^2 = (1-\rho) \Var_P[\phi_0(X)] + \rho \Var_Q[\psi_0(Y)],\quad \text{for any } \rho\in [0,1],$$
where we recall that $\phi_0 = \|\cdot\|^2-2\varphi_0$ and $\psi_0 = \|\cdot\|^2 - 2\varphi_0^*$, 
for any given Brenier potential $\varphi_0$ in the optimal transport problem from $P$ to $Q$. 
For the various estimators $\hat P_n$ and $\hat Q_m$ under consideration, we will derive central limit theorems
of the form 
\begin{align}
\label{eq:clt_one_sample}
\sqrt n \Big( W_2^2(\hat P_n, Q) - W_2^2(P,Q)\Big) &\rightsquigarrow N(0,\sigma_0^2),\quad \text{as } n\to\infty,
\quad\text{and} \\
\label{eq:clt_two_sample}
\sqrt{\frac{nm}{n+m}} \Big(W_2^2(\hat P_n, \hat Q_m) - W_2^2(P,Q)\Big)
&\rightsquigarrow N(0, \sigma_\rho^2),\quad 
\text{as } n,m\to\infty, \ \frac n {n+m} \to \rho,
\end{align}
for some $\rho\in [0,1]$. Our main result is the following.
 \begin{theorem}[Central Limit Theorems]
\label{thm:clt}
Assume that $P,Q \in \calPac(\Omega)$  
admit positive and bounded densities $p,q$   over $\Omega$. Then, the following assertions hold.
\begin{thmlist} 
\item \label{thm:clt--torus} (Density Estimation over the Torus) Let $\Omega = \bbT^d$ and
assume $p,q \in \calC^{\alpha-1}(\Omega)$ for some $\alpha > 1$ satisfying $2(\alpha+1) > d$.
Then, equations~\eqref{eq:clt_one_sample}--\eqref{eq:clt_two_sample}
hold when
$$(\hat P_n,\hat Q_m) = (\hat P_n\sker, \hat Q_m\sker).$$

\item \label{thm:clt--hypercube} (Density Estimation over the Hypercube) Let $\Omega = [0,1]^d$, and assume $p,q \in \calC^{\alpha-1}(\Omega)$
for some $\alpha > 1$ satisfying $2(\alpha+1) > d$. Assume additionally that 
$\varphi_0  \in \calC^{\alpha+1}(\Omega)$. Then, 
equation~\eqref{eq:clt_one_sample} holds when 
$$(\hat P_n,\hat Q_m) = (\hat P_n^{(\mathrm{bc})}, \hat Q_m^{\mathrm{(bc)}}).$$
Furthermore, equation~\eqref{eq:clt_two_sample} holds under the additional condition $\varphi_0^* \in \calC^{\alpha+1}(\Omega)$. 
\vspace{0.08in}
\item \label{thm:clt--empirical} (Empirical Measures) Let $\Omega$ be either $\bbT^d$ or $[0,1]^d$. Assume $d \leq 3$, 
and $\varphi_0 \in \calC^2(\Omega)$. 
Then equations~\eqref{eq:clt_one_sample}--\eqref{eq:clt_two_sample}
hold when 
$$(\hat P_n,\hat Q_m)=(P_n,Q_m).$$
\end{thmlist}
\end{theorem}

To the best of our knowledge, Theorem~\ref{thm:clt} provides the first known central limit theorems 
for nonparametric plugin estimators of the squared Wasserstein distance  in arbitrary dimension  $d \geq 1$ 
which are centered at their population counterpart $W_2^2(P,Q)$. 

We emphasize that the parametric scaling in the above result is made possible by the   
smoothness condition $2(\alpha+1) > d$. 
We do not generally expect that a central limit theorem for $W_2^2( \hat P_n,Q)$ centered
at $W_2^2(P,Q)$ can be obtained
when $d > 2(\alpha+1)$, as  the squared bias of this estimator may then dominate its variance. 
In contrast, even in the absence of  smoothness conditions, \cite{delbarrio2019a}   derived limit laws of the form 
\begin{equation}
\label{eq:delbarrio_clt}
\sqrt n \left( W_2^2(P_n, Q) - \bbE W_2^2(P_n,Q) \right) \rightsquigarrow N(0,\Var[\phi_0(X)]),
\quad n\to \infty,
\end{equation}
and two-sample analogues, for any $d \geq 1$. While such results are important and hold
under milder regularity conditions than those
of Theorem~\ref{thm:clt}, their centering sequence is a barrier to their use for statistical inference
for Wasserstein distances. The low-dimensional case $d \leq 3$ is an exception, 
in which the sequence $\bbE W_2^2(P,Q_n)$ can be replaced by $W_2^2(P,Q)$, as we show in
Theorem~\ref{thm:clt--empirical}. This  fact can  be deduced from  our bias bounds in Corollary~\ref{cor:one_sample_emp_hypercube}; a similar 
observation for $d \leq 3$ was also made in the recent work of~\cite{hundrieser2022},
under weaker assumptions than ours. 
 
Theorem~\ref{thm:clt} is a consequence of the stability bounds in Theorem~\ref{thm:stability}
and Proposition~\ref{thm:two_sample_stability}, which we use to show that $W_2^2(\hat P_n, Q)-W_2^2(P,Q)$ asymptotically
has same distribution as the linear functional $F(\hat P_n) = \int\phi_0 d(\hat P_n-P)$.
We defer the proof to Appendix~\ref{app:pf_clts}. 
Though our arguments differ significantly
from those used by \cite{delbarrio2019a}, 
this  functional $F$ also plays an important role in their work. Indeed, they prove that
$n\Var[W_2^2(P_n,Q) - F(P_n)] = o(1)$ 
under mild conditions. 
In Appendix~\ref{app:pf_clts_alternate}, we provide an alternate proof of Theorem~\ref{thm:clt} which does not make
use of our stability bounds, but which instead combines a generalization of the proof strategy of~\cite{delbarrio2019a}, 
together with our convergence rates for optimal transport maps in Theorems~\ref{thm:one_sample_wavelet}, \ref{thm:two_sample_density}
and~\ref{thm:two_sample_kernel}.

The variance $\sigma_\rho^2$ is positive
if and only if $\phi_0$ and $\psi_0$ are non-constant, thus
the distributional limits in Theorem~\ref{thm:clt} are non-degenerate whenever $P \neq Q$. 
When $P = Q$, it could already have been deduced from Lemma~\ref{lem:wavelet_wasserstein} that, for instance,
the correct scaling for  $W_2^2(\hat P_n\sbc, Q)$ is of larger order than $\sqrt n$. We leave open the question of obtaining
limit laws under this~regime.

The variances appearing in Theorem~\ref{thm:clt} can be consistently estimated 
using   estimators for the Kantorovich potentials $\phi_0$ and $\psi_0$. 
Indeed, using a qualitative stability result for Kantorovich potentials~\citep{santambrogio2015},
we show in Proposition~\ref{prop:variance_estimation} of Appendix~\ref{app:variance_estimation}
that for any of the estimators $(\hat P_n, \hat Q_m)$ described in Theorem~\ref{thm:clt}, 
if $(\hat\phi_{nm},\hat\psi_{nm})$ is a bounded pair of Kantorovich potentials 
in the optimal transport problem from $\hat P_n$ to $\hat Q_m$, then, 
\begin{align*} 
\hat\sigma_{0,nm}^2 &:= \Var_{U\sim P_n}[\hphi_{nm}(U)]
 \overset{p}{\longrightarrow} \sigma_0^2 ,~~
\text{and,}  \quad 
\hat\sigma_{1,nm}^2  := \Var_{V\sim Q_m}[ \hat\psi_{nm}  (V)] \overset{p}{\longrightarrow} \sigma_1^2 .
\end{align*} 
Letting $\hsigma_{nm}^2 = \frac{m\hsigma_{0,nm}^2 +n \hsigma_{1,nm}^2}{n+m}$, 
we  deduce  
from Theorem~\ref{thm:clt--hypercube} that  
for any $\delta \in (0,1)$,  
$$W_2^2(\hat P_n\sbc, \hat Q_m\sbc) \pm \hsigma_{nm} z_{\delta/2}\sqrt{\frac{n+m}{nm}}$$
is an asymptotic, two-sample $(1-\delta)$-confidence interval
for $W_2^2(P,Q)$,
assuming $P \neq Q$. Here $z_{\delta/2}$ denotes the $\delta/2$ quantile of the standard Gaussian distribution.
 To the best of our knowledge, this is the first practical confidence interval for the Wasserstein distance 
between absolutely continuous distributions in arbitrary dimension, albeit
under the strong assumptions that $2(\alpha+1) > d$, and that
the underlying domain $\Omega$ is known and of appropriate type.

\subsection{Efficiency Lower Bounds for Estimating the Wasserstein Distance}
\label{sec:efficiency}
Our aim is now to derive  efficiency lower bounds, showing that
the asymptotic variances in Theorem~\ref{thm:clt} cannot be improved by any other regular estimatorof $W_2^2(P,Q)$.  
In discussing semiparametric efficiency theory, we follow the definitions and notation
of~\cite{vandervaart1998,vandervaart2002}. We begin with a derivation of the efficient influence function
of the functional 
\begin{equation*}
\Phi_Q:\calP(\Omega) \to \bbR, \quad \Phi_Q(P)   = W_2^2(P,Q),
\end{equation*}
where $\Omega$ is either $\bbT^d$ or a subset of $\bbR^d$, and $Q \in \calPac(\Omega)$
is given. \citeauthor{santambrogio2015} (\citeyear{santambrogio2015}, Proposition 7.17) has previously derived the first variation of this functional. 
The
following is a version of their result, stated
in a language
suitable for our development.
\begin{lemma}[Efficient Influence Function]
\label{lem:influence}
Let $\Omega$ be $\bbT^d$ or any connected and compact subset of $\bbR^d$, and let $P,Q \in \calPac(\Omega)$. Assume that the density of at least one of $P$ and $Q$ 
is positive over $\Omega$. Let $(\phi_0,\psi_0)$ denote a pair of Kantorovich potentials
in the optimal transport problem from $P$ to $Q$, uniquely defined up to translation by a constant, and  define the map
\begin{alignat*}{2}
 \tilde \Phi_{(P,Q)}(x) = \phi_0(x) - \int\phi_0 dP,\quad x \in \Omega.
\end{alignat*}
Let $\dot\calP_P \subseteq L^2_0(P)$ be any tangent set containing $\tilde \Phi_{(P,Q)}$. Then, the functional $\Phi_Q$ is  differentiable
relative to $\dot\calP_P$, with efficient influence function given by $\tilde \Phi_{(P,Q)}$.
\end{lemma}
Lemma~\ref{lem:influence} is proved in Appendix~\ref{app:pf_efficiency}.
 The assumption that  $P$ or $Q$ have support equal to $\Omega$ 
is only used to ensure that  $\phi_0$ is unique, up to translation by a constant
(cf. Proposition 7.18 of~\cite{santambrogio2015}). 
While this condition is not necessary~\citep{staudt2022}, we retain it for simplicity since
we require it for our upper bounds. 

By combining this result with the Convolution Theorem (\cite{vandervaart1998}, Theorem~25.20), it immediately follows 
that any regular estimator sequence of $\Phi_Q(P)$ has asymptotic variance bounded below  
by $\Var_P[\phi_0(X)]/ n$. The one-sample plugin estimators in Theorem~\ref{thm:clt} are thus optimal 
among regular estimators. A similar remark was also made in the recent independent work of~\cite{goldfeld2022}, 
which studies efficient statistical inference for several variants of the Wasserstein distance.

We next complement this result with an asymptotic minimax lower bound, 
which relaxes the assumption of regularity of such estimator sequences,
at the expense of only comparing their worst-case risk. 
In this case, we also consider the two-sample setting. 
Using a construction of \cite{vandervaart1998}, 
we fix two differentiable paths $(P_{t,h_1})_{t \geq 0}$ and $(Q_{t,h_2})_{t \geq 0}$,
for any $(h_1,h_2) \in \bbR^2$, 
with respective score functions $h_1 \widetilde \Phi_{(P,Q)}$ and 
$h_2 \widetilde \Psi_{(P,Q)}$, where 
$\widetilde \Psi_{(P,Q)}(y) := \psi_0(y) - \int \psi_0 dQ.$
These paths are defined in equations~(\ref{eq:diff_opath_main_P}--\ref{eq:diff_opath_main_Q})
of Appendix~\ref{app:pf_efficiency}, and we use them 
to obtain the following
asymptotic minimax lower bound. 
\begin{theorem}[Asymptotic Minimax Lower Bound over $\bbT^d$]
\label{thm:asymptotic_minimax} 
Given $M,\gamma > 0$ and $\alpha > 1$, let $P, Q \in \calPac(\bbT^d)$
admit densities $p,q \in \calC^{\alpha-1}(\bbT^d;M,\gamma)$.  
Let $(\phi_0,\psi_0)$ denote a pair of Kantorovich potentials between $P$ and $Q$,
unique up to translation by a constant. 
Then, there exist $\widebar M, \widebar \gamma,\widebar u > 0$, depending only on $M, \gamma,\alpha$,
such that $P_{t,h_1}$ and $Q_{t,h_2}$ admit densities   in $\calC^{\alpha-1}(\bbT^d;\widebar M, \widebar \gamma)$,
for all $t > 0$ and $h_1,h_2 \in \bbR$ satisfying $t(|h_1|\vee  |h_2|) \leq  \widebar u$. Furthermore, 
\begin{thmlist} 
\item \label{thm:asymptotic_minimax--one_sample} (One Sample)
For any estimator sequence $(U_n)_{n\geq 1}$, we have
$$\sup_{\substack{\calI \subseteq \bbR \\ |\calI| < \infty}} \liminf_{n\to\infty} \sup_{h \in \calI} n\bbE_{n,h}
\big|U_n - \Phi_Q(P_{n^{-1/2}, h})\big|^2  \geq \Var_P[\phi_0(X)].$$
where   $\bbE_{n,h}$ denotes the expectation taken over the probability measure ${P_{n^{-1/2},h}^{\otimes n}}$.
\item \label{thm:asymptotic_minimax--two_sample} (Two Sample) For any estimator sequence $(U_{nm})_{n,m\geq 1}$, we have
\begin{equation*}
\begin{multlined}[\linewidth]
\sup_{\substack{\calI \subseteq \bbR^2\\ |\calI| < \infty}} \liminf_{n,m\to\infty} \sup_{(h_1,h_2)\in \calI} {\frac{nm}{n+m}}\bbE_{n,m,h_1,h_2} \left|U_{nm} - W_2^2(P_{n^{-1/2},h_1}, Q_{ m^{-1/2},h_2})\right|^2 \\[-0.2in]
\geq (1-\rho) \Var_P[\phi_0(X)] + \rho\Var_Q[\psi_0(Y)],
\end{multlined}
\end{equation*}
where the limit inferior is taken as $n/(n+m) \to \rho \in [0,1]$, and 
$\bbE_{n,m,h_1,h_2}$ denotes the expectation taken over 
the probability measure ${P_{n^{-1/2},h_1}^{\otimes n} \otimes Q_{m^{-1/2},h_2}^{\otimes m}}$.
\end{thmlist}
\end{theorem}
The proof of Theorem~\ref{thm:asymptotic_minimax} appears in Appendix~\ref{app:pf_efficiency}.
For technical purposes,
our statement assumes that 
$P,Q$ admit densities lying in a strict subset $\calC^{\alpha-1}(\bbT^d;M,\gamma)$
of  $\calC^{\alpha-1}(\bbT^d;\widebar M,\widebar \gamma)$, the latter being the class in which
our differentiable paths are shown to lie. 
With this caveat, 
our plugin estimators achieve the asymptotic minimax lower bounds of
Theorem~\ref{thm:asymptotic_minimax}. For example, under the conditions of
Theorem~\ref{thm:two_sample_kernel}, when $2(\alpha+1) > d$ we deduce that
\begin{align*}
\sup_{\substack{\calI \subseteq \bbR^2\\ |\calI| < \infty}} \liminf_{n,m\to\infty} \sup_{(h_1,h_2)\in \calI} {\frac{nm}{n+m}}\bbE_{n,m,h_1,h_2} \left| W_2^2(\hat P_n\sker, \hat Q_m\sker)- W_2^2(P_{n^{-1/2},h_1}, Q_{ m^{-1/2},h_2})\right|^2 \\[-0.1in]
 ~~= (1-\rho) \Var_P[\phi_0(X)] + \rho\Var_Q[\psi_0(Y)].
\end{align*} 

It can be verified that Theorem~\ref{thm:asymptotic_minimax} continues to hold
with $\bbT^d$ replaced by $[0,1]^d$, under
the additional condition that $\varphi_0,\varphi_0^* \in \calC^{\alpha+1}([0,1]^d)$.
We were unable, however, to derive differentiable paths $Q_{t,h} = (\nabla\varphi_{t,h})_\# P_{t,h}$ 
which simultaneously satisfy the H\"older continuity properties
of Theorem~\ref{thm:asymptotic_minimax} while also having  Brenier potentials $\varphi_{t,h},\varphi_{t,h}^*$ 
with uniformly bounded $\calC^{\alpha+1}([0,1]^d)$ norm.

\section{Discussion}
We have shown that several families of plugin estimators for smooth optimal transport maps are minimax
optimal. Our analysis   hinged upon stability arguments which relate this problem to that of estimating
the Wasserstein distance between two distributions, and, in turn, to that of estimating
a distribution under the Wasserstein distance. The latter question is well-studied in the literature, 
and formed a key component in deriving convergence rates for the former two problems. 
As a byproduct of our stability results, we derived central limit
theorems and efficiency lower bounds for estimating the Wasserstein distance between any two 
sufficiently smooth distributions. These results lead to the first practical confidence intervals for the Wasserstein distance in general dimension.

The  estimators in this work are simple to compute and   minimax optimal, but 
we make no claim that their computational efficiency is optimal. 
For example, our plugin estimators of the Wasserstein distance between
$(\alpha-1)$-smooth densities  can be approximated
by sampling $N$ observations from our density estimators, and 
computing the Wasserstein distance between the empirical measures formed by these observations,
which can be done in polynomial time with respect to $N$~\citep{peyre2019}. In order for this approximation
to achieve comparable risk to our theoretical estimators in the high-smoothness regime $\alpha \gtrsim d$,
one must take $N \asymp n^{cd}$ for some $c \geq 1$. Our estimator thus
requires computation time depending exponentially on~$d$.
{\cite{vacher2021,muzellec2021,lin2024} analyzed  alternative estimators 
based on kernel sum-of-squares}, which have more favorable computational properties; though their estimators are not minimax optimal, 
they can be computed in polynomial time if $\alpha \gtrsim d$.
It is an interesting open question to derive polynomial-time estimators in $d$ 
which are also minimax optimal.
{More broadly, there are other computationally efficient estimators for optimal transport maps
based on entropic
regularization~\citep{cuturi2013} and input convex neural networks~\citep{makkuva2020}
whose $L^2$ risks have very recently been studied~\citep{pooladian2021,divol2022a}, 
but are not yet known 
to achieve the minimax rate.}

In our analysis of smooth two-sample  optimal transport map  estimators, we required the fitted Brenier potential to be twice H\"older-smooth, for which we appealed
to Caffarelli's regularity theory. Since we do not know whether Caffarelli's boundary
regularity estimates hold uniformly in the various problem parameters, we resorted to working over $\bbT^d$, where a uniform analogue of Caffarelli's theory 
is available (cf. Theorem~\ref{thm:torus_regularity}). 
We showed in Section~\ref{sec:smooth_domains} that this analysis
can be extended to  generic domains $\Omega$ of $\bbR^d$, conditionally
on a uniform version of Caffarelli's boundary regularity theory.
To the best of our knowledge, it remains an interesting open to question
to verify whether this condition  indeed holds.

Finally, our work leaves open the question of estimating optimal transport maps 
when the ground cost function is not the squared Euclidean norm. While each of the plugin estimators
in this paper can be naturally defined  for generic cost functions, their theoretical analysis
presents a breadth of challenges. For example, although the regularity theory of Caffarelli has been
generalized to cover a large collection of cost functions \citep{ma2005},  
this collection does not include the costs $\|\cdot\|^p$ for $p \neq 2$ and $p > 1$, which are 
arguably most widely-used in statistical applications. For such costs, it remains unclear
what regularity conditions are sensible to place on the population optimal transport map in order
to obtain analogues of our risk bounds, and we hope to explore such questions in future work.

\begin{appendix}

\section{Smoothness Classes and  Density Estimation}
\label{app:smoothness_classes}
In this Appendix, we collect several definitions and properties of H\"older spaces, 
Besov spaces, and Sobolev Spaces, as well as properties of wavelet and kernel density estimators.

\subsection{H\"older Spaces} 
Given a closed   
set $\Omega\subseteq \bbR^d$, 
let $\calC_u(\Omega)$ denote the set of uniformly continuous real-valued functions on $\Omega$. 
For any function $f:\Omega \to \bbR$ which is differentiable 
up to order $k \geq 1$ in the interior of $\Omega$, 
and any multi-index $\gamma \in \bbN^d$, we write $|\gamma| = \sum_{i=1}^d \gamma_i$, and for all $|\gamma|\leq k$, 
$$D^\gamma f = \frac{\partial^{|\gamma|} f}{\partial x_1^{\gamma_1}\dots \partial x_d^{\gamma_d}}.$$
Given $\alpha \geq 0$, the H\"older space $\calC^\alpha(\Omega)$ is defined as the set
of functions $f \in \calC_u(\Omega)$ 
which are differentiable to
order $\lfloor \alpha\rfloor$
in the interior of $\Omega$, 
with derivatives extending continuously up to the boundary of $\Omega$, and such that the H\"older norm  
$$\norm f_{\calC^\alpha(\Omega)} = \sum_{j=0}^{\lfloor\alpha\rfloor} \sup_{|\gamma| = j} \|D^\gamma f\|_{\infty}
 + \sum_{|\gamma|=\lfloor \alpha \rfloor} \sup_{\substack{x,y \in \Omega^\circ \\ x\neq y}} \frac{|D^\gamma f(x) - D^\gamma f(y)|}{\norm{x-y}^{\alpha-\lfloor \alpha\rfloor}}$$
is finite. 
Furthermore, for any $\alpha \geq 0$, $\calC^\alpha(\bbT^d)$ (resp. $\calC_u(\bbT^d)$) is   defined as
the set of $\bbZ^d$-periodic functions $f:\bbR^d \to \bbR$ such that $f \in \calC^\alpha(\bbR^d)$
(resp. $f\in \calC_u(\bbR^d)$). 

Recall that $\calC^\alpha(\Omega;\lambda)$
denotes the closed $\calC^\alpha(\Omega)$ ball of radius $\lambda > 0$. We occasionally use the following simple observation. 
\begin{lemma}
\label{lem:holder_products}
Let $\Omega$ be $\bbT^d$ or a closed subset of $\bbR^d$. Then, for all $\lambda,\alpha > 0$, 
there exists a constant $C_{\lambda,\alpha} > 0$ such that
$$\sup_{f,g \in \calC^\alpha(\Omega;\lambda)} \|fg\|_{\calC^\alpha(\Omega)} \leq C_{\lambda,\alpha}.$$
\end{lemma}
Lemma~\ref{lem:holder_products} is stated for $\alpha < 1$ by \citeauthor{gilbarg2001}~(\citeyear{gilbarg2001}, equation~(4.7)),
and can easily be extended to all $\alpha \geq 1$ using the general Leibniz rule.
\subsection{Wavelets and Besov Spaces} 
\label{app:wavelets}
In Section~\ref{sec:one_sample_density} and Appendix~\ref{app:pf_density_based}, we make use
of the boundary-corrected wavelet system $\Psi^{\mathrm{bc}}$ over the unit cube $[0,1]^d$, 
and of the periodic wavelet system $\Psi^{\mathrm{per}}$ over the flat torus $\bbT^d$.
In this section, we provide further descriptions and properties of these wavelet bases,
before turning to definitions and characterizations of Besov spaces over $[0,1]^d$ and $\bbT^d$. 
For concreteness, we describe these constructions
in terms of the compactly-supported
$N$-th Daubechies scaling and wavelet functions $\zeta_0,\xi_0\in \calC^r(\bbR^d)$,
where $r = 0.18(N-1)$ for an integer $N \geq 2$
(\cite{daubechies1988}; \cite{gine2016}, Theorem~4.2.10). We also extend this definition
to the case $N = 1$ by taking $\zeta_0,\xi_0$ to be the (discontinuous) Haar
functions (\cite{gine2016}, p. 298).
Throughout the sequel and throughout the main manuscript, 
whenever we work with a Besov space $\calB_{p,q}^s$ or 
a H\"older space $\calC^s$ with $s > 0$, 
we tacitly assume that the  parameter $N$ is chosen such that the regularity $r$ is strictly greater than the 
parameters $\lceil s\rceil $ or $\alpha$, in which case it must at least hold that $N \geq 2$. 

Our exposition closely follows that of~\cite{gine2016}, and 
we also refer the reader to~\cite{cohen1993, cohen2003, hardle2012} 
and references therein for further details.

\subsubsection{Boundary-Corrected Wavelets on $[0,1]^d$}
\label{app:bc_wavelets} 
It is well-known that the $N$-th Daubechies wavelet system 
$$\zeta_{0k} = \zeta_0(\cdot-k), \quad \xi_{0jk} = 2^{\frac j 2}\xi_0(2^j(\cdot) - k),\quad j\geq 0, ~k\in \bbZ,$$
forms a basis of $L^2(\bbR)$, with the property
that $\{\zeta_{0k}:k \in \bbZ\}$ spans all
polynomials on $\bbR$ of degree at most $N-1$. While this family may 
easily be periodized 
to obtain a basis for $L^2([0,1])$, as in the following subsection,
doing so may not accurately reflect the regularity of functions in 
$L^2([0,1])$ via the decay of their wavelet coefficients, near the boundaries of the interval.
This consideration motivated \cite{meyer1991} and \cite{cohen1993} to introduce
the so-called boundary-corrected wavelet system on $[0,1]$, which
preserves the standard Daubechies scaling functions lying sufficiently far from the boundaries of the interval, 
and adds edge scaling functions such that their union continues to span all polynomials up to degree $N-1$ on $[0,1]$.  
In short, given a fixed integer 
$j_0 \geq \log_2 N$, the construction of \cite{cohen1993} leads to  smooth
scaling edge basis functions 
\begin{eqnarray*}
&\lzeta_{0j_0k}~~ &\text{   with support contained in }[0,(2N-1)/2^{j_0}],\\
&\rzeta_{0j_0k}~~ &\text{   with support contained in } [1-(2N-1)/2^{j_0},1], 
\end{eqnarray*}
which in turn can be used to define edge wavelet functions $\lxi_{0j_0k}, \rxi_{0j_0k}$,
for $k=0, \dots, N-1$. 
In this case, if one defines,
$$\zeta_{0jk}^a = 2^{\frac{j-j_0}{2}} \zeta_{0j_0k}^a\left(2^{j-j_0}(\cdot)\right), \quad 
  \xi_{0jk}^a = 2^{\frac{j-j_0}{2}} \xi_{0j_0k}^a\left(2^{j-j_0}(\cdot)\right),\quad
  \text{for all } j \geq j_0, \ a \in \{\text{left, right}\},$$
then the family
\begin{align*}
\Phi_0^{\mathrm{bc}} &{=} \{\zeta_{0j_0k}^{\mathrm{bc}}: 0 \leq k \leq 2^{j_0}-1\} 
             {=} \left\{\lzeta_{0j_0k}, \rzeta_{0j_0k}, \zeta_{0m}: 0 \leq k \leq  N-1, N \leq m \leq  2^{j_0} - N-1\right\},\\
\Psi_0^{\mathrm{bc}} &{=} \{ \xi_{0jk}^{\mathrm{bc}}:  0 \leq k \leq 2^j-1,j \geq j_0\} \\
             &{=} \left\{\lxi_{0jk}, \rxi_{0jk}, \xi_{0jm}: 0\leq k \leq  N-1, N \leq m \leq 2^{j_0} - N-1, j\geq j_0\right\},
\end{align*}
form an orthonormal  basis of $L^2([0,1])$, with the property 
that $\Phi^{\mathrm{bc}}$ spans all polynomials on $[0,1]$ of degree at most $N -1$.
We then define a tensor product wavelet basis of $L^2([0,1]^d)$
by setting for all $j \geq j_0$ and all $\ell = (\ell_1,\dots,\ell_d) \in \{0,1\}^d\setminus\{0\}$,
$$\zeta_{j_0k}^{\mathrm{bc}}(x) = \prod_{i=1}^d \zeta_{0j_0k_i}^{\mathrm{bc}}(x_i),\quad \text{and}\quad 
  \xi_{jk\ell}^{\mathrm{bc}}(x) = \prod_{i: \ell_i = 0} \zeta_{0jk_i}^{\mathrm{bc}}(x_i) \prod_{i:\ell_i = 1} \xi_{0jk_i}^{\mathrm{bc}}(x_i),\quad 
x \in [0,1]^d,$$
where in the definition of $\zeta_{j_0k}^{\mathrm{bc}}$, the index 
$k =(k_1, \dots, k_d)$ ranges over $\calK(j_0):=\{1, \dots, 2^{j_0}-1\}^d$, while in the definition of 
$\xi_{jk\ell}^{\mathrm{bc}}$, $k$ ranges over $\calK(j)$. In this case, the wavelet system
$$\Psi^{\mathrm{bc}}     = \Phi^{\mathrm{bc}} \cup \bigcup_{j=j_0}^\infty \Psi_j^{\mathrm{bc}}, \quad 
  \Phi^{\mathrm{bc}}     = \{\zeta_{j_0k}^{\mathrm{bc}}: k\in \calK(j_0)\}, \quad
  \Psi_{j}^{\mathrm{bc}} = \{ \xi_{jk\ell}^{\mathrm{bc}}:  k \in \calK(j)\},\quad j \geq j_0,$$
announced in Section~\ref{sec:one_sample_density} forms a basis of $L^2([0,1]^d)$.
We sometimes make use of the abbreviation $\Psi_{j_0-1} = \Phi$. 
 
\subsubsection{Periodic Wavelets on $\bbT^d$} 
\label{app:per_wavelets}
When working over $\bbT^d$, a simpler construction may be used due to the periodicity 
of the functions involved. Denote the periodization on $\bbT$ 
of dilations of the maps $\zeta_{0}, \xi_{0}$ by
$$\zeta_0^{\mathrm{per}} = \sum_{k \in \bbZ} \zeta_0(\cdot - k)=1, \quad
  \xi_{0j}^{\mathrm{per}} = \sum_{k \in \bbZ} 2^{j/2}   \xi_0(2^j(\cdot-k)),\quad j\geq 0.$$
  In this case, the collection
$$\Psi_0^{\mathrm{per}} = \left\{1, \xi_{0jk}^{\mathrm{per}} = \xi_{0j}^{\mathrm{per}}(\cdot - 2^{-j}k): 0 \leq k \leq 2^j-1, j\geq 0\right\}$$
forms an orthonormal  basis of $L^2(\bbT)$, which may again 
be extended to $L^2(\bbT^d)$ using tensor product wavelets. Specifically, if
$\xi_{jk\ell}^{\mathrm{per}} = \prod_{i=1}^d(\xi_{jk}^{\mathrm{per}})^{\ell_i}$ for all $\ell = (\ell_1, \dots, \ell_d) \in \{0,1\}^d\setminus\{0\}$, 
then 
$$\Psi^{\mathrm{per}} = \{1\} \cup \bigcup_{j=0}^\infty \Psi_j^{\mathrm{per}}, \quad  \text{with}\quad   
\Psi_j^{\mathrm{per}} = \big\{\xi_{jk\ell}: k \in \calK(j), \ell \in \{0,1\}^d\setminus\{0\}\big\},\ j \geq 0,$$
forms an orthonormal basis of $L^2(\bbT^d)$
(\cite{daubechies1992}, Section 9.3; \cite{gine2016}, Section 4.3).

\subsubsection{Properties of Boundary-Corrected and Periodic Wavelet Systems}
\label{sec:wavelet_common_properties}
In both of the preceding constructions, 
one obtains a family $\Phi$ of scaling functions and a sequence of families $(\Psi_j)_{j\geq j_0}$ of wavelet functions, such that
\begin{align*}
\Phi &= \Psi_{j_0-1} = \{\zeta_{k}: k \in \calK(j_0)\} = \begin{cases}
\Phi^{\mathrm{bc}}, & {\text{for }}\Psi = \Psi^{\mathrm{bc}} \\
\{1\},           & {\text{for }} \Psi = \Psi^{\mathrm{per}},
\end{cases} \\ 
\Psi_j &= \{\xi_{jk\ell}: k \in \calK(j), \ell \in \{0,1\}^d\setminus\{0\}\} = \begin{cases}
\Psi_{j}^{\mathrm{bc}}, & {\text{for }} \Psi = \Psi^{\mathrm{bc}}  \\
\Psi_{j}^{\mathrm{per}}, & {\text{for }} \Psi = \Psi^{\mathrm{per}},
\end{cases}\qquad j \geq j_0, \\
j_0 &= \begin{cases}
\lceil \log_2 N\rceil, &  {\text{for }}\Psi = \Psi^{\mathrm{bc}} \\
0,                     & {\text{for }} \Psi = \Psi^{\mathrm{per}},
\end{cases} \\
\calK(j) &= \{0, \dots, 2^{j}-1\}^d,\quad j \geq j_0.
\end{align*} 
In both cases, the wavelet system is defined over a domain $\Omega$, which
is to be understood as either $[0,1]^d$ in the boundary-corrected
case, or as $\bbT^d$ (which itself may be identified with $(0,1]^d$) in the periodic case. 
In either of these settings, the wavelet system
\begin{equation}
\label{eq:general_basis}
\Psi = \Phi \cup \bigcup_{j=j_0}^\infty \Psi_j
\end{equation}
forms a basis of $L^2(\Omega)$. 
The following simple result collects several properties and definitions which are common to both of the above bases.
\begin{lemma}
\label{lem:wavelet}
Let $N \geq 1$. There exist constants $C_1, C_2 \geq 1$ depending only on $N,d$ and on  the choice
of basis $\Psi \in \{\Psi^{\mathrm{bc}}, \Psi^{\mathrm{per}}\}$ such that the following properties hold.
\begin{thmlist} 
\item \label{lem:wavelet--basis_size}
{The cardinalities of $\Phi$ and $\Psi_j$ satisfy }
$|\Phi|  \leq C_1$, $|\Psi_j| \leq C_2 2^{dj}$ for all $j \geq j_0$.
\item \label{lem:wavelet--locality} For all $j \geq j_0$ and $\xi \in \Psi_j$, 
there exists a rectangle $I_\xi \subseteq \Omega$ such that $\diam(I_\xi) \leq C_1 2^{-j}$, 
$\supp(\xi_j) \subseteq I_\xi$, and $\left\|\sum_{\xi\in\Psi_j} I(\cdot  \in I_\xi)\right\|_{L^\infty} \leq C_2$.
\item \label{lem:wavelet--regularity} {Every element $\xi \in \Psi$
is contained in $\calC^r(\Omega)$.}
\item \label{lem:wavelet--polynomials} Polynomials of degree at most $N-1$ over
$\Omega$ lie in $\mathrm{Span}(\Phi)$. 
\item \label{lem:wavelet--derivatives} If $N \geq 2$, we have, 
$$\sup_{0 \leq |\gamma|\leq \lfloor r \rfloor} \sup_{\zeta \in \Phi} \norm{D^\gamma \zeta}_{L^\infty} \leq C_1,\quad 
  \sup_{0 \leq |\gamma|\leq \lfloor r \rfloor}
  \sup_{j \geq j_0} \sup_{\xi \in \Psi_j}2^{-j\left(\frac d 2 + |\gamma|\right)} \norm{D^\gamma \xi}_{L^\infty} \leq C_2.$$
\end{thmlist} 
\end{lemma}
{Notice that the only $\bbZ^d$-periodic polynomials on $\bbR^d$
 are constants, thus Lemma~\ref{lem:wavelet--polynomials}
 is nearly vacuous for the basis $\Psi^{\mathrm{per}}$.}
  
\subsubsection{Besov Spaces}
\label{app:besov_spaces}
We next define the Besov spaces $\calB_{p,q}^s(\Omega)$, for $s > 0$, $p,q \geq 1$. Once again, 
$\Omega$ is understood to be one of $[0,1]^d$ or $\bbT^d$, and $\Psi$ is understood to be the corresponding
wavelet basis as in equation~\eqref{eq:general_basis}. 
%
%
%
Let $f \in L^p(\Omega)$ admit the wavelet expansion 
$$f = \sum_{\zeta \in \Phi} \beta_\zeta \zeta + \sum_{j=j_0}^\infty \sum_{\xi \in \Psi_j} \beta_\xi \xi, \quad\text{over } \Omega,$$ 
with convergence in $L^p(\Omega)$,
where $\beta_\xi = \int \xi f$ for all $\xi \in \Psi$. 
Then, the Besov norm of $f$ 
may be defined by 
\begin{equation}
\label{eq:besov_norm}
\norm f_{\calB_{p,q}^s(\Omega)} = 
\norm{(\beta_\zeta)_{\zeta \in \Phi}}_{\ell_p} + \left\| \left( 2^{j(s + \frac d 2 - \frac d p)} \norm{(\beta_\xi)_{\xi\in\Psi_j}}_{\ell_p}\right)
_{j\geq j_0}\right\|_{\ell_q},
\end{equation}
and we define
$$\calB_{p,q}^s(\Omega)
 = \begin{cases}
 \left\{ f \in L^p(\Omega): \norm f_{\calB_{p,q}^s(\Omega)} < \infty \right\}, & 1 \leq p < \infty \\
 \left\{ f \in \calC_u(\Omega): \norm f_{\calB_{p,q}^{s}(\Omega)} < \infty \right\}, & p = \infty.
 \end{cases}$$ 
We extend the above definition to $s < 0$ by the duality $\calB_{p',q'}^{s}(\Omega) = \big(\calB_{p,q}^{-s}(\Omega)\big)^*$,
where $\frac 1 {p'} + \frac 1 p = \frac 1 {q'} + \frac 1 q = 1$. 
It can be shown that the resulting norm on the space $\calB_{p',q'}^s(\Omega)$ 
is equivalent to the sequence norm $\|\cdot\|_{\calB_{p',q'}^s(\Omega)}$ in equation~\eqref{eq:besov_norm}~(cf. 
\cite{cohen2003}, Theorem 3.8.1), thus we extend its definition to $s < 0$. 

We shall often make use of Besov spaces
in order to characterize H\"older continuous functions
in terms of the decay of their wavelet coefficients, via the following
classical result.
\begin{lemma}
\label{lem:besov_holder}
For all $0 < s < r$, and $d \geq 1$, we have
\begin{equation}
\label{eq:besov_holder_subseteq}
\calC^s([0,1]^d) \subseteq \calB_{\infty,\infty}^s([0,1]^d), \quad
  \calC^s(\bbT^d)  \subseteq \calB_{\infty,\infty}^s(\bbT^d),
\end{equation}
and there exist $C_1,C_2 > 0$
such that 
$$\norm\cdot_{\calB_{\infty,\infty}^s([0,1]^d)}
\leq C_1 \norm\cdot_{\calC^s([0,1]^d)}, \quad
 \norm\cdot_{\calB_{\infty,\infty}^s(\bbT^d)}
\leq C_2\norm\cdot_{\calC^s(\bbT^d)}.$$
If $s \not\in\bbN$, then equation~\eqref{eq:besov_holder_subseteq}
holds with equalities, and with equivalent norms.  
\end{lemma}
An analogue of Lemma~\ref{lem:besov_holder} is well-known to hold for the Daubechies
wavelet system over $\bbR^d$, in which case it can readily be proven using
an equivalent characterization of Besov spaces in terms
of moduli of smoothness~(\cite{gine2016}, Section 4.3.1).
Such characterizations are also available for the periodized
and boundary-corrected wavelet systems
(\cite{gine2016},  Theorem 4.3.26 and 
discussions in Sections~4.3.5--4.3.6), 
and at least in the periodized case can be shown to lead to Lemma~\ref{lem:besov_holder}
(\cite{gine2016}, 
equation (4.167)). 
For the boundary-corrected case, Lemma~\ref{lem:besov_holder}
is known to hold in the special case $d=1$ (\cite{cohen1993}, Theorem 4; 
\cite{gine2016}, equation (4.152)), 
but we do not know of a reference stating this precise
result when $d > 1$, in part due to the potential ambiguity of defining
the H\"older space $\calC^s([0,1]^d)$
over the closed set $[0,1]^d$. We thus
provide a self-contained proof of Lemma~\ref{lem:besov_holder}
in the boundary-corrected case for completeness, using standard arguments.

\noindent {\bf Proof of Lemma~\ref{lem:besov_holder} (Boundary-Corrected Case).}
Let $\Omega=[0,1]^d$. Suppose first that $f \in \calB^s_{\infty,\infty}(\Omega)$
for some $s\not\in\bbN$, with wavelet expansion
$$f = \sum_{\zeta \in \Phi^{\mathrm{bc}}} \beta_\zeta \zeta + \sum_{j=j_0}^\infty \sum_{\xi \in \Psi_j^{\mathrm{bc}}} \beta_\xi \xi.$$
We wish to show that $\|f \|_{\calC^s(\Omega)} \lesssim \|f \|_{\calB^s_{\infty,\infty}(\Omega)}$. 
By Lemma~\ref{lem:wavelet}, $\xi \in \calC^r(\Omega)$ for all $\xi \in \Psi^{\mathrm{bc}}$, where recall that $\lceil s \rceil < r$, thus we may define the map
$$f_\gamma
 =   \sum_{\zeta \in \Phi^{\mathrm{bc}}} \beta_\zeta D^\gamma \zeta + 
           \sum_{j=j_0}^\infty \sum_{\xi \in \Psi_j^{\mathrm{bc}}}  \beta_\xi  D^\gamma \xi,\quad
           \text{for all } 0 \leq |\gamma| \leq \lfloor s\rfloor.$$
Notice that $\|D^\gamma \zeta\|_{L^\infty} \lesssim 1$
for all $\zeta \in \Phi^{\mathrm{bc}}$, and for all $j\geq j_0, k\in \calK(j)$, $\ell \in \{0,1\}^d\setminus\{0\}$,
$$D^\gamma \xi_{jk\ell}^{\mathrm{bc}} = 2^{(j-j_0)\left(\frac{d}{2} + |\gamma|\right)} D^\gamma \xi_{j_0k\ell}^{\mathrm{bc}}(2^{j-j_0}(\cdot))$$
Then, it follows from  Lemma~\ref{lem:wavelet} that for all $x \in  \Omega^\circ$, 
\begin{align}
\nonumber 
|f_\gamma(x)|  
 &\leq     \sum_{\zeta \in \Phi^{\mathrm{bc}}} \left|\beta_\zeta D^\gamma \zeta(x)\right|+ 
           \sum_{j=j_0}^\infty \sum_{\xi \in \Psi_j^{\mathrm{bc}}} \left|\beta_\xi  D^\gamma \xi(x)\right| \\
\nonumber 
 &\lesssim \norm{(\beta_\zeta)_{\zeta\in\Phi^{\mathrm{bc}}}}_{\ell_\infty}|\Phi^{\mathrm{bc}}|  + 
			 \sum_{j=j_0}^\infty  \|(\beta_\xi)_{\xi \in\Psi_j^{\mathrm{bc}}}\|_{\ell_\infty}2^{(j-j_0)\left(\frac{d}{2} + |\gamma|\right)}   \sum_{\xi\in\Psi_j^{\mathrm{bc}}}        
 			I(|\xi(x)| > 0)\\
\nonumber
 &\lesssim \norm{(\beta_\zeta)_{\zeta\in\Phi^{\mathrm{bc}}}}_{\ell_\infty}  + 
            \sum_{j=j_0}^\infty 2^{j\left(\frac{d}{2} + |\gamma|\right)} \|(\beta_\xi)_{\xi \in\Psi_j^{\mathrm{bc}}}\|_{\ell_\infty} \\  
 &\lesssim \norm{(\beta_\zeta)_{\zeta\in\Phi^{\mathrm{bc}}}}_{\ell_\infty}  + 
            \norm{\left( 2^{j\left(\frac{d}{2} + s\right)} \|(\beta_\xi)_{\xi \in\Psi_j^{\mathrm{bc}}}\|_{\ell_\infty}\right)_{j\geq j_0}}_{\ell_\infty} \sum_{j=j_0}^\infty 2^{
            \frac{(|\gamma|-s)j}{2}} 
 \lesssim \norm f_{\calB_{\infty,\infty}^s(\Omega)},
 \label{eq:besov_holder_step_pf} 	
\end{align}
where on the final line, we used the fact that $s$ is not an integer, thus $|\gamma| < s$.
An analogous calculation reveals that
the series defining $f_\gamma$ converges uniformly
for any $0\leq |\gamma| \leq \lfloor s\rfloor$, 
thus it must follow that $f$ is differentiable up to order
$\lfloor s \rfloor$ with derivatives given by $D^\gamma f = f_\gamma$, 
which by equation~\eqref{eq:besov_holder_step_pf} must satisfy $|D^\gamma f(x)| \leq C \norm f_{\calB_{\infty,\infty}^s(\Omega)}$
for all $x \in \Omega^\circ$, 
for a constant $C > 0$ depending only on $d$ and $r$ 
We next prove that $D^\gamma f$ is uniformly $(s-\lfloor s\rfloor)$-H\"older
continuous over $ \Omega^\circ$, for all $|\gamma|=\lfloor s\rfloor$. For all $x,y \in \Omega^\circ$, we have,
\begin{align*}
|D^\gamma f(x) - D^\gamma f(y)|
 &\leq \sum_{\zeta \in \Phi^{\mathrm{bc}}} |\beta_\zeta| |D^\gamma \zeta (x) - D^\gamma\zeta(y)| + 
       \sum_{j=j_0}^\infty \sum_{\xi \in \Psi_j^{\mathrm{bc}}} |\beta_\xi| |D^\gamma \xi(x) - D^\gamma \xi(y)|.
\end{align*}
Since $\zeta \in \calC^r(\Omega)$, for all $\zeta \in \Phi^{\mathrm{bc}}$,  
$$\sum_{\zeta \in \Phi^{\mathrm{bc}}} |\beta_\zeta| |D^\gamma \zeta (x) - D^\gamma\zeta(y)|
\lesssim \norm f_{\calB_{\infty,\infty}^s(\Omega)} |\Phi^{\mathrm{bc}}| \norm{x-y} \lesssim \norm f_{\calB_{\infty,\infty}^s(\Omega)} \norm{x-y}.$$
Furthermore, using the definition of the boundary-corrected wavelet basis
and its locality property in Lemma~\ref{lem:wavelet--locality}, we have
\begin{align*}
\sum_{j=j_0}^\infty &\sum_{\xi \in \Psi_j^{\mathrm{bc}}} |\beta_\xi| |D^\gamma \xi(x) - D^\gamma \xi(y)| \\
  &= \sum_{j=j_0}^\infty \sum_{k=0}^{2^j-1} \sum_{l\in \{0,1]^d\setminus\{0\}}
   |\beta_{\xi_{jk\ell}}|2^{(j-j_0)\left(\frac{d}{2} + |\gamma|\right)} | D^\gamma \xi_{j_0k\ell}(2^{j-j_0}(x)) -
  D^\gamma \xi_{j_0k\ell}(2^{j-j_0}(y))|\\
 &\lesssim \sum_{j=j_0}^\infty  \|(\beta_\xi)_{\xi\in\Psi_j^{\mathrm{bc}}}\|_{\ell_\infty}
  2^{(j-j_0)\left(\frac{d}{2} + |\gamma|\right)} \big(\|2^{j-j_0}x - 2^{j-j_0}y\|  \wedge 1\big)
  \sum_{\xi \in \Psi_j^{\mathrm{bc}}} I(|\xi(x)|\vee | \xi(y)| > 0)\\ 
 &\lesssim \sum_{j=j_0}^\infty   \|(\beta_\xi)_{\xi\in\Psi_j^{\mathrm{bc}}}\|_{\ell_\infty}2^{(j-j_0)\left(\frac{d}{2} + |\gamma|\right)} \big(2^{j-j_0}\|x - y\|  \wedge 1\big) \\    
 &\lesssim \sum_{j=0}^\infty   \|(\beta_\xi)_{\xi\in\Psi_{j+j_0}^{\mathrm{bc}}}\|_{\ell_\infty}2^{j\left(\frac{d}{2} + |\gamma|\right)} \big(2^{j}\|x - y\|  \wedge 1\big) \\   
  &\lesssim  \sum_{j=0}^{J(x,y)} 
 		\|(\beta_\xi)_{\xi\in\Psi_{j+j_0}^{\mathrm{bc}}}\|_{\ell_\infty} 2^{j\left(\frac{d}{2} +  |\gamma| +1\right)}\| x -  y\|+
 		\sum_{j=J(x,y)}^\infty 
 		 \|(\beta_\xi)_{\xi\in\Psi_{j+j_0}^{\mathrm{bc}}}\|_{\ell_\infty}2^{j\left(\frac{d}{2} +  |\gamma|  \right)}, 
\end{align*}
where $J(x,y)$ is the smallest integer $j \geq 0$ such that $2^{j}| x - y| \geq 1$; in particular, 
\begin{equation}
\label{eq:Jxy}
2^{-J(x,y)} \leq \|x-y\| \leq 2^{-J(x,y)+1}.
\end{equation}
Now, since $2^{j(\frac d 2 + s)} \|(\beta_\xi)_{\xi\in\Psi_{j+j_0}^{\mathrm{bc}}}\|_{\ell_\infty} 
\leq \norm f_{\calB_{\infty,\infty}^s(\Omega)} < \infty$, and since $|\gamma| < s \not\in \bbN$, we obtain
\begin{align*}
 {\norm f^{-1}_{\calB_{\infty,\infty}^s(\Omega)}} |D^\gamma  f(x) - D^\gamma f(y)| 
 &\lesssim  \| x -  y\|\sum_{j=0}^{J(x,y)}   
 		2^{j\left( |\gamma| - s +1 \right)}+
 		\sum_{j=J(x,y)}^\infty
 		 2^{j\left( |\gamma| - s\right)} \\
 &\lesssim  \| x -  y\|  
 		2^{J(x,y)\left( |\gamma| - s +1 \right)}+
  		 2^{J(x,y)\left( |\gamma| - s\right)} 
 \lesssim     		 \norm{x-y}^{s-|\gamma|},
\end{align*} 
where the final inequality is due to equation~\eqref{eq:Jxy}.
It readily follows that $\norm f_{\calC^s(\Omega)} \lesssim \norm f_{\calB_{\infty,\infty}^s(\Omega)}$. Furthermore, 
since $D^\gamma f$ is uniformly H\"older continuous
over $(0,1)^d$, it is in particular uniformly continuous
and hence extends to a continuous function over $[0,1]^d$, 
thus $f \in \calC^s([0,1]^d)$. 
We next show that $\calC^s([0,1]^d) \subseteq \calB_{\infty,\infty}^s([0,1]^d)$ for all $s >0$, with the requisite H\"older norms. 
Assume $\norm f_{\calC^s(\Omega)} < \infty$, and let
$\beta_\xi = \int f \xi$ for all $\xi\in\Psi^{\mathrm{bc}}$. By definition of  the 
Besov norm, it will suffice to prove that
$$\big\|(\beta_\zeta)_{\zeta \in \Phi^{\mathrm{bc}}}\| \lesssim \norm f_{\calC^s([0,1]^d)}
, \quad 
  \big\|(\beta_\xi)_{\xi\in\Psi_j^{\mathrm{bc}}}\| \lesssim \norm f_{\calC^s([0,1]^d)}
  2^{-j\left(\frac d 2 +s\right)}, 
  \quad j \geq j_0.$$
The first bound is immediate, since $f$ is bounded above by 
$\norm f_{\calC^s([0,1]^d)}$ over $[0,1]^d$. 
To prove the second bound, let 
$x_0 \in (0,1)^d$, and let $\underline{s}$
denote the largest integer strictly less than $s$. 
By a Taylor expansion to order $\underline s$, there exists $c_s > 0$ 
such that 
\begin{equation}
\label{eq:holder_taylor}
\left|f(x) - \sum_{0\leq  |\gamma| \leq  \underline s} D^\gamma f(x_0) (x-x_0)^\gamma\right| \leq c_s \norm f_{\calC^s(\Omega)}\norm{x-x_0}^s,
\quad x \in \Omega,
\end{equation}
where $(x-x_0)^\gamma = \prod_{i=1}^d (x_i-x_{0i})^{\gamma_i}$.
In particular, for any given $\xi \in \Psi_j^{\mathrm{bc}}$, $j \geq j_0$, 
choose $x_0 \in I_\xi\cap (0,1)^d$, where $\diam(I_\xi)\lesssim 2^{-j}$ and $I_\xi$
is a set containing the support of $\xi$, as defined in Lemma~\ref{lem:wavelet--locality}.
We then have,
$$\begin{multlined}[0.9\textwidth]
\left|\int \xi f\right|
 \lesssim \left| \int \xi(x) \sum_{0 \leq |\gamma| \leq \underline s} D^\gamma f(x_0) (x-x_0)^\gamma dx\right|
     \\  + \|f\|_{\calC^s(\Omega)}\int |\xi(x)| \norm{x-x_0}^s dx
      =  \|f\|_{\calC^s(\Omega)}\int |\xi(x)| \norm{x-x_0}^s dx,
\end{multlined}$$
where the final equality uses the fact that polynomials of degree at most $\lfloor r\rfloor$ lie in $\mathrm{Span}(\Phi^{\mathrm{bc}})$
by Lemma~\ref{lem:wavelet--polynomials}, 
and are therefore orthogonal to $\xi$. We thus have, 
\begin{align*}
|\beta_\xi|
 &\lesssim  \|f\|_{\calC^s(\Omega)} \int_\Omega |\xi(x)| \norm{x-x_0}^s dx  \\
 &=  \|f\|_{\calC^s(\Omega)}\int_{I_\xi} |\xi(x)| \norm{x-x_0}^s dx  \\
 &\lesssim  \|f\|_{\calC^s(\Omega)}2^{dj/2} \diam(I_\xi)^{s}\calL(I_\xi)   
 \lesssim  \|f\|_{\calC^s(\Omega)}2^{-j\left(\frac d 2 + s\right)}.
\end{align*} 
The claim readily follows.
\qed

\subsection{Sobolev Spaces}
\label{app:sobolev} 
For our analysis of the kernel plugin estimators appearing in Section~\ref{sec:two_sample_combined}, 
we briefly recall a Fourier analytic description of the Sobolev spaces $H^s(\bbT^d) = \calB_{2,2}^s(\bbT^d)$ over the torus, 
and refer the reader to~\cite{roe1999,grafakos2009,bahouri2011} for further details. 
Given a function $\phi \in L^2(\bbT^d)$, 
denote its sequence of Fourier coefficients by
$$\calF[\phi](\xi) =  \int \phi(x)\exp(-2\pi ix^\top \xi) dx, \quad \xi \in \bbZ^d.$$
If instead $\phi\in L^1(\bbR^d)$, we continue to denote by $\calF[\phi]$ the Fourier transform of~$\phi$, 
now defined for all $\xi \in \bbR^d$. 
The  inhomogeneous Sobolev norm of order $s \in \bbR$ is defined by  
\begin{align}
\label{eq:sobolev_norm_torus}
\norm \phi_{H^s(\bbT^d)} = \big\|  \langle\cdot\rangle^s \calF[\phi](\cdot) \big\|_{\ell^2(\bbZ^d)},  
\end{align}
 where $\langle \xi \rangle = (1+\|\xi\|^2)^{1/2}$, and the inhomogeneous
Sobolev space $H^s(\bbT^d)$ is then defined as the completion of $\calC^\infty(\bbT^d)$ in the above norm. 
In the special case where $s \in \bbN$, one may equivalently write
$$H^s(\Omega) \equiv W^{s,2}(\bbT^d) = \left\{ f \in L^2(\bbT^d): D^{\gamma} f \in L^2(\bbT^d), 0 \leq |\gamma| \leq s\right\},$$
where differentiation is understood in the distributional sense, and the norm $\|\cdot\|_{H^s(\bbT^d)}$ is then equivalent to 
the norm
$$\|\phi\|_{W^{s,2}(\bbT^d)} =  \sum_{0 \leq |\gamma| \leq s} \|D^\gamma \phi\|_{L^2(\bbT^d)}.$$
We also 
denote the homogeneous Sobolev seminorm of a map $\phi\in L^2(\bbT^d)$ by
$$\norm \phi_{\dH{s}(\bbT^d)} = \big\|  \|\cdot\|^s \calF[\phi] (\cdot)\big\|_{\ell^2(\bbZ^d)}.$$
for any $s \in \bbR$, with the convention $0/0=0$. $\norm \cdot_{\dH{s}(\bbT^d)}$ is in fact a norm
on $L_0^2(\bbT^d)$, and we  define the homogeneous Sobolev space $\dot H^{s}(\bbT^d)$ 
as the completion of $L^2_0(\bbT^d) \cap \calC^\infty(\bbT^d)$ under this~norm. 
As before, one may equivalently write  for $s \in \bbN$, 
$$\dot H^s(\bbT^d) = \left\{ f \in L^2_0(\bbT^d): D^\gamma f \in L^2(\bbT^d), |\gamma| = s\right\}.$$

The following result summarizes some elementary identities (cf. 
Theorem 1.122 of~\cite{triebel2006} and Section 4.3.6 of~\cite{gine2016}).

\begin{lemma}
\label{lem:sobolev_facts} 
Let $s > 0$. Then, there exists a constant $C > 0$ depending only on $d$ and $s$ such that
$\|\cdot\|_{H^s(\bbT^d)} \leq C \|\cdot\|_{\calC^s(\bbT^d)}$, and hence $\calC^s(\bbT^d) \subseteq H^s(\bbT^d)$. 
Also, for any $s \in \bbR  $,
$$H^s(\bbT^d) =\calB_{2,2}^s(\bbT^d),$$
with equivalent norms.
\end{lemma}
 
 Finally, let us briefly mention a generalization of these spaces to domains of $\bbR^d$,
 which we will need for the proof of Theorem~\ref{thm:smooth_domains}.
 We refer to~\cite{triebel1995} for further details.   
 We define the  Bessel Sobolev norm of smoothness
 $s \in \bbR$ and integrability $1 < r < \infty$ as follows, for any tempered distribution $\phi$ over $\bbR^d$, 
 $$\|\phi\|_{H^{s,r}(\bbR^d)} = \big\| \calF^{-1} \big[ \langle \cdot\rangle^s \calF[\phi](\cdot)\big]\big\|_{L^r(\bbR^d)},$$
 and we let $H^{s,r}(\bbR^d)$  denote the completion of $\calC_c^\infty(\bbR^d)$ under the above norm. 
 In the special case $r=2$, it follows from Parseval's identity that 
$$\|\phi\|_{H^{s,r}(\bbR^d)} = \big\| \langle \cdot\rangle^s \calF[\phi](\cdot)\big\|_{L^2(\bbR^d)},$$
in analogy to equation~\eqref{eq:sobolev_norm_torus}. In this case, we omit this superscript and simply write $H^s(\bbR^d) = H^{s,2}(\bbR^d)$. 
Furthermore, for any domain $\Omega$ satisfying condition~\ref{assm:smooth_domain}, we define
$$\|\phi\|_{H^{s,r}(\Omega)}  = \inf_{\substack{f \in H^{s,r}(\bbR^d) \\ \phi = f|_\Omega}} \|f\|_{H^{s,r}(\bbR^d)},$$
where the restriction $f|_\Omega$ is to be understood in the sense
of distributions when $s < 0$. {The space} $H^{s,r}(\Omega)$ is then defined as the set of 
all restrictions $f|_\Omega$ of tempered distributions $f \in H^{s,r}(\Omega)$ for which the above norm is finite. 
Once again, we simply write $H^{s}(\Omega):= H^{s,2}(\Omega)$.

\subsection{Wavelet Density Estimation}
\label{app:wavelet_estim}
We next state several properties
of wavelet density estimators over $\Omega \in \{\bbT^d, [0,1]^d\}$,
with the corresponding basis $\Psi \in \{\Psi^{\mathrm{per}}, \Psi^{\mathrm{bc}}\}$ 
as in Section~\ref{sec:wavelet_common_properties}. 
Let $q \in L^2(\Omega)$ denote a probability density
with corresponding probability distribution $Q$, and with corresponding wavelet expansion
$$q = \sum_{\zeta \in \Phi} \beta_\zeta \zeta + \sum_{j=j_0}^\infty \sum_{\xi\in\Psi_j} \beta_\xi \xi.$$
Given an i.i.d. sample $Y_1, \dots, Y_n \sim Q$ with corresponding
empirical measure $Q_n = (1/n)\sum_{i=1}^n \delta_{Y_i}$, 
define the unnormalized and normalized wavelet density estimators of the density $q$ of $Q$,
\begin{equation} 
\label{eq:general_wavelet_density_estimator}
\tilde q_n = \sum_{\zeta \in \Phi} \hbeta_\zeta \zeta + \sum_{j=j_0}^{J_n} \sum_{\xi\in\Psi_j} \hbeta_\xi \xi,\qquad
\hat q_n = \frac{\tilde q_n I(\tilde q_n \geq 0)}{\int_{\tilde q_n \geq 0} \tilde q_n d\calL }, 
\end{equation}
where $J_n \geq j_0$ is a deterministic threshold, and 
$\hbeta_\xi = \int\xi dQ_n$ for all $\xi \in \Psi_j$,
$j_0 \leq j \leq J_n$.
The following simple
result guarantees that $\tilde q_n$ integrates to unity
since $q$ is a probability density. 
\begin{lemma}
\label{lem:wavelet_density_integral}
We have $\int_\Omega \tilde q_n = 1$. In particular, it follows
that 
$\sum_{\zeta \in \Phi} \hbeta_\zeta \int_\Omega \zeta= 1.$
\end{lemma} 
The proof of Lemma~\ref{lem:wavelet_density_integral} appears in 
Appendix~\ref{app:pf_wavelet_density_integral}.
In the special case of the periodic
wavelet system, for which $\Phi^{\mathrm{per}}$
consists only of the constant function 1, 
Lemma~\ref{lem:wavelet_density_integral}
implies that the corresponding estimated coefficient 
satisfies $\hbeta_1=1$ deterministically, 
thus the definition of $\tilde q_n$ in equation~\eqref{eq:general_wavelet_density_estimator}
coincides with that which will be given in Appendix~\ref{app:pf_density_based}. 

With this result in place, we turn to $L^\infty$
concentration results for $\tilde q_n$, as well as for Besov norms of $\tilde q_n$, which we frequently use throughout our proofs.
In what follows, write
$$q_{J_n}(y) = \bbE[\tilde q_{J_n}(y)] = \sum_{\zeta \in \Phi} \beta_\zeta \zeta + \sum_{j=j_0}^{J_n} \sum_{\xi\in\Psi_j} \beta_\xi\xi,
\quad y \in \Omega.$$
%
\begin{lemma}
\label{lem:wavelet_Linfty}
Let $N \geq 2$ and $q \in \calB_{\infty,\infty}^s(\Omega)$ for some $s > 0$. 
Then, there exist constants $v,b > 0$ depending only on the choice of wavelet
system, such that for any  $J_n \geq j_0$, and all $u > 0$, 
\begin{align}
\label{eq:hoeffding_wavelet_Linfty}
\bbP&\left(\sup_{\zeta \in \Phi} |\hbeta_\zeta - \beta_\zeta| \geq u \right) \lesssim 
  \exp\left\{ -\frac{nu^2}{b}\right\},\\
\bbP&\left(\sup_{\xi \in \Psi_j} |\hat\beta_\xi-\beta_\xi| \geq u\right) 
\lesssim 2^{\frac{dj}{2}} \exp\left\{ -\frac{nu^2}{v + 2^{jd/2} b u}\right\},
\quad j_0 \leq j \leq J_n.
\label{eq:bernstein_wavelet_Linfty}
\end{align}
Furthermore, if $2^{J_n} = c_0 n^{\frac 1 {d+2s}}$ for some $c_0 > 0$, 
then there exists a constant $C > 0$, depending on $c_0$ and on the choice of wavelet system $\Psi$, such that the following assertions hold. 
\begin{thmlist} 
\item \label{lem:wavelet_Linfty--holder} For all $0 < u \leq 1$, 
$$\bbP\Big( \|\tilde q_n\|_{\calB_{\infty,\infty}^{s/2}(\Omega)} \geq u + \|q \|_{\calB_{\infty,\infty}^{s/2}(\Omega)} \Big) \leq 
C  J_n 2^{dJ_n} \exp\big(-u^2 2^{ sJ_n}/C\big).$$
\item \label{lem:wavelet_Linfty--density}   For all $2^{-J_n} \leq u \leq 1$, 
$$\bbP\left(\|\tilde q_n - q_{J_n}\|_{L^\infty(\Omega)} \geq u\right) 
\leq C  J_n 2^{J_n d(d+3)}  \exp\big( -   n u^2 2^{-dJ_n} /C\big).$$
\end{thmlist}
\end{lemma} 
Lemma~\ref{lem:wavelet_Linfty--density} 
is implicit in the proofs of almost sure $L^\infty$ bounds
for wavelet estimators by 
\cite{masry1997} and  \cite{guo2019}, as well as \cite{gine2009a} when $d=1$. While these results
are based on wavelet estimators over $\bbR^d$, they can  readily
be adapted to the wavelet systems considered here,  
as consequences of inequalities~\eqref{eq:hoeffding_wavelet_Linfty}--\eqref{eq:bernstein_wavelet_Linfty}. 
For completeness, 
we provide a proof of Lemma~\ref{lem:wavelet_Linfty--density}, along with the remaining
assertions of Lemma~\ref{lem:wavelet_Linfty}, 
in Appendix~\ref{app:pf_wavelet_Linfty}.

Using Lemmas~\ref{lem:wavelet_density_integral}
and \ref{lem:wavelet_Linfty--density},
the following result is now straightforward.
\begin{lemma}
\label{lem:tilde_q_density}
Let $N \geq 2$.  Assume there exist $\gamma,s > 0$ such that
$q \geq 1/\gamma$ over $\Omega$, and such that $q \in \calB_{\infty,\infty}^s(\Omega)$. 
Then, there exists $c_1 > 0$ depending on $\gamma, \|q \|_{ \calB_{\infty,\infty}^s(\Omega)}$
such that with probability at least $1-c_1/n^2$, $\tilde q_n$ is  a valid probability density
over $\Omega$, and hence $\hat q_n = \tilde q_n$.  
If we instead have $N = 1$, then under no conditions on $q$ it holds that $\hat q_n = \tilde q_n$ almost surely. 
\end{lemma} 
Having now established that $\tilde q_n$ is a valid density with high probability, 
we may speak of its  convergence in Wasserstein distance. 
\cite{weed2019a} previously derived upper bounds on the risk, in Wasserstein distance over $[0,1]^d$, 
of a projection of $\tilde q_n$ onto the set of probability densities. 
Using Lemma~\ref{lem:tilde_q_density}, we are able to extend their result to the estimator
$\hat Q_n$, i.e. the distribution function of the density $\hat q_n$
defined in equation~\eqref{eq:general_wavelet_density_estimator}. We also state this result for 
a general exponent of the 2-Wasserstein distance.  
\begin{lemma}  
\label{lem:wavelet_wasserstein}
Let $\Psi = \Psi\pbc$ with $N \geq 2$. Assume that $q \in \calB_{\infty,\infty}^s([0,1]^d)$ for some $s > 0$.  
Assume further that $q \geq 1/\gamma$ over $[0,1]^d$ for some $\gamma > 0$.
Let $2^{J_n} \asymp n^{1/({d+2s})}$. Then, for any $\rho \geq 0$,  
there exists a constant $C > 0$ depending on $M,\gamma,\rho,s$ such that 
\vspace{-0.04in} 
\begin{equation}
\label{eq:wasserstein_density}
\bbE W_2^\rho(\hat Q_n, Q)  
\leq C 
\begin{cases}
n^{-\frac{\rho(s+1)}{2s+d}}, & d \geq 3 \\
(\log n/\sqrt n)^\rho, & d = 2 \\
1/n^{\rho/2}, & d = 1.
\end{cases}
\end{equation}
Furthermore, when $N=1$, equation~\eqref{eq:wasserstein_density} continues to hold with $s=0$ 
for any density  satisfying $\gamma^{-1} \leq q \leq \gamma$ over $[0,1]^d$, for some $\gamma > 0$.  
\end{lemma}   
The proof appears in Appendix~\ref{app:pf_wavelet_wasserstein}.
 
\subsubsection{Proof of Lemma \ref{lem:wavelet_density_integral}} 
\label{app:pf_wavelet_density_integral}
Recall that $\mathrm{Span}(\Phi)$ contains all polynomials
of degree at most $N-1$ over $\Omega$, by Lemma~\ref{lem:wavelet--polynomials}. 
In particular, it contains the constant function 1, 
thus if $\beta_\zeta' = \int_\Omega \zeta$, we obtain
$1 = \sum_{\zeta \in \Phi} \beta_\zeta' \zeta.$
It then follows by orthonormality of $\Psi$ that 
\begin{align*}
\int_\Omega \tilde q_n
 &= \int_\Omega \left(\sum_{\zeta \in \Phi} \beta_\zeta'\zeta\right)
                \left(\sum_{\zeta \in \Phi} \hbeta_\zeta\zeta + 
 			         \sum_{j=j_0}^{J_n}\sum_{\xi\in \Psi_j} \hbeta_\xi\xi\right)
 = \sum_{\zeta \in \Phi} \beta_\zeta' \hbeta_\zeta
 = \int\left(\sum_{\zeta \in \Phi} \beta_\zeta'  \zeta\right) dQ_n=1.
\end{align*}
This proves the claim.
\qed

\subsubsection{Proof of Lemma \ref{lem:wavelet_Linfty}} 
\label{app:pf_wavelet_Linfty}
Throughout the proof,  $b,v,c  > 0$ denote constants 
depending only  on $c_0$ and the choice of wavelet system, whose value may change from line to line.
To prove inequality~\eqref{eq:hoeffding_wavelet_Linfty}, recall first from Lemma~\ref{lem:wavelet--derivatives} that
\begin{align}
\label{eq:hoeffding_ub_pf}
\sup_{\zeta \in \Phi} \|\zeta\|_{L^\infty(\Omega)} \leq b, \qquad 
\sup_{j \geq j_0} 2^{-jd/2}\sup_{\xi \in \Psi_j} \|\xi\|_{L^\infty(\Omega)} \leq b.
\end{align}
By Hoeffding's inequality, equation~\eqref{eq:hoeffding_ub_pf}  implies that for all $u > 0$, 
\begin{equation} 
\label{eq:wavelet_hoeffding_in_proof}
\bbP\left(\sup_{\zeta \in \Phi} |\hbeta_\zeta - \beta_\zeta| \geq u \right) \leq 
\sum_{\zeta \in \Phi} \bbP\left( \left|\int \zeta d(Q_n-Q)\right| \geq u \right)\lesssim 
  \exp\left\{ -\frac{nu^2}{b^2}\right\},
\end{equation}
where we have used the fact that $|\Phi| \lesssim 1$ by Lemma~\ref{lem:wavelet--basis_size}.
 To prove equation~\eqref{eq:bernstein_wavelet_Linfty}, notice that for all $\xi \in \Psi_j$ and $j \geq j_0$, 
given $Y \sim Q$, 
\begin{align*}
\Var[\xi(Y)] \leq \int \xi^2(y) q(y)dy 
 \leq \|q \|_{L^\infty(\Omega)}   \int \xi^2(y) dy  = \|q \|_{L^\infty(\Omega)} \leq v,   
\end{align*}
where we used the fact that $q \in \calB_{\infty,\infty}^s(\Omega) \subseteq L^\infty(\Omega)$. 
Therefore, by Bernstein's inequality, we have for all $u > 0$ and $j_0\leq j \leq J_n$, 
\begin{align}
\label{eq:wavelet_bernstein_in_proof}
\bbP\left(  \sup_{\xi \in \Psi_j} |\hat\beta_\xi-\beta_\xi| \geq u\right) 
\leq   \sum_{\xi \in \Psi_j} \bbP \Big( |\hat\beta_\xi-\beta_\xi| \geq u\Big)  
 &\lesssim 2^{dj} \exp\left\{ -\frac{nu^2}{v + 2^{jd/2} b u }\right\}.
\end{align}
Here, the last inequality uses the fact that $|\Psi_j| \lesssim 2^{dj}$ by Lemma~\ref{lem:wavelet--basis_size} for all $j \geq j_0$.

 To prove part (i) from here, let $0<u \leq 1$.  A union bound combined
with the above display leads to
\begin{align}
\label{eq:wavelet_Linfty_third}
\bbP\left(\sup_{j_0 \leq j \leq J_n} \sup_{\xi \in \Psi_j} |\hat\beta_\xi-\beta_\xi| \geq u\right)  
 &\lesssim J_n 2^{dJ_n} \exp\left\{ -\frac{nu^2}{v + 2^{J_nd/2} b u }\right\},
\end{align}  
whence, since $2^{J_n} \asymp n^{\frac 1 {d+2s}}$,
\begin{align}
\label{eq:wavelet_Linfty_step}
\bbP\Bigg(2^{\frac{J_n(s + d)}{2}} &\sup_{j_0\leq j \leq J_n}  \big\|(\hat\beta_\xi-\beta_\xi)_{\xi\in\Psi_j}\big\|_{\ell_\infty} \geq u\Bigg) \\
\nonumber 
 &\lesssim J_n 2^{dJ_n}  \exp\left\{- \frac{nu^22^{-J_n(s +d) } }{v {+} b2^{\frac{dJ_n}{2}} 2^{-\frac{J_ns}{2}- \frac{dJ_n}{2}}u}\right\} 
\leq J_n 2^{dJ_n} \exp\left\{- cu^2 2^{J_ns} \right\}. 
\end{align}
Combining this fact with equation~\eqref{eq:wavelet_hoeffding_in_proof}, we have
\begin{align*}
\bbP&\left( \|\tilde q_n - q_{J_n}\|_{\calB_{\infty,\infty}^{s/2}} \geq u\right) \\
 &\leq \bbP\left(\big\|(\hbeta_\zeta - \beta_\zeta)_{\zeta \in \Phi}\big\|_{\ell_\infty} \geq u/2\right) +
       \bbP\left(2^{\frac{J_n(d+s)}{2}} \sup_{j_0 \leq j \leq J_n} \big\|(\hat\beta_\xi-\beta_\xi)_{\xi\in\Psi_j}\big\|_{\ell_\infty} \geq u/2\right)\\
 &\leq C J_n 2^{dJ_n} \exp\{-u^2 2^{J_n s}/C\},
\end{align*}
for a large enough constant $C > 0$. 
Thus, we have
\begin{align*}
 \|\tilde q_n\|_{\calB_{\infty,\infty}^{s/2}}
 &\leq   \|\tilde q_n - q_{J_n}\|_{\calB_{\infty,\infty}^{s/2}} +  
 		 \|q_{J_n}\|_{\calB_{\infty,\infty}^{s/2}}
  \leq u + \|q \|_{\calB_{\infty,\infty}^{s/2}} 
\end{align*} 
with probability at least $1 -  C J_n 2^{dJ_n} \exp\{-u^2 2^{J_n s}/C\}$.  
 Part  (i) thus follows. 
To prove part (ii), let $\delta_n \leq  2^{J_n(d+2)}/(4C_0)$, for a constant $C_0 > 0$ to be specified below. 
Notice that for all $x,y \in \Omega$, 
\begin{align*}
|\tilde q_n(x) - \tilde q_n(y)|
 &\leq  \left| \sum_{\zeta \in \Phi} \hbeta_\zeta (\zeta (x) - \zeta(y))\right| +  
    \left|\sum_{j=j_0}^{J_n} \sum_{\xi \in \Psi_j}   \hbeta_\xi   (\xi (x) - \xi(y))\right| \\ 
 &\lesssim \sum_{\zeta \in \Phi}  | \hbeta_\zeta| \|x-y\| +  
     \sum_{j=j_0}^{J_n}  2^{j\left(\frac d 2 + 1\right)}  |\hbeta_\xi | \|x-y\| \sum_{\xi \in \Psi_j} I(\xi(x)\wedge \xi(y) > 0)\\ 
 &\lesssim  \|x-y\| +  
     \sum_{j=j_0}^{J_n} 2^{j\left(\frac d 2 + 1\right)} \|\xi\|_{L^\infty(\Omega)}  \|x-y\|  \\
 &\lesssim  \sum_{j=j_0}^{J_n}  2^{j(d+1)}\|x-y\|
\lesssim 2^{J_n(d+1)}\|x-y\|,
\end{align*}
where we have again used the properties appearing in Lemma~\ref{lem:wavelet}.
Upon repeating an analogous calculation, we deduce that
both $\tilde q_n$ and $q_{J_n}$ are $C_0  2^{J_n(d+1)}$-Lipschitz.

Let $K_n = O(1/\delta_n^d) = O(2^{-J_nd(d+2)})$ denote the $\delta_n$-covering number of the unit cube $[0,1]^d$ with respect to the Euclidean norm, 
and let $\{x_{0k}: 1 \leq k \leq K_n\}$ be a corresponding $\delta_n$-cover.
Letting $I_k = \{x \in [0,1]^d: \norm{x-x_{0k}} \leq \delta_n\}$, we have (for both $\Omega \in \{[0,1]^d, \bbT^d\}$), 
\begin{align*}
\|\tilde q_n - q_{J_n}\|_{L^\infty(\Omega)}
 &\leq \max_{1 \leq k \leq K_n} \sup_{x \in I_k} |\tilde q_n(x) - q_{J_n}(x)| \\ 
 &\leq \max_{1 \leq k \leq K_n} \sup_{x \in I_{k}} |\tilde q_n(x) - \tilde q_{n}(x_{0k})|      \\ &\qquad + 
       \max_{1 \leq k \leq K_n} \sup_{x \in I_{k}} |q_{J_n}(x_{0k}) - q_{J_n}(x)| + 
       \max_{1 \leq k \leq K_n}| \tilde q_n(x_{0k}) - q_{J_n}(x_{0k})| \\       
 &\leq  2C_0 2^{J_n(d+1)} \delta_n    + \max_{1 \leq k \leq K_n}|\tilde q_n(x_{0k}) - q_{J_n}(x_{0k})|  \\
 &\leq 2^{-J_n}/2 +  
       \max_{1 \leq k\leq K_n}|\tilde q_n(x_{0k}) - q_{J_n}(x_{0k})|.
\end{align*}
Thus, for any $2^{-J_n} \leq u \leq 1$, using Lemma~\ref{lem:wavelet}
and the bounds~\eqref{eq:wavelet_hoeffding_in_proof}--\eqref{eq:wavelet_Linfty_third}, we have
\begin{align*}
\bbP&\left(\|\tilde q_n - q_{J_n}\|_{L^\infty(\Omega)} \geq u\right) \\
 &\leq \bbP\left(\max_{1 \leq k \leq K_n}| \tilde q_n(x_{0k}) - q_{J_n}(x_{0k}) | \geq u/2\right) \\
 &\leq \sum_{k=1}^{K_n} \bbP\left(\left| \sum_{\zeta \in \Phi} (\hbeta_\zeta - \beta_\zeta) \zeta(x_{0k}) + 
 \sum_{j=j_0}^{J_n} \sum_{\xi\in\Psi_j} (\hbeta_\xi - \beta_\xi) \xi(x_{0k})\right| \geq  u /2\right) \\ 
 &\leq \sum_{k=1}^{K_n} \bbP\left(\left| \sum_{\zeta \in \Phi} (\hbeta_\zeta - \beta_\zeta) \zeta(x_{0k})\right| \geq u/4\right) + \sum_{k=1}^{K_n} 
       \bbP\left( \left|\sum_{j=j_0}^{J_n} \sum_{\xi\in\Psi_j} (\hbeta_\xi - \beta_\xi) \xi(x_{0k})\right| \geq  u /4\right)\\
 &\leq K_n \bbP\left(\sup_{\zeta\in\Phi} |\hbeta_\zeta - \beta_\zeta| \geq c u\right) + 
                      K_n   \bbP\left(J_n2^{\frac{dJ_n}{2}} \sup_{j_0 \leq j \leq J_n} \sup_{\xi \in \Psi_j} |\hbeta_\xi - \beta_\xi| \geq  c u \right)\\
 &\lesssim K_n \exp(-nc^2u^2/b^2) + J_n K_n 2^{dJ_n} \exp\left(- nc^2u^2 2^{-dJ_n} /(J_n^2v + cbJ_nu)\right).
\end{align*} 
It follows that, for a sufficiently large constant $C > 0$, 
$$\bbP\left(\|\tilde q_n - q_{J_n}\|_{L^\infty(\Omega)} \geq u\right) 
\leq C  J_n 2^{J_n d(d+3)}  \exp\big( -   n u^2  2^{-dJ_n} /(J_nC)\big),$$
for all $2^{-J_n} < u \leq 1$. 
The claim readily follows.
\qed

\subsubsection{Proof of Lemma~\ref{lem:tilde_q_density}}
The claim for $N=1$ follows by definition of the Haar system, {since in this case $\tilde q_n$ is equal to
a histogram. }
We thus assume $s > 0$ and $N \geq 2$. 
Recall that $\tilde q_n$ integrates to unity by Lemma~\ref{lem:wavelet_density_integral},
thus it suffices to show that $\tilde q_n \geq 0$ with high probability.  
Apply Lemma~\ref{lem:wavelet_Linfty} to deduce that 
$$\|\tilde q_n - q_{J_n}\|_{L^\infty(\Omega)} \leq  \gamma^{-1}/4,$$ 
except on an event with probability at most $c_1/n^2$, for some $c_1 > 0$ depending on $\gamma^{-1}$
and $\|q\|_{\calB_{\infty,\infty}^s(\Omega)}$. Furthermore, using Lemma~\ref{lem:wavelet},
the bias of $\tilde q_n$ satisfies
\begin{align*}
\|q_{J_n} - q\|_{L^\infty(\Omega)} 
 &=  \sum_{j\geq J_n+1} 2^{\frac{dj}{2}} \|(\beta_\xi)_{\xi\in\Psi_j}\|_{\ell_\infty}   \\
 &\leq \|q\|_{\calB_{\infty,\infty}^s(\Omega)}  \sum_{j\geq J_n+1} 2^{\frac {dj}{2} - j(\frac{d}{2}+s)}
 \lesssim \|q\|_{\calB_{\infty,\infty}^s(\Omega)}   2^{-J_ns} \leq \gamma^{-1}/4,
\end{align*}
for all $n$ larger than a universal constant depending only on $\|q\|_{\calB_{\infty,\infty}^s(\Omega)}$. 
Therefore, after possibly increasing $c_1 > 0$, we have with probability at least $1-c_1/n^2$ that
for all $n \geq 1$, 
$$\|\tilde q_n - q\|_{L^\infty(\Omega)} \leq \gamma^{-1}/2.$$
Since $q \geq \gamma^{-1}$, we deduce that $\tilde q_n \geq \gamma^{-1}/2\geq 0$, over the same high probability event. \qed 

\subsubsection{Proof of Lemma~\ref{lem:wavelet_wasserstein}}
\label{app:pf_wavelet_wasserstein}
By Jensen's inequality, it suffices to assume that $\rho \geq 1$. It is straightforward to verify from Lemma~\ref{lem:wavelet}
that the wavelet system $\Psi^{\mathrm{bc}}$ satisfies Assumptions E.1--E.6  
of~\cite{weed2019a}, except Assumption E.2 in the special case $N=1$.  
We also have  $\gamma^{-1} \leq  q \leq \gamma$ 
over $[0,1]^d$. These conditions are sufficient to invoke their Theorem 4 for any $N \geq 1$, leading to
$$W_2 ( \hat Q_n, Q) \lesssim_\gamma \norm{\hat q_n - q}_{\calB_{2,1}^{-1}([0,1]^d)}.$$
Furthermore, it follows from \cref{lem:tilde_q_density} that
the  event $A_n=\{\hat q_n = \tilde q_n\}$  satisfies $\bbP(A_n^\cp) \lesssim n^{-2}$.
Let $q_{J_n} = \bbE[\tilde q_n]$, so that
\begin{align*}
\bbE W_2^\rho(\hat Q_n, Q) 
 &= \bbE\Big[W_2^\rho(\hat  Q_n, Q) I_{A_n}\Big] + \bbE\Big[ W_2^\rho(\hat Q_n, Q) I_{A_n^\cp}\Big] 
 \lesssim \bbE \|\tilde q_n - q \|_{\calB_{2,1}^{-1}([0,1]^d)}^\rho + n^{-2}.
\end{align*}
Now, we make use of the following result which
can be deduced from {  the proofs} of Theorem~1 and Proposition~4 of~\cite{weed2019a}.
\begin{lemma}[\cite{weed2019a}]
\label{lem:W2_bias_variance} 
Let $q$ be a density satisfying $\gamma^{-1} \leq q \leq \gamma$ over $[0,1]^d$. Assume further that 
$q \in \calB_{\infty,\infty}^s([0,1]^d)$ for some $s \geq 0$. Then, 
\begin{align*}
 \|q_{J_n} - q\|_{\calB_{2,1}^{-1}([0,1]^d)}^\rho &\lesssim 2^{-\rho J_n(s+1)}, \\
\bbE \big\|(\hbeta_\zeta-\beta_\zeta)_{\zeta\in\Phi\pbc}\big\|_{\ell_2}^\rho &\lesssim 1/n^{\rho/2},\quad 
\bbE \big\|(\hbeta_\xi-\beta_\xi)_{\xi\in\Psi_j\pbc}\big\|_{\ell_2}^\rho \lesssim \Big(2^{dj}/n^{1/2}\Big)^\rho,\quad j\geq j_0.
\end{align*}
\end{lemma}
Let $\rho'\geq 1$ satisfy  $\frac 1 \rho + \frac 1 {\rho'} = 1$.
Lemma~\ref{lem:W2_bias_variance} implies,
\begin{align*}
\bbE &\|\tilde q_n - q\|_{\calB_{2,1}^{-1}([0,1]^d)}^\rho \\
 &\lesssim  
   \bbE \|\tilde q_n - q_{J_n}\|_{\calB_{2,1}^{-1}([0,1]^d)}^\rho  + 
        \|q_{J_n} - q\|_{\calB_{2,1}^{-1}([0,1]^d)}^\rho \\
 &\lesssim \bbE \big\|(\hbeta_\zeta-\beta_\zeta)_{\zeta\in\Phi\pbc}\big\|_{\ell_2}^\rho  + 
 		  \bbE \left(\sum_{j=j_0}^{J_n} 2^{-j} \big\|(\hbeta_\xi-\beta_\xi)_{\xi\in\Psi_j\pbc}\big\|_{\ell_2}\right)^\rho + 
   2^{-\rho J_n(s+1)}  \\
 &\lesssim n^{-\frac \rho {2}} +  \left(\sum_{j=j_0}^{J_n} 2^{\rho(\eta-1)j} \bbE \big\|(\hbeta_\xi-\beta_\xi)_{\xi\in\Psi_j\pbc}\big\|^\rho_{\ell_2}\right)\left(\sum_{j=j_0}^{J_n} 2^{-\rho'\eta j} \right)^{\frac \rho {\rho'}} + 
   2^{-\rho J_n(s+1)}\\
 &\lesssim n^{-\frac \rho {2}} +  n^{-\frac \rho {2}} \left(\sum_{j=j_0}^{J_n} 2^{\rho(\eta+\frac d 2-1)j}  \right)\left(\sum_{j=j_0}^{J_n} 2^{-\rho'\eta j} \right)^{\frac \rho {\rho'}} + 
   2^{-\rho J_n(s+1)},
\end{align*}
for any $\eta  \in \bbR$. Now, when $d \geq 3$, choose 
$1 - \frac d 2 < \eta < 0$. In this case, the above display is of order
$$n^{-\frac \rho {2}} 2^{[\rho (\eta+\frac d 2 - 1) - \rho\eta]J_n} + 2^{-\rho J_n(s+1)}
 =  2^{\rho(\frac d  2 - 1)J_n}   + 
   2^{-\rho J_n(s+1)}
    \lesssim n^{-\frac{\rho(s+1)}{2s+d}},$$
    which proves the claim for $d \geq 3$. 
When $d \leq 2$, choose $\eta = 0$. Then, the penultimate display is dominated by its
second term, which is of order $n^{-\rho/2}$ when $d=1$ and of order $(\log n/\sqrt n)^\rho$ when $d=2$.
The claim follows. 
\qed

\subsection{Kernel Density Estimation}
We close this appendix with several properties of kernel density estimators.
We adopt the same notation as in Section~\ref{sec:two_sample_combined}. Specifically, 
$K : \bbR^d\to\bbR$ denotes an even kernel, we write $K_{h_n} = h_n^{-d} K(\cdot/h_n)$
for some bandwidth $h_n > 0$,
 and we consider the kernel density estimator 
$$\tilde q_n = Q_n \star K_{h_n} = \int_{\bbR^d} K_{h_n}(\cdot-z)dQ_n(z),$$
where $Q_n\in \calP(\bbT^d)$ denotes the empirical measure based on an i.i.d. sample $Y_1, \dots, Y_n \sim Q\in \calPac(\bbT^d)$. 
In the above display, recall that integration over $\bbR^d$ with respect to the measure $Q_n$ is understood as integration with respect
to its extension to $\bbR^d$ by $\bbZ^d$-periodicity, namely the measure
$$\frac 1 n \sum_{k\in \bbZ^d} \sum_{i=1}^n \delta_{Y_i + k}.$$
Equivalently, we may write
$$\tilde q_n = \int_{\bbT^d} K_{h_n}\sper(\cdot-z)dQ_n(z),\quad\text{where } K_{h_n}^{\text{(per})} = \sum_{k\in \bbZ^d} K_{h_n}(\cdot + k).$$
With the same conventions, we define
$$q_{h_n}(y) = \bbE[\tilde q_n(y)] = Q \star K_{h_n}(y),\quad y \in \bbT^d.$$ 
We begin by proving an $L^\infty(\bbT^d)$ concentration inequality for the estimator $\tilde q_n$ about its mean. 
Though such concentration inequalities  have previously
been established by~\cite{gine2002} under very general conditions on $K$, 
the following simple result will suffice for our purposes. 
\begin{lemma}
\label{lem:kernel_infty}
Assume $q \leq \gamma$ over $\bbT^d$ for some $\gamma > 0$, and that $K \in \calC^1(\bbR^d)$.  
 Then, there exists a constant $C > 0$ depending only on $\gamma, \norm K_{\calC^1(\bbR^d)}$ 
such that for all $h_n \leq u \leq 1$,  
\begin{align*}
\bbP\left(\|\tilde q_n - q_{h_n}\|_{L^\infty(\bbT^d)}\geq u\right) 
 \leq C h_n^{-d(d+2)} \exp\big(-nu^2h_n^d/C\big).
\end{align*} 
\end{lemma}
The proof appears in Appendix~\ref{app:pf_lem_kernel_infty}.
 When the true density $q$ is H\"older continuous with any positive exponent, and bounded below by a positive constant, it is easy to infer from this result that 
$\tilde q_n$ defines a valid density except on an event with exponentially small probability. 
We shall additionally require the following result, which ensures that the fitted density $\tilde q_n$ 
enjoys a nonzero amount of H\"older regularity. 
\begin{lemma}
\label{lem:kernel_smoothness}
Assume $nh_n^d \asymp n^a$ for some $a \in (0,1)$. Assume further that $q \in \calC^s(\bbT^d)$ for some $s > 0$, 
and that $K \in \calC^1(\bbR^d)\cap L^1(\bbR^d)$. 
Then, there exist constants $C,c_1 > 0, \beta \in (0, s\wedge 1)$ depending only on $\norm q_{\calC^s(\bbT^d)}, \norm K_{\calC^1(\bbR^d)},\norm K_{L^1(\bbR^d)}, a, s,d$ such that
for all $n \geq 1$,  with probability at least 
$1-c_1/n^2$, 
$$\norm {\tilde q_n}_{\calC^\beta(\bbT^d)} \leq C.$$
\end{lemma} 
The proof appears in Appendix~\ref{app:pf_lem_kernel_smoothness}.

\subsubsection{Proof of Lemma~\ref{lem:kernel_infty}}
\label{app:pf_lem_kernel_infty}
Let $\delta_n \leq   h_n^{d+2}/(4\norm{K}_{\calC^1(\bbR^d)})$. 
By a direct calculation, it can be seen that
$$\|\tilde q_n\|_{\calC^1(\bbT^d)} \vee \|q_{h_n}\|_{\calC^1(\bbT^d)} \leq \norm{K}_{\calC^1(\bbR^d)} h_n^{-(d+1)}.$$
Let $J_n = O(1/\delta_n^d) = O(h_n^{-d(d+2)})$ denote the $\delta_n$-covering number of the unit cube $[0,1]^d$ with respect to the Euclidean norm, 
and let $\{x_{0j}: 1 \leq j \leq J_n\}$ be a corresponding $\delta_n$-cover.
Letting $I_j = \{x \in [0,1]^d: \norm{x-x_{0j}} \leq \delta_n\}$, we have,
\begin{align*}
\|\tilde q_n - q_{h_n}\|_{L^\infty(\bbT^d)}
 &\leq \max_{1 \leq j \leq J_n} \sup_{x \in I_j} |\tilde q_n(x) - q_{h_n}(x)| \\ 
 &\leq \max_{1 \leq j \leq J_n} \sup_{x \in I_{j}} |\tilde q_n(x) - \tilde q_{n}(x_{0j})|      \\ &\qquad + 
       \max_{1 \leq j \leq J_n} \sup_{x \in I_{j}} |q_{h_n}(x_{0j}) - q_{h_n}(x)| + 
       \max_{1 \leq j \leq J_n}| \tilde q_n(x_{0j}) - q_{h_n}(x_{0j})| \\       
 &\leq  2\norm{K}_{\calC^1(\bbR^d)}h_n^{-(d+1)}\delta_n    + \max_{1 \leq j \leq J_n}|\tilde q_n(x_{0j}) - q_{h_n}(x_{0j})|  \\
 &\leq h_n/2 +  
       \max_{1 \leq j \leq J_n}|\tilde q_n(x_{0j}) - q_{h_n}(x_{0j})|.
\end{align*}
Thus, for any $h_n \leq u \leq 1$, 
\begin{align*}
\bbP&\left(\|\tilde q_n - q_{h_n}\|_{L^\infty(\bbT^d)} \geq u\right) \\
 &\leq \bbP\left(\max_{1 \leq j \leq J_n}| \tilde q_n(x_{0j}) - q_{h_n}(x_{0j}) | \geq u/2\right) \\
 &\leq \sum_{j=1}^{J_n} \bbP\left(|\tilde q_n(x_{0j}) - q_{h_n}(x_{0j})| \geq u/2\right) \\ 
 &\leq \sum_{j=1}^{J_n} \bbP\left(\left|\frac 1 n \sum_{i=1}^n \left[K_{h_n}^{(\mathrm{per})}\left(\norm{x_{0j}-X_i}\right)-\bbE\left\{ K_{h_n}^{(\mathrm{per})}\left(\norm{x_{0j}-X_i}\right)\right\}\right]\right| \geq  u /2\right) \\ 
 &\leq 2J_n \exp\left( - \frac{n u^2}{8\left(\gamma\|K^{(\text{per)}}\|_{L^\infty(\bbT^d)} h_n^{-d} + u h_n^{-d} \| K^{(\text{per})}\|_{L^\infty(\bbR^d)}/3\right)} \right),
\end{align*}
where we invoked Bernstein's inequality   by noting that
\begin{align*}
\left\|K_{h_n}^{(\text{per})} (\norm{x-\cdot})\right\|_{L^\infty(\bbT^d)} &\leq h_n^{-d} \|K^{(\text{per})}\|_{L^\infty(\bbT^d)}, \\
\text{and,} ~~\Var\left[ K_{h_n}^{(\text{per})} (\norm{x-X_i})\right]   
  &\leq  h_n^{-2d}       \int \left[K^{(\text{per)}}\left(\frac{\norm{x-y}}{h_n}\right) \right]^2q(y)dy  \\
  &\leq \gamma \|K^{\text{(per)}}\|_{L^\infty(\bbT^d)} h_n^{-d} \int K_{h_n}^{\text{(per)}}\left( \norm{x-y} \right)  dy  \\  
  &= \gamma \|K^{\text{(per)}}\|_{L^\infty(\bbT^d)} h_n^{-d}. 
\end{align*}
It follows that, for a sufficiently large constant $C > 0$ depending  on $\gamma$ and $\|K\|_{\calC^1(\bbR^d)}$,  we have
$$
\bbP\left(\|\tilde q_n - q_{h_n}\|_{L^\infty(\bbT^d)} \geq u\right) 
\leq C h_n^{-d(d+2)} \exp\big( -   n u^2 h_n^d /C\big).$$
The claim readily follows.\qed

\subsubsection{Proof of Lemma~\ref{lem:kernel_smoothness}}
\label{app:pf_lem_kernel_smoothness}
By Lemma~\ref{lem:kernel_infty}, there is a constant $c_1 > 0$ and an event $A_n$ satisfying $\bbP(A_n) \geq 1-1/n^2$ such that 
$$\norm{\tilde q_n - q_{h_n}}_{L^\infty(\bbT^d)} \leq \gamma_n = c_1\sqrt{\frac{ \log n}{nh_n^d}}.$$
All subsequent statements are made over the event $A_n$. 
Now, given $\beta \in (0,s\wedge 1)$ to be specified below, and $x,y \in \bbT^d$, we have
\begin{align*}
|\tilde q_n(x) - \tilde q_n(y)| 
 &\leq 2\|\tilde q_n - q_{h_n}\|_{L^\infty(\bbT^d)} + |q_{h_n}(x) - q_{h_n}(y)| \\ 
 &\leq 2 \gamma_n + \int_{\bbR^d}  \left| K (z)\big[q(x-h_nz) - q(y-h_nz)\big] \right|   dz\\
 &\leq 2 \gamma_n +  \|q\|_{\calC^\beta(\bbT^d)}\|K\|_{L^1(\bbR^d)} \norm{x-y}^\beta\\ 
 &\leq C_1 (\gamma_n + \norm{x-y}^\beta),
\end{align*}
for a large enough constant $C_1 > 0$. If $\norm{x-y}^\beta \geq \gamma_n$, then $\tilde q_n$ already satisfies the condition of $\beta$-H\"older continuity, thus it suffices
to assume $\norm{x-y}^\beta < \gamma_n$. 
Recall that 
$$\|\tilde q_n\|_{\calC^1(\bbT^d)} \leq \norm{K}_{\calC^1(\bbR^d)} h_n^{-(d+1)}.$$
We deduce that for all $x,y$ such that
$\norm{x-y}^\beta < \gamma_n$, 
\begin{align*}
|\tilde q_n(x) - \tilde q_n(y)| 
  \lesssim \frac{\norm{x-y}}{h_n^{d+1}} 
  \leq \frac{\gamma_n^{\frac {1}{\beta}-1}}{h_n^{d+1}} \norm{x-y}^\beta 
  \lesssim \frac{(n^a /\log n)^{\frac 1 2\left(1-\frac {1}{\beta}\right)}}{n^{(a-1)(d+1)}} \norm{x-y}^\beta 
  \lesssim  \norm{x-y}^\beta,
\end{align*}
for any small enough choice of $\beta$. The claim follows. 
\qed

\section{On the Variance of Kantorovich Potentials}
\label{app:kantorovich}
We state a straightforward technical result  which will be used throughout our proofs.
\begin{lemma}
\label{lem:kantorovich_L2}
Let $\Omega$ be equal to $[0,1]^d$ or $\bbT^d$. Given $P,Q \in \calPac(\Omega)$, 
let $(\phi_0,\psi_0)$ be a pair of Kantorovich potentials in the optimal transport 
transport problem from $P$ to $Q$. 
Assume further that the   density $q$ of $Q$ satisfies $\gamma^{-1} \leq q  \leq \gamma$
over $\Omega$, for some $\gamma > 0$. Define
$\widebar\psi_0 = \psi_0 - \int_\Omega \psi_0.$ 
Then, there exists a constant $C > 0$ depending only on $d$ such that
$$\|\widebar \psi_0\|_{L^2(Q)} \leq C \gamma W_2(P,Q).$$ 
In particular,
$$\Var_Q[\psi_0(Y)] \leq (C\gamma)^2 W_2^2(P,Q).$$ 
\end{lemma}
The proof will follow from Poincar\'e inequalities over $[0,1]^d$ and $\bbT^d$, which we
recall here as they will be needed again in the sequel. 
The following is a special case of the Poincar\'e inequality 
for convex domains (see for instance \cite{leoni2017}, Theorem 12.30).
\begin{lemma} 
\label{lem:poincare}
Let $0<a<b<\infty$ and $\Omega=[a,b]^d$. Then, there exists a constant $C > 0$ depending only on $d$
such that for all $f \in H^1(\Omega)$ satisfying $\int_\Omega f = 0$, 
$$\norm f_{L^2(\Omega)} \leq C (b-a)\norm{\nabla f}_{L^2(\Omega)}.$$
\end{lemma} 
We also state the following classical periodic Poincar\'e inequality 
(see for instance~\cite{steinerberger2016} for a simple proof).
\begin{lemma}
\label{lem:poincare-periodic}
Let $f \in H^1(\bbT^d)$ satisfy $\int_{\bbT^d}f= 0$. Then, 
$\norm{f}_{L^2(\bbT^d)} \leq \norm{\nabla f}_{L^2(\bbT^d)}.$
\end{lemma}   

{\bf Proof of Lemma~\ref{lem:kantorovich_L2}.} 
Since $\psi_0 \in H^1(\Omega)$ by definition, we may apply the Poincar\'e inequality over $\Omega$
(namely, Lemma \ref{lem:poincare} when $\Omega=[0,1]^d$, or Lemma~\ref{lem:poincare-periodic} when $\Omega=\bbT^d$).
 This fact, together with the assumption $\gamma^{-1} \leq q \leq \gamma$, implies
\begin{align*}
 \|\widebar \psi_0\|_{L^2(Q)}^2 \leq \gamma \|\widebar \psi_0\|_{L^2([0,1]^d)}^2
\leq C^2\gamma \|\nabla \psi_0\|_{L^2([0,1]^d)}^2
&\leq (C\gamma)^2 \|\nabla \psi_0\|_{L^2(Q)}^2  = (C\gamma)^2W_2^2(P,Q),
\end{align*}
which then also implies
$$\Var_Q[\psi_0(Y)] = \Var_Q[\widebar \psi_0(Y)] \leq  \|\widebar \psi_0\|_{L^2(Q)}^2 \leq (C\gamma)^2 W_2^2(P,Q),$$
as claimed.\qed

\section{Proofs of One-Sample Stability Bounds}
\label{app:pf_stability} 
\subsection{Proof of  Theorem \ref{thm:stability}} 
Recall that $\varphi_0$ denotes a Brenier potential from $P$ to $Q$, while
$\phi_0 = \norm\cdot^2-2\varphi_0$ and $\psi_0 = \norm\cdot^2 - 2\varphi_0^*$ denote
the corresponding Kantorovich potentials. 
Since we have assumed that both $P$ and $Q$ are absolutely continuous distributions, Brenier's Theorem implies that 
$S_0 = \nabla\varphi_0^*$ is the optimal transport map
from $Q$ to $P$. 
Since $\varphi_0$ is closed, the assumption
$$\frac 1 \lambda  I_d \preceq \nabla^2\varphi_0 \preceq \lambda I_d,$$
from condition~\ref{assm:curvature} also implies (\cite{hiriart-urruty2004}, Theorem 4.2.2),
$$\frac 1 \lambda I_d \preceq \nabla^2 \varphi_0^* \preceq   \lambda I_d.$$
Combining this bound with a second-order Taylor expansion of $\varphi_0^*$ leads to the following inequalities
\begin{equation}
\label{eq:curvature}
\frac 1 {2\lambda} \norm{x-y}^2 \leq \varphi_0^*(y) - \varphi_0^*(x) - \big\langle S_0(x), y-x\big\rangle 
   \leq \frac \lambda {2} \norm{x-y}^2,\quad x,y \in \Omega.
   \end{equation}
With these facts in place, we turn to proving the theorem, namely that
\begin{equation}
\label{eq:stability_pf_sufficient}
   \frac 1 \lambda \|\hat T - T_0\|_{L^2(P)}^2 \leq  W_2^2(P, \hat Q) -  W_2^2(P,Q)  - \int \psi_0 d(\hat Q-Q) \leq \lambda W_2^2(\hat Q,Q).
\end{equation} 
We begin with the first inequality.
Since $\hat T$ is the optimal transport map from $P$ to $\hat Q$, 
we have,
\begin{align*}
W_2^2(P, \hat Q) 
 &=  \int  \|\hat T(x) - x \|^2  dP(x) \\
 &= \int \|T_0(x) - x\|^2 dP(x)
  \\ &+ \int 2\big\langle  T_0(x) - x, \hat T(x) - T_0(x)\big\rangle dP(x) 
+ \int \| \hat T(x) - T_0(x)\|^2 dP(x) \\
 &= W_2^2(P,Q)+ \int 2\big\langle  T_0(x) - x, \hat T(x) - T_0(x)\big\rangle dP(x) 
+ \|\hat T - T_0\|_{L^2(P)}^2.
\end{align*}
To bound the cross term, notice that equation~\eqref{eq:curvature} implies
\begin{align*}
2\int \big\langle & T_0(x) - x, \hat T(x) - T_0(x)\big\rangle dP(x)   \\
 &= 2\int \big\langle T_0(x) - S_0(T_0(x)), \hat T(x) - T_0(x)\big\rangle dP(x) \\
 & \geq 2\int\bigg[ \big\langle T_0(x), \hat T(x) - T_0(x)\big\rangle \\ &\qquad\qquad + 
  			\varphi_0^*(T_0(x)) - \varphi_0^*(\hat T(x)) + \frac 1 {2\lambda} \|\hat T(x) - T_0(x)\|^2 \bigg]dP(x) \\
 &= \int\bigg[\|\hat T(x)\|^2 - \|T_0(x)\|^2 - \|\hat T(x) - T_0(x)\|^2  \\ &\qquad\qquad + 
  			2\varphi_0^*(T_0(x)) - 2\varphi_0^*(\hat T(x)) + \frac 1  \lambda   \|\hat T(x) - T_0(x)\|^2 \bigg]dP(x) \\
 &= \left(\frac 1 \lambda-1\right) \|\hat T - T_0\|_{L^2(P)}^2   + 
 \int \psi_0 d(\hat Q - Q).
\end{align*}
We deduce
$$W_2^2(P, \hat Q)  \geq W_2^2(P,Q) + \frac 1 \lambda \| \hat T - T_0\|_{L^2(P)}^2 + \int \psi_0 d(\hat Q - Q),$$ 
To prove the second inequality
in equation~\eqref{eq:stability_pf_sufficient}, let $\hpi$ denote an optimal coupling between $Q$ and $\hat Q$. 
Then, the measure $\hpi_{S_0} = (S_0, Id)_\# \hpi$ is a (possibly suboptimal) coupling between $P$ and $\hat Q$, thus
\begin{align}
\label{eq:one_sample_stability_key_step}
W_2^2(P, \hat Q)
 \leq \int \norm{x-z}^2 d\hpi_{S_0}(x,z) 
 = \int \norm{S_0(y)-z}^2 d\hpi(y,z).
 \end{align}
The claim is now a consequence of the following technical Lemma, which will be used again in the sequel. 
\begin{lemma}
\label{lem:stability_technical}
We have, 
$$W_2^2(P, \hat Q)
 \leq \int \norm{S_0(y)-z}^2 d\hpi(y,z) \leq W_2^2(P,Q)+ \int \psi_0 d(\hat Q - Q)  + \lambda W_2^2(\hat Q, Q) .$$
\end{lemma}

\subsection{Proof of Lemma~\ref{lem:stability_technical}} 
We have, 
\begin{align*}
\int& \norm{S_0(y)-z}^2 d\hpi(y,z)  \\
  &= \int \norm{S_0(y)-y}^2  dQ(y) {+} \int \norm{y - z}^2 d\hpi(y,z) {+} 
    2\int \big\langle S_0(y) - y, y -z\big\rangle  d\hpi(y,z) \\
  &= W_2^2(P,Q) + W_2^2(\hat Q, Q) + 2 \int \big\langle S_0(y) - y, y -z\big\rangle  d\hpi(y,z).
\end{align*}
Now, notice that by~\eqref{eq:curvature}, 
\begin{align*}
2\int \big\langle S_0(y), y-z \big\rangle d\hpi(y,z)
 &\leq 2\int \left[\varphi_0^*(y) - \varphi_0^*(z) + \frac \lambda {2} \norm{y-z}^2 \right] d\hpi(y,z) \\
 &= 2\int \varphi_0^* d(Q- \hat Q) +\lambda W_2^2(\hat Q, Q),
\end{align*} 
and,
\begin{align*}
2\int &\big\langle -y, y-z \big\rangle d\hpi(y,z) \\ 
&= \int \left[ \norm z^2 - \norm{z-y}^2 - \norm y^2 \right] d\hpi(y,z) 
= \int \norm \cdot^2  d(\hat Q-Q) - W_2^2(\hat Q, Q).
\end{align*}
Therefore,
\begin{align*}
W_2^2&(P, \hat Q) - W_2^2(P,Q) \\
  &\leq   
  \int \left(\norm \cdot^2-2\varphi_0^*\right)d(\hat Q-Q) +  \lambda W_2^2(\hat Q, Q)
  = \int \psi_0 d(\hat Q-Q) +  \lambda W_2^2(\hat Q, Q),
\end{align*} 
and the claim follows.\qed

\section{Proofs of Upper Bounds for One-Sample Empirical Estimators}
\label{app:one_sample_empirical}
In this Appendix, we prove Corollaries~\ref{cor:one_sample_emp_hypercube} and~\ref{cor:two_sided_wasserstein}. 

\subsection{Proof of Corollary~\ref{cor:one_sample_emp_hypercube}}
We shall make use of the notation introduced in Section~\ref{sec:one_sample_density} 
and Appendix~\ref{app:besov_spaces}, regarding wavelet density estimation over $[0,1]^d$. 
In particular, let $\Psi = \Psi^{\text{bc}}$ with $N=1$, so that $\Psi$ is the Haar wavelet basis on $[0,1]^d$. 
 
\begin{lemma}
\label{lem:wasserstein_wavelet_projection}
Let $J \geq 1$ be an integer. For any $\mu \in \calP([0,1]^d)$, let $\mu_J \in \calPac([0,1]^d)$
denote the measure admitting density
$$q_J = 1 + \sum_{j=0}^J \sum_{\xi\in\Psi_j} \xi \int \xi d\mu,$$
with respect to the Lebesgue measure on $[0,1]^d$. Then, $W_2(\mu,\mu_J) \leq \sqrt d 2^{-J}.$
\end{lemma} 

\begin{proof} 
The Lemma is a   consequence of dyadic partitioning arguments which
have previously been used by~\cite{boissard2014a, fournier2015, weed2019, lei2020}. In particular, for all $j \geq 0$, let
$\calQ_j$ denote the natural partition (up to intersections on Lebesgue null sets) of $[0,1]^d$ into $2^{dj}$ cubes of length $2^{-j}$. Then, Proposition 1
of~\cite{weed2019} implies
$$W_2^2(\mu,\mu_J) \leq d \left[2^{-2J} + \sum_{j=1}^J 2^{-2(j-1)}\sum_{S \in \calQ_j} |\mu(S) - \mu_J(S)|\right].$$
To prove the claim, it thus suffices to show that $\mu(S) = \mu_J(S)$ for all $S \in \calQ_j$ and $j=1, \dots, J$. 

Let $j \geq 0$, $S \in \calQ_j$, and recall that $I_S$ is the indicator function of $S$. 
Denote its expansion in the Haar basis by
$$I_S = \calL(S) + \sum_{\ell=0}^\infty \sum_{\xi \in \Psi_\ell} \gamma_\xi \xi, \quad \text{where } \gamma_\psi = \int I_S \psi, \psi \in \Psi.$$
Notice that
for any $\ell \geq j$ and $\xi \in \Psi_\ell$, we have  
$$\supp(\xi) \subseteq I_S, \quad \text{or} \quad \supp(\xi) \cap I_S = \emptyset.$$
Furthermore, since $\zeta = I_{[0,1]^d}$, and the Haar basis is orthonormal, 
we must have $\int_{[0,1]^d} \xi = 0$ for any $\xi \in \Psi_j$, $j \geq 0$. It must follow that
$$\gamma_\xi = \int I_S \xi = 0, \quad \text{for all } \xi \in \Psi_\ell, \ell \geq j,$$
that is, $I_S \in \text{Span}\left(\Phi \cup \bigcup_{\ell=0}^{j-1} \Psi_\ell\right)$.
We therefore have, for any $S \in \calQ_j$ and $j \leq J$, 
\begin{align*}
\mu_J(S)
 &= \int I_S(y) q_J(y) dy \\
 &= \calL(S) + \sum_{j=0}^J \sum_{\xi\in\Psi_j} \left(\int \xi d\mu\right) \left(\int I_S(y) \xi(y) dy\right) \\
 &= \calL(S) + \sum_{j=0}^J \sum_{\xi\in\Psi_j} \left(\int \xi d\mu\right) \gamma_\xi \\
 &= \int  \left(\calL(S) + \sum_{j=0}^J \sum_{\xi\in\Psi_j}\xi  \gamma_\xi  \right)d\mu = \int I_S d\mu = \mu(S).
\end{align*}
The claim follows.
\end{proof}
To prove the Corollary from here, let $2^{J_n} \asymp n^{1/d}$, and let $\hat Q_n$ be the distribution with density
$$\hat q_n(y) = 1 + \sum_{j=0}^{J_n} \sum_{\xi\in \Psi_j} \left(\int \xi dQ_n\right)\xi(y) , \quad y \in [0,1]^d.$$
Apply Lemma~\ref{lem:wasserstein_wavelet_projection} to the measure $\mu=Q_n$ to obtain 
$$W_2^2(Q_n, Q) \lesssim W_2^2(Q_n, \hat Q_{n}) + W_2^2(\hat Q_n, Q) \lesssim 2^{-2J_n} + W_2^2(\hat Q_n, Q)
\lesssim n^{-2/d} + W_2^2(\hat Q_n, Q).$$ 
Furthermore, recall that $\gamma^{-1} \leq q\leq \gamma$, 
thus we may apply Lemma~\ref{lem:wavelet_wasserstein}
to deduce 
$$\bbE W_2^2(\hat Q_n, Q)\lesssim 
\begin{cases}
n^{-2/d}, & d \geq 3 \\ 
(\log n)^2/n, & d = 2 \\
1/n,      & d = 1.
\end{cases}
$$
The claim follows.\qed 

\subsection{Proof of Corollary~\ref{cor:two_sided_wasserstein}}
By Theorem~\ref{thm:stability}, 
$$\bbE \big| W_2^2(P, Q_n) - W_2^2(P, Q)\big| 
 \leq  \bbE W_2^2(Q_n, Q)+\bbE\left| \int \psi_0 d(Q_n-Q)\right| .$$
By Jensen's inequality, the final term satisfies 
$$\bbE\left| \int \psi_0 d(Q_n-Q)\right| 
 \leq   n^{-\frac 1 2} \sqrt{\Var_Q[\psi_0(Y)]}.$$
Since $\psi_0$ is uniformly bounded by a constant depending only on $d$, the right-hand side of the above display is of order $n^{-1/2}$. 
Furthermore, by equation~\eqref{eq:wasserstein_empirical}, we have $\bbE W_2^2(Q_n, Q) \lesssim \kappa_n$, 
thus the first part of the claim follows.

Under the assumptions of the second part of the claim, we may instead use Corollary~\ref{cor:one_sample_emp_hypercube}
to obtain the stronger bound $\bbE W_2^2(Q_n, Q) \lesssim \widebar\kappa_n$, as
well as 
Lemma~\ref{lem:kantorovich_L2} to derive $\Var_Q[\psi_0(Y)] \lesssim W_2^2(P,Q)$.
The claim then follows. 
\qed

\section{Proofs of Upper Bounds for One-Sample Wavelet Estimators}
\label{app:one_sample_wavelet}
\subsection{Proof of Theorem~\ref{thm:one_sample_wavelet}}
Under the assumptions of part~(i), we may apply Theorem~\ref{thm:stability} and Lemma~\ref{lem:wavelet_wasserstein}
to obtain, 
$$\bbE \big\|\hat T_n - T_0\big\|_{L^2(P)}^2 \lesssim_\lambda   \bbE W_2^2(\hat Q_n, Q) \lesssim_{M,\gamma,\alpha} \trate,$$
which immediately leads to the first claim.  
To prove the second claim, recall that we have assumed $\alpha > 1$,
whence the assumption on $\varphi_0^*$ implies in particular that $\norm{\varphi_0^*}_{\calC^2(\Omega)} \leq \lambda$. 
Since the densities $p,q$ are bounded from below and above over $[0,1]^d$ by positive constants, 
it follows by Lemma~\ref{lem:gigli} 
that $\varphi_0$ satisfies condition~\ref{assm:curvature}, 
after possibly modifying the value of $\lambda$ in terms of $\gamma$.
 We may therefore invoke Theorem~\ref{thm:stability}  
to obtain,  
$$ L(\hat Q_n) \leq W_2^2(P, \hat Q_n) - W_2^2(P,Q) \leq \lambda W_2^2(\hat Q_n, Q) + L(\hat Q_n).$$
Let $C > 0$ be a constant depending only on $M,\lambda,\gamma,\alpha$, whose value may change
from line to line. By Lemma~\ref{lem:wavelet_wasserstein}, we have
$$\bbE W_2^2(\hat Q_n, Q) \leq  C\trate,\quad\text{and}\quad \bbE W_2^4(\hat Q_n, Q) \leq  C \sqtrate.$$
Furthermore, by Lemma~\ref{lem:L}, we have
\begin{align*}
\big| \bbE L(\hat Q_n)\big| &\leq C \trate \\
\Var\big[ L(\hat Q_n)\big] &\leq \frac 1 n \left(\Var_Q[\psi_0(Y)] + 2^{-2J_n\alpha}\right) \leq \frac{\Var_Q[\psi_0(Y)]}{n} + C\sqtrate \\
\bbE \big|  L(\hat Q_n)\big|^2 &=  \big| \bbE L(\hat Q_n)\big|^2 + \Var\big[ L(\hat Q_n)\big]
\leq \frac{\Var_Q[\psi_0(Y)]}{n} + C\sqtrate.
\end{align*}
Combining the preceding three displays, we deduce that  
\begin{align}
\label{eq:pf_one_sample_bias_setup}
\big|\bbE W_2^2(P, \hat Q_n) -  W_2^2(P,Q) \big| \leq  \lambda \bbE W_2^2(\hat Q_n, Q) {+} \big|\bbE L(\hat Q_n)\big|
{\leq} C  \trate,
\end{align}
and, 
\begin{align*}
\nonumber 
\bbE \big| &W_2^2(P, \hat Q_n) -  W_2^2(P,Q)\big|^2 \\ 
 &\leq \bbE\left[ \left(\lambda W_2^2(\hat Q_n, Q)+ \big|L(\hat Q_n)\big| \right)^2\right] \\ 
 &\leq \lambda^2 \bbE W_2^4(\hat Q_n, Q) + 2\lambda \bbE\left[ W_2^2(\hat Q_n, Q) \big| L(\hat Q_n)\big|\right]+ \bbE \big|L(\hat Q_n)\big|^2 \\
 &\leq \lambda^2 \bbE W_2^4(\hat Q_n, Q) + 2\lambda \sqrt{\left(\bbE W_2^4(\hat Q_n, Q)\right) \bbE \big|L(\hat Q_n)\big|^2}+ \bbE \big|L(\hat Q_n)\big|^2 \\
 &\leq C\lambda^2 \sqtrate + 2\lambda \sqrt{ C\sqtrate\left(C\sqtrate + \frac{\Var_Q[\psi_0(Y)]}{n}\right)} + \frac{\Var_Q[\psi_0(Y)]}{n} \\
 &\leq C^2  \sqtrate  + 2C\trate \sqrt{ \frac{\Var_Q[\psi_0(Y)]}{n}}+  \frac{\Var_Q[\psi_0(Y)]}{n} \\
 &\leq \left(  C  \trate  + \sqrt{\frac{\Var_Q[\psi_0(Y)]}{n}}\right)^2.
\end{align*}
The claim follows \qed 

It thus remains to prove Lemma~\ref{lem:L}. 

\subsection{Proof of Lemma~\ref{lem:L}}  
In order to bound the bias of $\int \psi_0 \hat q_n$, recall
from Lemma~\ref{lem:tilde_q_density}  
that the event $A_n=\{\hat q_n = \tilde q_n\}$ satisfies $\bbP(A_n^\cp) \lesssim n^{-2}$.
Since $\psi_0$ is bounded by a constant depending only on $d$, we have, 
\begin{align*}
\bbE \left| \int \psi_0 (\hat q_n - \tilde q_n)\right|
  &\leq \bbE \left(\left| \int \psi_0 (\hat q_n - \tilde q_n) \right|I_{A_n^\cp} \right) \lesssim \bbP(A_n^\cp) \lesssim 1/n^2.
\end{align*}
We deduce that 
\begin{align*}
\big|\bbE \big[L(\hat Q_n)\big]\big|  \lesssim \left| \int \psi_0 (\tilde q_n - q)\right|   + \frac 1 {n^2},
\end{align*} 
thus we are left with bounding the bias of $\int\psi_0 \tilde q_n$. 
  Recall that $\hat\beta_\xi$ is an unbiased estimator of $\beta_\xi$
for all $\xi \in\Psi$,
so that
$$q_{J_n} := \bbE[\tilde  q_n] = \sum_{\zeta \in \Phi} \beta_\zeta \zeta + \sum_{j=j_0}^{J_n} \sum_{\xi\in\Psi_j} \beta_\xi\xi.$$ 
Write the expansion of $\psi_0$ in the basis $\Psi$
as
$$\psi_0 = \sum_{\zeta \in \Phi} \gamma_\zeta\zeta  + \sum_{j =j_0}^\infty \sum_{\xi \in \Psi_j} \gamma_\xi \xi, \quad \text{where } \gamma_\xi = \int \psi_0 \xi \ 
\text{for all } \xi \in \Psi,$$
where the series converges uniformly due to the H\"older regularity of $\psi_0$, 
so that, 
\begin{align*}
\int \psi_0  (q-q_{J_n}) 
= \int \left(\sum_{\zeta \in \Phi} \gamma_\zeta\zeta  + \sum_{j =j_0}^\infty \sum_{\xi \in \Psi_j} \gamma_\xi \xi\right) 
         \left(\sum_{j=J_n+1}^\infty \sum_{\xi\in\Psi_j} \beta_\xi\xi\right)  
= \sum_{j=J_n+1}^\infty \sum_{\xi \in \Psi_j}\gamma_\xi \beta_\xi,
  \end{align*}
by orthonormality of the basis $\Psi$.
By \cref{lem:wavelet--basis_size} in Appendix~\ref{app:wavelets}, we have $|\Psi_j| \lesssim 2^{dj}$, therefore
\begin{equation}
\label{eq:step1_one_sample_density}
\left|\int \psi_0 (q-q_{J_n})\right| \leq 
\sum_{j=J_n+1}^\infty \sum_{\xi \in \Psi_j}|\gamma_\xi \beta_\xi|
  \lesssim \sum_{j=J_n+1}^\infty 2^{dj} \|(\gamma_\xi)_{\xi\in\Psi_j}\|_{\infty} \|(\beta_\xi)_{\xi\in\Psi_j}\|_\infty.
  \end{equation}
On the other hand, 
we have 
$\norm\cdot_{\calB_{\infty,\infty}^s(\Omega)} \lesssim \norm\cdot_{\calC^s(\Omega)}$
for all $s > 0$ by Lemma~\ref{lem:besov_holder}. Therefore, by assumption on $q$ and $\varphi_0^*$, we obtain
\begin{equation}
\label{eq:step3_one_sample_density}
\begin{aligned}
\|(\beta_\xi)_{\xi\in\Psi_j}\|_{\ell_\infty} &\leq \norm q_{\calB_{\infty,\infty}^{\alpha-1}(\Omega)}2^{-j[ (\alpha-1) + \frac d 2]}
\lesssim 2^{-j[ (\alpha-1) + \frac d 2]}, \\
\|(\gamma_{\xi})_{\xi\in\Psi_j}\|_{\ell_\infty} &\leq \norm {\psi_0}_{\calB_{\infty,\infty}^{\alpha+1}(\Omega)} 2^{-j[ (\alpha+1) + \frac d 2]}
\lesssim 2^{-j[ (\alpha+1) + \frac d 2]},
\end{aligned}
\end{equation}
for all $j \geq j_0$. 
Combine equations~\eqref{eq:step1_one_sample_density}--\eqref{eq:step3_one_sample_density}
to deduce
\begin{align*}
\big| \bbE L(\hat Q_n)\big| \lesssim \sum_{j=J_n+1}^\infty 2^{dj} 2^{-j[ (\alpha+1) + \frac d 2]}2^{-j[ (\alpha-1) + \frac d 2]} 
 \lesssim \sum_{j=J_n+1}^\infty  2^{-2j \alpha}  
  \lesssim 2^{-2J_n\alpha} \asymp n^{-\frac{2\alpha}{2(\alpha-1)+d}}.
\end{align*}
We next bound the variance $\Var_Q[L(\hat Q_n)]$. Denote~by
$$\psi_{J_n} = \sum_{\zeta\in \Phi} \gamma_\zeta \zeta +  \sum_{j=j_0}^{J_n} \sum_{\xi\in\Psi_j} \xi  \gamma_\xi$$ 
the projection of $\psi_0$ onto $\mathrm{Span}\left(\Phi \cup \bigcup_{j=j_0}^{J_n} \Psi_j\right)$. 
By again applying Lemma~\ref{lem:tilde_q_density}, 
it is a straightforward observation that
\begin{align*}
\left|\Var\left[\int \psi_0 \hat q_n\right] - \Var\left[ \int \psi_0 \tilde q_n\right]\right|
  \lesssim n^{-2},
\end{align*}
thus it suffices to show that $\Var\left[ \int \psi_0 \tilde q_n\right] = \Var_Q[\psi_0(Y)] /n + O(2^{-2J_n\alpha}/n)$. Notice that
\begin{align}
\nonumber 
\int \psi_0\tilde  q_n
 &= \sum_{\zeta \in \Phi} \hbeta_\zeta \int\psi_0\zeta + \sum_{j=j_0}^{J_n} \sum_{\xi\in\Psi_j} \hbeta_\xi \int \psi_0\xi  \\
\nonumber 
 &= \sum_{\zeta \in \Phi} \hbeta_\zeta \gamma_\zeta + \sum_{j=j_0}^{J_n} \sum_{\xi\in\Psi_j} \hbeta_\xi \gamma_\xi \\
 &=\frac 1 n \sum_{i=1}^n \left[\sum_{\zeta \in \Phi} \zeta(Y_i)\gamma_\zeta + \sum_{j=j_0}^{J_n} \sum_{\xi\in\Psi_j} \xi(Y_i) \gamma_\xi\right] = \int \psi_{J_n} dQ_n,
 \label{eq:pf_variance_kantorovich_bound_step}
\end{align}  
whence, 
\begin{align*}
\Var\left[\int \psi_0\tilde q_n\right]
 &= \frac 1 n \Var_Q[\psi_{J_n}(Y)] \\ 
 &=  \frac 1 n \Var_Q[\psi_0(Y)] + \frac 1 n (\Var_Q[\psi_{J_n}(Y)]-\Var_Q[\psi_0(Y)]).
\end{align*} 
It thus remains to bound the final term. Notice that
\begin{align*}
\Big| \Var_Q&[\psi_{J_n}(Y)] - \Var_Q[\psi_0(Y)]\Big| \\
 &{\lesssim} \Big| \bbE_Q [\psi_{J_n}^2(Y) - \psi_0^2(Y)]\Big| +  \Big| \bbE_Q [\psi_{J_n}(Y) - \psi_0(Y)]\Big| = (I) + (II).
\end{align*}
We begin by bounding $(I)$. Letting $g_n = (\psi_{J_n} + \psi_0)q$, we have, 
\begin{align*}
(I) = \left| \int (\psi_{J_n} - \psi_0)(\psi_{J_n} + \psi_0) q \right| = \left| \int (\psi_{J_n} - \psi_0)g_n  \right|.
\end{align*}
It is clear that $\|\psi_{J_n}\|_{\calB^{\alpha+1}_{\infty,\infty}([0,1]^d)} \leq \|\psi_0\|_{\calB^{\alpha+1}_{\infty,\infty}([0,1]^d)}
\lesssim \lambda $,
thus for any fixed $\epsilon > 0$ sufficiently small, the map $\psi_{J_n}+\psi_0$ lies in $\calC^{\alpha+1-\epsilon}([0,1]^d)$ with uniformly bounded norm,
by Lemma~\ref{lem:besov_holder}. Note that one may take $\epsilon=0$ if $\alpha$ is not an integer. 
On the other hand, we also have $\|q\|_{\calC^{\alpha-1}([0,1]^d)} \leq M$.
Deduce that
$$\sup_{n\geq 1}\|g_n\|_{\calB_{\infty,\infty}^{\alpha-1}([0,1]^d)} \lesssim
\sup_{n \geq 1} \|g_n\|_{\calC^{\alpha-1}([0,1]^d)}\lesssim 1,$$
where the first inequality again uses Lemma~\ref{lem:besov_holder},
and the second inequality follows from  Lemma~\ref{lem:holder_products}.
Now, let $\alpha_{n,\xi} = \int \xi g_n$ for all $\xi \in \Psi$. 
By following the same argument as in the first
part of this proof, and using again the fact that
 $\|\psi_0\|_{\calB_{\infty,\infty}^{\alpha+1}([0,1]^d)} \lesssim \lambda$, we may deduce that 
\begin{align*}
(I)
 &\leq \sum_{j=J_n+1}^\infty  \sum_{\xi \in \Psi_j} |\gamma_\xi\alpha_{n,\xi}| \\
 &\leq  \|\psi_0\|_{\calB_{\infty,\infty}^{\alpha+1}([0,1]^d)}\| g_n\|_{\calB^{\alpha-1}_{\infty,\infty}([0,1]^d)}
  \sum_{j=J_n+1}^{\infty} 2^{dj}  2^{-j[ (\alpha+1) + \frac d 2]}2^{-j[ (\alpha-1) + \frac d 2]}  \lesssim 2^{-2J_n\alpha}.
\end{align*}
Likewise,  we have
\begin{align*} 
(II) = \left|\int (\psi_{J_n} - \psi_0) q\right| \lesssim 2^{-2J_n\alpha},
\end{align*}
and the claim follows from here.\qed

\section{Proofs of Two-Sample Stability Bounds}
\subsection{Proof of Proposition~\ref{thm:two_sample_stability}} 
\label{app:pf_two_sample_stability}
Due to the absolute continuity of $P$ and $Q$, the optimal transport map 
from $Q$ to $P$ is given by $S_0 = \nabla\varphi_0^*$. 
Furthermore, by absolute continuity of $P$, there exists 
an optimal transport map $\hsigma$ from  $P$ to $\hat P$. We clearly have,  
$$(\hsigma \circ S_0)_\# Q = \hat P.$$
Also let $\hpi \in \Pi(Q, \hat Q)$ be the optimal coupling between $Q$ and $\hat Q$, so that 
$$(\hsigma \circ S_0, Id)_\# \hpi \in \Pi(\hat P, \hat Q).$$
We deduce,
\begin{align}
\nonumber 
W_2^2(\hat P, \hat Q)
 &\leq \int \norm{\hsigma\circ S_0(y)-z}^2 d\hpi(y,z) \\
\nonumber 
 &= \int \Big[ \norm{\hsigma\circ S_0(y)-S_0(y)}^2 + \norm{S_0(y)-z}^2 \Big] d\hpi(y,z) \\& +
  2\int \langle \hsigma\circ S_0(y)-S_0(y) , S_0(y)-z\rangle d\hpi(y,z).
  \label{eq:pf_two_sample_stability_cross}
\end{align}
Notice that
\begin{equation}
  \label{eq:pf_two_sample_stability_noncross1}
\int \norm{\hsigma\circ S_0(y)-S_0(y)}^2  d\hpi(y,z)
 = \int \|\hat\sigma(x) -x\|^2dP(x) = W_2^2(\hat P, P).
 \end{equation}
Furthermore, we have
\begin{equation}
  \label{eq:pf_two_sample_stability_noncross2}
  \int \|S_0(y)-z\|^2 d\hpi(y,z) \leq W_2^2(P,Q) + \int \psi_0 d(\hat Q - Q)+  \lambda   W_2^2(\hat Q, Q),
\end{equation}
by Lemma~\ref{lem:stability_technical}. Additionally, 
the cross term in equation~\eqref{eq:pf_two_sample_stability_cross} is bounded as follows.
\begin{lemma}
\label{lem:two_sample_stability_technical}
We have,  
$$
\begin{multlined}[0.9\textwidth]
2\int \langle \hsigma\circ S_0(y)-S_0(y) , S_0(y)-z\rangle d\hpi(y,z) \\
\leq \int \phi_0 d(\hat P - P) + 2 W_2(\hat P,P) W_2(\hat Q,Q) + (\lambda-1) W_2^2(\hat P, P).
\end{multlined}$$
\end{lemma}
We prove Lemma~\ref{lem:two_sample_stability_technical} in Appendix~\ref{app:two_sample_stability_technical}
below. 
By equations~(\ref{eq:pf_two_sample_stability_cross}--\ref{eq:pf_two_sample_stability_noncross2})
and Lemma~\ref{lem:two_sample_stability_technical}, we obtain
\begin{align*}
W_2^2(\hat P, \hat Q) 
 &\leq W_2^2(P, Q) + \lambda W_2^2(\hat P, P) + \lambda  W_2^2(\hat Q, Q)  \\ &\qquad\qquad\qquad \qquad + \int \psi_0 d(\hat Q-Q) + \int \phi_0 d(\hat P-P)
 + 2W_2(\hat P, P)W_2(\hat Q,Q) \\
 &\leq W_2^2(P, Q) +  \lambda  \left[W_2(\hat P, P) + W_2(\hat Q, Q)\right]^2 + \int \psi_0 d(\hat Q-Q) + \int \phi_0 d(\hat P-P).
\end{align*}
This proves the upper bound of the claim. To prove the lower bound, notice that, by the Kantorovich duality, 
\begin{align*} 
W_2^2(\hat P, \hat Q) 
 &\geq \int \phi_0 d\hat P + \int \psi_0 d\hat Q \\
 &= \int\phi_0 dP + \int\psi_0 dQ + \int \phi_0 d(\hat P - P) + \int \psi_0 d(\hat Q - Q)  \\
 &= W_2^2(P, Q) + \int \phi_0 d(\hat P - P) + \int \psi_0 d(\hat Q - Q).
\end{align*}
The claim follows.
\qed

\subsection{Proof of Lemma~\ref{lem:two_sample_stability_technical}}
\label{app:two_sample_stability_technical}
Write
\begin{equation}
\label{eq:decomp_tech_two_sample_stab}
2\int \langle \hsigma\circ S_0(y)-S_0(y) , S_0(y)-z\rangle d\hpi(y,z) = (I) + (II) + (III),
\end{equation}
where
\begin{align*}
(I) &= 2\int \langle \hsigma\circ S_0(y)-S_0(y) , y -z\rangle d\hpi(y,z)\\
(II)  &= 2\int \langle \hsigma\circ S_0(y)-S_0(y) , -y\rangle d\hpi(y,z)\\
(III)&= 2\int \langle \hsigma\circ S_0(y)-S_0(y) , S_0(y)\rangle d\hpi(y,z). 
\end{align*}
Regarding $(I)$, the Cauchy-Schwarz inequality implies 
\begin{align}
\nonumber 
(I) &\leq  2
\left(\int \norm{\hsigma\circ S_0(y)-S_0(y)}^2 d\hpi(y,z)\right)^{\frac 1 2}
       \left(\int \norm{y-z}^2 d\hpi(y,z)\right)^{\frac 1 2} \\
   &=2
   \nonumber 
\left(\int \norm{\hsigma(x)-x}^2 dP(x)\right)^{\frac 1 2}
       \left(\int \norm{y-z}^2 d\hpi(y,z)\right)^{\frac 1 2}\\
 &= 2W_2(\hat P, P) W_2(\hat Q,Q).       
       \label{eq:pf_two_sample_stab_cross1}
\end{align}
Regarding term $(II)$, recall that $\varphi_0$ satisfies assumption~\ref{assm:curvature}, thus we have 
$$\frac 1 {2\lambda} \norm{x-y}^2 \leq \varphi_0(y) - \varphi_0(x) - \big\langle T_0(x), y-x\big\rangle 
\leq \frac \lambda 2 \|x-y\|^2, \quad x,y \in \Omega,$$
We deduce that, \begin{align}
\nonumber 
(II) &=2 \int \langle \hsigma\circ S_0(y)-S_0(y) , -T_0\circ S_0(y)\rangle d\hpi(y,z) \\
\nonumber 
 &\leq 
 2\int \left[ \varphi_0(S_0(y)) - \varphi_0(\hsigma\circ S_0(y)) + \frac \lambda 2 \norm{S_0(y) - \hsigma\circ S_0(y)}^2\right] d\hpi(y,z) \\
 &=  \int  2\varphi_0 d(P - \hat P)   +  \lambda   W_2^2(\hat P, P).
\label{eq:pf_two_sample_stab_cross2}
\end{align}
Finally, term $(III)$ satisfies
\begin{align}
\nonumber 
(III)
 &= \int \left[ \norm{ \hsigma\circ S_0(y)}^2 - \norm{\hsigma\circ S_0(y) - S_0(y)}^2 - \norm{S_0(y)}^2 \right] d\hpi(y,z)\\
 &= \int \norm\cdot^2 d(\hat P - P) - W_2^2(\hat P, P).
\label{eq:pf_two_sample_stab_cross3}
\end{align}
Combine equations~\eqref{eq:pf_two_sample_stab_cross1}--\eqref{eq:pf_two_sample_stab_cross3} 
with equation~\eqref{eq:decomp_tech_two_sample_stab} to deduce the claim.\qed  
 
\subsection{Proof of Proposition \ref{prop:curvature_empirical}} 
\label{app:pf_two_sample_stability_empirical}
Once again, denote by $S_0 = \nabla\varphi_0^*$ the optimal transport map from $Q$ to $P$. 
Recall from the proof of Theorem~\ref{thm:stability} (equation~\eqref{eq:curvature}) that, due to assumption~\ref{assm:curvature},
$$\frac 1 { 2\lambda} \norm{x-y}^2 \leq  \varphi_0^*(y) - \varphi_0^*(x) - \big\langle S_0(x), y-x\big\rangle \leq \frac \lambda  {2} \norm{x-y}^2,$$
for all $x,y \in \Omega$. 
Now, we have,
\begin{align*}
W_2^2(P_n, Q_m) 
 &=  \sum_{i=1}^n \sum_{j=1}^m \hat\pi_{ij} \|X_i - Y_j\|^2 \\
 &=  \sum_{i=1}^n \sum_{j=1}^m \hat\pi_{ij} \bigg[\|T_0(X_i) - X_i \|^2 \\ &\hspace{0.9in}+ 2\langle T_0(X_i) - X_i, Y_j - T_0(X_i)\rangle + 
 									 \norm{Y_j - T_0(X_i)}^2\bigg].
\end{align*}
Notice that 
\begin{align*}
\bbE\left[\sum_{i=1}^n \sum_{j=1}^m \hat\pi_{ij} \|T_0(X_i) - X_i \|^2\right]
 &= \bbE\left[\sum_{i=1}^n \left(\sum_{j=1}^m \hat\pi_{ij}\right) \|T_0(X_i) - X_i \|^2\right] \\
 &= \bbE\left[\frac 1 n \sum_{i=1}^n \|T_0(X_i) - X_i \|^2  \right] 
 = W_2^2(P,Q),
\end{align*}
where we have used the marginal constraint on the coupling $\hat\pi$ in the first equality
of the above display. Recalling that $\Delta_{nm} = \sum_{i=1}^n\sum_{j=1}^m \hpi_{ij} \|X_i-Y_j\|^2$,
 thus we obtain,	
\begin{align*}
\bbE \Big[ W_2^2(P_n,& Q_m)  - W_2^2(P,Q)\Big]  \\
  &= \bbE [\Delta_{nm}] + 2\bbE \left[  \sum_{i=1}^n  \sum_{j=1}^m \hat\pi_{ij} \langle T_0(X_i) - X_i, Y_j - T_0(X_i)\rangle 
 									  \right]\\
  &= \bbE [\Delta_{nm}] + 2\bbE \left[  \sum_{i=1}^n  \sum_{j=1}^m \hat\pi_{ij}  \langle T_0(X_i) - S_0(T_0(X_i)) , Y_j - T_0(X_i)\rangle 
 									  \right].
\end{align*}
Now,
\begin{equation}
\label{eq:empirical_stability_step}
2\langle - S_0(T_0(X_i)) , Y_j - T_0(X_i)\rangle \geq 2\varphi_0^*(T_0(X_i)) - 2\varphi_0^*(Y_j) +  \frac 1 \lambda   \norm{T_0(X_i) - Y_j}^2,
\end{equation}
whence, we obtain,
$$\begin{multlined}[0.75\textwidth]				  
\bbE \Big[ W_2^2(P_n, Q_m)  - W_2^2(P,Q)\Big]  
  \geq  \bbE [\Delta_{nm}] + \bbE \Bigg[ \sum_{i=1}^n \sum_{j=1}^m \hat\pi_{ij}\bigg(   2\varphi_0^*(T_0(X_i)) - 2\varphi_0^*(Y_j) \\ + 
\frac 1 \lambda \norm{T_0(X_i) - Y_j}^2 + 2\langle T_0(X_i), Y_j - T_0(X_i)\rangle \bigg) \Bigg]
\end{multlined}$$
Now, notice that
$$2\langle T_0(X_i), Y_j - T_0(X_i)\rangle = -\norm{T_0(X_i) - Y_j}^2 + \norm{Y_j}^2 - \norm{T_0(X_i)}^2.$$
Thus, continuing from before, we have
\begin{align*} 							
\bbE \Big[ W_2^2&(P_n, Q_m)  - W_2^2(P,Q)\Big] 	  	\\	  								  
  &\geq\frac 1 \lambda \bbE [\Delta_{nm}] + \bbE \left[ \sum_{i=1}^n  \sum_{j=1}^m \hat\pi_{ij}\Big(2\varphi_0^*(T_0(X_i)) - 2\varphi_0^*(Y_j) + 
  \|Y_j\|^2- \norm{T_0(X_i)}^2 \Big) \right] \\ 
  &=\frac 1 \lambda \bbE [\Delta_{nm}] + \bbE \left[ \frac 1 m \sum_{j=1}^m \Big(\norm{Y_j}^2 - 2\varphi_0^*(Y_j)\Big)\right] 
  \\ &\hspace{0.95in}  - \bbE\left[\frac 1 n \sum_{i=1}^n \Big(\norm{T_0(X_i)}^2 - 2\varphi_0^*(T_0(X_i))  \Big) \right] 
  =\frac 1 \lambda\bbE [\Delta_{nm}].
\end{align*}
This proves one of the inequalities of the claim. To obtain the other, return to equation~\eqref{eq:empirical_stability_step}
and notice that one also has
\begin{equation*}
2\langle - S_0(T_0(X_i)) , Y_j - T_0(X_i)\rangle \leq 2\varphi_0^*(T_0(X_i)) - 2\varphi_0^*(Y_j) +   \lambda   \norm{T_0(X_i) - Y_j}^2.
\end{equation*}
The proof then proceeds analogously. This proves that
$$\bbE[\Delta_{nm}] \asymp_\lambda \bbE \Big[ W_2^2(P_n, Q_m) - W_2^2(P,Q) \Big].$$
To conclude, apply Proposition~\ref{thm:two_sample_stability} to deduce
\begin{align*}
\bbE \Big[ W_2^2&(P_n, Q_m) - W_2^2(P,Q)\Big] \\
&\leq \bbE \int \phi_0 d(P_n - P) + \bbE \int \psi_0 d(Q_m - Q) + 
2\lambda\Big[\bbE W_2^2(P_n, P) + \bbE W_2^2(Q_m,Q) \Big]\\
&=2\lambda\Big[\bbE W_2^2(P_n, P) + \bbE W_2^2(Q_m,Q) \Big].
\end{align*}
The above display is of the order $\kappa_{n\wedge m}$ due to equation~\eqref{eq:wasserstein_empirical}. 
When we additionally assume that $\Omega = [0,1]^d$ and $\gamma^{-1} \leq p,q \leq \gamma$, 
we may instead bound it from above by $\widebar \kappa_{n\wedge m}$,
due to Corollary~\ref{cor:one_sample_emp_hypercube}. The claim follows.\qed

\subsection{Proof of Proposition~\ref{prop:torus_stability}} 
The claim follows along the same lines as the proofs of 
Theorem~\ref{thm:stability} and Proposition~\ref{thm:two_sample_stability}, 
thus we only provide a brief proof of the analogue of Theorem~\ref{thm:stability}
over the torus. 
It will suffice to prove
\begin{equation} 
\label{eq:stability_pf_torus_sufficient}
   \frac 1 \lambda \|\hat T - T_0\|_{L^2(P)}^2 \leq  \calW_2^2(P, \hat Q) - \calW_2^2(P,Q)  - \int \psi_0 d(\hat Q-Q) \leq \lambda \calW_2^2(\hat Q,Q).
\end{equation}
Recall that $\hat T$ is the optimal transport map from $P$ to $\hat Q$. 
By Proposition~\ref{prop:torus_ot--map_locality}, we therefore have {$P$-almost surely}
$$d_{\bbT^d}(\hat T(x),x) = \|\hat T(x) - x\|,\quad d_{\bbT^d}(T_0(x),x) = \norm{T_0(x) - x},
\quad x \in \bbT^d.$$
It follows that
\begin{align*}
\calW_2^2(P, \hat Q)  - \calW_2^2(P, Q)
 = \int\|\hat T(x) - x\|^2 dP(x)  - \int \|T_0(x) - x\|^2 dP(x).
 \end{align*}
From here, it follows identically as in the proof  of Theorem~\ref{thm:stability} that
\begin{align*}
\calW_2^2(P, \hat Q)  - \calW_2^2(P, Q)
  \geq \frac 1 \lambda \|\hat T - T_0\|_{L^2(P)}^2 + \int \psi_0 d(\hat Q - Q).
\end{align*}
To prove the second inequality in equation~\eqref{eq:stability_pf_torus_sufficient},  
let $\hpi$ denote an optimal coupling between $Q$ and $\hat Q$ with respect to the cost $d_{\bbT^d}^2$. 
Notice similarly as before that Proposition~\ref{prop:torus_ot--map_locality} implies
$$\calW_2^2(P, Q) = \int \|S_0(y) - y\|^2 dQ(y), \quad \calW_2^2(Q, \hat Q) = \int \|y-z\|^2 d\hpi(y,z),$$ 
thus, since $(S_0, Id)_\# \hpi \in \Pi(P,\hat Q)$, and using the fact that
$d_{\bbT^d} \leq \|\cdot\|$, we have
\begin{align*}
\calW_2^2&(P, \hat Q) \\
 &\leq \int d_{\bbT^d}^2(S_0(y),z) d\hpi(y,z) \\
 &\leq \int \|S_0(y)-z\|^2 d\hpi(y,z) \\
 &= \int \norm{S_0(y)-y}^2  dQ(y) + \int \norm{y - z}^2 d\hpi(y,z) + 
    2\int \big\langle S_0(y) - y, y -z\big\rangle  d\hpi(y,z) \\
  &= \calW_2^2(P,Q) + \calW_2^2(\hat Q, Q) + 2 \int \big\langle S_0(y) - y, y -z\big\rangle  d\hpi(y,z).
\end{align*}
By the same argument as in Theorem~\ref{thm:stability}, the cross term is bounded
above by $(\lambda - 1)\calW_2^2(\hat Q, Q)+\int\psi_0d(\hat Q-Q)$, thus the claim follows. 
\qed

\section{Proofs of Upper Bounds for Two-Sample Empirical Estimators}
In this Appendix, we prove Propositions~\ref{prop:transport_1nn} and~\ref{prop:transport_least_squares}. 
We begin with the following result.
\begin{lemma}
\label{lem:voronoi}
Let $\Omega$ satisfy conditions~\ref{assm:supp_global}--\ref{assm:supp_standard}.
Let $P \in \calPac(\Omega)$ admit a density $p$ such that $\gamma^{-1} \leq p \leq \gamma$
for some $\gamma > 0$. Let $V_1, \dots, V_n$ denote the Voronoi partition in equation~\eqref{eq:voronoi},
based on an i.i.d. sample $X_1, \dots, X_n \sim P$. Then,
there exist constants $C_1, C_2 > 0$ depending only on $d,\gamma,\epsilon_0,\delta_0$ such that
{the following assertions hold.}
\begin{thmlist} 
\item \label{lem:voronoi--mass} For all $\delta \in (0,1)$, we have, 
$$\bbP\left(\max_{1 \leq i \leq n} P(V_i) \geq \frac {C_1} n\Big[d\log n +  \log\left(  1 /\delta\right)\Big]\right) \leq \delta.$$
\item \label{lem:voronoi--diam} We have,
$$\bbE\left[ \max_{1 \leq i \leq n} \diam(V_i)^2\right] \leq C_2 \left(\frac {\log n}{n}\right)^{\frac 2 d}.$$
\end{thmlist}
\end{lemma} 
\textbf{Proof of Lemma~\ref{lem:voronoi}.}
 We shall make use of the relative Vapnik-Chervonenkis
 inequality~\citep{vapnik2013, bousquet2003}, in the following form stated by \cite{chaudhuri2010}. 
\begin{lemma}
\label{lem:relative_vc}
Let $\calB$ denote the set of balls in $\bbR^d$. Then, 
there exists a universal constant $C > 0$ such that 
for every $\delta \in (0,1)$, we have with probability at least $1-\delta$ that for all $B \in \calB$,
$$P(B) \geq \frac{C}{n}\left[d\log n +  \log\left(\frac 1 \delta\right)\right] ~ \Longrightarrow ~ P_n(B) > 0.$$
\end{lemma}
We now turn to the proof. Recall that $\Omega$ is a standard set
by  condition~\ref{assm:supp_standard}, and recall the constants
$\epsilon_0,\delta_0 > 0$ therein. For any $1 \leq i \leq n$ and 
{$x \in V_i\setminus\{X_i\}$,
let $\rho_i(x) = (\epsilon_0/2d)\norm{x-X_i}$.
Since $\diam(\Omega) \leq\sqrt  d$ by condition~\ref{assm:supp_global}, 
we have $\rho_i(x) \leq \epsilon_0$.  
We also have $\rho_i(x) < \norm{x-X_i}$, thus the   balls $B(x,\rho_i(x))$ 
of radius $\rho_i(x)$ centered at $x$ contain no sample points. Therefore, 
by  Lemma~\ref{lem:relative_vc}, we have that
for any $\delta \in (0,1)$, with probability at least $1-\delta$,
\begin{equation}
\label{eq:voronoi_pf_step}
\max_{1 \leq i \leq n} \sup_{x \in V_i} P\big(B(x_i,\rho_i(x))\big) \leq  \frac{C}{n}\left[d\log n +  \log\left(\frac 1 \delta\right)\right].
\end{equation}
Now, since $\gamma^{-1}\leq p \leq \gamma$, 
the assumption of standardness on $\Omega$ leads to the bound
$$P(B(x,\rho_i(x))) \geq   \gamma^{-1} \calL(B(x,\rho_i(x))\cap\Omega)
\geq \delta_0\gamma^{-1} \calL(B(x, \rho_i(x))) \asymp \rho_i^d(x),$$
thus equation~\eqref{eq:voronoi_pf_step} reduces to
$$\max_{1 \leq i \leq n} \sup_{x \in V_i} \rho_i^d(x) \leq \frac{C}{n}\left[d\log n +  \log\left(\frac 1 \delta\right)\right].$$
Deduce from here that with probability at least $1-\delta$, for any $1 \leq i \leq n$ and $x,y \in V_i$, 
$$\|x-y\| \leq \|x-X_i\|+ \|y-Y_i\| \lesssim  \rho_i(x) + \rho_i(y) \lesssim \left[\frac{d\log n +  \log\left(  1 /\delta\right)}{n}\right]^{\frac 1 d}.$$
It follows that for some  $C_1 > 0$ not depending on $\delta$, 
we have with probability at least $1-\delta$,
$$\max_{1 \leq i \leq n}\diam(V_i) \leq C_1 \left[\frac{d\log n +  \log\left(  1 /\delta\right)}{n}\right]^{\frac 1 d}.$$
}
To prove claim (i), notice that
since the density of $P$ is bounded from above, we also have
with probability at least $1-\delta$,
$$\max_{1 \leq i \leq n} P(V_i) 
\leq \gamma \max_{1 \leq i \leq n} \calL(V_i)
\lesssim \max_{1 \leq i \leq n} \diam(V_i)^d \lesssim  \frac 1 n\Big[d\log n +  \log\left(  1 /\delta\right)\Big].$$
To prove claim (ii), let $t_n = ( 2C_1^d(d+2)\log n/n)^{2/d}$. 
Set $\delta = n^d \exp\left(-\frac{u^d n}{C_1^d}\right)$ for any $u > 0$ to obtain
\begin{align*}
\bbE\left[\max_{1 \leq i \leq n}\diam(V_i)^2\right]
 &= \int_0^\infty \bbP\left(\max_{1 \leq i \leq n}\diam(V_i)^2 \geq u\right)du  \\
 &\leq t_n  + n^d \int_{t_n}^\infty   \exp\left(-\frac{u^{\frac d 2} n}{C_1^d}\right)du  \\
 &= t_n  + \frac {4n^d} d\int_{t_n^{d/4}}^\infty   \exp\left(-\frac{v^{ 2} n}{C_1^d}\right)v^{\frac{4}{d}-1} dv  \\
 &\lesssim t_n  + n^d \int_{t_n^{d/4}}^\infty   \exp\left(-\frac{v^{ 2} n}{2C_1^d}\right) dv  \\
 &\lesssim t_n  + \frac{n^d}{\sqrt n} \exp\left(-\frac{t_n^{d/ 2} n}{2C_1^d}\right)  \lesssim \left(\frac{\log n}{n}\right)^{\frac 2 d}.
\end{align*} 
The claim follows.
\qed

 \subsection{Proof of Proposition~\ref{prop:transport_1nn}}
\label{app:pf_transport_1nn}
Abbreviate $\hat T_{nm}^{\mathrm{1NN}}$ by $\hat T_{nm}$. We have,  
\begin{align*}
\big\| \hat T_{nm} - T_0\big\|_{L^2(P)}^2  
 &= \sum_{i=1}^n \int_{V_i} \|\hat T_{nm}(x) - T_0(X_i) + T_0(X_i) - T_0(x)\big\|^2 dP(x) \\
 &\lesssim \sum_{i=1}^n \int_{V_i} \Big[ \|\hat T_{nm}(x) - T_0(X_i)\|^2 + \|T_0(X_i) - T_0(x)\|^2   \Big] dP(x).
\end{align*}
To bound the first term, notice that,
\begin{align*}
 \sum_{i=1}^n \int_{V_i}  \|\hat T_{nm}(x) - T_0(X_i)\|^2 dP(x) 
  &=  \sum_{i=1}^n \int_{V_i}   \Bigg\|\sum_{j=1}^m (n \hpi_{ij})Y_j - T_0(X_i)\Bigg\|^2 dP(x) \\
  &=  \sum_{i=1}^n P(V_i)  \Bigg\|\sum_{j=1}^m (n\hpi_{ij}) Y_j - T_0(X_i)\Bigg\|^2   \\
  &\leq  \sum_{i=1}^n P(V_i) \sum_{j=1}^m (n\hpi_{ij})  \left\|Y_j - T_0(X_i)\right\|^2,
\end{align*}
by convexity of $\norm\cdot^2$. Therefore, setting $M_n = \max_{1 \leq i \leq n} P(V_i)$, we obtain
\begin{align*}
\big\| \hat T_{nm} - T_0\big\|_{L^2(P)}^2   \leq 
n \Delta_{nm} \left(\max_{1 \leq i \leq n} P(V_i)\right) + 
\sum_{i=1}^n \int_{V_i}  \|T_0(X_i) - T_0(x)\|^2    dP(x).
 \end{align*}
Since $T_0$ is $\lambda$-Lipschitz by condition~\ref{assm:curvature}, 
the claim is now a consequence of the following simple Lemma, which we isolate for future reference.
\begin{lemma}
\label{lem:two_sample_empirical_technical}
Under the conditions of the first claim of Proposition~\ref{prop:transport_1nn}, we have 
for any $\lambda$-Lipschitz map $F:\Omega\to\Omega$,
\begin{align*}
\bbE\left[\sum_{i=1}^n \int_{V_i} \|F(X_i) - F(x)\|^2 dP(x)\right] &\lesssim_{\lambda,\gamma}(\log n/n)^{2/d},\\
\bbE\left[n\Delta_{nm} \left(\max_{1 \leq i \leq n} P(V_i)\right) \right]   &\lesssim_{\lambda,\gamma,\epsilon_0,\delta_0} (\log n) \kappa_{n\wedge m}.
\end{align*}
If we additionally assume that $\Omega=[0,1]^d$ and $\gamma^{-1} \leq q \leq \gamma$ over $\Omega$, then
\begin{align*}
\bbE\left[n\Delta_{nm} \left(\max_{1 \leq i \leq n} P(V_i)\right) \right]   &\lesssim_{\lambda,\gamma,\epsilon_0,\delta_0} (\log n) \widebar \kappa_{n\wedge m}.
\end{align*}
\end{lemma}
 \subsubsection{Proof of Lemma~\ref{lem:two_sample_empirical_technical}}
The first quantity is easily bounded as follows,
\begin{align*}
\bbE\left[\sum_{i=1}^n \int_{V_i} \|F(X_i) - F(x)\|^2 dP(x)\right]
 &\leq \lambda^2 \bbE\left[\sum_{i=1}^n \int_{V_i} \|X_i - x\|^2 dP(x)\right] \\
 &\leq \lambda^2  \bbE   \left[\sum_{i=1}^n P(V_i) \diam(V_i)^2 \right]  \\
 &\leq \lambda^2 \bbE \left[\max_{1 \leq i \leq n} \diam(V_i)^2\right] 
 \lesssim \left(\frac{\log n}{n}\right)^{\frac 2 d},
\end{align*}
where the final inequality is due to Lemma~\ref{lem:voronoi--diam}.
To bound the second quantity, let~$M_n = \max_{1 \leq i \leq n} P(V_i)$. 
By Lemma~\ref{lem:voronoi--mass}
with $\delta = 1/n^2$, there is a large enough
constant $c > 0$ such that if $m_n = c \log n/n$, then
$\bbP(M_n \geq m_n ) \leq 1/n^2.$  
We have,
\begin{align*}
\bbE\left[nM_n \Delta_{nm}\right] 
 &=  \bbE\left[nM_n I(M_n \geq m_n) \Delta_{nm} \right] + 
 	    \bbE\left[nM_n I(M_n < m_n) \Delta_{nm} \right].
\end{align*}
Notice that $\Delta_{nm}$ is bounded above by $\diam(\Omega)^2$, and $0 \leq M_n \leq 1$, thus,
by Proposition~\ref{prop:curvature_empirical},  
\begin{align*}
\bbE\left[nM_n \Delta_{nm}\right] 
 \lesssim  n\bbP(M_n \geq m_n) +
 	    m_n n\bbE\left[ \Delta_{nm}\right] 
 \lesssim  \frac 1 n  + (\log n)\bbE\left[ \Delta_{nm}\right] 
 \lesssim (\log n)  \kappa_{n\wedge m},
\end{align*}
as desired. 
The final claim follows analogously.
\qed

\subsection{Proof of Proposition \ref{prop:transport_least_squares}}
Abbreviate $\hat T_{nm}^{\mathrm{LS}}$ by $\hat T_{nm}$. Notice first that we have
\begin{align}
\label{eq:step_least_squares} 
\nonumber
\big\|\hat T_{nm} - T_0\big\|_{L^2(P_n)}^2
 &=\frac 1 n \sum_{i=1}^n \big\|\hat T_{nm}(X_i) - T_0(X_i)\big\|^2 \\
\nonumber &= \sum_{i=1}^n\sum_{j=1}^m\hat\pi_{ij}\big\|\hat T_{nm}(X_i) - T_0(X_i)\big\|^2 \\
\nonumber &\lesssim 
\sum_{i=1}^n\sum_{j=1}^m\hat\pi_{ij}\big\|\hat T_{nm}(X_i) - Y_j\big\|^2
+ 
\sum_{i=1}^n\sum_{j=1}^m\hat\pi_{ij}\big\|Y_j - T_0(X_i)\big\|^2 \\
 &\leq 2\sum_{i=1}^n\sum_{j=1}^m\hat\pi_{ij}\big\|Y_j - T_0(X_i)\big\|^2 =2\Delta_{nm},
\end{align}
where the final inequality follows by definition of $\hat T_{nm}$, since $\varphi_0 \in \calJ_\lambda$
under assumption~\ref{assm:curvature}. Therefore,
$$\begin{multlined}[\textwidth]
\big\|\hat T_{nm} - T_0\big\|_{L^2(P)}^2 
= \sum_{i=1}^n \int_{V_i} \|\hat T_{nm} - T_0\big\|^2 dP  
\lesssim  \sum_{i=1}^n \int_{V_i} \Big[
 									\|\hat T_{nm}(x) - \hat T_{nm}(X_i)\|^2  \\ +
 								   \| \hat T_{nm}(X_i) - T_0(X_i)\|^2 + 
 								   \|T_0(X_i)-T_0(x)\|^2\Big] dP(x).
 								   \end{multlined}
$$
By definition of $\calJ_\lambda$ and by assumption~\ref{assm:curvature}, $\hat T_{nm}$ and $T_0$
are both $\lambda$-Lipschitz, thus by Lemma~\ref{lem:two_sample_empirical_technical},  
\begin{align*}
\bbE \big\|\hat T_{nm} - T_0\big\|_{L^2(P)}^2  
 &\lesssim \left(\frac{\log n}{n}\right)^{\frac 2 d} + \bbE\left[ \sum_{i=1}^n \int_{V_i} \| \hat T_{nm}(X_i) - T_0(X_i)\|^2  dP(x)\right]\\
 &\leq \left(\frac{\log n}{n}\right)^{\frac 2 d} + 
          \bbE\left[n \left(\max_{1 \leq i \leq n} P(V_i)\right)\big\|\hat T_{nm} - T_0\big\|^2_{L^2(P_n)}\right]\\
 &\lesssim \left(\frac{\log n}{n}\right)^{\frac 2 d} + 
          \bbE\left[n \left(\max_{1 \leq i \leq n} P(V_i)\right)\Delta_{nm}\right],
\end{align*}
where we used equation~\eqref{eq:step_least_squares}. Lemma~\ref{lem:two_sample_empirical_technical} may
now be applied to bound the right-hand term in the above display, leading to the claim.\qed

\section{Upper Bounds for Two-Sample Wavelet Estimators}
\label{app:pf_density_based}
{
In this section, we state and prove a result deferred from Section~\ref{sec:two_sample_combined},
regarding two-sample plugin estimators based on wavelet density estimation over the torus. 

Unlike the boundary-corrected wavelet system used in~Section~\ref{sec:one_sample_density},
it will be convenient to introduce a simpler basis  
which guarantees that the density estimators are periodic.
Recall that we described in Appendix~\ref{app:per_wavelets}  how the standard Daubechies wavelet system
may be periodized to obtain a set of $\bbZ^d$-periodic functions 
$$\Psi^{\mathrm{per}} = \{1\} \cup \bigcup_{j=0}^\infty \Psi_j^{\mathrm{per}}, \quad \text{where}\quad  
  \Psi_j^{\mathrm{per}} = \big\{\xi_{jk\ell}^{\mathrm{per}}:  0\leq k \leq  2^{j-1}, \ell \in \{0,1\}^d\setminus\{0\}\big\},\ j \geq 0,$$
which forms an orthonormal basis of $L^2(\bbT^d)$~\citep{daubechies1992,gine2016}.
Whenever the  densities $p,q$ lie in $L^2(\bbT^d)$, they admit wavelet expansions of the form
$$p  = 1 + \sum_{j=0}^\infty \sum_{\xi \in \Psi_j\pper} \alpha_\xi \xi,\quad 
  q  = 1 + \sum_{j=0}^\infty \sum_{\xi \in \Psi_j\pper} \beta_\xi \xi,$$
where $\alpha_\xi = \int \xi dP$ and $\beta_\xi = \int \xi dQ$. 
We then define the wavelet density estimators
 $$\tilde p_n\sper =  1+ \sum_{j =0}^{J_n} \sum_{\xi \in \Psi_j\pper} \hat{\alpha}_\xi \xi,\quad 
   \tilde  q_m\sper =  1+ \sum_{j =0}^{J_m} \sum_{\xi \in \Psi_j\pper} \hat{\beta}_\xi \xi,
$$
where $ \hat \alpha_\xi = \int\xi dP_n$ and $\hat\beta_\xi = \int\xi dQ_m$. 
By orthonormality of $\Psi\pper$, it is straightforward to see that $\tilde p_n\sper, \tilde q_m\sper$
integrate to unity, but may nevertheless be negative.  
We therefore define the final density estimators by 
\begin{equation}
\label{eq:final_estimator_periodic}
\hat p_n\sper \propto \tilde p_n\sper I(\tilde p_n\sper \geq 0),\quad 
\hat q_m\sper \propto \tilde q_m\sper I(\tilde q_m\sper \geq 0),
\end{equation}
where the proportionality constants are to be chosen such that $\hat p_n\sper$ and $\hat q_m\sper$ are probability
densities, which respectively induce
probability distributions $\hat P_n\sper, \hat Q_m\sper \in \calPac(\bbT^d)$. 
Once again, we drop all superscripts ``per'' whenever the choice of wavelet basis is unambiguous.   
We state the following bound for the two-sample
estimator~$\hat T_{nm}\equiv \hat T_{nm}\sper$ in equation~\eqref{eq:two_sample_density_estimator},
together with the associated plugin estimator of the squared Wasserstein distance.
Recall the sequence $\trate$ defined in Theorem~\ref{thm:one_sample_wavelet}.  
 \begin{theorem}[Two-Sample Wavelet Estimators]
\label{thm:two_sample_density} 
Let $P,Q \in \calPac(\bbT^d)$ admit  
densities $p,q\in \calC^{\alpha-1}(\bbT^d; M,\gamma)$
for some $\alpha > 1$ and $M,\gamma > 0$. Assume $2^{J_n} \asymp n^{\frac 1 {d + 2(\alpha-1)}}$. 
 Then, there exists a constant $C > 0$ depending only
 on $ M, \gamma, \alpha$ such that the following statements hold.
\begin{thmlist}
\item \label{thm:two_sample_density--map} (Optimal Transport Maps) We have, 
$$\bbE \big\|\hat T_{nm} - T_0\big\|_{L^2(P)}^2  \leq
C  \nmtrate.$$
\item \label{thm:two_sample_density--wasserstein} (Wasserstein Distances)  When $\alpha\not\in \bbN$, we have
\begin{equation*}  
\begin{aligned}
\big|\bbE \calW_2^2(\hat P_n, \hat Q_m) - \calW_2^2(P,Q)\big| &\leq 
C \nmtrate,
\\
\bbE \big|\calW_2^2(\hat P_n, \hat Q_m) - \calW_2^2(P,Q)\big|^2 &\leq
\left[ C\nmtrate {+}  \sqrt{\frac{\Var_P[\phi_0(X)]}{n}+\frac{\Var_Q[\psi_0(Y)]}{m}}\right]^2.
\end{aligned}
\end{equation*} 
\end{thmlist} 
\end{theorem} 
The proof appears in Appendix~\ref{app:pf_two_sample_density}.
Theorem~\ref{thm:two_sample_density} shows that the plugin estimators $\hat T_{nm}$ and $\calW_2^2(\hat P_n,\hat Q_m)$
achieve analogous convergence rates as   in the one-sample setting.   
Similarly as in Section~\ref{sec:one_sample_density}, we 
may deduce from Theorem~\ref{thm:two_sample_density--wasserstein} and Lemma~\ref{lem:kantorovich_L2}   that
$$\bbE \big|\calW_2^2(\hat P_n, \hat Q_m) - \calW_2^2(P,Q)\big| 
 \lesssim_{M,\gamma, \alpha} \nmtrate + (n\wedge m)^{-1/2} \calW_2(P,Q).$$ 
 Thus, in the high-smoothness regime $2(\alpha+1) > d$, 
the risk of $\calW_2^2(\hat P_n, \hat Q_m)$ decays at a rate which adapts to the magnitude of the Wasserstein distance between $P$ and $Q$. 

If one is willing to place assumptions on the regularity of the potentials
$\varphi_0$ and $\varphi_0^*$, Theorem~\ref{thm:two_sample_density--wasserstein} 
may be extended
to the case where the sampling domain is taken to be the unit cube $[0,1]^d$, as we show next. 
Such a result is  made possible by the fact
that Proposition~\ref{thm:two_sample_stability}
does not require any regularity of the fitted potentials. 
On the other hand, we do not know how to obtain an analogue of 
Theorem~\ref{thm:two_sample_density--map} over domains in $\bbR^d$.
Let $P,Q \in \calPac([0,1]^d)$,
and denote by $\hat P_n\sbc$ and $\hat Q_m\sbc$ the boundary corrected
wavelet estimators   defined in Section~\ref{sec:one_sample_density}. 
 \begin{proposition}
 \label{thm:two_sample_density--wasserstein_cube} 
Let $P,Q \in \calPac([0,1]^d)$ admit  
densities $p,q\in \calC^{\alpha-1}([0,1]^d; M,\gamma)$
for some $\alpha > 1$ and $M,\gamma > 0$. Assume further that 
for some $\lambda > 0$, 
\begin{equation}
\label{eq:assumption_varphi0_conj}
\varphi_0,\varphi_0^* \in  \calC^{\alpha+1}([0,1]^d;\lambda).
\end{equation}
Let $2^{J_n} \asymp n^{\frac 1 {d + 2(\alpha-1)}}$. 
 Then, there exists a constant $C > 0$ depending only
 on $ M, \lambda, \gamma, \alpha$ such that,
\begin{equation*}
\begin{aligned}
\big|\bbE W_2^2(\hat P_n\sbc, \hat Q_m\sbc) - W_2^2(P,Q)\big| &\leq 
C \nmtrate,
\\
\bbE \big| W_2^2(\hat P_n\sbc, \hat Q_m\sbc) - W_2^2(P,Q)\big|^2 &\leq
  \left[C\nmtrate {+} \sqrt{ \frac{\Var_P[\phi_0(X)]}{n}  {+}  \frac{\Var_Q[\psi_0(Y)]}{m}}\right]^2.
\end{aligned}
\end{equation*} 
\end{proposition}
The proof follows along similar lines as that of Theorem~\ref{thm:two_sample_density--wasserstein}, which 
will be given below,
and is therefore omitted. 

Condition~\eqref{eq:assumption_varphi0_conj} places a smoothness assumption   
on $\varphi_0^*$ in addition to~$\varphi_0$. 
If our analysis could be carried out over a domain $\Omega \subseteq \bbR^d$
with smooth boundary, then, under appropriate boundary conditions
on the potentials and under the assumptions made on $p,q$, 
standard Schauder theory~\citep{gilbarg2001} could be applied to the Monge-Amp\`ere equation
to obtain that $\varphi_0^* \in \calC^{\alpha+1}(\Omega)$ as soon 
as $\varphi_0 \in \calC^2(\Omega)$, 
with uniform H\"older norms (see Proposition 9.1 of~\cite{caffarelli1995}). 
We do not know whether analogues of such results can be applied over the hypercube $[0,1]^d$, 
thus we have placed assumptions both on $\varphi_0$ and its convex conjugate.
%
%


We now turn to the proof of Theorem~\ref{thm:two_sample_density}.  
} We first note that the one-sample results from Section~\ref{sec:one_sample_density}
may readily be extended to the optimal transport problem over $\bbT^d$.
\begin{proposition}
\label{prop:one_sample_torus}
Assume $P,Q \in \calPac(\bbT^d)$ admit 
densities $p,q \in \calC^{\alpha-1}(\bbT^d;M,\gamma)$ 
for some $\alpha > 1$, $\alpha\not\in\bbN$, and $M,\gamma > 0$. Let $\hat q_m=\hat q_m\sper$ be the periodic wavelet estimator
defined in equation~\eqref{eq:final_estimator_periodic}, and let $\hat Q_m$ be the induced probability distribution. 
Let
$$\widebar T_m  = \argmin_{T \in \calT(P,\hat Q_m)} \int d_{\bbT^d}^2(x,T(x))dP(x).$$
Furthermore, let $2^{J_m} \asymp m^{\frac 1 {2(\alpha-1) + d}}$. Then, there exists a constant $C > 0$ depending only on $M,\gamma,\alpha$ such that the following
statements hold. 
\begin{thmlist} 
\item\label{prop:one_sample_torus_wasserstein_under}
We have $\bbE \calW_2^2(\hat Q_m, Q) \leq C  \mtrate$ and $\bbE \calW_2^4(\hat Q_m, Q) \leq C \sqmtrate$. 
\item \label{prop:one_sample_torus_functional} We have, 
\begin{align*} 
&\left|\bbE \int \psi_0 d(\hat Q_m - Q)\right| \leq C 2^{-2J_m\alpha} \\
&\left| \Var\left[\int\psi_0 d(\hat Q_m - Q)\right]  - \frac{\Var_Q[\psi_0(Y)]}{m}\right| 
\leq  \frac {C2^{-2J_m\alpha}}{m}.
\end{align*}
\item  \label{prop:one_sample_torus_wasserstein} 
We have, 
\begin{align*}  
\bbE \|\widebar T_m - T_0\|_{L^2(P)}^2 &\leq C \mtrate, \\
\big| \bbE  \calW_2^2(P, \hat Q_m) - \calW_2^2(P,Q)\big|&\leq C\mtrate,\\
\bbE \big| \calW_2^2(P, \hat Q_m) - \calW_2^2(P,Q)\big|^2 &\leq
    \left[ C \mtrate + \sqrt{\frac{\Var_Q[\psi_0(Y)]}{m}}\right]^2. 
\end{align*}
\end{thmlist}
\end{proposition}
Notice that the only properties of the boundary-correct wavelet basis used in the proofs of Lemma~\ref{lem:wavelet_wasserstein}
and Theorem~\ref{thm:one_sample_wavelet} are those contained in Lemmas~\ref{lem:wavelet} and Lemma~\ref{lem:wavelet_Linfty}
of Appendix~\ref{app:wavelets},
which are also stated to hold for the periodic wavelet basis. The proof of Proposition~\ref{prop:one_sample_torus}
is therefore a direct extension of these results. 
Notice that, unlike
Theorem~\ref{thm:one_sample_wavelet}, we no longer require any conditions on the smoothnes of~$\varphi_0$ itself,
due to the torus regularity result of Theorem~\ref{thm:torus_regularity}. 
Indeed, 
under the assumptions of Proposition~\ref{prop:one_sample_torus}, the latter implies that
there exists a constant $C' > 0$ depending only on $\alpha, \gamma, M$ such that $\norm{\varphi_0}_{\calC^{\alpha+1}(\bbT^d)} \leq C'$, assuming $\alpha\not\in\bbN$.

\subsection{Proof of Theorem~\ref{thm:two_sample_density}}
\label{app:pf_two_sample_density} 
Throughout the proof, we use the abbreviations 
$$F(\hat P_n) = \int \phi_0 d(\hat P_n-P), \quad L(\hat Q_m) = \int \psi_0 d(\hat Q_m-Q).$$
We begin by proving part (ii). Under the assumptions of this case,  
Theorem~\ref{thm:torus_regularity}
implies that  $\norm{\varphi_0}_{\calC^{\alpha+1}(\bbT^d)} \leq M_0$ 
for a universal constant $M_0 > 0$ depending only
on $\alpha,\gamma$ and $M$. 
In particular, it also follows from Proposition~\ref{prop:torus_ot--curvature}
that $\varphi_0$ is strongly convex, and thus satisfies
condition~\ref{assm:curvature} for some $\lambda > 0$ depending only on $M_0$ and $\gamma$. We may therefore invoke the two-sample
stability bound over $\bbT^d$ in Proposition~\ref{prop:curvature_empirical} (arising from Proposition~\ref{thm:two_sample_stability}) to deduce
$$
\begin{multlined}[0.9\textwidth]
F(\hat P_n) + L(\hat Q_m) \leq 
  \calW_2^2(\hat P_n, \hat Q_m) 
 - \calW_2^2(P,Q) \\
 \leq  F(\hat P_n) + L(\hat Q_m)
  +  2\lambda\left[\calW_2^2(\hat P_n, P) +  \calW_2^2(\hat Q_m, Q)\right].
  \end{multlined}$$
From Proposition~\ref{prop:one_sample_torus_functional}, it can be deduced that 
\begin{align}
\label{eq:pf_one_sample_wavelet_bias_functional}
\big| \bbE F(\hat P_n)\big| \vee \big|\bbE  L(\hat Q_m)  \big| &\lesssim \nmtrate \\
\label{eq:pf_one_sample_wavelet_variance_functional_P}
\Var\big[ F(\hat P_n)\big] &\leq \frac{\Var_P[\phi_0(X)]}{n} + C \sqtrate \\
\label{eq:pf_one_sample_wavelet_variance_functional_Q}
\Var\big[ L(\hat Q_m)\big] &\leq \frac{\Var_Q[\psi_0(Y)]}{m} + C \sqmtrate,
\end{align}
%
for a constant $C > 0$ depending only on $M,\gamma,\alpha$, whose value we allow to change from line to line in the 
remainder of the proof. Thus, recalling Proposition~\ref{prop:one_sample_torus_wasserstein_under}, 
\begin{align*}
\big| \bbE \calW_2^2(\hat P_n,& \hat Q_m) - \calW_2^2(P, Q)\big| \\
 &\lesssim \big|\bbE  F(\hat P_n) \big| + 
      \big| \bbE L(\hat Q_m)\big| + 
      \bbE \calW_2^2(\hat P_n, P) + \bbE \calW_2^2(\hat Q_m, Q)  \lesssim \nmtrate.
\end{align*} 
Furthermore,
\begin{align*}
\bbE\big| & \calW_2^2(\hat P_n, \hat Q_m) - \calW_2^2(P, Q)\big|^2 \\
 &\leq  \bbE\left[ \left(\big| F(\hat P_n)\big| + \big|L(\hat Q_m)\big| + 2\lambda\big(\calW_2^2(\hat P_n, P) + \calW_2^2(\hat Q_m,Q)\big)\right)^2\right]
  =: (I) + (II) + (III),
\end{align*}
where
\begin{align*}
(I)  &= \bbE \left[ \left( \big|F(\hat P_n)\big| + \big|L(\hat Q_m)\big|\right)^2\right] \\
(II) &= 4\lambda^2 \bbE \left[  \left( \calW_2^2(\hat P_n, P) + \calW_2^2(\hat Q_m, Q)\right)^2 \right] \\
(III)&= 4\lambda \bbE\left[
			\left( \calW_2^2(\hat P_n, P) + \calW_2^2(\hat Q_m, Q)\right)
			\left( \big| F(\hat P_n)\big| +  \big|L(\hat Q_m) \big|\right)
			\right].
\end{align*}
Regarding term $(I)$,  recall that we have assumed that $X_i$ is independent of  $Y_j$ for all $i,j=1, \dots, n$.
Therefore, using equations~(\ref{eq:pf_one_sample_wavelet_bias_functional}--\ref{eq:pf_one_sample_wavelet_variance_functional_Q}),
\begin{align*}
(I)
 &= \bbE\big[ F^2(\hat P_n)\big]  + \bbE \big[L^2(\hat Q_m)\big] + 2 \bbE\big|F(\hat P_n)L(\hat Q_m)\big|\\
 &= \bbE\big[ F^2(\hat P_n)\big]  + \bbE \big[L^2(\hat Q_m)\big] + 2 \bbE\big|F(\hat P_n)\big|\bbE\big|L(\hat Q_m)\big|\\
 &= \Var\big[F(\hat P_n)\big] + \Var\big[L(\hat Q_m)\big] + \big| \bbE F(\hat P_n)\big|^2 + \big| \bbE L(\hat Q_m)\big|^2
        + 2 \bbE\big|F(\hat P_n)\big|\bbE\big|L(\hat Q_m)\big| \\  
 &\leq \frac{\Var_P[\phi_0(X)]}{n} + \frac{\Var_Q[\psi_0(Y)]}{m} + C\sqnmtrate.
\end{align*}
Furthermore, by  Proposition~\ref{prop:one_sample_torus_wasserstein_under},  we   have
\begin{align*}
(II) \leq 8\lambda^2 \left(\bbE \calW_2^4(\hat P_n, P) + \bbE \calW_2^4(\hat Q_m, Q)\right)
 \leq C \sqnmtrate,
\end{align*}
and, using the Cauchy-Schwarz inequality and equations~(\ref{eq:pf_one_sample_wavelet_bias_functional}--\ref{eq:pf_one_sample_wavelet_variance_functional_Q}), we obtain
\begin{align*}
(III)
 &\leq C \sqrt{\left(\bbE \calW_2^4(\hat P_n, P) + \bbE \calW_2^4(\hat Q_m, Q)\right)
 					   \left(\bbE\big| F(\hat P_n)\big|^2 + \bbE\big|L(\hat Q_m)\big|^2\right)} \\
 &\leq C \sqrt{\sqnmtrate \left(C\sqnmtrate + \frac{\Var_P[\phi_0(X)]}{n} + \frac{\Var_Q[\psi_0(Y)]}{m} \right) } \\
 &\leq C \sqnmtrate + C \nmtrate \sqrt{\frac{\Var_P[\phi_0(X)]}{n} + \frac{\Var_Q[\psi_0(Y)]}{m} }.           
 \end{align*}
 Deduce that
$$(I)+(II)+(III) \leq \left(C \nmtrate + \sqrt{\frac{\Var_P[\phi_0(X)]}{n} + \frac{\Var_Q[\psi_0(Y)]}{m} }\right)^2.$$
Claim (ii) follows from here. 

To prove part (i), we shall
make use of the one-sample optimal transport problem from $P$ to $\hat Q_m$. 
Denote by $\widebar \varphi_m$ an optimal Brenier potential for this problem, 
so that $\widebar  T_m = \nabla \widebar\varphi_m$ is the optimal transport
map pushing $P$ forward onto $\hat Q_m$, with respect to 
the cost function $d_{\bbT^d}^2$.
Furthermore, denote~by 
$$\widebar\phi_m = \norm\cdot^2-2\widebar\varphi_m, \quad 
  \widebar\psi_m = \norm\cdot^2-2\widebar\varphi_m^*,$$ 
a corresponding pair of optimal Kantorovich potentials. 
We proceed with three steps. 

\noindent {\bf Step 1: Regularity of the Fitted Potentials.}
Recall that $\alpha > 1$, and fix $  \epsilon \in (0, 1 \wedge \frac{\alpha-1}{2})$. 
By Lemma~\ref{lem:wavelet_Linfty} and Lemma~\ref{lem:tilde_q_density}, under our choice of 
threshold $J_n$, and under the assumption $p,q \in \calC^{\alpha-1}(\bbT^d;M,\gamma)$, 
it can be deduced that the event
\begin{align*}
E_{nm} = 
\left\{\tilde p_n = \hat p_n \right\} &\cap 
\left\{\tilde q_m = \hat q_m \right\} \\ 
 &\cap \left\{\tilde p_n,\tilde q_m\geq 1/(2\gamma)~ \text{over } \bbT^d\right\} \\
 &\cap  
\left\{\|\tilde p_n\|_{\calB_{\infty,\infty}^{\epsilon}(\bbT^d)} \leq 2 \|p\|_{\calB_{\infty,\infty}^{\alpha-1}(\bbT^d)}\right\}  \cap 
\left\{\|\tilde q_m\|_{\calB_{\infty,\infty}^{\epsilon}(\bbT^d)} \leq 2 \|q\|_{\calB_{\infty,\infty}^{\alpha-1}(\bbT^d)}\right\}
\end{align*}
satisfies $\bbP(E_{nm}^\cp)\lesssim (n\wedge m)^{-2}$. 
Note that $\epsilon\not\in\bbN$, thus by Lemma~\ref{lem:besov_holder},
we have on the event~$E_{nm}$,  
$$\norm{\hat q_m}_{\calC^\epsilon(\bbT^d)} \lesssim  
  \norm{\hat q_m}_{\calB_{\infty,\infty}^\epsilon(\bbT^d)} \lesssim 
  \norm{q}_{\calB_{\infty,\infty}^{\alpha-1}(\bbT^d)} \lesssim 
  \norm{q}_{\calC^{\alpha-1}(\bbT^d)}\leq M,$$ 
and similarly for $\hat p_n$. Thus, there exists $M_1 > 0$
depending only on $M,\gamma$   such that
$$\norm{\hat p_n}_{\calC^{\epsilon}(\bbT^d)},\norm{\hat q_m}_{\calC^{\epsilon}(\bbT^d)}
\leq M_1,
\quad \text{on } E_{nm}.$$  
Under the preceding display, together with the smoothness assumptions
on the population densities $p,q$ themselves, and the fact that $\hat p_n,\hat q_m,p,q \geq 1/(2\gamma)$ over
$\bbT^d$ on the event
$E_{nm}$, we may apply the regularity Theorem~\ref{thm:torus_regularity}
to deduce that there exists a constant $M_2 > 0$ depending only on $M_0,M_1, \gamma$ 
such that for all $n,m \geq 1$, 
\begin{align} 
\label{eq:regularity_fitted_potentials}
\norm{\varphi_0}_{\calC^{2+\epsilon}([0,1]^d)}   \vee
\norm{\hat\varphi_{nm}}_{\calC^{2+\epsilon}([0,1]^d)}   \vee
\norm{\widebar\varphi_{m}}_{\calC^{2+\epsilon}([0,1]^d)}    \leq M_2,
\end{align}
on $E_{nm}$. Deduce from Proposition~\ref{prop:torus_ot--potential_periodicity} 
that the Hessians of the above potentials are uniformly bounded  
over $\bbR^d$. Further apply Proposition~\ref{prop:torus_ot--curvature}
to deduce that  $\hat \varphi_{nm}$ and $\widebar \varphi_m$ satisfy
the curvature condition~\ref{assm:curvature} almost surely, up to modifying the value of $\lambda > 0$
in terms of $M_2$ and $\gamma$, namely:
\begin{align}
\label{eq:fitted_curvature_2sample}
\lambda^{-1} I_d \preceq \nabla^2\varphi_0(x),\nabla^2\widebar\varphi_m(x),\nabla^2\hat\varphi_{nm}(x) \preceq \lambda I_d,\quad 
\text{for all }x \in \bbR^d;~ n,m\geq 1,
\end{align}
on the event $E_{nm}$.

\noindent {\bf Step 2: Reduction to Optimal Transport Problems with Same Source Distribution.}
 In order to appeal to the one-sample stability bounds of Theorem~\ref{thm:stability}, write 
\begin{align}
\label{eq:step_reduction_two_sample_same_source1}
\big\|\hat T_{nm} - T_0\big\|_{L^2(P)}^2
 &\lesssim \big\|\hat T_{nm} - \widebar T_m\big\|_{L^2(P)}^2 + 
           \big\|\widebar T_{m} - T_0\big\|_{L^2(P)}^2.
 \end{align}
The first term in the above display compares transport maps which are optimal for distinct
source distributions. We therefore proceed with the following reduction, over the event $E_{nm}$:
\begin{align}
\nonumber
\big\|\hat T_{nm} - \widebar T_m\big\|_{L^2(P)}^2
 &= \int_{\bbT^d} \big\|\hat T_{nm}(x) - \widebar T_m(x)\big\|^2 dP(x) \\
\nonumber
 &= \int_{\bbT^d}  \big\|\hat T_{nm}(\widebar T_m^{-1}(y)) - y \big\|^2 d\hat Q_m(y) \\ 
 &= \int_{\bbT^d}  \big\|\hat T_{nm}(\widebar T_m^{-1}(y)) - \hat T_{nm}(\hat T_{nm}^{-1}(y)) \big\|^2 d\hat Q_m(y),
\label{eq:step_reduction_two_sample_same_source2}
\end{align}
where the second line follows from the fact that $({{\widebar T}_m})_\# P = \hat Q_m$,
and the third follows by invertibility of $\hat T_{nm}$, which is ensured
by the strong convexity of $\hat\varphi_{nm}$ in equation~\eqref{eq:fitted_curvature_2sample}. 
This same equation implies that,  on the event $E_{nm}$, 
$\hat T_{nm}=\nabla\hat\varphi_{nm}$ is Lipschitz with a uniform constant. It follows that 
\begin{align}
\label{eq:step_reduction_two_sample_same_source3}
\big\|\hat T_{nm} - \widebar T_m\big\|_{L^2(P)}^2
 &\lesssim \int_{\bbT^d}  \big\| \hat T_{nm}^{-1}(y) - \widebar T_m^{-1}(y)\big\|^2 d\hat Q_m(y)
  = \big\| \hat T_{nm}^{-1}- \widebar T_m^{-1} \big\|_{L^2(\hat Q_m)}^2.
\end{align}
{\bf Step 3: Stability Bounds.} Due to the inequalities~\eqref{eq:fitted_curvature_2sample}, 
the stability bounds of Proposition~\ref{prop:torus_stability} (arising from Theorem~\ref{thm:stability}) imply
 \begin{align}
 \big\|\hat T_{nm}^{-1}-\widebar T_m^{-1}\big\|_{L^2(\hat Q_m)}^2
 &\leq  \lambda^2  \calW_2^2(\hat P_n, P),   \quad 
   \big\|\widebar T_{m} - T_0\big\|_{L^2(P)}^2
  \leq  \lambda^2  \calW_2^2(\hat Q_m, Q).
\end{align}
Thus, combined with equations~\eqref{eq:step_reduction_two_sample_same_source1} and \eqref{eq:step_reduction_two_sample_same_source3}, we have on the event $E_{nm}$,
\begin{align*}
\big\|\hat T_{nm} - T_0\big\|_{L^2(P)}^2
 &\lesssim \calW_2^2(\hat P_n, P) + \calW_2^2(\hat Q_m, Q).
 \end{align*}
We deduce, 
\begin{align*}
\bbE \big\|\hat T_{nm} - T_0\big\|_{L^2(P)}^2
 &= \bbE \left[  \big\|\hat T_{nm} - T_0\big\|_{L^2(P)}^2 I_{E_{nm}}\right] + 
    \bbE \left[  \big\|\hat T_{nm} - T_0\big\|_{L^2(P)}^2 I_{E_{nm}^\cp}\right] \\
 &\lesssim \bbE \left[  \calW_2^2(\hat P_n, P) I_{E_{nm}}\right] +
           \bbE \left[  \calW_2^2(\hat Q_m, Q) I_{E_{nm}}\right] +
           \bbP(E_{nm}^\cp) \\
 &\leq \bbE \left[  \calW_2^2(\hat P_n, P)  \right] +
           \bbE \left[  \calW_2^2(\hat Q_m, Q) \right] +
           \bbP(E_{nm}^\cp) \\           
 &\lesssim \trate + \mtrate + (n\wedge m)^{-2} \lesssim \nmtrate,
\end{align*}
where we made use of Proposition~\ref{prop:one_sample_torus_wasserstein_under}  on the final line. 
The claim follows.
\qed

\section{Proofs of Upper Bounds for Two-Sample Kernel Estimators}
\label{app:pf_kernel_based}

The goal of this Appendix is to prove Theorem~\ref{thm:two_sample_kernel}. 
For ease of notation, we omit the superscript ``ker'' in all kernel-based estimators, and write
$$p_{h_n}(x) = \bbE[\tilde p_n(x)] = (p \star K_{h_n})(x), \quad 
  q_{h_m}(y) = \bbE[\tilde q_m(y)] = (q \star K_{h_m})(y),\quad x,y \in \bbT^d.$$
We begin with the following technical result.
\begin{lemma}
\label{lem:ker_key_lemma} 
Let $t,s > 0$, and assume $p \in H^s(\bbT^d)$. Assume further that the kernel $K$ satisfies 
condition~\hyperref[assm:kernel]{\textbf{K1($s+t,\kappa$)}} for some $\kappa > 0$.
Then, for any $h_n > 0$, 
$$\norm{p_{h_n} - p}_{\dH{-t}(\bbT^d)}  \leq \kappa \norm p_{H^s(\bbT^d)} h_n^{s+t}$$
\end{lemma} 
{\bf Proof of Lemma~\ref{lem:ker_key_lemma}.} 
By definition of the $\dH{-t}(\bbT^d)$ norm, 
$$\norm{p_{h_n}-p}_{\dH{-t}(\bbT^d)}  = \norm{\|\cdot\|^{-t} \calF[p_{h_n}-p](\cdot)}_{\ell^2(\bbZ^d)}.$$
Furthermore, using standard properties of the Fourier transform,
$$\calF[p_{h_n}-p]  = \calF[p \star K_{h_n} - p]  = (\calF [K](h_n\cdot) - 1) \calF [p].$$
Thus, using assumption~\hyperref[assm:kernel]{\textbf{K1($s+t,\kappa$)}}, we have
\begin{align*}
\norm{p_{h_n}-p}_{\dH{-t}(\bbT^d)}^2
 &= \sum_{\xi\in \bbZ^d} \frac 1 {\norm \xi^{2t}}\big| \calF[K](h_n\xi) - 1\big|^2  \calF[p]^2(\xi) \\
 &\leq  \kappa^2  \sum_{\xi\in \bbZ^d} \frac{\norm{h_n \xi}^{2(s+t)}}{\norm \xi^{2t}} \calF[p]^2(\xi) \\
 &=  \kappa^2 h_n^{2(s+t)} \sum_{\xi\in \bbZ^d} \norm \xi^{2s} \calF[p]^2(\xi) \leq \kappa^2 h_n^{2(s+t)} \norm p_{H^s(\bbT^d)}^2,
\end{align*}
as claimed.\qed 

While Lemma~\ref{lem:ker_key_lemma} will be needed in the proof of Theorem~\ref{thm:two_sample_kernel} below, 
we begin by showing how it may also be used to derive a rate of convergence of $\hat Q_n$
under the Wasserstein distance. The following result was anticipated by~\cite{divol2021}, who derived a Fourier-analytic
proof of the convergence rate of the empirical measure under the Wasserstein distance on $\bbT^d$. 
Our proof follows along similar lines, and is simplified by the fact that we work only with the Wasserstein distance of second order, 
but is complicated by the fact that we require general exponents $\rho\geq 0$.  
\begin{lemma}
\label{lem:w2_kernel_rate} 
Let $s>0$. Assume $P \in \calPac(\bbT^d)$ admits a density $p$ such that
$$\norm{p}_{H^{s}(\bbT^d)}  \leq R < \infty, \qquad 0 < \gamma^{-1} \leq p \leq \gamma < \infty.$$
Assume further that the kernel $K$ satisfies 
condition~\hyperref[assm:kernel]{\textbf{K1($s+1,\kappa$)}} for some $\kappa > 0$.
Set $h_n \asymp n^{\frac 1 {2s+d}}$. Then, for any $\rho \geq 0$, 
$$\bbE \calW_2^\rho(\hat P_n, P) \lesssim_{R,\rho,\gamma,s} 
\begin{cases}
n^{-\frac{\rho(s+1)}{2s+d}}, & d \geq 3 \\
(\log n/n)^{\rho/2}, & d = 2 \\
(1/n)^{\rho/2}, & d = 1.
\end{cases}$$
\end{lemma}
\begin{proof} 
By Jensen's inequality, it suffices to prove the claim for $\rho \geq 2$. 
It is a direct consequence of Lemma~\ref{lem:kernel_infty} 
and the assumption $ \gamma^{-1}\leq p \leq \gamma$ that the event $A_n=\{\hat p_n = \tilde p_n\}$ satisfies $\bbP(A_n) \lesssim 1/n^2$. 
Furthermore, recall from equation~\eqref{eq:peyre}, arising from the work of~\cite{peyre2018}, that 
$$\calW_2(\hat P_n, P) \lesssim \|\hat p_n - p\|_{\dH{-1}}.$$
We therefore have, 
\begin{align}
\nonumber
\bbE \calW_2^\rho(\hat P_n, P) 
 &= \bbE \Big[ \calW_2^\rho(\hat P_n, P) I_{A_n}\Big] + \bbE \Big[ \calW_2^\rho(\hat P_n, P) I_{A_n^\cp}\Big] \\
\nonumber
 &\lesssim \bbE \Big[ \|\hat p_n-p\|_{\dH{-1}}^\rho I_{A_n}\Big] +1/n^2 \\
\nonumber
 &= \bbE \Big[ \|\tilde p_n-p\|_{\dH{-1}}^\rho I_{A_n}\Big] +1/n^2 \\ 
\nonumber
 &\lesssim  \|p_{h_n} - p\|_{\dH{-1}}^\rho + \bbE \|\tilde p_n - p_{h_n}\|_{\dot H^{-1}}^\rho  + 1/n^2 \\
 &\lesssim  h_n^{\rho(s+1)} + \bbE \|\tilde p_n - p_{h_n}\|_{\dot H^{-1}}^\rho  + 1/n^2,
 \label{eq:pf_kernel_first_reduction}
\end{align} 
where we used Lemma~\ref{lem:ker_key_lemma} on the final line, together with the assumption~\hyperref[assm:kernel]{\textbf{K1($s+1,\kappa$)}}.
To bound the variance term,  
write  $\bbE  \| \tilde p_n - p_{h_n}\|_{\dH{-1}}^\rho  \lesssim S_{n,1} + S_{n,2}$, where
\begin{align*} 
S_{n,1} &:=  \bbE \left[\left(\sum_{\xi\in \bbZ^d,\|h_n\xi\|\leq 1} \norm \xi^{-2} \big|\calF[\tilde p_n - p_{h_n}](\xi)\big|^2\right)^{\frac \rho 2}\right],\\
S_{n,2} &:= \bbE \left[\left(\sum_{\xi\in \bbZ^d,\|h_n\xi\|>  1} \norm \xi^{-2} \big|\calF[\tilde p_n - p_{h_n}](\xi)\big|^2\right)^{\frac \rho 2}\right].
\end{align*}
We begin by bounding $S_{n,1}$. Recall that 
$$\calF[\tilde p_n - p_{h_n}](\xi) = \calF[K](h_n\xi) \frac 1 n\sum_{j=1}^n \left( e^{-2\pi i \langle X_j,\xi\rangle} - \calF[p](\xi)\right),\quad \xi\in \bbZ^d,$$
where $i^2 = -1$. 
In fact, since $\tilde p_n$ and $p_{h_n}$ integrate to the same constant, we have $\calF[\tilde p_n - p_{h_n}](0)=0$. 
Furthermore, let $\rho' \in \bbR$ satisfy $\frac 1 \rho + \frac 1 {\rho'} = \frac 1 2$.  Then, 
for any $\eta \in \bbR$, we have by H\"older's inequality, 
\begin{align*} 
S_{n,1} 
 &= \bbE \left[\left(\sum_{\xi\in \bbZ^d,\|h_n\xi\|\leq 1} \norm \xi^{-2\eta} 
 \norm \xi^{2(\eta-1)}\big|\calF[\tilde p_n - p_{h_n}](\xi)\big|^2\right)^{\frac \rho 2}\right]\\
 &\leq \left(\sum_{\substack{\xi\in \bbZ^d,\xi \neq 0 \\ \|h_n\xi\|\leq 1}} \norm \xi^{-\rho'\eta} \right)^{\frac \rho {\rho'}} \bbE \left[\sum_{\substack{\xi\in \bbZ^d,\xi \neq 0 \\ \|h_n\xi\|\leq 1}} \norm \xi^{\rho(\eta-1)} \big|\calF[\tilde p_n - p_{h_n}](\xi)\big|^\rho\right] \\
 &=\left(\sum_{\substack{\xi\in \bbZ^d,\xi \neq 0 \\ \|h_n\xi\|\leq 1}}\norm \xi^{-\rho'\eta} \right)^{\frac \rho {\rho'}}  
         \sum_{\substack{\xi\in \bbZ^d,\xi \neq 0 \\ \|h_n\xi\|\leq 1}}\norm \xi^{\rho(\eta-1)}\big|\calF[K](h_n\xi)\big|^\rho
      \bbE\left| \frac 1 n\sum_{j=1}^n Z_j(\xi)\right|^\rho,
\end{align*} 
where $Z_j(\xi) = e^{-2\pi i \langle X_j,\xi\rangle} - \calF[p](\xi)$, for all $j=1,\dots, n$ and $\xi\in \bbZ^d$. 
Since $\rho \geq 2$, it can be deduced from
Rosenthal's inequalities~\citep{rosenthal1970,rosenthal1972} that, 
$$\bbE\left| \frac 1 n \sum_{j=1}^n Z_j(\xi) \right|^\rho 
 \lesssim n^{-\frac \rho 2} \left(\bbE |Z_1(\xi)|^2\right)^\rho + n^{1-\rho}\bbE|Z_1(\xi)|^\rho.$$
 Notice that $|Z_1(\xi)| \leq 2$ for any $\xi \in \bbZ^d$, and $\rho/2 \leq \rho-1$, thus we deduce from the previous two displays that, 
\begin{align}
\label{eq:pf_kernel_Sn1_sums0} 
S_{n,1} 
 &\lesssim n^{-\frac \rho 2}\left(\sum_{\substack{\xi\in \bbZ^d,\xi \neq 0 \\ \|h_n\xi\|\leq 1}} \norm \xi^{-\rho'\eta} \right)^{\frac \rho {\rho'}} 
                               \sum_{\substack{\xi\in \bbZ^d,\xi \neq 0 \\ \|h_n\xi\|\leq 1}} \norm \xi^{\rho(\eta-1)}\big|\calF[K](h_n\xi)\big|^\rho \\
 &\lesssim n^{-\frac \rho 2}\left(\sum_{\substack{\xi\in \bbZ^d,\xi \neq 0 \\ \|h_n\xi\|\leq 1}} \norm \xi^{-\rho'\eta} \right)^{\frac \rho {\rho'}} 
                               \sum_{\substack{\xi\in \bbZ^d,\xi \neq 0 \\ \|h_n\xi\|\leq 1}} \norm \xi^{\rho(\eta-1)},
\label{eq:pf_kernel_Sn1_sums}                         
\end{align} 
where the final inequality follows from the fact that the Fourier transform of $K$ is bounded over the unit ball, since $K\in \calC_c^\infty(\bbR^d)$.
When $d \geq 3$, due to the condition $\frac 1 \rho + \frac 1 {\rho'} = \frac 1 2$, we may   choose
$\eta$ satisfying
\begin{align}
\label{eq:eta_choice_d_r}
d\left(\frac 1 d - \frac 1 \rho\right) < \eta < \frac d {\rho'}.
\end{align}
In particular, we then have $-d < \rho(\eta-1)$ and $-d < -\rho'\eta$, so that 
$$S_{n,1} \lesssim n^{-\frac \rho 2} \left(h_n^{\rho'\eta-d}\right)^{\frac \rho {\rho'}} h_n^{-\rho (\eta-1) - d}
 = n^{-\frac \rho 2} h_n^{\rho-d(\frac \rho {\rho'}+ 1)} 
 = n^{-\frac \rho 2} h_n^{\rho(1-\frac d 2)} .$$ 
If $d=2$, we choose $\eta$ such that the strict inequalities
in equation~\eqref{eq:eta_choice_d_r} both hold with equality. In this case, we have $\rho'\eta = d$ and $\rho (\eta-1)=-d$,  thus
$$S_{n,1} \lesssim n^{-\frac \rho 2} \left(\sum_{\xi\in \bbZ^d,\|h_n\xi\|\leq 1} \norm \xi^{-d} \right)^{\frac \rho {\rho'}+1} 
\lesssim n^{-\frac \rho 2}\log(h_n^{-1})^{\frac \rho {\rho'} + 1} 
                         = \left(\log(h_n^{-1})/n\right)^{\frac \rho 2}.$$
Finally, if $d = 1$, choose $\eta$ such that 
\begin{align} 
1 - \frac 1 \rho > \eta > \frac 1 {\rho'}.
\end{align}
In this case, both sequences in equation~\eqref{eq:pf_kernel_Sn1_sums} are summable, and
we obtain $S_{n,1} \lesssim n^{-\rho/2}$. In summary, we deduce
\begin{equation}
\label{eq:pf_kernel_Sn1}
S_{n,1} \lesssim \beta_n := n^{-\frac \rho 2} \begin{cases}
h_n^{\rho(1-\frac d 2)}, & d \geq 3 \\
\big(\log(h_n^{-1})\big)^{\rho/2}, & d = 2 \\
1, & d = 1.
\end{cases}
\end{equation}
We next bound $S_{n,2}$. Let $\eta < d/\rho'$. Apply a similar reduction as in equation~\eqref{eq:pf_kernel_Sn1_sums0}, 
to obtain
\begin{align*}
S_{n,2}  
 &\lesssim n^{-\frac \rho 2}\left(\sum_{\xi\in \bbZ^d,\|h_n\xi\|> 1} \norm \xi^{-\rho'\eta} \right)^{\frac \rho {\rho'}} 
                         \left(\sum_{\xi\in \bbZ^d,\|h_n\xi\|> 1} \norm \xi^{\rho(\eta-1)}\big|\calF[K](h_n\xi)\big|^\rho\right).
\end{align*}                         
Since $K \in \calC_c^\infty(\bbR^d)$, 
notice that $K$ and $\calF[K]$ are Schwartz
functions. In particular, $\calF[K](\xi) \lesssim \|\xi\|^{-\ell}$ for any $\ell > 0$.
Choose $\ell > 0$ such that $\rho(\eta - 1-\ell) < -d$.  
We then have, 
\begin{align*}
S_{n,2}   
 &\lesssim n^{-\frac \rho 2}\left(\sum_{\xi\in \bbZ^d,\|h_n\xi\|> 1} \norm \xi^{-\rho'\eta} \right)^{\frac \rho {\rho'}} 
                         \left(h_n^{-\rho\ell}\sum_{\xi\in \bbZ^d,\|h_n\xi\|> 1} \norm \xi^{\rho(\eta-1-\ell)}\right) \\                         
 & \lesssim n^{-\frac \rho 2} \left(h_n^{\rho'\eta -d} \right)^{\frac \rho {\rho'}} h_n^{-\rho\ell} h_n^{-\rho(\eta-1-\ell) - d} \lesssim 
 n^{-\frac  \rho 2} h_n^{\rho(1-\frac d 2)}\lesssim \beta_n.
\end{align*} 
Combine this bound with those of equations~\eqref{eq:pf_kernel_first_reduction} and~\eqref{eq:pf_kernel_Sn1} 
$$\bbE \calW_2^\rho(\hat P_n, P) \lesssim h_n^{\rho(s+1)} + \beta_n + 1/n^2 
\lesssim \begin{cases} 
n^{-\frac{\rho(s+1)}{2s+d}}, & d \geq 3 \\
(\log n/n)^{\rho/2}, & d = 2 \\
(1/n)^{\rho/2}, & d = 1.
\end{cases}$$
The claim follows.
\end{proof} 

\paragraph*{} We are now in a position to prove Theorem~\ref{thm:two_sample_kernel}.

\subsection{Proof of Theorem~\ref{thm:two_sample_kernel}}
In view of Lemmas~\ref{lem:kernel_infty}, \ref{lem:kernel_smoothness}, \ref{lem:ker_key_lemma}, and \ref{lem:w2_kernel_rate}, 
the proof of the claim is analogous to that of Theorem~\ref{thm:two_sample_density}, thus we only
provide brief justifications. 

Regarding part (i), apply Lemmas~\ref{lem:besov_holder} and \ref{lem:kernel_infty}--\ref{lem:kernel_smoothness}
to deduce that there exists $\epsilon \in (0, 1 \wedge \frac{\alpha-1}{2})$ and an event of probability
at least $1-1/n^2$ over which $\hat p_n, \hat q_m$ coincide with $\tilde p_n, \tilde q_m$ respectively, 
are bounded from below by $(2\gamma)^{-1}$, and are of class $\calC^{\epsilon}(\bbT^d)$,
with H\"older norm  uniformly bounded in $n$. By Theorem~\ref{thm:torus_regularity}, 
it follows that, over this same high-probability event, any mean-zero Brenier potential  in the optimal transport
problem from $P$ to $\hat Q_m$, or from $\hat P_n$ to $\hat Q_m$, is of class $\calC^{2+\epsilon}(\bbT^d)$, again with a uniformly bounded H\"older
norm. Arguing as in Step 1 of the proof of Theorem~\ref{thm:two_sample_density--map}, 
we deduce that these potentials achieve the conclusion of equation~\eqref{eq:fitted_curvature_2sample} therein. 
The same argument as in Steps~2--3 of that proof, coupled with Lemma~\ref{lem:w2_kernel_rate} stating the convergence rate of the kernel
density estimator in Wasserstein distance,  can then be used to deduce that
the optimal transport map $\hat T_{nm}$ from $\hat P_n$ to $\hat Q_m$ satisfies 
$$\bbE \big\|\hat T_{nm} - \hat T_0\big\|^2_{L^2(P)} \lesssim
\bbE W_2^2(\hat P_n, P) + \bbE W_2^2(\hat Q_m,Q) + \frac 1 {(n\wedge m)^2} 
\lesssim  \nmtratesker.$$
In applying Lemma~\ref{lem:w2_kernel_rate}, we note that
our stated assumption \hyperref[assm:kernel]{\textbf{K1($2\alpha,\kappa$)}}  
implies \hyperref[assm:kernel]{\textbf{K1($\alpha+1,\kappa'$)}}  
for a constant $\kappa' > 0$ depending only on $\alpha$ and $\kappa$. 
This proves part (i). To prove part (ii), we use the following observation.
\begin{lemma}
\label{lem:L_ker}
Under the assumptions of Theorem~\ref{thm:two_sample_kernel}, we have
\begin{alignat}{3}
\label{eq:pf_kernel_claim1}
\bbE\left[\int \phi_0 (\hat p_n - p)\right]   = O(h_n^{2\alpha}),
\quad
\Var\left[\int \phi_0 (\hat p_n - p)\right] = \frac{\Var_P[\phi_0(X)]}{n} + O\left(\frac{h_n^{2\alpha}}{n}\right),
\end{alignat}
where the implicit constants depend only on $M,\gamma,\alpha$.
\end{lemma}
Using Lemmas~\ref{lem:w2_kernel_rate}--\ref{lem:L_ker}, 
the same argument as in the proof of Theorem~\ref{thm:two_sample_density--wasserstein}
leads to the claim of part (ii). \qed 

\paragraph*{}It thus remains to prove Lemma~\ref{lem:L_ker}.

\subsection{Proof of Lemma~\ref{lem:L_ker}}
Using Lemma~\ref{lem:kernel_infty}, {the densities $\tilde p_n$ and $\hat p_n$ coincide with high probability, }
thus arguing similarly as in the proof of Lemma~\ref{lem:L}, 
it will suffice to prove that 
\begin{align}
\int \phi_0 (p-p_{h_n})    = O(h_n^{2\alpha}),
\quad
\Var\left[\int \phi_0 (\tilde p_n - p_{h_n})\right] = \frac{\Var_P[\phi_0(X)]}{n} + O\left(\frac{h_n^{2\alpha}}{n}\right).
\end{align} 
Under the condition $\alpha \not\in \bbN$, $\alpha > 1$, 
we deduce from Theorem~\ref{thm:torus_regularity} that there exists  $\lambda > 0$ 
depending only on $M,\gamma,\alpha$ such that
\begin{equation}
\label{eq:kernel_pf_potential_smoothness}
\phi_0,\psi_0 \in \calC^{\alpha+1}(\bbT^d;\lambda).
 \end{equation}
Now, by Parseval's Theorem, 
\begin{align*}
\left|\int_{\bbT^d} \phi_0 (p-p_{h_n})\right|
 &= \left|\sum_{\xi \in \bbZ^d} \calF[\phi_0](\xi) \calF[p - p_{h_n}](\xi) \right|  \\
 &\leq 
\big\| \|\cdot\|^{\alpha+1} \calF[\phi_0](\cdot)\big\|_{\ell^2(\bbZ^d)} \big\| \|\cdot\|^{-(\alpha+1)} \calF[p-p_{h_n}](\cdot)\big\|_{\ell^2(\bbZ^d)} \\
 &= \|\phi_0\|_{\dH{\alpha+1}(\bbT^d)}\|p - p_{h_n}\|_{\dH{-(\alpha+1)}(\bbT^d)}
 \lesssim \|p - p_{h_n}\|_{\dH{-(\alpha+1)}(\bbT^d)},
\end{align*}
where we used 
equation~\eqref{eq:kernel_pf_potential_smoothness} and the fact that $\|\phi_0\|_{\dot H^{\alpha+1}(\bbT^d)}
\leq \|\phi_0\|_{H^{\alpha+1}(\bbT^d)} 
\lesssim 
\|\phi_0\|_{\calC^{\alpha+1}(\bbT^d)}$ (cf. Lemma~\ref{lem:sobolev_facts}).
Apply 
Lemma~\ref{lem:ker_key_lemma}, under the assumption~\hyperref[assm:kernel]{\textbf{K1($2\alpha,\kappa$)}}
and the smoothness assumption on $p$, 
to deduce
\begin{align*}
\left|\int_{\bbT^d} \phi_0 d(p-p_{h_n})\right| \lesssim h_n^{2\alpha} \lesssim \nmtratesker.
\end{align*}
To prove the variance bound, notice that
$$\Var\left[\int \phi_0(\tilde p_n-p_{h_n})\right]
 = \Var\left[\int (\phi_0 \star K_{h_n}) d(P_n - P)\right]
 = \frac 1 n \Var_P[\phi_{h_n}(X)],$$
 where $\phi_{h_n} = \phi_0\star K_{h_n}$. Thus, reasoning as in the proof of Lemma~\ref{lem:L}, we have
\begin{align*}
\bigg|\Var&\left[\int \phi_0(\tilde p_n-p_{h_n})\right]-
  \frac 1 n \Var_P[\phi_{0}(X)]\bigg| \\
   &\leq \frac 1 n \Big|  \Var_P[\phi_{h_n}(X)] - \Var_P[\phi_{0}(X)]\Big| \\
   &\leq \frac 1 n \Big|\bbE[\phi_{h_n}^2(X) - \phi_0^2(X)]\Big| + \frac 1 n \Big|\bbE[\phi_{h_n}(X) - \phi_0(X)]\Big| = \frac 1 n\big[(I) + (II)\big].
\end{align*}
We shall again bound term $(I)$, and a similar proof can be used for term $(II)$. Notice that
\begin{align*}
(I)
 = \left| \int (\phi_{h_n} - \phi_0)(\phi_{h_n} + \phi_0)p\right|
 \leq \|\phi_{h_n} - \phi_0\|_{\dot H^{-(\alpha-1)}(\bbT^d)}\|(\phi_{h_n} + \phi_0)p\|_{\dot H^{\alpha-1}(\bbT^d)}.
\end{align*} 
It is a straightforward observation that $\|\phi_{h_n}\|_{\calC^{\alpha+1}(\bbT^d)} \leq \|\phi_0\|_{\calC^{\alpha+1}(\bbT^d)}$
for all $n \geq 1$, thus the function $(\phi_{h_n} + \phi_0)p$ has uniformly bounded $\calC^{\alpha-1}(\bbT^d)$
norm, by Lemma~\ref{lem:holder_products}. Since $\phi_0 \in \calC^{\alpha+1}(\bbT^d;\lambda)$, we deduce   that 
\begin{align*}
(I) \lesssim \|\phi_{h_n} - \phi_0\|_{\dot H^{-(\alpha-1)}(\bbT^d)} \lesssim h_n^{2\alpha},
\end{align*} 
by Lemma~\ref{lem:ker_key_lemma}. 
The claim follows from here.
\qed 

{ 
\subsection{Further Results}
In this section, we state for completeness several additional results on estimating optimal transport
maps and Wasserstein distances over $\bbT^d$, which mirror our results
over domains of $\bbR^d$ across Sections~\ref{sec:one_sample}--\ref{sec:two_sample}. Throughout what follows, let $P,Q \in \calPac(\bbT^d)$ 
admit densities $p,q$, and let $X_1,\dots,X_n \sim P$ and $Y_1,\dots,Y_m \sim Q$ be i.i.d. samples which
are independent of each other. Let $P_n = (1/n)\sum_{i=1}^n \delta_{X_i}$ and $Q_m = (1/m) \sum_{j=1}^m \delta_{Y_j}$. 
As in the previous subsections, we omit the superscript ``ker'' on the estimators $\hat P_n\sker$ and $\hat Q_m\sker$. 
Let $T_m$ be the optimal transport
map from $P$ to $Q_m$, and let 
$$\Delta_{nm} = \sum_{i=1}^n \sum_{j=1}^m \hat\pi_{ij} \|T_0(X_i) - Y_j\|^2,$$
where 
$$\hat\pi \in \argmin_{\pi \in \calQ_{nm}} \sum_{i=1}^n \sum_{j=1}^m \pi_{ij} d_{\bbT^d}^2(X_i,Y_j).$$
Furthermore, let $\widebar T_m$ be the optimal transport map from $P$ to $\hat Q_m$. 
 We begin by stating a one-sample analogue 
of Theorem~\ref{thm:two_sample_kernel}.
 \begin{proposition}[One-Sample Kernel Estimators]
\label{thm:one_sample_kernel}
Let $P,Q\in \calPac(\bbT^d)$ and assume $p,q \in \calC^{\alpha-1}(\bbT^d; M,\gamma)$,
for some $\alpha > 1$, $\alpha\not\in\bbN$, and $M,\gamma > 0$.
Let  $h_m \asymp m^{-1 /({d+2(\alpha-1)})}$.  Then, there exists a constant $C > 0$ depending only on
$M,\gamma,\alpha$ such that the following assertions hold.
\begin{thmlist}
\item \label{thm:one_sample_kernel--map} (Optimal Transport Maps) We have,
\begin{equation*}
\bbE \big\|\widebar T_m - T_0\big\|_{L^2(P)}^2
\leq C \mtratesker.
\end{equation*}
\item \label{thm:one_sample_kernel--wasserstein} (Wasserstein Distances) We have,
\end{thmlist}
\vspace{-0.1in}
\begin{align*}
\big| \bbE W_2^2(P, \hat Q_m) - W_2^2(P,Q) \big| &\leq C \mtratesker,\\
\bbE \big| W_2^2(P, \hat Q_m) - W_2^2(P,Q) \big|^2
&\leq \left[C \mtratesker  +  \sqrt{\frac {\Var_Q[\psi_0(Y)]}{m}}\right]^2.
\end{align*}
\end{proposition} 
Next, we state convergence rates for empirical estimators.  In what follows, we use the abbreviation
$$\tilde \kappa_n = \begin{cases}
1/n, & d = 1, \\ 
\log n/n , & d=2, \\ 
n^{-2/d}, & d \geq 3.
\end{cases}$$
\begin{proposition}[One-Sample Empirical Estimators over $\bbT^d$]
\label{prop:one_sample_torus_empirical}
Let $P,Q \in \calPac(\bbT^d)$ admit densities $p,q$ satisfying
$$\gamma^{-1} \leq p,q \leq \gamma,\quad\text{over } \bbT^d,$$
for some $\gamma > 0$. Assume further that $\phi_0 \in \calC^2(\bbT^d)$. Then, 
$$\bbE \|\hat T_m - T_0\|_{L^2(P)}^2 \asymp \bbE \big[ W_2^2(P,Q_m) - W_2^2(P,Q)\big] \asymp \bbE W_2^2(Q_m,Q) 
\lesssim \tilde \kappa_m,$$
and, 
$$\bbE [\Delta_{nm}] \asymp \bbE \big[ W_2^2(P_n,Q_m) - W_2^2(P,Q)\big] \lesssim   \tilde \kappa_{n\wedge m}.$$
\end{proposition}
From Proposition~\ref{prop:one_sample_torus_empirical}, one may also deduce rates of convergence for the  nearest-neighbor
estimator discussion in Section~\ref{sec:two_sample_empirical},  over $\bbT^d$. We omit the details for the sake of brevity.
 }

{ 

\section{Plugin Estimation over Smooth Domains}  
\label{app:smooth_domains}
{The goal of this appendix is to prove 
Theorem~\ref{thm:smooth_domains}. Let us summarize our proof strategy.
\begin{enumerate}
\item[(i)] In Section~\ref{app:spectrally_defined_sobo}, we define a scale
of spectrally-defined Sobolev 
spaces $\calH^{s,r}(\Omega)$, which are well-suited to the analysis
of the density estimator $\hat q_n^{(\mathrm{lap})}$.  
We show that these spaces coincide with a scale of subspaces $H_N^{s,r}(\Omega)$ of the usual
Sobolev spaces $H^{s,r}(\Omega)$. 
\item[(ii)] In Section~\ref{app:density_estimation_under_spectral_norm}, 
we   bound the risk of the density estimator $\hat q_n^{(\mathrm{lap})}$
under the norm of the space $\calH^{s,r}(\Omega)$, for a wide range of parameters $s,r$.
\item[(iii)] In Section~\ref{app:regularity_of_spectral_estimator},
we use parts (i)--(ii)
to derive a convergence rate of $\hat q_n^{(\mathrm{lap})}$ under a suitable H\"older norm.
This
will allow us to deduce that with probability tending to one, 
$\hat q_n^{(\mathrm{lap})}$ 
satisfies the regularity conditions
needed for Caffarelli's regularity theory~(condition~\ref{assm:caffarelli}).
\item[(iv)] In Section~\ref{app:W2_conv_of_spectral_estimator}, we combine
parts (i)--(iii) to obtain a convergence rate of 
$\hat q_n^{(\mathrm{lap})}$ under the Wasserstein distance, using
its equivalence to a negative Sobolev norm (cf. equation~\eqref{eq:peyre}). 
\item[(v)] In Section~\ref{app:pf_thm_smooth_domains}, we combine these ingredients to deduce the claim,
by the same strategy as in the proof of Theorems~\ref{thm:two_sample_density}
and~\ref{thm:two_sample_kernel}.
\end{enumerate}

Throughout our development, an important role will be played
by the Neumann boundary condition,  together with 
assumption~\ref{assm:smooth_domain}. On the one hand, 
we will see that these conditions
are sufficient for the space $\calH^{s,r}(\Omega)$ to admit
a Littlewood-Paley characterization; cf.~Lemmas~\ref{lem:multiplier} and~\ref{lem:littlewood_paley}.
On the other hand, these conditions ensure that the eigenvalues 
and spectral function
of the Neumann Laplacian grow at a sufficiently slow rate to obtain
our stated convergence rates for density estimation;
cf. Lemmas~\ref{lem:spectral_function}--\ref{lem:weyl}.  
}
 
\subsection{Spectrally-Defined Sobolev Spaces}
\label{app:spectrally_defined_sobo}
To facilitate our analysis of the estimators $\tilde p_n, \tilde q_m$, 
we will begin by showing that the Bessel potential Sobolev spaces $H^{s,r}(\Omega)$,
defined in Appendix~\ref{app:sobolev}, can be characterized via the spectrum of the Neumann Laplacian.
More specifically, we will work with the following subspaces of $H^{s,r}(\Omega)$, 
which have suitably vanishing Neumann trace. 
{Throughout what follows, we always denote by $\nu$ 
an outward-pointing unit normal vector to $\partial\Omega$, and by 
$\partial u/\partial\nu$ the 
weak normal derivative operator, whose trace on $\partial\Omega$  is  well-defined  and takes values in
$L^r(\partial\Omega)$ 
whenever $u \in H^{t,r}(\Omega)$ with $t > 1 + 1/r$~\citep[Theorem 4.6]{taira2016}.}
We will assume
 $s \geq 0$ and $r \geq 2$ throughout this section. Furthermore, we adopt the nonstandard notation
$H_0^{s,r}(\Omega) = H^{s,r}(\Omega) \cap L_0^r(\Omega)$.
\begin{definition}[\cite{triebel1995}, Section 4.3.3]
Let $\Omega$ be a domain satisfying condition~\ref{assm:smooth_domain}.
For all $s \geq 0$ and $r \geq 2$, let $H_N^{s,r}(\Omega)$ be defined as follows.
\begin{enumerate}
\item[(i)] If $s - 1/r < 1$,  set $H_N^{s,r}(\Omega) = H_0^{s,r}(\Omega)$.  
\vspace{0.05in}
\item[(ii)] If $2k +1< s-1/r < 2(k+1)+1$ for some $k \geq 0$,  set
$$H_N^{s,r}(\Omega) = \left\{ u \in H_0^{s,r}(\Omega): \frac{\partial  \Delta^j u}{\partial \nu} = 0 \text{ on } \partial\Omega,
~ 0 \leq j \leq k\right\}.$$
\item[(iii)] If $2k +1= s-1/r$ for some $k \geq 0$, extend $\nu$ continuously to $\Omega$, and set
$$H_N^{s,r}(\Omega) = \left\{ u \in H_0^{s,r}(\Omega): \frac{\partial  \Delta^j u}{\partial \nu} = 0 \text{ on } \partial\Omega,
~ 0 \leq j < k,\, \frac{\partial \Delta^k u}{\partial \nu} \in H^{\frac 1 r,r}(\bbR^d)  \right\},$$
where $\partial \Delta^k u/\partial \nu$ is extended by zero outside of $\Omega$.
\end{enumerate}
\end{definition}
Note that, in part (iii), the normal vector $\nu$  may be extended smoothly away from the boundary
since we assumed $\partial\Omega$ is $\calC^\infty$, thus $\nu$ is itself smooth. The particular choice of extension does not
alter the definition of the space. 
As we shall see, the relevance of the space $H_N^{s,r}(\Omega)$ lies in the fact that the 
fractional Neumann Laplacian $(-\Delta)^{s/2}$, defined next,
is an isomorphism of $H_N^{s,r}(\Omega)$ onto $L_0^r(\Omega)$. 



 Recall that $0<\lambda_1 \leq \lambda_2 \leq \dots$ is the sequence of eigenvalues corresponding
to the eigenbasis $\{\eta_\ell\}_{\ell=1}^\infty$. 
Define the spectral fractional Laplacian  for all $u \in L_0^2(\Omega)$ by 
$$(-\Delta)^{s/2} u = \sum_{\ell=1}^\infty \lambda_\ell^{s/2} \calL_\ell[u] \eta_\ell,\quad \text{where } ~\calL_\ell[u] := \langle u,\eta_\ell\rangle_{L^2(\Omega)},~~ \ell=1,2,\dots$$
Furthermore, for $s \geq 0$ and $r \geq 2$, let $\calH^{s,r}(\Omega)$ denote the Banach space of functions $u \in L_0^r(\Omega)$ such that 
the norm 
$$\|u\|_{\calH^{s,r}(\Omega)} := \big\| (-\Delta)^{s/2} u \big\|_{L^r(\Omega)} $$
is finite. 
In the special case $r=2$, 
$\calH^{s,r}(\Omega)$ becomes a Hilbert space, as noted by~\cite{dunlop2020}. In this case we omit the superscript ``$r$'' and simply write $\calH^s(\Omega):= \calH^{s,2}(\Omega)$.
The corresponding inner product on this space is given by
$$\langle u,v\rangle_{\calH^{s}(\Omega)} = \sum_{\ell=1}^\infty \lambda_\ell^{s} \calL_\ell[u] \calL_\ell[v],\quad u,v\in \calH^s(\Omega).$$
It is easy to see by Parseval's identity that $\|\cdot \|_{\calH^{s}(\Omega)} = \sqrt{\langle\cdot,\cdot\rangle_{\calH^{s}(\Omega)}}$.

Our first aim is to show that the spaces $\calH^{s,r}(\Omega)$ and $H_N^{s,r}(\Omega)$ coincide, with equivalent norms.  
To simplify our proof, we will focus only on the ranges of $s$ and $r$ which we will need in our development. 
\begin{proposition}
\label{prop:equiv_sobolev}
Let $\Omega$ satisfy condition~\ref{assm:smooth_domain}. Assume that one of the following conditions holds:
\begin{align}
\begin{cases}
s \in [0,2] \\
r \geq 2 
\end{cases} 
\quad \text{or} \quad 
\begin{cases}
s \geq 0 \\
r =2. 
\end{cases} 
\end{align}
Then, with equivalent norms,
\begin{equation}
\label{eq:equality_sobolev_spaces}
H_N^s(\Omega) = \calH^{s}(\Omega).
\end{equation}
\end{proposition}
Proposition~\ref{prop:equiv_sobolev} was   stated without proof by~\citet[Section 5]{seeley1972}. 
For completeness, we provide a self-contained proof 
below. 
Let us also note that, in the Hilbertian case $r=2$, Proposition~\ref{prop:equiv_sobolev}
was  established for integer exponents $s$ by~\citet[Lemma 7.1]{dunlop2020}, and for $s \in [0,2]$ by~\cite{kim2020}.  
For the case $r >2$, a  
result similar to~Proposition~\ref{prop:equiv_sobolev} was 
proven by~\citet[Theorem 3.1]{cao2020}, however they considered the setting where $\Omega$ is the entire Euclidean space $\bbR^d$, 
thus they employed a different definition of the fractional Laplacian.

Before turning to the proof, let us begin by stating a generalization of Mikhlin's multiplier theorem for the  Neumann Laplacian,
which we will use repeatedly in the following subsections. 
This result follows from Theorem 1.3 of~\cite{xu2011}, or Theorem 7.9 of~\cite{kerky2015}.
\begin{lemma}[Mikhlin's Multiplier Theorem for the Neumann Laplacian]
\label{lem:multiplier}  
Let $m \in \calC^\infty(\bbR_+)$ satisfy Mikhlin's multiplier condition:
\begin{align}\label{eq:cond_mikhlin}
|D^\alpha m(x)| \leq c |x|^{-|\alpha|}, \quad\text{for all } \alpha=1, \dots, d +1,\,x \in \bbR_+.
\end{align} 
Then, for any $1 < r < \infty$, there exists $C > 0$ depending on $\Omega,c,r$ such that for any $f \in L_N^r(\Omega)$, 
\begin{equation}
\label{eq:multiplier}
\left\|\sum_{\ell=1}^\infty m(\sqrt{\lambda_\ell}) \calL_\ell[f] \eta_\ell\right\|_{L^r(\Omega)} \leq C 
\|f\|_{L^r(\Omega)}.
\end{equation} 
In this case, we say that $m$ is an $L^r_N(\Omega)$ multiplier.
\end{lemma}

We now turn to the proof of Proposition~\ref{prop:equiv_sobolev}.  
\subsubsection{Proof of Proposition~\ref{prop:equiv_sobolev}}
\label{app:pf_equiv_sobolev}
Let us begin with the case where $s \in [0,2]$ and $r \geq 2$. 
The result is trivial when $s=0$, in which case we have
$$\calH^{0,r}(\Omega) = H_N^{0,r}(\Omega) = L^r_0(\Omega).$$
Next, we prove the claim when $s=2$, in which case the space $H_N^{s,r}(\Omega)$ takes the form
$$H_N^{2,r}(\Omega) = \left\{ u \in H^{2,r}(\Omega): \frac{\partial u}{\partial \nu} = 0~\text{ on } \partial\Omega\right\}.$$
By~\citet[Theorem 4.2.4, p.~316]{triebel1995}, we have that $H^{2,s}(\Omega) = W^{2,s}(\Omega)$
with equivalent norms, where $W^{k,r}(\Omega)$ is the standard $L^r(\Omega)$ Sobolev norm with integer smoothness parameter $k \in \bbN$.
Thus, for all $u \in H_N^{2,r}(\Omega)$, we have
$$\|u\|_{\calH^{2,r}(\Omega)} = \|\Delta u\|_{L^r(\Omega)} \lesssim  \|u\|_{W^{2,r}(\Omega)}
\lesssim \|u\|_{H^{2,r}(\Omega)}.$$
It is a standard fact that $-\Delta$ is a bijection
of $H_N^{2,r}(\Omega)$ onto $L_0^r(\Omega)$~(e.g.~\cite{franke1995}). Furthermore, the above
display shows that this mapping is continuous, thus, by the Banach isomorphism theorem, 
 the operator $(-\Delta)^{-1}:L^r_0(\Omega)\to H_N^{2,r}(\Omega)$ is bounded. Equivalently, for all $w \in \calH^{2,r}(\Omega)$, 
we obtain $\|w\|_{H^{2,r}(\Omega)} \lesssim \|w\|_{\calH^{2,r}(\Omega)}$. We deduce that
$\calH^{2,r}(\Omega) = H_N^{2,r}(\Omega)$, with equivalent norms.

It thus remains to prove the claim for all $s \in (0,2)$, which we shall do using an interpolation argument. 
Given two complex Banach spaces
$A$ and $B$, let $(A,B)_{[\theta]}$ denote the complex interpolation space
between $A$ and $B$, for any $\theta \in [0,1]$~\citep{bergh1976}. 
It is well-known that the complex interpolation of any two Bessel potential spaces $H^{s_0,r}(\Omega)$
and $H^{s_1,r}(\Omega)$ is itself a Bessel potential space
(see for instance~\cite{triebel1995}, 
Theorem 4.3.1/1). 
The following is an analogue of this result for spaces with zero Neumann trace.
\begin{lemma}[\cite{seeley1972}]\label{lem:interpolation_bessel}
Let $1 < r < \infty$. Then, for all $s \geq 0$ and $\theta \in (0,1)$, 
$$H_N^{\theta s,r}(\Omega) = \big(L_0^{r}(\Omega), H_N^{s,r}(\Omega)\big)_{[\theta]}.$$
\end{lemma}
By combining Lemma~\ref{lem:interpolation_bessel} with
what we have shown above, it holds  for any $s \in [0,2]$,
$$H^{s,r}_N(\Omega) = \big(L_0^{r}(\Omega), H_N^{2,r}(\Omega)\big)_{[s/2]} = 
\big( L_0^{r}(\Omega), \calH^{2,r}(\Omega)\big)_{[s/2]},$$
To complete the proof of the claim, it thus suffices to prove that 
\begin{equation}\label{eq:interpolation_of_our_space}
\calH^{s,r}(\Omega) = \big(L_0^{r}(\Omega), \calH^{2,r}(\Omega)\big)_{[s/2]}
\end{equation}
for any $s \in [0,2]$. We will do so by following
 similar lines as the proof of Theorem 6.4.5 of~\cite{bergh1976}. Specifically, the following can be inferred
 from their Theorem~6.4.2.
\begin{lemma}
\label{lem:retract}
Suppose there exists a collection of complex Banach spaces $(B_s)_{s\in[0,2]}$ with
$B_s\subseteq B_{s'} \subseteq L_N^r(\Omega)$ for all $0 \leq s' \leq s \leq 2$,
fulfilling the following properties for all $s \in [0,2]$:
\begin{enumerate}
\item[(i)] $B_s = (B_0,B_2)_{[s/2]}$. 
\item[(ii)] There exists a continuous linear map $\mathscr{I}: \calH^{s,r}(\Omega) \to B_s$. 
\item[(iii)] There exists a continuous linear map $\mathscr{P} : B_s \to \calH^{s,r}(\Omega)$ such that
{ $\mathscr{P}\circ \mathscr{I} = 
\mathrm{Id}_{\calH^{s,r}(\Omega)}$}.
\end{enumerate}
Then, equation~\eqref{eq:interpolation_of_our_space} holds for all $s \in [0,2]$.
\end{lemma} 
The claim will therefore follow if we can exhibit a collection of Banach spaces $(B_s)_{s \in [0,2]}$
satisfying the properties of Lemma~\ref{lem:retract}. 
To do so, it will be convenient to show that  $\calH^{s,r}(\Omega)$ lies in the Triebel-Lizorkin family of spaces. 
Indeed, it is well-known that the standard Sobolev space $H^{s,r}(\Omega)$
is equal to the Triebel-Lizorkin space $F_{r,2}^s(\Omega)$~\citep{triebel1995}, and an analogue of this
fact for the space $\calH^{s,r}(\Omega)$ has been derived by~\cite{kerky2015}. 
We state a variant of their result   below, beginning with  some notation. 
Let $\xi_0,\xi \in \calC^\infty(\bbR_+)$
be an admissible pair of Littlewood-Paley functions, so that $\supp(\xi_0) \subseteq [0,2]$, 
$\supp(\xi) \subseteq [1/2,2]$, and 
$\sum_{j\geq 0} \xi_j(\lambda) = 1$ for all $\lambda \in \bbR$, where we write $\xi_j = \xi(2^{-j}(\cdot))$
  (cf. Lemma 6.1.7 of~\cite{bergh1976} for a construction of such 
functions).
We then have the following statement. 
\begin{lemma}[\cite{kerky2015}]
\label{lem:littlewood_paley}
Let $1 < r < \infty$. Then, for all $u \in L_0^r(\Omega)$,
$$\|u\|_{\calH^{s,r}(\Omega)} \asymp \|u\|_{\mathscr{F}_{r,2}^s(\Omega)}:= 
\left\| \left(\sum_{j=0}^\infty   \bigg|2^{js}\sum_{\ell=1}^\infty  \xi_j(\lambda_\ell^{1/2}) \calL_\ell[u] \eta_\ell(\cdot) \bigg|^2 \right)^{\frac 1 2}\right\|_{L^r(\Omega)}$$
\end{lemma}
\begin{proof}[Proof of Lemma~\ref{lem:littlewood_paley}] 
It follows from Theorem 7.8 of~\cite{kerky2015} that
$$ \|u\|_{\mathscr{F}_{r,2}^s(\Omega)}\asymp\left\|(\mathrm{Id}-\Delta)^{s/2}u\right\|_{L^r(\Omega)} 
= \left\|\sum_{\ell=1}^\infty (1+\lambda_\ell)^{s/2} \calL_\ell[u] \eta_\ell \right\|_{L^r(\Omega)}$$
where we used the fact the Neumann Laplacian over a convex domain $\Omega$ with smooth boundary
satisfies the conditions of the operator~$L$ in the introduction
of~\cite{kerky2015}. Indeed, the Gaussian
upper bounds on the heat kernel generated by the Neumann Laplacian are given for instance in Theorem~3.3.5 of~\cite{davies1989}, 
while the H\"older continuity of the heat kernel can be deduced, as in Proposition~3.1 of~\cite{sturm1996},
from the parabolic Harnack inequality (see for instance Theorem 5.3.5 of~\cite{davies1989}). 
See also 
\citet[Theorem 3.1]{saloff2010}, and remarks thereafter. 
To prove our claim, it thus suffices to show that
\begin{equation}
\label{eq:homog_inhomog}
\|(-\Delta)^{s/2} u\|_{L^r(\Omega)} \asymp \|(\mathrm{Id}-\Delta)^{s/2} u \|_{L^r(\Omega)}.
\end{equation}
Notice first that the maps $m(\lambda) = (1+\lambda^{s/2}) / (1+\lambda)^{s/2}$
and $1/m(\lambda)$, $\lambda \in \bbR_+$ satisfy the conditions of Lemma~\ref{lem:multiplier}, thus
$$\|(\mathrm{Id}-\Delta)^{s/2} u\|_{L^r(\Omega)} \asymp \left\| \sum_{\ell=1}^\infty (1+\lambda_\ell^{s/2}) \calL_\ell[u] \eta_\ell\right\|_{L^r(\Omega)}
\asymp \|u\|_{L^r(\Omega)} + \| (-\Delta)^{s/2} u\|_{L^r(\Omega)}.$$
 It thus suffices to show that $\|u\|_{L^r(\Omega)} \lesssim \| (-\Delta)^{s/2} u\|_{L^r(\Omega)}$.
This follows from the fact that 
 any map $m\in \calC^\infty(\bbR_+)$, of the form $m(\lambda) = \lambda^{-s/2}$
for $\lambda > \lambda_1/2$, is an $L_N^r(\Omega)$ multiplier. The claim follows.
\end{proof}
Let $s \in [0,2]$. Denote by $\ell_2^s$ the set
of real-valued sequences $a =(a_j)_{j\geq 0}$ such that
$$\|a\|_{\ell^s_2} := \left(\sum_{j=0}^\infty (2^{js} |a_j|)^2\right)^{\frac 1 2 } < \infty$$
Furthermore, let $L_0^r(\ell_2^s)$ denote the space of all sequences 
$F=(f_j)_{j\geq 0} \subseteq L_0^r(\Omega)$
such that 
$$\|F\|_{L^r(\ell_2^s)}^r := \int_\Omega \|F(x)\|_{\ell_2^s}^r dx < \infty.$$
By Theorems 5.1.2 and 5.6.3 of~\cite{bergh1976}, the Banach spaces $B_s := L_0^r(\ell_2^s)$, for $0\leq s \leq 2$, 
satisfy condition~(i) of Lemma~\ref{lem:retract}. Furthermore, the map 
$$\mathscr{I}:\calH^{s,r}(\Omega) \to B_s, \quad \mathscr{I}:u \mapsto \left(\sum_{\ell=1}^\infty \xi_j(\sqrt{\lambda_\ell}) \calL_\ell[u] \eta_\ell(\cdot)\right)_{j\geq 0}$$
satisfies, by Lemma~\ref{lem:littlewood_paley}, $\|u\|_{\calH^{s,r}(\Omega)} \asymp \|u\|_{\mathscr{F}_{r,2}^s(\Omega)} = \|\mathscr{I} u\|_{B_s}$, and thus satisfies condition~(ii) of Lemma~\ref{lem:retract}. 
Finally, define the map
$$\mathscr{P}:B_s \to \calH^{s,r}(\Omega),\quad \mathscr{P}:(f_j)_{j\geq 0}\mapsto 
\sum_{\ell=1}^\infty \sum_{j=0}^\infty \tilde\xi_j(\sqrt{\lambda_\ell}) \calL_\ell[f_j]\eta_\ell,$$
where
\begin{align*}
\tilde \xi_j &= \xi_{j-1} + \xi_j + \xi_{j+1}, \quad j=1,2,\dots
\end{align*}
with the convention that $\xi_{-m} = 0$ for any {$m > 0$}. Notice that for all $u \in \calH^{s,r}(\Omega)$, we have
\begin{align*}
\mathscr{P} \mathscr{I} u 
 &= \sum_{\ell=1}^\infty \sum_{j=0}^\infty  \tilde \xi_j(\sqrt{\lambda_\ell})\xi_j(\sqrt{\lambda_\ell})\calL_\ell[u] \eta_\ell.
\end{align*}
Since $\xi_j$ has disjoint support from $\xi_k$ for any $j,k \in \bbN$, $|j-k| \geq 2$, we have
$$\tilde \xi_j(\sqrt{\lambda_\ell}) \xi_j(\sqrt{\lambda_\ell}) =
 \sum_{k=0}^\infty \xi_k(\sqrt{\lambda_\ell}) \xi_j(\sqrt{\lambda_\ell}) = \xi_j(\sqrt{\lambda_\ell}),$$
 where we used the fact that $\{\xi_j\}_{j\geq 0}$ forms a partition of unity. By reapplying this property, we obtain 
\begin{align*}
\mathscr{P} \mathscr{I} u 
 &= \sum_{\ell=1}^\infty \sum_{j=0}^\infty  \xi_j(\sqrt{\lambda_\ell})\calL_\ell[u] \eta_\ell = \sum_{\ell=1}^\infty \calL_\ell[u] \eta_\ell = u.
\end{align*}
It thus remains to show that $\mathscr{P}$ is a bounded linear operator. 
To do so, we will make use of the Hardy-Littlewood maximal function
$$M f(x) = \sup_{B\in \calB_x} \frac 1 {\calL(B)}\int_B f(y)dy, \quad x \in \Omega,$$
for any $f \in L^1(\Omega)$, where $\calB_x$ is the set of balls of the form $\{y\in \Omega: \|x-y\|< \delta\}$, $\delta > 0$. 
For our purposes, the utility of the maximal function lies in the fact that it   induces a bounded operator from $B_s$ into 
itself (cf. Theorem~5.6.6 of~\cite{grafakos2009}):
there exists a constant $C > 0$ depending on $\Omega,r$ such that
for any $s \in [0,2]$ and $(f_j)_{j\geq 0} \in B_s$, it holds that
\begin{equation}\label{eq:maximal_fn}
\left\| \left( \sum_{j=0}^\infty (2^{js} |M(f_j)|)^2\right)^{\frac 1 2}\right\|_{L^r(\Omega)}
\leq C \left\| \left(\sum_{j=0}^\infty (2^{js}|f_j|)^2\right)^{\frac 1 2}\right\|_{L^r(\Omega)}.
\end{equation}
Let us now turn to bounding the operator norm of $\mathscr{P}$. Using Lemma~\ref{lem:littlewood_paley}, we have
\begin{align}
\label{eq:bounding_mathscrP}
\nonumber
\|\mathscr{P} (f_j)_{j}\|_{\calH^{s,r}(\Omega)}
 &\asymp \left\| \left( \sum_{j=0}^\infty \left( 2^{js} \sum_{\ell=1}^\infty \sum_{k=0}^\infty \xi_j(\sqrt{\lambda_\ell})  \tilde \xi_k(\sqrt{\lambda_\ell}) \calL[f_k] \eta_\ell\right)^2\right)^{\frac 1 2}\right\|_{L^r(\Omega)} \\ 
 &= \left\| \left( \sum_{j=0}^\infty \left( 2^{js} \sum_{\ell=1}^\infty \sum_{k=j-2}^{j+2} \xi_j(\sqrt{\lambda_\ell})  \tilde \xi_k(\sqrt{\lambda_\ell}) \calL[f_k] \eta_\ell\right)^2\right)^{\frac 1 2}\right\|_{L^r(\Omega)},
\end{align} 
where we used the fact that $\xi_j$ has disjoint support from $\tilde \xi_k$ whenever $|j-k| \geq 3$. 
For any $j,k \geq 0$, let $\Lambda_{jk}$ be the operator defined by
$$\Lambda_{jk} f = \sum_{\ell=1}^\infty \xi_j(\sqrt{\lambda_\ell})\tilde \xi_k(\sqrt{\lambda_\ell}) \calL[f]\eta_\ell,$$
for any $f \in L_N^r(\Omega)$. 
In order to bound the right-hand side of equation~\eqref{eq:bounding_mathscrP}, we will relate
the operator $\Lambda_{jk}$ to the maximal function $M$, in the following Lemma. 
This result is largely inspired by~\citet[Eq.\,(5.9)]{georgiadis2023}. 
\begin{lemma}\label{lem:Lambda_M}
There exists a constant $C > 0$ such that for all $s \in [0,2]$, $j \geq 0$, 
$f \in L^r_N(\Omega)$, and $x \in \Omega$,  
$$|\Lambda_{jk} f(x)| \leq C M f(x).$$
\end{lemma}
Before proving the Lemma, let us show how it implies the claim. Write $f_{-1}=f_{-2}=0$. Continuing from equation~\eqref{eq:bounding_mathscrP}, we obtain
from Lemma~\ref{lem:Lambda_M} that
\begin{align*}
\|\mathscr{P}(f_j)_j\|_{\calH^{s,r}(\Omega)} 
 &\lesssim   \left\| \left( \sum_{j=0}^\infty \left( 2^{js}\sum_{k=j-2}^{j+2} M f_k \right)^2\right)^{\frac 1 2}\right\|_{L^r(\Omega)} \\
 &\lesssim   \left\| \left( \sum_{j=0}^\infty \left( 2^{js}  M f_j \right)^2\right)^{\frac 1 2}\right\|_{L^r(\Omega)} \\
 &\lesssim \|(f_j)_j\|_{B_s},
\end{align*}
where the final inequality follows from equation~\eqref{eq:maximal_fn}.
This proves the boundedness of the operator $\mathscr{P}$, and it remains to prove Lemma~\ref{lem:Lambda_M}, which we shall do next.

By Theorem 3.1 of~\cite{kerky2015} and the definition of the functions~$\xi_j$, the operator $\Lambda_{jk}$
is an integral operator, 
$$\Lambda_{jk} f(x) = \int_\Omega \Gamma_{jk}(x,y) f(y)dy,\quad f \in L^r(\Omega), x \in \Omega,$$
where the kernel $\Gamma_{jk}$ is real-valued and enjoys the bound
$$|\Gamma_{jk}(x,y)| \lesssim_a \calL\big(B(x,2^{-j})\big)^{-1}(1+2^j \|x-y\|)^{-a},$$
for any $a$ sufficiently large,  where $B(x,\delta) = \{y \in \Omega: \|x-y\| < \delta\}$.
Note that the implicit constant above is independent of $j,k$. 
Under condition~\ref{assm:smooth_domain}, 
we have $\calL\big(B(x,2^{-j})\big)^{-1} \lesssim 2^{jd}$.
Letting $D = \diam(\Omega)$, we thus have uniformly in $x \in \Omega$,
\begin{align*}
|\Lambda_{jk}f(x)|
 &= \left| \int_\Omega \Gamma_{jk} (x,y) f(y)dy \right| \\ 
 &\leq \sum_{k=0}^{j-1} \int_{\{y\in \Omega : 2^{k-j} \leq \frac{\|x-y\|}{D} \leq 2^{k-j+1}\}}2^{jd}(1+2^j \|x-y\|)^{-a}| f(y)|dy \\ 
 &\leq \sum_{k=0}^{j-1} \int_{\{y\in \Omega :  \frac{\|x-y\|}{D} \leq 2^{k-j+1}\}}2^{jd-ka}| f(y)|dy \\ 
 &\lesssim \sum_{k=0}^{j-1}   2^{jd-ka +  d(k-j+1)} Mf(x) \\ 
 &\lesssim Mf(x) \sum_{k=0}^{\infty}   2^{-k(a-d)}.
\end{align*}
Choosing $a > d$, we obtain the conclusion of Lemma~\ref{lem:Lambda_M}, and hence of Proposition~\ref{prop:equiv_sobolev}
in the regime $s \in [0,2]$, $r \geq 2$. 

It remains to prove Proposition~\ref{prop:equiv_sobolev} in the regime $s \geq 0$ when $r=2$. 
By~\citet[Theorem 7.1]{dunlop2020},
we already know that the claim holds for integer values of $s$. We 
will again use an interpolation argument to deduce the claim for non-integer values of $s$. 
Indeed, as noted by~\cite{dunlop2020}, for $s_1 \in \bbN$, and any $0 \leq s \leq s_1$, 
the space $\calH^s(\Omega)$ can be written as the following real interpolation space~\citep{bergh1976}:
$$\calH^{s}(\Omega) = \big( L_0^2(\Omega), \calH^{s_1}(\Omega)\big)_{s/s_1,2}.$$
Theorem 7.1 of~\cite{dunlop2020} thus implies
$$\calH^{s}(\Omega) = \big( L_0^2(\Omega), H_N^{s_1}(\Omega)\big)_{s/s_1,2}.$$
The right-hand side of the above display is  equal to $H_N^s(\Omega)$ by~\cite{lofstrom1992}, and the claim follows
by taking $s_1$ arbitrarily large.\qed

\subsection{Density Estimation under the $\calH^{s,r}(\Omega)$ Norms}
\label{app:density_estimation_under_spectral_norm}
Our aim in this subsection is to prove the following result, which is our main technical
tool for deriving Theorem~\ref{thm:smooth_domains}. Let $p_{L_n}(x):=\bbE[\tilde p_n(x)]$ for all $x \in \Omega$. 
\begin{proposition}
\label{prop:density_risk_smooth_domain}
Let $r \geq 2$, $M,s,c> 0$, and assume that $p \in \calC_N^s(\Omega;M)$. Assume that either $s \leq 2$ or $r=2$. 
Assume further that $L_n = c n^a$ for some $0<a<1$. Then, for any given $-\infty <  t < s$, 
there exists a constant $C > 0$ depending on $\Omega, M,d,t,s,r,c,a$ such that
\begin{align*}
\| p_{L_n} - p\|_{\calH^{t,r}(\Omega)}^r &\leq C  L_n^{-\frac{r(s-t)}{d}},
\end{align*}
and, if we further assume
 $t \geq 0$,  
\begin{align*}
\bbE \|\tilde p_n - p_{L_n} \|_{\calH^{t,r}(\Omega)}^r &\leq C   n^{-\frac r 2}  L_n^{r\left(\frac{t}{d} + \frac 1 2\right)}.
\end{align*} 
In particular, if $L_n^{1/d} \asymp n^{\frac 1 {d+2s}}$, then for  $t \geq 0$, 
$$\bbE  \|\tilde p_n - p \|_{\calH^{t,r}(\Omega)}^r \lesssim n^{-\frac{r(s-t)}{2s+d}}.$$
\end{proposition}

\subsubsection{Proof of Proposition~\ref{prop:density_risk_smooth_domain}} 
We begin by bounding the variance term, assuming $t \geq 0$. 
Recall that $L_n \asymp n^a$ with $0 < a < 1$, and define $r_0= 2/(1-a)$. 
We begin by proving the claim when $r \geq r_0$. 

We wish to bound the quantity
$$V_n = \bbE \big\|(-\Delta)^{t/2}[\tilde p_n-p_{L_n}]\big\|_{L^r(\Omega)}^r
 = \left\| \sum_{\ell=1}^{L_n} \calL_\ell[\tilde p_n-p_{L_n}] \omega_\ell \lambda_\ell^{t/2} \eta_\ell\right\|_{L^r(\Omega)}^r.$$ 
Notice that
$$\calL_\ell[\tilde p_n-p_{L_n}]  = \halpha_\ell - \alpha_\ell = \frac 1 n \sum_{i=1}^n (\eta_\ell(X_i)-\bbE[\eta_\ell(X_i)]),$$
so that
$$V_n = \int_\Omega\left| \frac 1 n \sum_{i=1}^n U_{n,i}(x)\right|^r dx,
\quad \text{with } U_{n,i}(x) = \sum_{\ell=1}^{L_n} \omega_\ell \lambda_\ell^{t/2} \eta_\ell(x)\big(\eta_\ell(X_i) - \bbE[\eta_\ell(X_i)]\big).$$
By Rosenthal's inequalities~\citep{rosenthal1970,rosenthal1972}, we deduce that
\begin{equation}\label{eq:rosenthal}
V_n \lesssim n^{-r/2} \int_\Omega \bbE[|U_{n,1}(x)|^2]^{r/2}dx + n^{1-r} \int_\Omega \bbE|U_{n,1}(x)|^rdx.
\end{equation} 
 We will provide a somewhat crude bound
   on $\bbE|U_{n,1}(x)|^r$, 
 followed by a sharp bound on $\bbE|U_{n,1}(x)|^2$.
 The density $p$ is H\"older-smooth over $\Omega$, and is in particular bounded. We thus have
 for any $x \in \Omega$, 
\begin{align*}
\bbE|U_{n,1}(x)|^r 
 &\lesssim \left\| \sum_{\ell=1}^{L_n} \omega_\ell \lambda_\ell^{t/2} \eta_\ell(x) \eta_\ell\right\|_{L^r(\Omega)}^r \\
 &\leq \lambda_{L_n}^{rt/2} \sup_{y \in \Omega} \left(\sum_{\ell=1}^{L_n} |\eta_\ell(x)||\eta_\ell(y)|\right)^r
 \leq \lambda_{L_n}^{rt/2} \left(e_{L_n}(x,x)e_{L_n}(y,y)\right)^{\frac r 2},
 \end{align*} 
 where we define the {\it spectral function} of the Neumann Laplacian by 
$$e_L(x,y) = \sum_{\ell=1}^L \eta_\ell(x) \eta_\ell(y),\quad x,y \in \Omega, ~ L \geq 1.$$
We will make use of the following bound on the spectral function. 
\begin{lemma}[\cite{hormander2007}, Theorem 17.5.3]
\label{lem:spectral_function}
There exists a constant $C > 0$ such that for all $L \geq 1$, 
$$\| e_L\|_{L^\infty(\Omega\times \Omega)} \leq C \lambda_L^{\frac{d}{2}}.$$
\end{lemma}
In order to bound the eigenvalue appearing on the right-hand side of the above Lemma, we make use of Weyl's Law for the Neumann
Laplacian (see for instance~\cite{dunlop2020}, Lemma 7.10, and references therein). 
\begin{lemma}[Weyl's Law]
\label{lem:weyl}
There exists a constant $c > 0$ depending only on $\Omega$ such that
$$\ell^{2/d}/c \leq \lambda_\ell \leq c \ell^{2/d}, \quad \ell=1,2,\dots$$
\end{lemma}
By Lemmas~\ref{lem:spectral_function}--\ref{lem:weyl}, we 
obtain 
\begin{equation}
\label{eq:U1r_bound}
\bbE|U_{n,1}(x)|^r \lesssim \lambda_{L_n}^{\frac{rt}{2} + \frac {dr} 2} \asymp L_n^{r\left(\frac{t}{d} + 1\right)}.
\end{equation} 
Using Plancherel's identity, a sharper bound in available in the quadratic case:
\begin{align}
\nonumber 
\bbE|U_{n,1}(x)|^2
 &\lesssim \left\| \sum_{\ell=1}^{L_n} \omega_\ell \lambda_\ell^{t/2} \eta_\ell(x) \eta_\ell\right\|_{L^2(\Omega)}^2 \\
\nonumber 
 &=   \sum_{\ell=1}^{L_n} \omega_\ell^2 \lambda_\ell^{t} \eta_\ell^2(x)   \\
\nonumber 
 &\lesssim \lambda_{L_n}^t  e_{L_n}(x,x) \\
 \label{eq:plancherel_bound_U1}
 &\lesssim \lambda_{L_n}^{t + \frac d 2} \asymp L_n^{\frac{2t}{d} + 1}.
 \end{align}
Combining equation~\eqref{eq:rosenthal} with the bounds~\eqref{eq:U1r_bound}
and~\eqref{eq:plancherel_bound_U1}, we have thus shown:
$$V_n \lesssim n^{-\frac r 2}  L_n^{r\left(\frac{t}{d} + \frac 1 2\right)}  + n^{1-r} L_n^{r\left(\frac{t}{d} + 1\right)}.$$
Since $r \geq r_0$, the second term is of lower order than the first, and we obtain the claimed bound
$$V_n \lesssim n^{-\frac r 2}  L_n^{r\left(\frac{t}{d}+\frac 1 2\right)}.$$
If we instead have $r < r_0$, then Jensen's inequality and the above bound imply
\begin{align*}
\bbE\big\|(-\Delta)^{t/2}[\tilde p_n-p_{L_n}]\big\|_{L^r(\Omega)}^r
 &\leq \bbE\big\|(-\Delta)^{t/2}[\tilde p_n-p_{L_n}]\big\|_{L^{r_0}(\Omega)}^r \\
 &\leq \Big(\bbE\big\|(-\Delta)^{t/2}[\tilde p_n-p_{L_n}]\big\|_{L^{r_0}(\Omega)}^{r_0}\Big)^{\frac{r}{r_0}}\\
 &\lesssim \Big( n^{-\frac {r_0} 2}   L_n^{r_0\left(\frac{t}{d}+\frac 1 2\right)}\Big)^{\frac{r}{r_0}}\\ 
 &=  n^{-\frac r 2}  L_n^{r\left(\frac{t}{d}+\frac 1 2\right)}. 
\end{align*}
This completes our bound of the fluctuations when $t \geq 0$.
 
We now turn to bounding the bias term, where we now allow $t$ to be any real number. 
Our main technical tool will be the multiplier result in Lemma~\ref{lem:multiplier}. 
Define the map 
\begin{align}
\label{eq:multiplier_smooth_domain}
m(x) = |x|^{t-s}(1-\tau(|x|^2)),\quad x \in \bbR,
\end{align}
where we recall that $\tau$ is the function used to define
the weights $\omega_j$. Notice that $m(x) = 0$ for all $|x|\leq 1/2$. 
Furthermore, it is clear that $m \in \calC^\infty(\bbR_+)$, 
and that $m$ satisfies Mikhlin's condition~\eqref{eq:cond_mikhlin}. 
It is then also clear that the map $m(\cdot/\sqrt{\lambda_{L_n}})$ satisfies this condition.

Now, with the convention $\omega_\ell = 0$ for all $\ell \geq L_n+1$, 
$$(-\Delta)^{t/2} [p - p_{L_n} ]= \sum_{\ell=1}^\infty (1-\omega_\ell) \lambda_\ell^{t/2} \alpha_\ell \eta_\ell 
 = \lambda_{L_n}^{-\frac{s-t}{2}} \sum_{\ell=1}^\infty m(\sqrt{\lambda_\ell/\lambda_{L_n}}) \lambda_\ell^{s/2}  \alpha_\ell\eta_\ell, $$
thus, applying Lemma~\ref{lem:multiplier} to the multiplier $m(\cdot/\sqrt{\lambda_{L_n}})$, we obtain 
\begin{align*}
\big\|(-\Delta)^{t/2} [p_{L_n} - p]\big\|_{L^r(\Omega)}
 &= \left\| \lambda_{L_n}^{-\frac{s-t}{2}} \sum_{\ell=1}^\infty 
                        m(\sqrt{\lambda_\ell/\lambda_{L_n}}) \lambda_\ell^{s/2}  \alpha_\ell\eta_\ell\right\|_{L^r(\Omega)} \\
 &\lesssim \lambda_{L_n}^{-\frac{s-t}{2}} \left\|  \sum_{\ell=1}^\infty  \lambda_\ell^{s/2}  \alpha_\ell\eta_\ell \right\|_{L^r(\Omega)} \\
 &\lesssim  {L_n}^{-\frac{s-t}{d}} \left\| p\right\|_{\calH^{s,r}(\Omega)},
\end{align*} 
where we again used Weyl's law. It thus suffices to show that $\left\| p\right\|_{\calH^{s,r}(\Omega)}$ is finite.
To this end, let $r_1$ be defined as $r$ if $s-1/r$ is not an odd integer, and otherwise define $r_1$ as $r+\delta$ for any small enough $\delta < 1$. 
Then, by  Proposition~\ref{prop:equiv_sobolev} (with $s \leq 2$ or $r=r_1=2$) \
and the definitions of the spaces $\calC_N^s(\Omega)$ and $H_N^{s,r}(\Omega)$, we have
\begin{align*}
\left\| p\right\|_{\calH^{s,r}(\Omega)}
\leq 
\left\| p\right\|_{\calH^{s,r_1}(\Omega)}
 &\lesssim  \left\| p\right\|_{H^{s,r_1}(\Omega)}
 \leq   \left\| p\right\|_{\calC^s(\Omega)} \leq M.
\end{align*} 
The claim thus follows
\qed 

\begin{remark}\label{rem:fefferman}
Suppose that instead of the estimator 
$$\tilde p_n = \sum_{\ell=1}^{L_n} \omega_\ell\hat\alpha_\ell\eta_\ell$$
we had used the traditional truncated series estimator
$$\bar p_n = \sum_{\ell=1}^{L_n} \hat\alpha_\ell\eta_\ell,$$
which corresponds to choosing the nonsmooth function $\tau(x) = I(|x| < 1)$
in the definition of the weights $\omega_\ell$. 
This choice 
would  prevent the function $m$ in equation~\eqref{eq:multiplier_smooth_domain} from satisfying
the conditions of Lemma~\ref{lem:multiplier}. In fact, if one
were to replace $\Omega$ by $\bbT^d$, then the eigenvalues $\lambda_\ell$ would be of the form $\|2\pi \xi_\ell\|^2$
for some enumeration $\xi_1,\xi_2,\dots$ of $\bbZ_*^d$. In this case, we have for all $\ell \geq 1$, 
$$1-\tau(\lambda_\ell/\lambda_{L_n}) = I(\|\xi_\ell\| \leq \|\xi_{L_n}\|).$$
Viewed as a function of $\xi_\ell$, the right-hand side is
the indicator function of a ball, which  is well-known not to be an $L^r(\bbT^d)$ Fourier multiplier
for $r > 2$ and $d > 1$ 
\citep{fefferman1971}, thus the expression $m$ in~\eqref{eq:multiplier_smooth_domain} is also not an $L^r(\bbT^d)$-multiplier
in this case. This suggests that our current proof technique cannot be used for the traditional series
estimator.
\end{remark} 

We now supplement Proposition~\ref{prop:density_risk_smooth_domain} with a simple variance
bound for negative values of $t$, but now focusing on the case $r=2$. 

\begin{proposition}\label{prop:smooth_domain_neg_exp}
Let $M,s> 0$, and assume that $p \in \calC_N^s(\Omega;M)$. 
Then, for any given $t < 0$, 
there exists a constant $C > 0$ depending on $\Omega, M,d,t,s$ such that 
\begin{align*}
 \bbE \|\tilde p_n - p_{L_n}\|_{\calH^{t}(\Omega)}^2  
\lesssim \frac 1 n\begin{cases}
L_n^{\frac{2t}{d}+1}, & 2|t| < d, \\ 
\log(L_n), & 2|t| = d, \\
1, & 2|t| > d.
\end{cases}
\end{align*} 
In particular, if $L_n^{1/d} \asymp n^{\frac 1 {d+2s}}$, then for  $t \geq 0$, 
$$\bbE  \|\tilde p_n - p \|_{\calH^{t}(\Omega)}^2 \lesssim 
\begin{cases}
n^{-\frac{2(s-t)}{2s+d}}, & 2|t| < d \\ 
\log n / n,       & 2|t| = d \\ 
1/n, &                      2|t| > d
\end{cases}
.$$
\end{proposition}

\begin{proof}[Proof of Proposition~\ref{prop:smooth_domain_neg_exp}]
Notice that
\begin{align*}
\bbE \|\tilde  p_n - p_{L_n}\|_{\calH^{t}(\Omega)}^2
 &= \bbE\left[\sum_{\ell=1}^{L_n} \lambda_\ell^{t} \omega_\ell^2 (\halpha_\ell-\alpha_\ell)^2\right] 
 \leq \sum_{\ell=1}^{L_n}  \lambda_\ell^{t}  \Var[\halpha_\ell] = \sum_{\ell=1}^{L_n} \lambda_\ell^{t}\frac{ \Var[\eta_\ell(X)]}{n}.
\end{align*}
Since $p$ is bounded from above by $M$ over $\Omega$, we have 
$$\Var[\eta_\ell(X)] \leq \bbE[\eta_\ell^2(X)] \leq M \|\eta_\ell\|_{L^2(\Omega)}^2 = M,$$
thus, together with Weyl's Law, we have
\begin{align*}
\bbE \|\tilde  p_n - p_{L_n}\|_{\calH^{-t}(\Omega)}^2
 \leq \frac M n\sum_{\ell=1}^{L_n}\ell^{-2|t|/d} .
\end{align*}
The claim  follows from here. 
\end{proof} 

\subsection{Regularity of the Density Estimator}
\label{app:regularity_of_spectral_estimator} 
With Proposition~\ref{prop:density_risk_smooth_domain} in hand, we can prove the following result, 
which will allow us to invoke condition~\ref{assm:caffarelli} directly on the density
estimators.
\begin{lemma}\label{lem:tilde_q_smooth}
Let  $M,\gamma,s,c> 0$, and assume that $p \in \calC_N^s(\Omega;M,\gamma)$. 
Assume $L_n^{1/d} = c n^{\frac{1}{d+2s}}$.
Then, there exist  constants 
$C,\epsilon > 0$ depending on $\Omega, M,\gamma,d,s,c$ such that the following assertions
hold on an event of probability at least $1 - C/n^2$:
\begin{enumerate}
\item[(i)] $\tilde p_n \geq 1/C$ over $\Omega$. In particular, $\tilde p_n = \hat p_n$. 
\item[(ii)] $\|\tilde p_n\|_{\calC^\epsilon(\Omega)} \leq C$. 
\end{enumerate}
\end{lemma}
\begin{proof}[Proof of Lemma~\ref{lem:tilde_q_smooth}] 
Let $t = (s/2) \wedge 1$ and $\epsilon =t/2$. By a Sobolev embedding (cf.~\cite{triebel1995},
Theorem~4.6.1), 
we have  for all $r \geq r_0:= 2d/\epsilon$,
$$ \|\tilde p_n-p\|_{\calC^\epsilon(\Omega)}\lesssim \|\tilde p_n - p\|_{H^{t,r}(\Omega)}
\asymp \|\tilde p_n-p\|_{\calH^{t,r}(\Omega)},$$
where the final order assessment follows from Proposition~\ref{prop:equiv_sobolev}.
By Proposition~\ref{prop:density_risk_smooth_domain}, we thus obtain
$$\bbE \|\tilde p_n-p\|_{\calC^\epsilon(\Omega)}^r \lesssim 
 n^{-\frac{r(s-t)}{2s+d}}.$$ 
Let $u = \big(\|p^{-1}\|_{L^\infty(\Omega)}^{-1} \wedge \|p\|_{\calC^\epsilon(\Omega)}\big)/2$.
Then, by Markov's inequality, 
we have 
\begin{align*}
\bbP\Big( \|\tilde p_n - p\|_{\calC^\epsilon(\Omega)} \geq u  \Big)
 \leq \frac{\bbE\|\tilde p_n - p\|_{\calC^\epsilon(\Omega)}^r}{u^r} 
 \lesssim \frac{n^{-\frac{r(s-t)}{2s+d}}}{u^r}.
\end{align*}
Since $t < s$, we may choose $r$ large enough such that $\frac{r(s-t)}{2s+d} \geq 2$. Thus we 
readily deduce that for a large enough constant $C > 0$, 
we have with probability at least $1-C/n^2$ that 
$$\|\tilde p_n - p\|_{\calC^\epsilon(\Omega)} \leq u.$$
Over the above high-probability event, we have on the one hand
$$\inf_{x \in \Omega} \tilde p_n(x) 
\geq \inf_{x \in \Omega}  p(x)  - \|\tilde p_n - p\|_{L^\infty(\Omega)} 
\geq \inf_{x \in \Omega}  p(x)  - \|\tilde p_n - p\|_{\calC^\epsilon(\Omega)} 
\geq \inf_{x \in \Omega}  p(x) / 2,$$
from which part (i) of Lemma~\ref{lem:tilde_q_smooth} follows. On the other hand, 
$$\|\tilde p_n\|_{\calC^\epsilon(\Omega)} \leq 
   \|p\|_{\calC^\epsilon(\Omega)} + \|\tilde p_n - p\|_{\calC^\epsilon(\Omega)}
\leq 2\|p\|_{\calC^\epsilon(\Omega)} ,$$
from which part (ii) of Lemma~\ref{lem:tilde_q_smooth} follows. 
\end{proof}

\subsection{Convergence Rate under the Wasserstein Distance}
\label{app:W2_conv_of_spectral_estimator}
With the help of Proposition~\ref{prop:density_risk_smooth_domain} 
we can   obtain a bound on the risk
of $\hat P_n$ in Wasserstein distance. 
\begin{lemma}\label{lem:w_convergence_domain}
Let  $M,\gamma,s> 0$,  and assume that $p \in \calC_N^s(\Omega;M,\gamma)$. 
Assume $L_n^{1/d} = c n^{\frac{1}{d+2s}}$. Then, there exists
a  constant 
$C > 0$ depending on $\Omega, M,\gamma,d,s,r,c$ such that
$$\bbE W_2^2(\hat P_n,P)
\lesssim 
\begin{cases}
n^{-\frac{2(s+1)}{2s+d}}, & d > 2 \\ 
\log n / n,       & d=2 \\ 
1/n,              & d=1.
\end{cases}
$$
\end{lemma}
\begin{proof}[Proof of Lemma~\ref{lem:w_convergence_domain}]
Let $A_n$ be the event over which the two assertions of Lemma~\ref{lem:tilde_q_smooth} hold. 
By the bound~\eqref{eq:peyre} due to~\cite{peyre2018} (in its form
stated in Theorem 5.34 of~\cite{santambrogio2015}), it holds that
\begin{align*}
\bbE W_2^r(\hat P_n, P) 
 \lesssim 
\bbE \big[W_2^r(\hat P_n, P)I(A_n)\big] + n^{-2}
 \lesssim   \bbE\big[ \|\tilde  p_n - p\|_{\calH^{-1}(\Omega)}^r \big] + n^{-2}.
\end{align*} 
The claim thus follows from Proposition~\ref{prop:smooth_domain_neg_exp}.
\end{proof}

\subsection{Proof of Theorem~\ref{thm:smooth_domains}}
\label{app:pf_thm_smooth_domains}
Using Lemmas~\ref{lem:tilde_q_smooth} and \ref{lem:w_convergence_domain}, 
Theorem~\ref{thm:smooth_domains--one_sample} follows by the   same argument 
as Theorem~\ref{thm:one_sample_wavelet--map}, and Theorem~\ref{thm:smooth_domains--two_sample} follows by the same
argument as  Theorems~\ref{thm:two_sample_kernel} and \ref{thm:two_sample_density}.
In the latter case, one replaces the application of Caffarelli's regularity theory (Theorem~\ref{thm:torus_regularity})
by an application of condition~\ref{assm:caffarelli}.
We omit further details for brevity.\qed  

}

\section{Proofs of Central Limit Theorems}
\label{app:pf_clts}

The aim of this appendix is to prove Theorem~\ref{thm:clt}. We also  state and prove
Proposition~\ref{prop:variance_estimation}, regarding the question of 
variance estimation. We begin by deriving limit laws for the functional $\int\phi_0(\hat p_n-p) $,
which form an important component of our central limit theorems.
Here, $\hat p_n$ is one of
the estimators $\hat p_n\sbc, \hat p_n\sper$, and 
$\hat p_n\sker$,  
which respectively arise from the classical boundary-corrected, periodic, and  kernel density estimators
$\tilde p_n\sbc, \tilde p_n\sper$,  and
$\tilde p_n\sker$.  
We also write  
$$p_{J_n}\sbc = \bbE[\tilde p_n\sbc], \quad 
  p_{J_n}\sper = \bbE[\tilde p_n\sper], 
\quad 
p_{h_n}\sker = \bbE[\tilde p_n\sker].$$
We have the following. 
\begin{lemma}
\label{lem:linear_functional_clt}
Let $\epsilon,s > 0$, and let $h_n^{-1} \asymp 2^{J_n} \uparrow \infty$. 
\begin{thmlist}
\item\label{lem:linear_functional_clt--bc}  (Unit Hypercube) Let $p \in \calC^\epsilon([0,1]^d)$ be positive over $[0,1]^d$. 
Assume that $\phi_0 \in \calC^{s}([0,1]^d)$ satisfies $\Var_P[\phi_0(X)]  > 0$. Then, as  $n \to \infty$, 
\begin{align*}
\sqrt n \int \phi_0 (\hat p_n\sbc - p_{J_n}\sbc) &\rightsquigarrow N(0, \Var_P[\phi_0(X)]).
\end{align*}
\item\label{lem:linear_functional_clt--per} (Flat Torus) Let $p\in \calC^\epsilon(\bbT^d)$ 
be positive over $\bbT^d$. Assume that  $\phi_0 \in \calC^s(\bbT^d)$
satisfies  $\Var_P[\phi_0(X)] > 0$. Then,  as $n \to \infty$,
\begin{align*}
\sqrt n\int \phi_0 (\hat p_n\sper - p_{J_n}\sper) &\rightsquigarrow N(0, \Var_P[\phi_0(X)]), \\
\sqrt n\int \phi_0 (\hat p_n\sker - p_{h_n}\sker) &\rightsquigarrow N(0, \Var_P[\phi_0(X)]).
\end{align*}
\end{thmlist} 
\end{lemma}
\subsection{Proof of Lemma~\ref{lem:linear_functional_clt}}
The proof is standard, thus we only prove  claim (i).
 The remaining
claims can be proven similarly. 
For simplicity, we write $\Psi_{j_0-1}\pbc = \Phi\pbc$ throughout the proof.
Reasoning as in the proof of Lemma~\ref{lem:L}, 
and in particular using Lemma~\ref{lem:tilde_q_density}, it holds that
\begin{align*}
\sqrt n \int \phi_0 (\hat p_n\sbc - p_{J_n}\sbc)
 &= \sqrt n \int \phi_0 (\tilde p_n\sbc - p_{J_n}\sbc) + \sqrt n \int \phi_0 (\hat p_n\sbc - \tilde p_n\sbc) \\
 &=  \sqrt n \int \phi_0 (\tilde p_n\sbc - p_{J_n}\sbc) + o_p(1) \\ 
 &= \frac 1 {\sqrt n} \sum_{i=1}^n (Z_{n,i} - \bbE[ Z_{n,i}]),
\end{align*}
where we write 
$$Z_{n,i} =   \sum_{j=j_0-1}^{J_n}\sum_{\xi\in\Psi_j\pbc} \xi(X_i)\gamma_\xi,
\quad i=1,\dots,n,$$
and where  $\gamma_\xi = \int \phi_0 \xi$ for all $\xi \in \Psi\pbc$.  
By Lyapunov's central limit theorem~(\cite{billingsley1968}, Theorem 7.3), it holds that
\begin{equation}
\label{eq:lyapunov}
\frac{1}{\sqrt{\sum_{i=1}^n \Var[Z_{n,i}]}}\sum_{i=1}^n (Z_{n,i} - \bbE[Z_{n,i}])  
 \rightsquigarrow N(0,1),
\end{equation}
provided that for some $p > 2$,
\begin{equation}
\label{eq:lyapunov_condition}
\frac{\sum_{i=1}^n \bbE\left[|Z_{n,i}-\bbE Z_{n,i}|^p\right]}
       {\left(\sum_{i=1}^n \Var[Z_{n,i}]\right)^{p/2} }
      \to 0.
\end{equation}      
Now, using Lemma~\ref{lem:wavelet}, it holds that  
\begin{align}
\label{eq:pf_kantorovich_variance_bound_unif_Z}
\nonumber
\sup_{n \geq 1} \sup_{1 \leq i \leq n} |Z_{n,i} |
 &\leq \sup_{n \geq 1}\left\|\sum_{j=j_0-1}^{J_n}\sum_{\xi\in\Psi_j\pbc}
  \xi \gamma_\xi\right\|_{\infty} \\
\nonumber 
 &\leq  \sum_{j=j_0-1}^\infty \|(\gamma_\xi)_{\xi\in\Psi_j\pbc}\|_{\ell_\infty}
 				\left(\sup_{\xi\in\Psi_j\pbc} \|\xi\|_{\infty} \right) \left\|\sum_{\xi \in \Psi_j\pbc} I(|\xi| > 0)\right\|_{\infty} \\
 &\lesssim   \sum_{j=j_0-1}^\infty 2^{-j\left(\frac d 2 +s\right)} 2^{\frac{dj}{2}} \lesssim 
 \sum_{j=j_0-1}^\infty  2^{-js} <\infty.
\end{align}
On the other hand, under the stated conditions, it follows from Lemma~\ref{lem:L} that 
\begin{equation} 
\label{eq:variance_n_o_1}
\sum_{i=1}^n \Var[Z_{n,i}]  
 = n(\Var_P[\phi_0(X)]+ o(1)).
 \end{equation} 
Since $\Var_P[\phi_0(X)] > 0$, the denominator in equation~\eqref{eq:lyapunov_condition}
is of the order $n^{p/2}$, while the numerator is of   order $n$ by equation~\eqref{eq:pf_kantorovich_variance_bound_unif_Z}.
It follows that Lyapunov's condition~\eqref{eq:lyapunov_condition} holds for all $p > 2$.
The claim thus follows from equations~\eqref{eq:lyapunov} and~\eqref{eq:variance_n_o_1}. 
\qed

\subsection{Proof of Theorem~\ref{thm:clt}} 
Assume first that $\sigma_0,\sigma_1 > 0$. We begin with part (i). 
Under the stated conditions on the densities, it follows from Theorem~\ref{thm:torus_regularity}
that $\varphi_0$ satisfies condition~\ref{assm:curvature} for some $\lambda > 0$.  
{Apply the stability bound of Theorem~\ref{thm:stability} to obtain,
$$0 \leq W_2^2 (\hat P_n^{(\mathrm{ker})}, Q) - W_2^2(P,Q) 
 - \int \phi_0 d(\hat P_n^{(\mathrm{ker})} - P)\leq W_2^2(\hat P_n^{(\mathrm{ker})}, P).$$
Using the convergence
rate of $\hat P_n^{(\mathrm{ker})}$ under $W_2^2$ in Proposition~\ref{lem:w2_kernel_rate}, and 
Lemma~\ref{lem:L_ker}
regarding the bias of $\int \phi_0 d\hat P_n^{(\mathrm{ker})}$, we obtain
$$W_2^2 (\hat P_n^{(\mathrm{ker})}, Q) - W_2^2(P,Q)  = \int \phi_0 (\hat p_n^{(\mathrm{ker})}-p_{h_n}^{(\mathrm{ker})}) +  O_p\left(n^{-\frac{2\alpha}{2(\alpha-1)+d}} \vee \frac{(\log n)^2}{n}\right).$$
Using the assumption $2(\alpha+1) > d$, deduce that 
$$\sqrt n \left(W_2^2 (\hat P_n^{(\mathrm{ker})}, Q) - W_2^2(P,Q) \right) =\sqrt n \int \phi_0  (\hat p_n^{(\mathrm{ker})}-p_{h_n}^{(\mathrm{ker})}) +  o_p(1).$$
Apply  Lemma~\ref{lem:linear_functional_clt} to deduce that
$$\sqrt n \left(W_2^2 (\hat P_n^{(\mathrm{ker})}, Q) - W_2^2(P,Q) \right) \rightsquigarrow N(0,\sigma_0^2),
\quad \text{as } n\to \infty.$$
By the same reasoning, but now using the two-sample stability bound of Proposition~\ref{thm:two_sample_stability}, we also have
\begin{align*}
\sqrt{\frac{nm}{n+m}} &\left(W_2^2 (\hat P_n^{(\mathrm{ker})}, Q_m^{(\text{ker})})- W_2^2(P,Q) \right) 
\\ &= 
\sqrt {(1-\rho)   n} \int \phi_0  (\hat p_n^{(\mathrm{ker})}-p_{h_n}^{(\mathrm{ker})}) + 
\sqrt{\rho m} \int \psi_0 (\hat q_m^{(\mathrm{ker})}-q_{h_m}^{(\mathrm{ker})}) + o_p(1),
\end{align*}
as $n,m \to \infty$ such that $n/(n+m) \to \rho \in [0,1]$.
By Lemma~\ref{lem:linear_functional_clt} and the independence
of $X_1, \dots, X_n, Y_1, \dots, Y_m$, we deduce that 
\begin{align*}
\sqrt{\frac{nm}{n+m}} \left(W_2^2 (\hat P_n^{(\mathrm{ker})}, Q_m^{(\text{ker})})- W_2^2(P,Q) \right) 
\rightsquigarrow N(0, \sigma_\rho^2),
\end{align*}
as $n,m\to\infty$ such that $n/(n+m) \to \rho$. This proves claim (i) for kernel estimators. 
Claim~(ii)  regarding the boundary-corrected wavelet estimators $(\hat P_n,\hat Q_m)$
now follows analogously by using 
Lemma~\ref{lem:L} 
to bound the bias of $\int \phi_0 d\hat P_n$, 
 Lemma~\ref{lem:wavelet_wasserstein} 
to bound the convergence rate
of $\hat P_n$ in Wasserstein distance,
Proposition~\ref{thm:two_sample_density--wasserstein_cube},
to bound the bias of the plugin estimator of the Wasserstein distance, and 
Lemma~\ref{lem:linear_functional_clt} to obtain the limiting distribution
of $\int \phi_0 (\tilde p_n - \bbE[\tilde p_n])$. 
}

Finally, to prove part (v), apply Corollary~\ref{cor:one_sample_emp_hypercube} and the result of~\cite{divol2021}
to deduce that, for $\Omega \in \{[0,1]^d, \bbT^d\}$, since the densities $p$ and $q$ are bounded and bounded away from zero, we have
$$\sqrt n  W_2^2(P_n, P) = o_p(1), \quad \sqrt m W_2^2(Q_m, Q) = o_p(1), $$
as $n,m \to \infty$, whenever $d \leq 3$. Therefore, using  Theorem~\ref{thm:stability}, 
Proposition~\ref{thm:two_sample_stability}, 
and Proposition~\ref{prop:torus_stability}, we obtain
\begin{align*}
\sqrt n\Big( W_2^2(P_n, Q) - W_2^2(P,Q)\Big)  &= \sqrt n \int \phi_0 d(P_n - P) + o_p(1), \\
\sqrt{\frac{nm}{n+m}} \Big( W_2^2(P_n, Q_m) - W_2^2(P,Q)\Big)  &= \sqrt{(1-\rho) n} \int \phi_0 d(P_n - P) \\ &+ 
\sqrt{\rho m} \int\psi_0 d(Q_m-Q) + o_p(1).
\end{align*}
Claim (v) then follows by the classical central limit theorem. 

It thus remains to consider the situation
where $\sigma_1 = 0$ or $\sigma_0 = 0$. Notice that the Kantorovich potentials $\phi_0$ and $\psi_0$
are almost everywhere constant if and only if $P = Q$. As a result, the statements  ``$\sigma_0 = 0$'', ``$\sigma_1 = 0$'', 
and ``$P = Q$'' are equivalent, thus it remains to prove the claim when $P = Q$. In this case, 
it suffices to show that $\sqrt n W_2^2(\hat P_n, P) = o_p(1)$ and 
$\sqrt{\frac{nm}{n+m}} W_2^2(\hat P_n, \hat Q_m) = o_p(1)$ for the various estimators
$\hat P_n$ and $\hat Q_m$ under consideration. But these assertions are a direct consequence
of the aforementioned convergence rates of these estimators in Wasserstein distance,
under the assumptions of each of parts (i)--(v).
The claim thus follows.
\qed  

\begin{remark}[Periodic Wavelet Estimators]
Using Proposition~\ref{prop:one_sample_torus} and Lemma~\ref{lem:linear_functional_clt}, 
it is easy to see that Theorem~\ref{thm:clt--torus} continues to hold
when the kernel density estimators $(\hat P_n\sker,\hat Q_m\sker)$ are replaced
by the periodic wavelet estimators $(\hat P_n\sper,\hat Q_m\sper)$, defined in Appendix~\ref{app:pf_density_based}. 
\end{remark} 
 
\subsection{Variance Estimation}\label{app:variance_estimation}
We now state   a simple result regarding the estimation of variances appearing in Theorem~\ref{thm:clt}. 
In what follows, let $\Omega$ be equal to $\bbT^d$, or to a compact and connected subset of $\bbR^d$, and let
$P,Q \in \calPac(\Omega)$. Let $(\phi_0,\psi_0)$ be a pair of Kantorovich potentials in the optimal transport problem
from $P$ to $Q$. Furthermore, let
$$X_1,\dots, X_n \sim P, \quad Y_1,\dots, Y_m \sim Q$$
be i.i.d. samples which are independent of each other, and let $P_n$ and $Q_m$ denote their respective empirical measures. 
\begin{proposition}\label{prop:variance_estimation}
Let the distributions $P,Q \in \calPac(\Omega)$ 
have positive densities over $\Omega$. Let $\hat P_n$ and $\hat Q_m$ be 
estimators such that
$$W_2(\hat P_n,P) = o_p(1), \quad \text{and}\quad W_2(\hat Q_m,Q) = o_p(1),$$
as $n,m\to\infty$. Let $(\hat \phi_{nm},\hat\psi_{nm})$
be a uniformly bounded  pair of Kantorovich potentials in the 
optimal transport problem from $\hat P_n$ to $\hat Q_m$. Then, as $n,m\to \infty$,
\begin{alignat*}{2} 
\hat\sigma_{0,nm}^2 &:= \Var_{U\sim P_n}[\hphi_{nm}(U)]
 &&\overset{p}{\longrightarrow} \Var[\phi_0(X)] , \text{ and}, \\ 
\hat\sigma_{1,nm}^2 &:= \Var_{V\sim Q_m}[ \hat\psi_{nm}  (V)] &&\overset{p}{\longrightarrow} \Var[\psi_0(Y)] .
\end{alignat*} 
\end{proposition} 
Note that the assumption of uniform boundedness of the fitted potentials can always be satisfied, due
to the compactness of $\Omega$~\cite[Remark 1.13]{villani2003}. 
In particular, the conditions of Proposition~\ref{prop:variance_estimation} are met for any of 
the estimators $(\hat P_n,\hat Q_m)$ appearing in the statement of Theorem~\ref{thm:clt}.

\begin{proof}[Proof of Proposition~\ref{prop:variance_estimation}]
To prove the claim, we shall make use of the following stability result for Kantorovich potentials 
over compact metric spaces, due to~\citeauthor{santambrogio2015} (\citeyear{santambrogio2015}, Theorem 1.52), which we only state in the generality required for our proofs. 
\begin{lemma}[\cite{santambrogio2015}]
\label{lem:stability_potentials}
Let $P,Q \in \calPac(\Omega)$,
and assume that at least one of $P$ and $Q$ has support equal to $\Omega$. Let
$(\widebar P_k)_{k\geq 1} , (\widebar Q_k)_{k\geq 1} \subseteq \calP(\Omega)$ be sequences 
which respectively converge to $P,Q$ weakly. Let $(\phi_k,\psi_k)$
denote a pair of   Kantorovich potentials in the optimal transport
problem from $\widebar P_k$ to $\widebar Q_k$, for all $k \geq 1$. Then, 
it holds that $\phi_k \to \phi_0$ and $\psi_k \to \psi_0$ as $k \to \infty$, the convergence
being uniform over $\Omega$, for some pair of Kantorovich potentials $(\phi_0,\psi_0)$  in the optimal transport
problem from $P$ to $Q$, which is uniquely defined up to translation by constants. 
\end{lemma} 
We will prove the claim for $\hat\sigma_{0,nm}^2$, and a symmetric argument
may be used for $\hat\sigma_{1,nm}^2$. By Lemma~\ref{lem:stability_potentials}, there exists a random variable $a_{nm}$ such that
$\hat\phi_{nm}^o := \hat\phi_{nm} + a_{nm}$ and $\hat\psi_{nm}^o:=\hat\psi_{nm}-a_{nm}$ 
converge uniformly to $\phi_0$ and $\psi_0$
respectively.
We have, 
\begin{align} 
\nonumber |\hsigma_{0,{nm}}^2 - \sigma_0^2|
 &= \left| \Var_{P_n}[\hphi_{nm}(U)] - \Var_P[\phi_0(X)]\right| \\
\nonumber 
 &= \left| \Var_{P_n}[\hphi_{nm}^o(U)] - \Var_P[\phi_0(X)]\right| \\ 
\nonumber 
 &\lesssim \left| \int (\hphi_{nm}^o)^2 dP_n - \int_{\bbT^d} \phi_0^2 dP\right| + 
          \left| \int  \hphi_{nm}^o   dP_n - \int_{\bbT^d} \phi_0 dP\right| \\
\nonumber 
 &\lesssim \left| \int(\hphi_{nm}^o)^2 d(P_n - P)\right| + 
           \left| \int \hphi_{nm}^o d(P_n - P)\right| +  
              \big\|\hphi_{nm}^o  - \phi_0\big\|_{L^2(P)}. 
\end{align} 
Since $\hat\phi_{nm}^o$ is convex up to translation
by a quadratic function,  and uniformly bounded, 
it must be 
Lipschitz with respect to $\|\cdot\|$ over the compact set $\Omega$, 
with a uniform constant depending only on the diameter of this 
set~(\cite{hiriart-urruty2004}, Lemma 3.1.1, p. 102). 
Thus, $(\hat\phi_{nm}^o)^2$ is also Lipschitz over $\Omega$ with uniform constant. 
The set of Lipschitz functions 
with a uniformly bounded  Lipschitz constant, over any given compact domain, forms a Glivenko-Cantelli class~(\cite{vandervaart1996}, Theorem 2.7.1),  
thus the first two terms on the right-hand side of the above
display vanish in probability. The final term vanishes due to the uniform convergence of $\hat\phi_{nm}^o$ to $\phi_0$.
\end{proof} 
%
%
%
%
 
\section{Proofs of Efficiency Lower Bounds}
\label{app:pf_efficiency}
Throughout this appendix, given $Q \in \calPac(\Omega)$, we abbreviate the functional $\Phi_Q$ by $\Phi$, 
and the influence functions $\tilde\Phi_{(P,Q)}$ and $\tilde\Psi_{(P,Q)}$ by $\tilde\Phi$ and $\tilde\Psi$,
respectively. 
 
We begin by defining the differentiable paths $(P_{t,h_1})_{t \geq 0}$
and $(Q_{t,h_2})_{t \geq 0}$, for all $(h_1,h_2) \in \bbR^2$, as announced
in Section~\ref{sec:efficiency}. 
We follow a construction from  
Example 1.12 of~\cite{vandervaart2002}.  
Recall that $P,Q \in \calPac(\Omega)$ admit respective densities $p,q$. 
Let $\zeta \in \calC^\infty(\bbR)\cap\calC^2(\bbR)$ be a bounded nonnegative  map, which is bounded
away from zero over $\bbR$ by a positive constant, and which satisfies
 $\zeta(0) = \zeta'(0) =\zeta''(0) = 1$. For any functions $f \in L^2_0(P)$
 and $g \in L^2_0(Q)$, define $P_t^f, Q_t^g \in \calPac(\Omega)$ to be
 the distributions with densities 
\begin{align}
\label{eq:main_diff_path_construction}
p_{t}^f(x) \propto \zeta(tf(x)) p(x), \quad
q_{t}^g(y) \propto \zeta(tg(y)) q(y),
\end{align}
for all $x,y \in \Omega$ and $t \geq 0$. Since $\zeta$ is bounded away from zero over $\bbR$, 
notice that the implicit normalizing constants in the above display are bounded from above by a constant
which does not depend on $t,f,g,p,q$. 
We now turn to the proofs of Lemma~\ref{lem:influence} and Theorem~\ref{thm:asymptotic_minimax}.
   
\subsection{Proof of Lemma~\ref{lem:influence}}
Let $f \in \dot\calP_P$ be an arbitrary score function, 
and abbreviate the differentiable path $P_t := P_t^f$, 
and its density $p_t := p_t^f$, for all $t \geq 0$. Here, we use the definition 
in equation~\eqref{eq:main_diff_path_construction}.  
Let $(\phi_t,\psi_t)$ denote a pair of Kantorovich potentials in 
the optimal transport problem from $P_t$ to $Q$, which we may and do choose to be
uniformly bounded by $\diam(\Omega)^2$, and hence uniformly bounded in $t$.
By the Kantorovich duality, one has
\begin{equation}
\label{eq:pf_influence_fn_kantorovich}
\begin{aligned}
\Phi(P_t) - \Phi(P) 
 &= \sup_{(\phi,\psi) \in \calK} \left[\int \phi dP_t + \int \psi dQ\right] - \int \phi_0 dP - \int\psi_0 dQ \\
 &\geq  \int \phi_0 d P_t + \int \psi_0 dQ - \int \phi_0 dP - \int \psi_0 dQ
 =   \int \phi_0 d(P_t - P), \\
\Phi(P_t) - \Phi(P) &= 
   \int \phi_t dP_t + \int \psi_t dQ - \sup_{(\phi,\psi) \in \calK}\left[\int \phi dP - \int\psi dQ\right]
 \leq \int \phi_t d(P_t - P).
\end{aligned}
\end{equation} 
By construction, the map $t\in [0,\infty) \mapsto p_t(x)$ is differentiable 
for every $x \in \Omega$, and letting $\Delta_t(x) = (p_t(x) - p(x))/t$, we have
\begin{equation}
\label{eq:pf_influence_fn_limit_qt}
\begin{aligned}
\lim_{t\to0} \Delta_t(x) = \frac{\partial}{\partial t} p_t(x) \bigg|_{t=0}
 = f(x)p(x).
 \end{aligned}
\end{equation}
Now, notice that for all $t \geq 0$, 
$$|\Delta_t(x)| \lesssim \left|\frac{\zeta(tf(x)) - 1}{t} \right|p(x) = \left|\frac{\zeta(tf(x)) - \zeta(0)}{t} \right| p(x)
\leq \|\zeta\|_{\calC^1(\bbR^d)} f(x)p(x).$$
Since $f \in L^2_0(P) \subseteq L^1_0(P)$, we deduce that $\Delta_t(x)$ is dominated by an integrable function, 
uniformly in $t$. Since $\phi_0$ is uniformly bounded, we also deduce
that the map $|\phi_0||\Delta_t - fp|$ is dominated by an integrable function.  
We then have, by equation~\eqref{eq:pf_influence_fn_limit_qt} and the Dominated Convergence Theorem,
\begin{align}
\nonumber 
\liminf_{t\to 0} \frac{\Phi(P_t) - \Phi(P)}{t} 
 &\geq  \liminf_{t \to 0}  \int_\Omega \phi_0 \Delta_t d\calL \\
\nonumber 
 &= \int_\Omega \phi_0 f dP +  \liminf_{t \to 0}  \int_\Omega \phi_0 \left[\Delta_t - fp\right] d\calL \\
\nonumber 
 &\geq \int_\Omega \phi_0 f dP -  \limsup_{t \to 0}  \int_\Omega |\phi_0| \left|\Delta_t - fp\right|d\calL \\
 &\geq \int_\Omega \phi_0 f dP -  \int_\Omega |\phi_0|  \limsup_{t \to 0}  \left|\Delta_t - fp\right|d\calL = \int \phi_0 f dP,
 \label{eq:pf_influence_liminf}
\end{align}
and similarly, 
\begin{align*}
\limsup_{t \to 0} \frac{\Phi(P_t) - \Phi(P)}{t}
 &\leq \limsup_{ t \to 0} \int \phi_t \Delta_t d\calL \\
 &\leq \limsup_{t\to 0}\int \phi_t fdP  + \left(\sup_{t \geq 0} \norm{\phi_t}_{L^\infty(\Omega)}\right) \limsup_{t \to 0}\int \left|\Delta_t - fp\right| d\calL\\
 &= \limsup_{t\to 0}  \int \phi_tfdP. 
 \end{align*}
Let $t_k\downarrow 0$ be a sequence achieving the limit superior, in the sense that $\lim_{k\to \infty} \int \phi_{t_k} fdP
= \limsup_{t\to 0} \int \phi_t dP$. Up to taking a subsequence  of $(t_k)$, Lemma~\ref{lem:stability_potentials} implies that 
$\phi_{t_k}$ converges uniformly to a Kantorovich potential $f_0$ from $P$ to $Q$, 
which is unique up to translation by a constant, and which therefore takes the form
$f_0 = \phi_0 + a$ for some $a \in \bbR$.  The limit superior clearly continues to be achieved along this subsequence,
thus we replace it by $(\phi_{t_k})$ without loss of generality. We thus have
$$\limsup_{t\to 0} \int \phi_t fdP = 
 \lim_{k\to\infty} \int \phi_{t_k} f dP = \int \left(\lim_{k\to\infty} \phi_{t_k}\right) f dP = \int (\phi_0 +a)f dP
  = \int \phi_0 f dP,$$
where the interchange of limit and integration holds again by the Dominated Convergence Theorem,
since $\phi_t$ are uniformly bounded, and $f \in L_0^2(P)$. Combine
this fact with equation~\eqref{eq:pf_influence_liminf} to deduce that
$$\lim_{t \to 0}\frac{\Phi(P_t) - \Phi(P)}{t} = \int \phi_0 f dP.$$
It follows that $\tilde \Phi = \phi_0 - \int \phi_0 dP$ is an influence function
of $\Phi$ with respect to $\dot\calP_P$. Since we assumed that
$\tilde \Phi \in \dot\calP_P$, it must in fact be the case
that $\tilde\Phi$ is the unique efficient influence function of $\Phi$
with respect to $\dot\calP_P$~\citep{vandervaart2002}, and the claim follows. 
 \qed

\subsection{Proof of Theorem~\ref{thm:asymptotic_minimax}} 
We shall use the following abbrevations of the differentiable
paths defined in equation~\eqref{eq:main_diff_path_construction}. For any $h \in \bbR$ and $t \geq 0$, if
$f = h\tilde \Phi$ and $g = h\tilde \Psi$, we write 
\begin{alignat}{3}
\label{eq:diff_opath_main_P}
P_{t,h} &:= P_t^{f}, \quad p_{t,h}(x) &&:= p_t^f(x) &&= c_h(t) \zeta(th\tilde\Phi(x))p(x), \\
\label{eq:diff_opath_main_Q}
Q_{t,h} &:= Q_t^{g}, \quad q_{t,h}(y) &&:= q_t^g(y) &&= k_h(t) \zeta(th\tilde\Psi(y))q(y), 
\end{alignat}
for all $x,y \in \bbT^d$, where the normalizing constants are explicitly denoted 
$$c_{h}(t) = \left(\int_{\bbT^d}\zeta(th\tilde\Phi(x))dP(x)\right)^{-1},\quad 
  k_{h}(t) = \left(\int_{\bbT^d} \zeta(th\tilde\Psi(y)) dQ(y)\right)^{-1}.$$
In this case,  the collections $\{(P_{t,h})_{t \geq 0}: h \in \bbR\}$ and
$\{(Q_{t,h})_{t \geq 0}: h \in \bbR\}$
respectively have score functions given by the tangent spaces
$$\dot\calP_P = \{h \tilde\Phi: h  \in \bbR\},\quad 
  \dot\calP_Q = \{h \tilde\Psi: h  \in \bbR\}.$$

We begin by showing that there exist $\widebar M, \widebar\gamma,\widebar u  > 0$ such that 
$p_{t,h} \in \calC^{\alpha-1}(\bbT^d;\widebar M, \widebar \gamma)$ 
uniformly in $t|h|\leq \widebar u$. An identical argument may then be used to show
that $q_{t,h} \in \calC^{\alpha-1}(\bbT^d;\widebar M, \widebar \gamma)$ for all appropriate $t,h$.
Our proof then proceeds by proving parts (i) and (ii). 

Since $p \geq \gamma^{-1}$, and since $\zeta$ is  
bounded from below by a positive constant, 
it is clear that there must exist $\widebar \gamma > 0$ depending on $\gamma$ 
and $\zeta$ such that 
\begin{equation}
\label{eq:pth_lower_bound}
\bar \gamma^{-1} \leq p_{t,h}~~~\text{over } \bbT^d,~ \text{for all } t \geq 0,h\in \bbR .
\end{equation}
We next prove the uniform H\"older continuity of $p_{t,h}$. We begin by studying the H\"older continuity of the map $\zeta(th\widetilde \Phi(\cdot))$.
Since 
$p,q \in \calC^{\alpha-1}(\bbT^d;M,\gamma)$, and since we assumed $\alpha \not\in\bbN$, 
we have by Theorem~\ref{thm:torus_regularity} that, for some constant $\lambda > 0$ depending
only on $M,\gamma, \alpha$, 
\begin{equation}
\label{eq:pf_one_sample_eff_caff}
\|\tilde\Phi \|_{\calC^{\alpha+1}(\bbT^d)}  \leq \lambda.
\end{equation}
Furthermore, recall that $\zeta \in \calC^\infty(\bbR)$. 
Thus, by the multivariate Fa\`a di Bruno formula~(see, for instance, \cite{encinas2003,constantine1996}), it holds that for all multi-indices $1 \leq |\beta| \leq \lfloor \alpha+1\rfloor$,
$$D^\beta \zeta(th \tilde\Phi(\cdot)) = \beta! \sum_{\ell=0}^{|\beta|} (th)^{\ell}\zeta^{(\ell)}(th \tilde\Phi(\cdot)) 
\sum_{(e_j),(\tau_j)} \prod_{j=1}^d \frac 1 {e_j!} \left(\frac 1 {\tau_j!} 
D^{\tau_j} \tilde\Phi(\cdot)\right)^{e_j},$$ 
where the second summation is taken over
all indices $(e_j)_{1 \leq j \leq d} \subseteq  \bbN$ and
multi-indices $(\tau_j)_{1 \leq j \leq d} \subseteq  \bbN^d$
such that for some $1 \leq s \leq d-1$, $\tau_j = 0$ and $e_j = 0$ for all $1 \leq j \leq s$, 
$e_j \neq 0$ and $\tau_j \neq 0$ for $s+1 \leq j \leq d$, and 
for which it holds that $\sum_{j=1}^d e_\tau = \ell$  and
$\sum_{j=1}^d e_j \tau_j = \beta.$
Furthermore,  $\beta! = \beta_1 ! \dots \beta_d!$, and   $\tau_j!$ is defined similarly for all $j$.  
Since  $\zeta \in \calC^\infty(\bbR)$, its derivatives of all orders less than $\alpha+1$ are uniformly
bounded over any fixed compact set. Since $\widetilde \Phi$ is   bounded, we deduce that for any $\widebar u > 0$, 
$$\sup_{0 \leq \ell \leq  \lfloor \alpha+1\rfloor} \sup_{\substack{t \geq 0 , h \in \bbR \\ t|h| \leq \widebar u}} \| (th)^\ell \zeta^{(\ell)}(th \widetilde \Phi(\cdot))\|_\infty \lesssim_{\widebar u,\alpha} 1.$$
Furthermore, we have $\|D^\tau \widetilde \Phi\|_\infty \leq \lambda$ for any $0 \leq |\tau| \leq \lfloor \alpha+1\rfloor$. 
This fact together with the preceding two displays implies
\begin{equation}
\label{eq:pf_semipar_holder_step_zeta}
\sup_{0 \leq |\beta| \leq \lfloor \alpha+1\rfloor} \sup_{\substack{t \geq 0 , h \in \bbR \\ t|h| \leq \widebar u}} \| D^\beta \zeta(th \tilde\Phi(\cdot))\|_\infty \lesssim_{\lambda,\widebar u, \alpha} 1.
\end{equation}
Now, recall that $p_{t,h}(\cdot) = c_h(t) \zeta(th\tilde\Phi(\cdot)) p(\cdot)$, and
 that $c_h(t)$ is uniformly bounded in $h$ and $t$ because $\zeta$ is bounded away from zero by a positive constant.
Thus, using the above display, the fact that $p \in \calC^{\alpha-1}(\bbT^d;M)$, and Lemma~\ref{lem:holder_products},
we deduce there exists a constant $\widebar M > 0$, depending only on $M,\widebar u,\alpha$ and the choice of $\zeta$, 
such that
$$\sup_{\substack{t \geq 0 , h \in \bbR \\ t|h| \leq \widebar u}} \|p_{t,h}\|_{\calC^{\alpha-1}(\bbT^d)} \leq \widebar M.$$
Combine this fact with equation~\eqref{eq:pth_lower_bound} to deduce that
$$p_{t,h} \in \calC^{\alpha-1}(\bbT^d; \widebar M, \widebar \gamma), \quad \text{for all } t\geq 0, h\in \bbR,t|h| \leq \widebar u.$$

We now prove part (i).  
Since $\tilde\Phi \in \dot\calP_P$, it follows from Lemma~\ref{lem:influence}  that 
$\tilde \Phi$
is the efficient influence function of $\Phi$
relative to $\dot\calP_P$. Since $\dot\calP_P$ is a vector space, it  
follows from Theorem~25.21 of~\cite{vandervaart1998} that for any estimator sequence $U_n$, 
$$\sup_{\substack{\calI \subseteq \bbR \\ |\calI| < \infty}} \liminf_{n\to\infty} \sup_{h \in \calI} n\bbE_{n,h}
\big|U_n - \Phi_Q(P_{n^{-1/2}, h})\big|^2  \geq \Var_P[\phi_0(X)],$$
where the infimum is over all estimator sequences.  
 
We next prove part (ii).  Inspired by the proof of Theorem~11 of \cite{berrett2019},
our goal will be to invoke a more general version of  Theorem~25.21 of~\cite{vandervaart1998}, 
given in Theorem~3.11.5 of~\cite{vandervaart1996}, whose statement we briefly summarize here. 
Let $H$ be a Hilbert space with inner product $\langle\cdot,\cdot\rangle_H$, 
and norm $\norm\cdot_H$. 
Let $(\calX_n, \calA_n, \mu_{n,h}:h\in H)$ be a sequence of asymptotically normal
experiments (as defined in Section 3.11 of  \cite{vandervaart1996}). 
A parameter sequence $(\kappa_n(h):h\in H) \subseteq \bbR$ is said to be 
regular if there exists a nonnegative sequence $(r_n)$ such that
$$r_n\big( \kappa_n(h) - \kappa_n(0)\big) \to \dot\kappa(h), \quad h \in H,$$
for a continuous linear map $\dot\kappa: H \to \bbR$.  
Denote by $\dot\kappa^*:\bbR \to H$ the adjoint of $\dot\kappa$, 
namely the map satisfying $\langle \dot\kappa^*(b^*), h\rangle_H = b^* \dot\kappa(h)$
for all $h \in H$.  
\begin{lemma}[\cite{vandervaart1996}, Theorem 3.11.5]
\label{lem:abstract_minimax}
Let the sequence of experiments $(\calX_n, \calA_n, \mu_{n,h}:h\in H)$ be asymptotically normal, 
and let the parameter sequence $(\kappa_n(h):h\in H)$ be regular. Suppose there exists a Gaussian random variable $G$
such that for all $b^* \in \bbR$, $b^* G \sim N(0, \|\dot\kappa^*(b^*)\|_H^2)$. 
Then, for any estimator sequence $(U_n)_{n\geq 1}$,
$$\sup_{\substack{\calI\subseteq H \\ |\calI| < \infty}} \liminf_{n\to\infty} \sup_{h\in \calI} r_n^2\bbE_{\mu_{n,h}}  (U_n - \kappa_n(h))^2 \geq \Var[G].$$ 
\end{lemma}
Returning to the proof, define the Hilbert space
$H = \bbR^2$  with inner product 
$$\langle(h_1, h_2), (h_1', h_2')\rangle_H = h_1h_1'\Var_P[\phi_0(X)] + h_2h_2'\Var_Q[\psi_0(Y)], \quad (h_1,h_2),(h_1',h_2') \in H,$$
and the sequence of experiments
$$\mu_{n,h} = P_{n^{- 1 /2},h_1}^{\otimes n} \otimes Q_{m^{-1 /2},h_2 }^{\otimes m},
\quad h=(h_1,h_2) \in \bbR^2,$$
endowed with the standard Borel $\sigma$-algebra. 
Here, $m$ is viewed as a function of $n$ which satisfies $n/(n+m) \to \rho \in [0,1]$ as $n\to \infty$. 
The following result can be deduced from Section~7.5
of~\cite{berrett2019} with minor modifications, using our assumptions placed on~$\zeta$. 
\begin{lemma}[\cite{berrett2019}]  
\label{lem:asymp_normal_experiments}
The sequence of experiments $(\mu_{n,h}: h\in H)$ is asymptotically
normal.
\end{lemma}
For all  $h = (h_1,h_2) \in H$, let $\kappa_n(h) = \Psi(P_{n^{-1/2},h_1}, Q_{m^{-1/2},h_2})$,
where again $m$ is treated as a function of $n$. Notice that
$\kappa_n(0) = \Psi(P,Q)$ for any $n \geq 1$. By following the same
argument as in the proof of Lemma~\ref{lem:influence}, using the Kantorovich
duality and the stability result for Kantorovich potentials in  Lemma~\ref{lem:stability_potentials}, one has
$$\kappa_n(h) - \kappa_n(0) 
 = \int \phi_0 d(P_{ n^{-1/2},h_1}-P) + \int \psi_0 d(Q_{m^{-1/2},h_2} - Q) + o(1).$$
 Now,  since $\zeta(0)=\zeta'(0)=1$ and $\int \tilde\Phi dP = 0$, we have  for all $t \geq 0$, 
\begin{align*}
\Bigg|\frac 1 {c_{h_1}(t)} - 1 \Bigg|
= \left| \int \left[ \zeta(t h_1 \tilde\Phi(x)) - 1 - th_1\tilde\Phi(x) \right]dP(x) \right| 
\lesssim \|\zeta\|_{\calC^2(\bbR)} t^2 h_1^2\|\tilde\Phi\|_{L^2(P)}. 
\end{align*}
Recall that $p$ and $\zeta$ are bounded, and that
$c_{h_1}(n^{-1/2})$ is uniformly bounded in $h_1$ and $n$, thus for all $x \in \bbT^d$, 
\begin{align*} 
p_{n^{-1/2},h_1}(x) - p(x) 
 &= p(x)\left[c_{h_1}(n^{-1/2})\zeta\big(h_1n^{-1/2}\tilde\Phi(x)\big)-1\right] \\
 &= p(x)\left[\zeta\big(h_1n^{-1/2}\tilde\Phi(x)\big)-1 \right] +  O\left(\frac{\|p\|_{\infty} \|\zeta\|_\infty h_1^2}{n}\right) \\
 &= p(x) h_1n^{-1/2}\tilde\Phi(x) +  O\left(h_1^2/n\right),
\end{align*}
where we again used the fact that $\zeta(0) = \zeta'(0)=1$. 
Similarly, for all $y \in \bbT^d$, 
$$q_{m^{-1/2},h_2}(y) - q(y) = q(y)h_2 m^{-1/2}\tilde\Psi(y) + O\left(h_2^2/m\right),$$
implying that, 
\begin{align*}
\kappa_n(h) - \kappa_n(0) 
&= h_1 n^{-1/2} \int \phi_0 \tilde\Phi dP + h_2m^{-1/2} \int \psi_0 \tilde\Psi dQ  
+ O(h_1^2/n+h_2^2/m)\\
&= h_1 n^{-1/2} \Var_P[\phi_0(X)] + h_2m^{-1/2} \Var_Q[\psi_0(Y)]+ O(h_1^2/n+h_2^2/m).
\end{align*}
We deduce that  
$$\sqrt{\frac{nm}{m+m}}  (\kappa_n(h) - \kappa_n(0) )
 \longrightarrow 
   \dot\kappa(h):= \langle (\sqrt{1- \rho}, \sqrt{\rho}), (h_1,h_2)\rangle_H,$$
   as $n,m\to\infty$ such that $n/(n+m) \to \rho$. 
It follows that the sequence of parameters $(\kappa_n(h):h\in H)$ is regular. 
Furthermore, the adjoint of $\dot\kappa$ is easily seen to be
$\dot\kappa^*(b^*) = b^*(\sqrt{1-\rho},\sqrt{\rho})$, for all $b^* \in \bbR$,
and one has
$$\norm{\dot\kappa^*(b^*)}_H^2 = b^* \Big((1-\rho) \Var_P[\phi_0(X)] + \rho \Var_Q[\psi_0(Y)]\Big).$$
The claim now follows from Lemma~\ref{lem:abstract_minimax}.\qed

\section{Alternate Proofs of Central Limit Theorems}
\label{app:pf_clts_alternate}
In this Section, we provide an alternate proof of Theorem~\ref{thm:clt}
which does not rely on our stability bounds in Theorem~\ref{thm:stability}
and~Proposition~\ref{thm:two_sample_stability}. We instead   follow the strategy 
developed by~\cite{delbarrio2019a} for obtaining limit
laws of the   process $\sqrt n(W_2^2(P_n, Q)  -W_2^2(P,Q))$.
For the sake of brevity, we only prove the one-sample case of Theorem~\ref{thm:clt--hypercube}, and the remaining assertions
of Theorem~\ref{thm:clt} can be handled similarly.  
Throughout this section, we abbreviate $\Psi = \Psi\pbc$ and $\hat P_n = \hat P_n\sbc$.

We shall make use of the classical Efron-Stein inequality (see for instance~\cite{boucheron2013}, Theorem~3.1) for bounding
the variance of functions of independent
random variables, stated as follows.
\begin{lemma} [Efron-Stein Inequality]
\label{lem:efron-stein}
Let $X_1, X_1', X_2, X_2', \dots, X_n, X_n'$ be independent random variables, and let 
$R_n = f(X_1, \dots, X_n)$ be a square-integrable function of $X_1, \dots, X_n$.
Let 
$$R_{ni}' = f(X_1, \dots, X_{i-1}, X_i', X_{i+1}, \dots, X_n),\quad i=1, \dots, n.$$ Then, 
$$\Var[R_n] \leq \sum_{i=1}^n  \bbE(R_n-R_{ni}')_+^2.$$
\end{lemma}
 
With these results in place, we turn to proving the one-sample case of Theorem~\ref{thm:clt--hypercube}.
In view of Lemma~\ref{lem:wavelet_wasserstein}, it suffices to assume $P \neq Q$, in which case $\Var[\phi_0(X)] > 0$. 
We abbreviate $\hat P_n = \hat P_n\sbc$, and we begin with the following result. 
\begin{proposition}
\label{prop:wasserstein_linearization}
Assume the same conditions as~Theorem~\ref{thm:clt--hypercube}. Define 
\begin{align*}
R_n    &= W_2^2(\hat P_n, Q) - \int \phi_0 d\hat P_n. 
\end{align*}
Then, as $n\to \infty$, $n\Var(R_n) \to 0$.
\end{proposition} 

\subsection{Proof of Proposition~\ref{prop:wasserstein_linearization}}
\label{app:pf_wasserstein_linearization}
Let $X_1' \sim P$ denote a random variable independent of $X_1, \dots, X_n$, and let
$$P_n' = \frac 1 n \delta_{X_1'} + \frac 1 n \sum_{i=2}^n \delta_{X_i}$$
denote the corresponding 
empirical measure. Let $\hat P_n'$ be the distribution with density
$$\hat p_n' = \sum_{\zeta\in\Phi}\hbeta_\zeta'\zeta + 
\sum_{j=j_0}^{J_n}\sum_{\xi\in\Psi_j} \hbeta_\xi'\xi 
 = \sum_{j=j_0-1}^{J_n}\sum_{\xi\in\Psi_j} \hbeta_\xi'\xi, \quad 
\text{where } \hbeta_\xi' = \int \xi d\hat P_n', \ \xi \in\Psi,$$
where we write $\Psi_{j_0-1} = \Phi$ for ease of notation. Set
$$R_n' = W_2^2(\hat P_n', Q) - \int \phi_0d\hat P_n'.$$
By  Lemma~\ref{lem:efron-stein}, 
it will suffice to prove that $n^2 \bbE(R_n-R_n')_+^2 =o(1)$.  
Let $(\hat\phi_n,\hat\psi_n)$ be a pair of Kantorovich potentials between $\hat P_n$ and $Q$. Without
loss of generality, we may assume that $\int \hat\phi_n d\calL = \int \phi_0 d\calL$ for all
$n \geq 1$. By the Kantorovich duality, we have
\begin{align*}
W_2^2(\hat P_n,  Q) &= \int \hat \phi_n d\hat P_n+ \int \hat \psi_n d Q, \\
W_2^2(\hat P_n', Q) &= 
 \sup_{(\phi,\psi) \in\calK} \int \phi d \hat P_n' + \int \psi d Q \\ &\geq 
 \int \hat \phi_n d \hat P_n' + \int \hat \psi_nd Q = 
 W_2^2(\hat P_n, Q) + \int \hphi_n d(\hat P_n' - \hat P_n).
\end{align*}
It follows that,  on the event $E_n$, 
\begin{align*}
R_n-R_n' \leq \int (\hat \phi_n - \phi_0)d(\hat P_n - \hat P_n').
\end{align*}
In view of Lemma~\ref{lem:efron-stein}, the claim will
follow if we are able to show that $n^2 \bbE(R_n-R_n')_+ = o(1)$. 
Arguing similarly as in the proof of, for instance, Lemma~\ref{lem:L},  
it holds that $\bbP(\hat p_n= \tilde p_n) \lesssim n^{-3}$.
Using this fact and the above inequality,  it will suffice to prove that the quantity
$$\Delta_n := n^2 \bbE\left(\int (\hat \phi_n - \phi_0)(\tilde p_n - \tilde p_n')d\calL\right)_+^2$$
vanishes as $n\to\infty$. 
To this end, notice that
\begin{align*}
\int (\hat \phi_n - \phi_0)(\tilde p_n - \tilde p_n')d\calL
 &=  \int (\hat \phi_n - \phi_0)\left( 
		\sum_{j=j_0-1}^{J_n}\sum_{\xi\in\Psi_j} (\hbeta_\xi- \hbeta_\xi')\xi\right) \\
 &= \frac 1 n \sum_{j=j_0-1}^{J_n}\sum_{\xi\in\Psi_j}(\xi(X_1)-\xi(X_1'))  \int (\hat \phi_n - \phi_0)
		 \xi.
\end{align*} 
Using the locality of the wavelet basis (Lemma~\ref{lem:wavelet--locality}) and the Cauchy-Schwarz inequality, we obtain 
\begin{align*}
\Delta_n
 &\lesssim   J_n \sum_{j=j_0-1}^{J_n}\sum_{\xi\in\Psi_j} \bbE[\xi^2(X)] \int \big\|\hat \phi_n - \phi_0\big\|^2|\xi|^2  
 \lesssim   {J_n}  \sum_{j=j_0-1}^{J_n} \sum_{\xi\in\Psi_j}
 			\int \big\|\hat \phi_n - \phi_0 \big\|^2|\xi|^2.
\end{align*} 			
In the final step, we again used Lemma~\ref{lem:wavelet--locality}  together with the fact that  $p$ is bounded over $[0,1]^d$ 
(since $p \in \calC^{\alpha-1}([0,1]^d)$), implying that
$$\bbE[\xi^2(X)] = \int \xi^2(x)p(x)dx \lesssim \int\xi^2(x)dx =1.$$
By Lemma~\ref{lem:wavelet},  for all $\xi \in \Psi_j$ and $j \geq j_0$,
we have $\supp(\xi) \subseteq I_\xi$ for a rectangle $I_\xi \subseteq [0,1]^d$
satisfying $\diam(I_\xi) \lesssim 2^{-j}$, and 
$\|\xi\|_{L^\infty(I_\xi)} \lesssim 2^{dj/2}$. Thus, 
\begin{align*}
n^2&\bbE(R_n-R_n')_+^2 
 \lesssim   {J_n}  \sum_{j=j_0-1}^{J_n}  2^{dj} \int_{I_\xi} \big\|\hat \phi_n - \phi_0 \big\|^2.
\end{align*} 	
Apply the Poincar\'e inequality in Lemma~\ref{lem:poincare} together with the bound $\diam(I_\xi) \lesssim 2^{-j}$ to deduce
\begin{align*}
n^2 \bbE(R_n-R_n')_+^2 
 {\lesssim}   {J_n}  \sum_{j=j_0-1}^{J_n}  2^{(d-2)j} \int_{I_\xi} \big\|\nabla(\hat \phi_n - \phi_0 )\big\|^2 
 {\lesssim}   {J_n}  \sum_{j=j_0-1}^{J_n}  2^{(d-2)j} \big\|\hat T_n - T_0\big\|_{L^2(P)}^2,
\end{align*} 	
where the final inequality holds due to the assumption that $p$ has a positive density over
$[0,1]^d$, which, due to the continuity of $p $, implies that there is a constant $\gamma^{-1} > 0$
such that $p \geq \gamma^{-1}$ over $[0,1]^d$. 
Apply Theorem~\ref{thm:one_sample_wavelet} to deduce that 
\begin{align*} 
n^2 \bbE(R_n-R_n')_+^2 
 \lesssim   {J_n}  \sum_{j=j_0-1}^{J_n}  2^{(d-2)j} \big\|\hat T_n - T_0\big\|_{L^2(P)}^2 
 \lesssim J_n \left(2^{J_n(d-2-2\alpha)} \vee \frac{(\log n)^2}{n}\right).  
\end{align*}
Since $d < 2(\alpha+1)$ and $J_n \asymp \log n$, the above display is of order $o(1)$,
thus the claim follows from Lemma~\ref{lem:efron-stein}.\qed

To prove the claim from here, write 
\begin{align*}
\sqrt n
 &\Big( W_2^2(\hat P_n, Q) - \bbE W_2^2(\hat P_n ,Q)\Big)   
 = \sqrt n \int \phi_0   (\hat p_n   - p_{J_n})
 + \sqrt n\Big(R_n - \bbE[R_n]\Big),
 \end{align*}
 where recall that 
 $p_{J_n} = \bbE[\hat p_n]$. 
It follows from Proposition~\ref{prop:wasserstein_linearization} that the final term of the above display
 converges to zero in probability. Furthermore, $\sqrt n \int \phi_0 (\hat p_n - p_{J_n}) \rightsquigarrow N(0,\Var[\phi_0(X)])$
 by Lemma~\ref{lem:linear_functional_clt}. Combining these facts with the bias bound of Theorem~\ref{thm:one_sample_wavelet}, 
 the claim follows. \qed

\end{appendix}

\begin{acks}[Acknowledgments] 
The authors would like to thank Alden Green for bringing their attention
to the paper~\cite{hendriks1990}, 
and Ziv Goldfeld and Kengo Kato for a discussion related
to the results of Section~\ref{sec:efficiency}. The authors are grateful for the constructive 
comments of the Editor, Associate Editor, and four anonymous reviewers, which significantly improved the quality of this manuscript.
TM also wishes to thank Aram-Alexandre Pooladian for conversations related
to this work, and for his comments on an earlier version of this manuscript.
TM was supported in part by the Natural Sciences and Engineering Research Council of Canada,
through a PGS D scholarship.
TM, SB and LW were supported
in part by National Science Foundation grants
DMS-1713003 and DMS-2113684. SB was additionally supported by a Google Research Scholar Award and an Amazon Research Award.
JNW gratefully acknowledges the support of National Science Foundation grant DMS-2015291.
\end{acks}


\bibliographystyle{imsart-nameyear} 
\bibliography{manuscript_full}       
 
\end{document}